\documentclass[12pt]{article}
\usepackage{amsmath}
\usepackage{latexsym}
\usepackage{amssymb}
\newtheorem{thm}{Theorem}[section]
\newtheorem{la}[thm]{Lemma}
\newtheorem{Defn}[thm]{Definition}
\newtheorem{Remark}[thm]{Remark}
\newtheorem{Note}[thm]{Note}
\newtheorem{prop}[thm]{Proposition}

\newtheorem{cor}[thm]{Corollary}
\newtheorem{Example}[thm]{Example}
\newtheorem{Examples}[thm]{Examples}
\newtheorem{Problems}[thm]{Problems}

\newtheorem{Problem}[thm]{Problem}
\newtheorem{Convention}[thm]{Convention}
\newtheorem{Number}[thm]{\!\!}
\newenvironment{defn}{\begin{Defn}\rm}{\end{Defn}}

\newenvironment{rem}{\begin{Remark}\rm}{\end{Remark}}
\newenvironment{numba}{\begin{Number}\rm}{\end{Number}}
\newenvironment{proof}{{\noindent\bf Proof.}}%
                  {\nopagebreak\hspace*{\fill}$\Box$\medskip\medskip\par} \newcommand{\Punkt}{\nopagebreak\hspace*{\fill}$\Box$}
\newcommand{\wb}{\overline}
\newcommand{\ve}{\varepsilon}
\newcommand{\at}{\symbol{'100}}

\newcommand{\wt}{\widetilde}
\newcommand{\tensor}{\otimes}
\newcommand{\impl}{\Rightarrow}

\newcommand{\mto}{\mapsto}

\newcommand{\isom}{\cong}
\DeclareMathOperator{\Ad}{Ad}
\newcommand{\N}{{\mathbb N}}
\newcommand{\R}{{\mathbb R}}

\newcommand{\cG}{{\mathcal G}}
\newcommand{\cE}{{\mathcal E}}
\newcommand{\K}{{\mathbb K}}

\newcommand{\Q}{{\mathbb Q}}
\newcommand{\bS}{{\mathbb S}}
\newcommand{\Z}{{\mathbb Z}}
\newcommand{\C}{{\mathbb C}}

\newcommand{\cU}{{\cal U}}
\newcommand{\cO}{{\cal O}}
\newcommand{\cV}{{\cal V}}
\newcommand{\cW}{{\cal W}}
\newcommand{\cX}{{\cal X}}
\newcommand{\cC}{{\cal C}}
\newcommand{\cQ}{{\cal Q}}
\newcommand{\cP}{{\cal P}}
\newcommand{\cg}{{\mathfrak g}}
\newcommand{\ch}{{\mathfrak h}}
\newcommand{\cn}{{\mathfrak n}}
\newcommand{\ck}{{\mathfrak k}}

\newcommand{\cq}{{\mathfrak q}}

\DeclareMathOperator{\Diff}{Diff}
\newcommand{\dl}{{\displaystyle \lim_{\longrightarrow}}}
\newcommand{\pl}{{\displaystyle \lim_{\longleftarrow}}}

\DeclareMathOperator{\Aut}{Aut}
\newcommand{\wh}{\widehat}
\newcommand{\one}{{\bf 1}}
\newcommand{\sub}{\subseteq}
\DeclareMathOperator{\GL}{GL}

\DeclareMathOperator{\im}{im}

\DeclareMathOperator{\pr}{pr}

\DeclareMathOperator{\id}{id}

\newcommand{\cB}{{\cal B}}

\newcommand{\cD}{{\cal D}}
\newcommand{\cF}{{\cal F}}

\newcommand{\cR}{{\cal R}}
\newcommand{\cL}{{\cal L}}
\newcommand{\cT}{{\cal T}}
\newcommand{\esssup}{\mbox{ess\,sup}}

\DeclareMathOperator{\Spann}{span}
\DeclareMathOperator{\Lip}{Lip}

\DeclareMathOperator{\conv}{conv}

\DeclareMathOperator{\ev}{ev}

\newcommand{\sbull}{{\scriptscriptstyle \bullet}}

\DeclareMathOperator{\op}{op}

\DeclareMathOperator{\Germ}{Germ}

\DeclareMathOperator{\evol}{evol}
\DeclareMathOperator{\Evol}{Evol}
\DeclareMathOperator{\Gau}{Gau}

\DeclareMathOperator{\graph}{graph}
\DeclareMathOperator{\Hol}{Hol}

\begin{document}
\begin{center}
{\Large\bf Measurable regularity properties of\\[2mm]
infinite-dimensional Lie groups}\\[8mm]
{\bf Helge Gl\"{o}ckner}\vspace{7mm}
\end{center}
\begin{abstract}\noindent
Let $G$ be a Banach-Lie group with Lie algebra
$\cg$, and $p\in [1,\infty]$.
Then the space $AC_{L^p}([0,1],\cg)$
of absolutely continuous functions
$\gamma\colon [0,1]\to \cg$ with $\gamma'\in L^p([0,1],\cg)$
is a Banach-Lie algebra.
Let $AC_{L^p}([0,1],G)_0=\langle \exp_G\circ\gamma\colon
\gamma\in AC_{L^p}([0,1],\cg)\rangle$ be the integral subgroup
of $C([0,1],G)$ with Lie algebra $AC_{L^p}([0,1],\cg)$.
We show that each $\gamma\in L^p([0,1],\cg)$
has a left evolution $\Evol(\gamma)\in AC_{L^p}([0,1],G)_0$,
and that the map $\Evol\colon L^p([0,1],\cg)\to AC_{L^p}([0,1],G)_0$
is smooth.
Similar results are obtained for important classes of
Fr\'{e}chet-Lie groups and more general Lie groups,
notably for diffeomorphism groups
of paracompact finite-dimensional smooth manifolds
and gauge groups of principal bundles
with Banach structure groups.
The measurable regularity properties considered
imply validity of the Trotter product formula
and the commutator formula.
\end{abstract}
{\footnotesize {\em Classification}:
22E65 (primary);
%
32A12,
%
34G10,
%
46G20,
%
46H05,
%
58B10.\\[1mm]
%
{\em Key words}: Infinite-dimensional Lie group, Banach-Lie group,
Fr\'{e}chet-Lie group,
regular Lie group, regularity, logarithmic derivative,
product integral, evolution, initial value problem,
parameter dependence, measurable map, current group, loop group,
gauge group, diffeomorphism group, projective limit, direct limit, inductive limit, direct
sum, weak direct product, extension, test function group,
locally m-convex algebra,
continuous inverse algebra,
commutator formula,
Trotter formula, Lebesgue space, regulated function, absolute continuity,
Carath\'{e}odory solution, measurable right-hand-side, control theory}\\[6mm]
\noindent
{\bf Introduction and statement of the main results}\\[3mm]
To enable proofs for fundamental Lie theoretic
facts in infinite dimensions,
John Milnor~\cite{Mil} introduced
the concept of regularity for infinite-dimensional
Lie groups (compare also \cite{OMO}, \cite{OM2} for related earlier
work).
Let $G$ be a Lie group modelled on a locally convex space~$E$,
with identity element~$e$, tangent bundle $T(G)$ and
Lie algebra $\cg:=L(G):=T_e(G)\isom E$.
Given $g,h\in G$ and $v\in T_h(G)$,
let $\lambda_g\colon G\to G$, $x\mto gx$ be left translation
and $g.v:=T_h(\lambda_g)(v)\in T_{gh}(G)$.
Thus $h^{-1}.v\in T_e(G)=\cg$.
If $\gamma\colon [0,1]\to\cg$ is a continuous map,
then there exists at most one
$C^1$-map $\eta\colon [0,1]\to G$
with
%
\begin{equation}\label{defformreg}
\eta'(t)=\eta(t).\gamma(t)\quad\mbox{for all $\,t\in [0,1]$, \,and\, $\eta(0)=1$.}
\end{equation}
If such an $\eta$ exists, it is called the \emph{evolution of $\gamma$},
and denoted by $\Evol(\gamma):=\eta$.
Let $k\in \N_0\cup\{\infty\}$.
The Lie group $G$ is called $C^k$-\emph{regular}
if each $\gamma\in C^k([0,1],\cg)$ admits an evolution $\Evol(\gamma)$
and the map $\Evol\colon C^k([0,1],\cg)\to C^{k+1}([0,1], G)$
is smooth
with respect to the natural Lie group structure on the mapping group
$C^{k+1}([0,1],G)$ (cf.\ \cite{MFO}, \cite{SEM}, and \cite{Nee}).
Then $C^k$-regularity implies $C^\ell$-regularity for all
$\ell\in \N_0\cup\{\infty\}$ with $\ell \geq k$ (cf.\ \cite{SEM}).
The $C^\infty$-regular Lie groups
are simply called \emph{regular}
(cf.\ \cite{SEM}, \cite{Mil}, and \cite{Nee},
where also applications of regularity are described).
For the purposes of representation theory,
the strongest notion, $C^0$-regularity,
is particularly useful~\cite{SAL}.
Recently, it also proved valuable to consider
weakened topologies on $C^0([0,1],\cg)$ (like
the $L^1$-topology) \cite{SEM}.\\[2.3mm]
Again for the purposes of representation theory,
it would be useful to have even stronger
regularity properties available.
And also from the point of view of control theory,
it is natural to allow more general functions $\gamma$
on the right hand side of~(\ref{defformreg}),
e.g.\ step functions (with steps when a control is switched
on or off).
The current article is devoted to
such generalizations.
We mention that initial value problems of the form (\ref{defformreg}),
for $G$ a finite-dimensional Lie group,
are a familiar topic in Geometric Control Theory.\footnote{In this context,
$\gamma$ in (\ref{defformreg})
is usually parametrized by certain control functions.}
See, e.g., \cite[Chapter 12]{Jur}
for optimal control problems
in this context,
when the controls are furnished by
left-invariant vector fields on~$G$.
For general aspects of
control theory on finite-dimensional
spaces with measurable right hand sides
(irrespective of Lie groups)
and the corresponding initial value problems,
see~\cite{Son}.\\[2.3mm]
If $E$ is a Fr\'{e}chet space (resp., a Banach space)
and $p\in [1,\infty]$,
let $L^p([0,1],E)$ be the space of all
(equivalence classes of) measurable functions
$\gamma\colon [0,1]\to E$ with separable image
such that $q\circ \gamma$
is in the Lebesgue space $\cL^p([0,1],\R)$
for each continuous seminorm~$q$ on~$E$
(cf.\ \cite{HaP} and \cite{HBK}
if $E$ is Banach).
We let $AC_{L^p}([0,1],E)$ be the space of all functions
$\eta\colon [0,1]\to E$ of the form
\[
\eta(t)=v+\int_0^s\gamma(s)\, d\lambda_1(s)
\]
with $v\in E$ and $[\gamma]\in L^p([0,1],E)$
(where $\lambda_1$ is Lebesgue-Borel measure on~$\R$).
Then $\eta'(t)=\gamma(t)$ for $\lambda_1$-almost all
$t\in [0,1]$ (see Lemma~\ref{funda}; cf.\ \cite{HBK} if $E$ is Banach).
The spaces $L^p([0,1],E)$ and $AC_{L^p}([0,1],E)$ are Fr\'{e}chet
spaces (resp., Banach spaces) in a natural way
(see, e.g., Lemma~\ref{labase1}).
Testing in local charts,
one can also speak of $AC_{L^p}$-maps to
a manifold $M$ modelled on a Fr\'{e}chet space (cf.\ Definition~\ref{ACintoM}).
We show that $AC_{L^p}([0,1],G)$
is a Fr\'{e}chet-Lie group
(resp., Banach-Lie group),
for each Fr\'{e}chet-Lie group (resp., Banach-Lie group)~$G$
(Proposition~\ref{propACLie}).
Let $G$ be a Fr\'{e}chet-Lie group,
with Lie algebra~$\cg$. Let $p\in [1,\infty]$.
We say that $G$ is \emph{$L^p$-regular}
if each $[\gamma]\in L^p([0,1],\cg)$ has a (necessarily
unique) left evolution
$\Evol([\gamma]):=\eta\in AC_{L^p}([0,1],G)$
such that
\[
\eta(0)=e\quad\mbox{and}\quad \eta'(t)=\eta(t)\cdot \gamma(t)\quad
\mbox{$\lambda_1$-almost everywhere,}
\]
and the map $\Evol\colon L^p([0,1],\cg)\to AC_{L^p}([0,1],G)$
is smooth. Let $L^\infty_{rc}([0,1],\cg)$
be the space of all elements of $L^\infty([0,1],\cg)$
having a representative $\gamma$ with relatively compact
image. Then $L^\infty_{rc}([0,1],\cg)$
is a closed vector subspace of $L^\infty([0,1],\cg)$.
A function $\gamma\colon [0,1]\to \cg$
is called \emph{regulated} if it is
a uniform limit of a sequence of $E$-valued
step functions.
Then the space $R([0,1],\cg)$
of equivalence classes of regulated
maps is a closed vector subspace of $L^\infty_{rc}([0,1],\cg)$
(and hence also of $L^\infty([0,1],\cg)$).
If $L^\infty([0,1],\cg)$ is replaced with
$L^\infty_{rc}([0,1], \cg)$ and
$R([0,1], \cg)$ in the above
definition, then
$G$ is called \emph{$L^\infty_{rc}$-regular}
and \emph{$R$-regular} (or \emph{regulated regular}), respectively.\\[4mm]
{\bf Theorem A.}
\emph{For each Fr\'{e}chet-Lie group~$G$ and each $p\in [1,\infty]$,
we have the following implications}:
\[
\mbox{\emph{$G$ is $L^p$-regular $\impl$ $G$ is $L^q$-regular for all
$q\in [1,\infty]$ such that $q\geq p$};}\vspace{-1mm}
\]
\[
\mbox{\emph{$G$ is $L^\infty$-regular $\impl$ $G$ is $L^\infty_{rc}$-regular
$\impl$ $G$ is $R$-regular};}
\]
\[
\mbox{\emph{$G$ is $R$-regular $\impl$ $G$ is $C^0$-regular}.}\vspace{2mm}
\]
We say that a locally convex space $E$ is \emph{integral
complete} if each continuous curve $\gamma\colon [0,1]\to E$
has a weak integral $\int_a^b\gamma(t)\, dt\in E$.
It is known that $E$ is integral complete
if and only if $E$ has the \emph{metric convex compactness
property} (metric CCP) in the sense that the closed
convex hull
\[
\wb{\conv(K)}\sub E
\]
is compact for each compact metrizable subset $K\sub E$
(see \cite{Wei}).
The following implications are known for a locally convex space
$E$:
\begin{eqnarray*}
\mbox{$E$ is complete}\; \impl\;\mbox{$E$ is quasi-complete} &\impl &
\mbox{$E$ is sequentially complete}\\
& \impl &  \mbox{$E$ is integral
complete;}
\end{eqnarray*}
moreover, none of these implications are equivalences~\cite{Voi}.
If we try to strengthen the concept of $C^0$-regularity
for a Lie group $G$ modelled on a locally convex space~$E$,
then $E$ has to be integral complete
(as $E\cong L(G)$ is integral complete
for each $C^0$-regular Lie group~$G$~\cite{SEM}).
We mention that Hausdorff locally convex spaces
$L^\infty_{rc}([0,1],E)$ can be defined for
$E$ an arbitrary locally convex space,
formed by equivalence classes
of all measurable functions $\gamma\colon [0,1]\to E$
such that the closure $\wb{\gamma([0,1])}$
of the image is compact and metrizable (see \cite{MEA}).
If $E$ is integral complete, then it is possible to define
spaces $AC_{L^\infty_{rc}}([0,1],E)$
(see Section~\ref{secAC}),
giving rise to Lie groups $AC_{L^\infty_{rc}}([0,1],G)$
and notions of $L^\infty_{rc}$-regularity
and $R$-regularity
for arbitrary Lie groups~$G$ modelled
on integral complete
locally convex spaces (see Sections~\ref{secACG}
and~\ref{secintroreg}).\\[2.3mm]
We say that a locally convex space $E$ has the
\emph{Fr\'{e}chet exhaustion property}
(FEP) if every closed vector subspace $S\sub E$
having a dense countable subset
can be written as the union
\[
S=\bigcup_{n\in \N}F_n
\]
of an ascending sequence
$F_1\sub F_2\sub\cdots$ of vector subspaces $F_n\sub E$
which are Fr\'{e}chet spaces in the induced topology
(see Definition~\ref{deffep}).
For every locally convex space $E$ with the (FEP),
we are able to give a sense to
$L^p([0,1],E)$ with $p\in [1,\infty]$ (see \ref{Lpfepvsp}).
The class of (FEP)-spaces
subsumes all Fr\'{e}chet spaces
and strict (LF)-spaces $E=\dl\,E_n$.\vspace{-1.3mm}
Moreover, the space $\cX_c(M)$ of compactly supported
smooth vector fields on a paracompact finite-dimensional
smooth manifold~$M$ is an (FEP)-space
(see Lemma~\ref{lafep} for these assertions).
The latter is the modelling space for the Lie group
\[
\Diff_c(M)
\]
of all diffeomorphisms $\phi\colon M\to M$ such that,
for some compact set $K\sub M$, we have
$\phi(x)=x$ for all $x\in M\setminus K$
(cf.\ \cite{Mic}, \cite{DIF}, or \cite{Shm}).\\[2.3mm]
Two main results are devoted to measurable
regularity properties of\linebreak
Banach-Lie groups
and diffeomorphism groups.\footnote{The $L^\infty_{rc}$-regularity of Banach-Lie
groups was first announced in~\cite{MEA}
(without proof, and using different terminology);
the $L^\infty_{rc}$-regularity of $\Diff_c(M)$ was conjectured there.}
For $M$ a paracompact finite-dimensional
smooth manifold and $K\sub M$ a compact set,
let $\Diff_K(M)$ be the group of all smooth diffeomorphisms
$\phi\colon M\to M $ such that $\phi(x)=x$
for all $x\in M\setminus K$. We obtain the following result:\\[4mm]
{\bf Theorem B.}
\emph{For each paracompact finite-dimensional
smooth manifold~$M$, the Lie group $\Diff_c(M)$
is $L^1$-regular.
Also $\Diff_K(M)$ is $L^1$-regular,
for each compact subset $K\sub M$.}\\[4mm]
We first prove Theorem~B for
the instructive cases of
$\Diff_K(\R^n)$, $\Diff_c(\R^n)$ and $\Diff(\bS_1)$,
before turning to general~$M$. We also prove:\\[4mm]
%
%
%
%
%
{\bf Theorem C.}
\emph{Every Banach-Lie group is
$L^1$-regular.}\\[4mm]
%
%
Using a projective limit argument,
we deduce that also some Fr\'{e}chet-Lie
groups are $L^1$-regular.
For example, the unit group $A^\times$
is $L^1$-regular for each continuous inverse
algebra~$A$ which is a Fr\'{e}chet space
and locally m-convex (in the sense of \cite{EAM}):
see Proposition~\ref{algsreg}. Recall that a continuous
inverse algebra is a unital associative real
(or complex) algebra, endowed with
a locally convex vector topology
for which the unit group $A^\times\sub A$ is open
and both the inversion map $A^\times \to A$, $a\mto a^{-1}$
and the algebra multiplication $A\times A\to A$ are continuous;
then $A^\times$ is a Lie group (see \cite{ALG}
and the references therein).
Similarly, we find that
the mapping group
\[
C^\infty_K(M,H):=\{\gamma\in C^\infty(M,H)\colon \gamma|_{M\setminus K}=e\}
\]
is $L^1$-regular
for each finite-dimensional smooth manifold~$M$,
Banach-Lie group~$H$ and compact set $K\sub M$ (Proposition~\ref{mapKbanreg}).\\[2.3mm]
We then turn to
measurable regularity
properties of ascending unions of Lie groups,
and related topics.
Notably, we find that the
weak direct product Lie group $\bigoplus_{j\in J}H_j$
(as introduced in \cite{MEA})
is $L^1$-regular for each family
of $L^1$-regular Lie groups~$H_j$
modelled on sequentially complete (FEP)-spaces
(Proposition~\ref{mrgsum}).
Together with a result on the inheritance of measurable
regularity properties by certain Lie subgroups
(Proposition~\ref{equarg}),
this entails:\\[3.5mm]
{\bf Theorem D.}
\emph{The test function group
$C^k_c(M,H)$ is
$L^1$-regular,
for each paracompact
finite-dimensional smooth manifold~$M$,
Banach-Lie group~$H$ and $k\in \N_0\cup\{\infty\}$.}\\[3.5mm]
Here $C^k_c(M,H)$ is the Lie group
of all $C^k$-maps $\gamma\colon M\to H$
such that $\gamma^{-1}(H\setminus\{e\})\sub M$
is relatively compact;
it is modelled on the locally convex direct limit
$C^k_c(M,L(H))=\dl\, C^k_K(M,L(H))$\vspace{-1mm}
(see, e.g., \cite{GCX} for the
Lie group structure in the main case that $M$ is $\sigma$-compact).
The conclusion of Theorem~D remains valid
for Lie groups of compactly supported
gauge transformations of principal bundles
with Banach structure groups (see\linebreak
Corollary~\ref{gaucreg}).\\[2.3mm]
We also have a result ensuring measurable regularity properties for ascending unions
of Banach-Lie groups under suitable hypotheses
(Proposition~\ref{unionbanreg}).
It entails:\\[3.5mm]
{\bf Theorem E.}
\emph{Let $G_1\sub G_2\sub \cdots$ be finite-dimensional Lie groups,
such that the inclusion maps $G_n\to G_{n+1}$
are smooth group homomorphisms for all $n\in \N$.
Then the direct limit Lie group $\dl\,G_n=\bigcup_{n\in \N}G_n$}\vspace{-1.3mm}
(as in \cite{DIR}, \cite{DI2})
\emph{is
$L^1$-regular.}\\[3mm]
In particular, $G=\dl\,G_n$\vspace{-1mm} is always $C^0$-regular.
So far, it was only known that
$G=\dl\,G_n$\vspace{-1.3mm} is $C^1$-regular~\cite{DI2}.\footnote{In
the special situation of \cite{Da3}
(which, in particular, entails that $\exp_G$
is a local diffeomorphism at $0$) $C^0$-regularity was already available.}\\[2.3mm]
Proposition~\ref{unionbanreg} also implies (see Corollary~\ref{anmupreg}):\\[4mm]
%
%
{\bf Theorem F.}
\emph{For each compact real analytic manifold~$M$
and each Banach-Lie group $H$,
the Lie group $C^\omega(M,H)$
of all real analytic $H$-valued maps on~$M$
is $L^\infty_{rc}$-regular.}\\[4mm]
The $C^0$-regularity of $C^\omega(M,H)$ is already stated in \cite{DGS}.\\[2.3mm]
All measurable regularity
properties we consider are extension properties:\\[4mm]
{\bf Theorem G.}
\emph{Consider an extension}
\[
\{1\}\to N\stackrel{j}{\to} G\to Q\stackrel{q}{\to}\{1\}
\]
\emph{of Lie groups modelled on integral complete
locally convex spaces, such that $q$ admits smooth
local sections.
If both $N$ and $Q$ are $L^\infty_{rc}$-regular}
(\emph{resp., $R$-regular}),
\emph{the also $G$
is $L^\infty_{rc}$-regular}
(\emph{resp., $R$-regular}).
\emph{If $p\in [0,\infty]$ and both $N$ and $Q$ are $L^p$-regular
Fr\'{e}chet-Lie groups, then also
$G$ is an $L^p$-regular Fr\'{e}chet--Lie group.
If $N$ and $Q$ are $L^p$-regular Lie groups modelled
on sequentially complete} (FEP)-\emph{spaces,
then also $G$ is modelled on a sequentially complete} (FEP)-\emph{space
and $L^p$-regular.}\\[4mm]
Even the weakest measurable regularity established here,
$R$-regularity, has remarkable consequences.
Let $G$ be a Lie group modelled
on a locally convex space
such that $G$
has a smooth exponential function\footnote{This
ensures that $\R\to G$, $t\mto \exp_G(tv)$
is a smooth one-parameter group of $G$
for each $v\in \cg$
with $\frac{d}{dt}\big|_{t=0}\exp_G(tv)=v$,
and that every smooth one-parameter group
of $G$ is of this form.}
$\exp_G\colon \cg\to G$ on $\cg=L(G)$.
Following \cite{SAL},
$G$ is said to have the \emph{Trotter property}
if, for all $v,w\in \cg$,
\[
(\exp_G(tv/n)\exp_G(tw/n))^n
\]
converges to $\exp_G(t(v+w))$ as $n\to\infty$,
uniformly for $t$ in compact subsets
of~$\R$.
We say that $G$ has the \emph{strong Trotter property}
if even\footnote{This implies the Trotter property,
as we can take $\gamma(t):=\exp_G(tv)\exp_G(tw)$.}
\begin{equation}\label{evenuni}
(\gamma(t/n))^n\to \exp_G(t \gamma'(0))\quad
\mbox{as $n\to\infty$,}
\end{equation}
uniformly for $t$ in compact
subsets of $[0,\infty[$,
for each $C^1$-curve $\gamma \colon [0,1]\to G$
such that $\gamma(0)=e$.
If, for all $v,w\in \cg$,
\[
\big(\exp_G(\sqrt{t}/n)\exp_G(\sqrt{t}/n)
\exp_G(-\sqrt{t}/n)\exp_G(-\sqrt{t}/n)\big)^{n^2}
\to \exp_G(t[v,w])
\]
as $n\to\infty$,
uniformly in $t$ in compact subsets of $[0,\infty[$,
then we say that $G$ has the \emph{commutator property}.
We say that $G$ has the \emph{strong commutator property}
if
\[
\big(\gamma(\sqrt{t}/n)\eta(\sqrt{t}/n)
\gamma(\sqrt{t}/n)^{-1}\eta(\sqrt{t}/n)^{-1}\big)^{n^2}
\to \exp_G(t[\gamma'(0),\eta'(0)])\;\mbox{as $n\to\infty$,}
\]
uniformly for $t$ in compact
subsets of $[0,\infty[$,
for all $C^1$-curves $\gamma,\eta\colon [0,1]\to G$
such that $\gamma(0)=\eta(0)=e$.
Both the Trotter property and the
commutator property are useful in representation theory
(see \cite{SAL} and ongoing work by K.-H. Neeb).
We already explained that the strong Trotter property implies
the Trotter property.
Likewise, the strong commutator property
implies the commutator property.
We show:\\[4mm]
{\bf Theorem H.}
\emph{Let $G$ be a Lie group modelled on a locally convex space.
If $G$ has the strong Trotter property,
then $G$ has the strong commutator property.}\\[4mm]
{\bf Theorem I.}
\emph{Let $G$ be a Lie group modelled on an
integral complete locally convex space.
If $G$ is $R$-regular, then $G$
has the strong Trotter property $($and hence also
the strong commutator property$)$.}\\[4mm]
We mention that the notion of $R$-regularity
provides a link to the original notion
of regularity (called $\mu$-regularity
in~\cite{Nee})
in the works by Omori and collaborators
(see \cite{OMO} and \cite{OM2}),
which was based on the convergence
of certain ``product integrals.''
In the special case of
Fr\'{e}chet-Lie groups,
the assertion on the strong Trotter property
in Theorem~G for $R$-regular Lie groups
is a counterpart of
the corresponding result for
$\mu$-regular Fr\'{e}chet-Lie groups in
\cite[Lemma 1.1]{OM2}.\\[4mm]
Combining Theorems~D
and~I, we see
that the test function group $C^k_c(M,H)$ has the strong Trotter
property for each paracompact finite-dimensional
smooth manifold~$M$ and Banach-Lie group~$H$.
Generalizing the case of $C^\infty_c(M,H)$,
also the gauge group $\Gau(P)$
of a principal bundle $P\to M$ with structure group~$H$
(with $\Gau_c(P)$ as an open Lie subgroup)
is $L^1$-regular and hence has the strong Trotter property, for each
paracompact finite-dimensional smooth manifold~$M$
and Banach-Lie group~$H$.
%
%
Also $\Diff(M)$ (with $\Diff_c(M)$
as an open Lie subgroup) is $L^1$-regular and hence
has the strong Trotter property.
Now the full automorphism group $\Aut(P)$ is a Lie group
extension
\[
\{1\}\to \Gau(P)\to \Aut(P)\to\Diff(M)_P\to\{1\}
\]
for a suitable open subgroup $\Diff(M)_P$ of $\Diff(M)$
(cf.\ \cite{Jak} for the essential special case that
$M$ is $\sigma$-compact).
Since being
$L^1$-regular
is an extension
property, we deduce:\\[4mm]
{\bf Theorem J.}
\emph{Let $P\to M$ be a smooth principal bundle
over a paracompact finite-dimensional
smooth manifold~$M$ whose structure group is a Banach-Lie group.
Then $\Aut(P)$ is $L^1$-regular
and hence $\Aut(P)$ has the strong Trotter property and the strong
commutator property.}\\[4mm]
Our proof of Theorem~I shows that the
convergence in (\ref{evenuni})
is even uniform for $\gamma$
in compact sets. This implies:\\[4mm]
{\bf Theorem K.}
\emph{If a Lie group $H$ is $R$-regular,
then $C_K(X,H)$ and $C_c(X,H)$
have the strong Trotter property,
for every locally compact topological
space~$X$ and compact subset $K\sub X$.
If, moreover, $X$ is paracompact, then
also $C_c(X,H)$ has the strong Trotter property}.\\[4mm]
Here $C_K(X,H):=\{\gamma\in C(X,H)\colon \gamma|_{X\setminus K}=e\}$
is endowed with its natural Lie group structure
(see, e.g., \cite{GCX});
likewise for $C_c(X,H)=\bigcup_K C_K(X,H)$
(cf.\ \cite{GCX} for the essential case
when $X$ is $\sigma$-compact).\\[2.3mm]
%
%
%
Note that
the Lebesgue spaces $L^p([0,1],E)$,
$L^\infty_{rc}([0,1],E)$ and $R([0,1],E)$
we consider only serve as a tool to
define strengthened regularity properties,
where they appear as the domains of certain evolution maps.
For this purpose,
properties like completeness
of the spaces (which might fail
unless we assume that $E$ is a Fr\'{e}chet space)
are irrelevant.
Rather, it is important that we have good
results on continuity and differentiability
properties for mappings between such spaces
or families of such.
Results of this type do not seem available
if, instead, one would define vector-valued
$L^p$-spaces as
completions of tensor products $L^p[0,1]\tensor E$
with respect to suitable tensor norms
(which might look more natural
from the point of view of linear functional
analysis).
Compare~\cite{FMP} for another viable type of vector-valued $L^p$-maps
based on Suslin-measurability.\\[3mm]
For previous work concerning
differential equations in finite-dimensional
(or Banach) spaces with measurable
right hand sides, see e.g.\ \cite{KaS},
\cite{HBK}, \cite{MaS}, \cite{Son}
and the references therein.\\[3mm]
{\bf Structure of the article.}
After a preparatory section
with selected material on
Lebesgue spaces
and infinite-dimensional
calculus (Section~\ref{prels}),
we study differentiability
properties of mappings like
\[
\wt{f}\colon
C([a,b],V)\times L^p([a,b],E_2)\to L^p([a,b],F),\quad
(\eta,[\gamma])\mto [f\circ (\eta,\gamma)],
\]
e.g.\ if $E_2$ and $F$ are Fr\'{e}chet spaces,
$V$ is an open subset of a locally convex space~$E_1$
and
\[
f\colon V\times E_2\to F
\]
a smooth map which is linear in its second argument
(Section~\ref{mpbelbg}).
If $E$ is a Fr\'{e}chet space,
then we can consider each of the spaces
\[
L^p([a,b], E),\quad L^\infty_{rc}([a,b],E)\quad\mbox{and}\quad
R([a,b],E)
\]
as a vector subspace $\cE([a,b],E)$ of $L^1([a,b],E)$.
The assignment is functorial
both in $E$ and in $[a,b]$;
e.g., we have a continuous linear map
\[
L^p([a,b],\lambda)\colon L^p([a,b],E_1)\to L^p([a,b], E_2),\quad
[\gamma]\mto [\lambda\circ \gamma]
\]
for each continuous linear map $\lambda\colon E_1\to E_2$
between Fr\'{e}chet spaces.
And we can pull back functions along an affine-linear map
$f\colon [c,d]\to [a,b]$:
\[
L^p(f,E)\colon L^p([a,b],E)\to L^p([c,d],E),\quad [\gamma]\mto
[\gamma\circ f].
\]
We therefore speak of a \emph{bifunctor} $\cE$
on Fr\'{e}chet spaces.
Using such a bifunctor,
we call a function
\[
\eta\colon [a,b]\to E
\]
$\cE$-absolutely continuous (and write $\eta\in AC_\cE([a,b],E)$)
if
\[
\eta(t)=\eta(a)+\int_a^t \gamma(s)\,ds
\]
for some $[\gamma]\in \cE([a,b],E)$.
The theory of vector-valued
absolutely continuous functions is developed
in Section~\ref{secAC}.
Notably, we study differentiability properties
of non-linear
mappings on spaces of absolutely continuous
functions and
describe conditions
ensuring that absolutely continuous functions
to manifolds can be defined.
This enables a Lie group structure on
\[
AC_\cE([0,1],G)
\]
to be constructed
for suitable bifunctors $\cE$ and Lie groups~$G$
(Section~\ref{secACG}),
which are then used to define and study
$\cE$-regular Lie groups (Section~\ref{secintroreg}).
To this end, certain axioms and properties
need to be imposed on the bifunctors under consideration.
Thus, we shall encounter the Locality Axiom,
the Pushforward Axioms, the Subdivision Property,
and the requirement that smooth functions act smoothly
on $AC_\cE$.
We shall see that all of these axioms and requirements
are satisfied by $L^p$, $L^\infty_{rc}$
and~$R$.\footnote{Another less central axiom,
the Embedding Axiom,
is satisfied by $L^p$ as a bifunctor on Fr\'{e}chet spaces
and $L^\infty_{rc}$ as a bifunctor on integral complete
locally convex spaces
(but possibly not by~$R$ or by $L^p$ as a bifunctor
on other spaces).
Beyond Fr\'{e}chet spaces,
the axiom is not used for major results.}
Thus $\cE$-regularity
provides a common roof (and uniform proofs)
for the concepts of $L^p$-regular, $L^\infty_{rc}$-regular
and $R$-regular Lie groups.
As part of the general theory,
proofs for Theorems~A and~G
are obtained in Section~\ref{secintroreg}.
We then prove Theorem~C and further
results on Banach-Lie groups and local Banach-Lie groups
(Section~\ref{secbana}).
Next, we study measurable regularity
properties of projective limits of Lie groups
(Section~\ref{secPL})
and of ascending unions
of Lie groups (Section~\ref{secDL}),
including proofs for Theorems~D, E and~F.
In Section~\ref{secRS},
we prove the $L^1$-regularity of $\Diff_K(\R^n)$,
$\Diff(\bS_1)$, and $\Diff_c(\R^n)$,
before proving Theorem~B
(in full generality) in Section~\ref{secDiM}.
The proof uses some basic uniqueness results
for solutions to initial value problems
with measurable right hand sides,
provided in Section~\ref{L1lip}.
We then establish Theorems~H and~I
(Section~\ref{secregreg}).
The proofs for the preparatory
Section~\ref{prels}
have been relegated to an appendix
(Appendix~A),
as well as proofs of auxiliary
results on Lebesgue spaces
of projective limits needed
in Section~\ref{secPL}
(see Appendix~B)
and some calculations concerning diffeomorphism
groups (Appendix~C).\\[3mm]
\emph{Acknowledgement.}
The author thanks K.-H. Neeb,
who suggested the consideration
of measurable regularity properties
and mentioned potential relations to
control theory as well as the Trotter product
and commutator formulas.
\section{Preliminaries and notation}\label{prels}
In this section,
we fix our notation and terminology concerning
topology, infinite-dimensional calculus
and Lebesgue spaces of vector-valued measurable mappings.
Several basic facts will be stated for later
use. Many of these are easy to take on faith,
whence we relegate proofs to the appendix
(Appendix~\ref{appsecprel}).\\[2.3mm]
We write $\N=\{1,2,\cdots\}$ and $\N_0:=\N\cup\{0\}$.
If $f\colon X\to Y$ is a function, we write
$\graph(f):=\{(x,f(x))\colon x\in X\}$
for its graph.
All vector spaces encountered in the article
are real vector spaces, unless we explicitly say
they are complex vector spaces.
We use `locally convex space' as
a shorthand for `locally convex topological vector space.'
All topological spaces and locally convex
spaces occuring in the article
are assumed Hausdorff,
except for the $\cL^p$-spaces and $\cL^\infty_{rc}$-spaces
presently encountered,
which are merely a preliminary for the
definition of the Hausdorff
$L^p$-spaces (and $L^\infty_{rc}$-spaces)
we are really interested in.
If $(X,d)$ is a metric space, $x\in X$ and $r>0$, we write
\[
B^d_r(x):=\{y\in X\colon d(x,y)<r\}\quad\mbox{and}\quad
\wb{B}^d_r(x):=\{y\in X\colon d(x,y)\leq r\}
\]
for the open ball and closed ball, respectively.
If $q$ is a seminorm on a vector space~$E$,
we write
\[
B^q_r(x):=\{y\in E\colon q(y-x)<r\}\quad\mbox{and}\quad
\wb{B}^q_r(x):=\{y\in E\colon q(y-x)\leq r\}
\]
for $x\in E$, $r>0$.
If $E=\R^n$, we let $\|.\|_\infty$ be the maximum norm on $\R^n$
and abbreviate $B_r(x):=B^{\|.\|_\infty}_r(x)$
as well as $\wb{B}_r(x):=\wb{B}^{\|.\|_\infty}_r(x)$.
If $\lambda\colon E\to F$ is a continuous linear map
between normed spaces $(E,\|.\|_E)$
and $(F,\|.\|_F)$, we write $\|\lambda\|_{op}$ for its operator norm.
Some basic concepts and facts from
topology will be useful.
\begin{numba}\label{Wallace}
We recall
the Wallace Lemma \cite[Chapter~5, Theorem~12]{Kel}:\\[2.3mm]
\emph{Let $X_1$ and $X_1$ be topological spaces,
$K_1\sub X_1$ and $K_2\sub X_1$ be compact subsets
and $U\sub X_1\times X_2$ be an open set such that
$K_1\times K_2\sub U$.
Then there exist open subsets $U_1\sub X_1$ and $U_2\sub X_2$
such that $K_1\times K_2\sub U_!\times U_2\sub U$.}
\end{numba}
\begin{numba}\label{topPL}
Let $(J,\leq)$ be a directed set,
$((X_j)_{j\in J},(\phi_{i,j})_{i\leq j})$
be a projective system\footnote{Thus $\phi_{j,j}=\id_{X_j}$
for all $j\in J$ and $\phi_{i,j}\circ \phi_{j,k}=\phi_{i,k}$
for all $i,j,k\in J$ such that $i\leq j\leq k$.}
of topological spaces $X_j$
and continuous mappings $\phi_{i,j}\colon X_j\to X_i$
for $i,j\in J$ such that $i\leq j$.
Let $(X,(\phi_j)_{j\in J})$ be a projective limit
of the above system in the category of topological
spaces and continuous mappings.\footnote{That is, the
map $(\phi_j)_{j\in J} \colon
X\to \{(x_j)_{j\in J}\in \prod_{j\in J}X_j\colon\,(\forall i,j\in J)\;i\leq j\;\impl\;
x_i=\phi_{i,j}(x_j)\}$ is a homeomorphism.}
The $\phi_{i,j}$ will be called \emph{bonding maps},
and the $\phi_j$ will be called \emph{limit maps}.
Then the following holds:
\begin{itemize}
\item[(a)]
\emph{For each $x\in X$,
the sets $\phi_j^{-1}(U)$ form a basis of
open neighbourhoods of~$x$, for $j\in J$ and $U$ ranging through the
set of open neighbourhoods of $\phi_j(x)$ in~$X_j$.}
\item[(b)]
\emph{A subset $D\sub X$ is dense in~$X$ if and only
if $\phi_j(D)$ is dense in $\phi_j(X)$ for each $j\in J$.}
\end{itemize}
\end{numba}
\begin{numba}
By a \emph{topological embedding},
we mean a map $f\colon X\to Y$ between
topological spaces such that the co-restriction
$f|^{f(X)}\colon X\to f(X)$ is a homeomorphism
with respect to the topology induced by~$Y$ on the image
$f(X)=\im(f)$.
\end{numba}
\begin{numba}
A topology on a real vector space~$E$
is called a \emph{vector topology}
if it turns~$E$ into a topological vector space.
\end{numba}
\begin{numba}
Throughout, we are working in the setting of abstract set theory.
Thus, if $(X_1,\Sigma_1)$ and $(X_2,\Sigma_2)$
are measurable spaces (viz.\ $X_i$
is a set and $\Sigma_i$ a $\sigma$-algebra on~$X_i$),
we call a map $f\colon X_1\to X_2$ measurable
if $f^{-1}(A)\in \Sigma_1$ for all $A\in \Sigma_2$
(see, e.g., \cite[\S7]{Bau}).
We also say that $f\colon (X_1,\Sigma_1)\to (X_2,\Sigma_2)$
is measurable in the situation.
If $\cE$ is a set of subset of a set~$X$,
we write $\sigma(\cE)$ for the $\sigma$-algebra
on~$X$ generated by~$\cE$.
As usual, if $X$ is a topological space, we write $\cB(X):=\sigma(\cO)$ for
the $\sigma$-algebra of Borel sets of~$X$,
generated by the set $\cO$ of all open subsets of~$X$.
If $(X,\Sigma)$ is a measurable space and $Y$ a subset
of $X$, then the trace
\[
\Sigma|_Y:=\{A\cap Y\colon A\in \Sigma\}
\]
is a $\sigma$-algebra on~$Y$. If $Y\in \Sigma$, then
\[
\Sigma|_Y=\{A\in \Sigma\colon A\sub Y\}
\]
(see \cite[\S1]{Bau}).
If $(X_i,\Sigma_i)$
are measurable spaces, as usual we let
$\Sigma_1\tensor \Sigma_2$ be the product $\sigma$-algebra,
i.e., the $\sigma$-algebra
on $X_1\times X_2$ generated by
$\{A_1\times A_2\colon A_1\in \Sigma_1,\,A_2\in \Sigma_2\}$.
\end{numba}
Some simple basic facts will be used:
\begin{numba}\label{basicsmeas}
\begin{itemize}
\item[(a)]
If is $(X,\Sigma)$ is a measurable space
and $Y\sub X$,
then the inclusion map $j\colon (Y,\Sigma|_Y)\to (X,\Sigma)$
is measurable (as $j^{-1}(A)=A\cap Y\in \Sigma|_Y$ for each $A\in\Sigma$)
and hence $f|_Y=f\circ j\colon (Y,\Sigma|_Y)$
is measurable for each measurable map $f\colon (X,\Sigma)\to (X_2,\Sigma_2)$
to a measurable space.
\item[(b)]
Let $(X_i,\Sigma_i)$ be measurable spaces for $i\in \{1,2\}$
and $f\colon X_1\to X_2$ be a map such that $f(X_1)\sub Y$
for a subset $Y\sub X_2$.
Then $f\colon (X_1,\Sigma_1)\to (X_2,\Sigma_2)$
is measurable if and only if the co-restriction
$f|^Y\colon (X_1,\Sigma_1)\to(Y,\Sigma_2|_Y)$
is measurable.\footnote{If $f|^Y$ is measurable, then also
$f=j\circ f|^Y$ with the measurable incluison map $j\colon Y\to X_2$.
If $f$ is measurable, then each $B\in \Sigma_2|_Y$ is of the form
$B=Y\cap A$ with some $A\in \Sigma_2$
and hence $(f|^Y)^{-1}(B)=f^{-1}(B\cap f(X_1))=f^{-1}(A\cap f(X_1))=f^{-1}(A)\in \Sigma_1$,
showing that~$f|^Y$ is measurable.}
\item[(c)]
Let $X$ be a set and $\cE$ be a set of subsets
of~$X$. Then $\sigma(\cE)|_Y=\sigma(\{A\cap Y\colon A\in \cE\})$
for each subset $Y\sub X$.
In particular:
\item[(d)]
If $X$ is a topological space
and we endow a subset $Y\sub X$ with the
induced topology, then
$\cB(Y)=\cB(X)|_Y$.
Hence, if $Y\in \cB(X)$, then $\cB(Y)=\{A\in \cB(X)\colon A\sub Y\}$.
\item[(e)]
Let $(X, \Sigma)$ be a measurable space, $Y$ be a set and $\cE$
be a set of subsets of~$Y$.
Then a map $f\colon (X,\Sigma)\to (Y,\sigma(\cE))$ is measurable
if and only if $f^{-1}(A)\in \Sigma$ for each $A\in \cE$
\cite[Satz 7.2]{Bau}.
In particular, a map
\[
f=(f_1,f_2)\colon (X,\Sigma)\to (X_1\times X_2, \Sigma_1\tensor \Sigma_2)
\]
is measurable if and only if both $f_1\colon (X,\Sigma)\to (X_1,\Sigma_1)$
and $f_2\colon (X,\Sigma)\to (X_2,\Sigma_2)$
are measurable.
\item[(f)]
If $X_1, X_2$ are topological spaces
and $X_2$ has a countable basis for its topology, then
$\cB(X_1\times X_2)=\cB(X_1)\tensor\cB(X_2)$
(see, e.g., \cite[Lemma~2.7]{MEA}).
\end{itemize}
\end{numba}
For a classical discussion of vector-valued integrals
in Banach spaces, see \cite[Chapter III]{HaP}; see also
\cite{HBK} for further relevant aspects.
It is essential for our purposes to go beyond
the classical frame of Banach spaces and consider
$\cL^p$-functions (and related functions) also
with values in Fr\'{e}chet spaces
(and, later, in even more general locally convex spaces).
\begin{numba}\label{thespaces1}
Let $(X,\Sigma,\mu)$ be a measure space,
$E$ be a Fr\'{e}chet space, $p\in [1,\infty[$ and $P(E)$ be the set of all
continuous seminorms $q\colon E\to [0,\infty[$.
We let $\cL^p(X,\mu,E)$ be the set of all
measurable functions $\gamma\colon (X,\Sigma)\to (E,\cB(E))$
such that
\begin{itemize}
\item[(a)]
$\|\gamma\|_{\cL^p,q}:=\sqrt[p]{\int_X q(\gamma(x))^p\,d\mu(x)}
<\infty$ for each $q\in P(E)$; and
\item[(b)]
The image $\gamma(X)$ is separable
(i.e., it has a dense countable subset).
\end{itemize}
\end{numba}
\begin{numba}
For $E$ a Fr\'{e}chet space,
let $\cL^\infty(X,\mu, E)$
be the set of all measurable mappings
$\gamma\colon (X,\Sigma)\to (E,\cB(E))$
with separable, bounded image.
For $\gamma\in \cL^\infty(X,\mu,E)$
and $q$ a continuous seminorm on~$E$, we write
\begin{equation}\label{defLinftysem}
\|\gamma\|_{\cL^\infty,q}:=\|q\circ\gamma\|_{\cL^\infty}=\esssup_\mu (q\circ\gamma)\,.
\end{equation}
Let $\cL^\infty_{rc}(X,\mu,E)$
be the set of all measurable maps
$\gamma\colon (X,\Sigma)\to (E,\cB(E))$
with relatively compact image.
Linear combinations
of measurable $E$-valued maps with separable
image being measurable (cf.\ Lemma~\ref{Sepp}),
$\cL^q(X,\mu,E)$, $\cL^\infty(X,\mu,E)$ and
$\cL^\infty_{rc}(X,\mu,E)$
are vector subspaces of~$E^X$. As is clear,
%
\begin{equation}\label{inclu1}
\cL^\infty_{rc}(X,\mu,E)\sub \cL^\infty(X,\mu,E),
\end{equation}
with equality if and only if all bounded subsets of~$E$
are relatively compact, i.e., $E$ is semi-Montel.
This holds for example if $E$
is Schwartz, nuclear or finite-dimensional.\footnote{As every Fr\'{e}chet
space is barrelled, it is semi-Montel iff it is Montel.}
If $\mu$ is a finite measure, then
\begin{equation}\label{inclu2}
\cL^\infty(X,\mu,E)\sub \cL^p(X,\mu,E)\sub \cL^1(X,\mu,E)\;\;\mbox{for all $p\in [1,\infty]$}
\end{equation}
For $p\in [1,\infty]$,
we endow $\cL^p(X,\mu,E)$
with the (not necessarily Hausdorff) locally convex vector topology
defined by the seminorms $\|.\|_{\cL^p,q}$, for $q\in P(E)$.
We give
$\cL^\infty_{rc}(X,\mu,E)$
the topology induced by~$\cL^\infty(X,\mu,E)$.
Then the inclusion maps in (\ref{inclu1}) and (\ref{inclu2})
are continuous.
\end{numba}
\begin{numba}
If $E$ is an arbitrary locally convex space,
then $\cL^\infty_{rc}(X,\mu,E)$ is defined as the
set of all measurable functions $\gamma\colon X\to E$
such that the closure $\wb{\im(\gamma)}$
of the image of~$\gamma$ in~$E$ is compact
and metrizable~\cite{MEA}.
Also in this generality,
$\cL^\infty_{rc}(X,\mu,E)$
is a vector subspace of~$E^X$
(see \cite{MEA}).
We define seminorms $\|.\|_{\cL^\infty,q}$
on $\cL^\infty_{rc}(X,\mu,E)$
as in~(\ref{defLinftysem})
and use these to endow
$\cL^\infty_{rc}(X,\mu,E)$
with a vector topology.
\end{numba}
The following two lemmas
are useful tools for dealing with the mappings
$\gamma\in \cL^\infty_{rc}(X,\mu,E)$.
See, e.g., \cite[Lemma~2.1]{MEA}
for the first fact (which can be deduced from
\cite[Theorem 4.2.13]{Eng}):
\begin{la}\label{quotcpmet}
If $K$ is a metrizable compact topological space
and $f\colon K\to Y$
a continuous map to a Hausdorff topological
space~$Y$, then also the image $f(K)$
is compact and metrizable.\,\Punkt
\end{la}
\begin{la}\label{imagesepmet}
Let $E$ be a locally convex space
and $K\sub E$ a subset which is
compact and metrizable in the induced topology.
Let $E_K:=\Spann(K)\sub E$
be the vector subspace spanned by~$K$
and $\cO_K$ be the induced topology on~$E_K$.
Then $E_K$ can be given a separable metrizable
locally convex vector topology $\cO'$
such that $\cO'\sub \cO_K$.
\end{la}
As before, $(X,\Sigma,\mu)$ is a measure space.
\begin{numba}
If $Y$ is a topological space
and $\gamma\colon X\to Y$ a measurable
map, we write $[\gamma]$ for
the set of all measurable mappings
$\gamma_1\colon X\to Y$
such that $\gamma(x)=\gamma_1(x)$
for $\mu$-almost all $x\in X$.
\end{numba}
\begin{numba}
If $E$ is a Fr\'{e}chet space
and $p\in [1,\infty]$, we set
\[
L^p(X,\mu,E):=\{[\gamma]\colon \gamma\in \cL^p(X,\mu,E)\}.
\]
Let $N_p\sub \cL^p(X,\mu,E)$
be the vector subspaces of all $\gamma$ in
$\cL^p(X,\mu,E)$
such that $\gamma(x)=0$ for $\mu$-almost all $x\in X$.
Then the map
\begin{equation}\label{quo1}
\cL^p(X,\mu,E)/{N_p}\to L^p(X,\mu,E),\quad
\gamma+N_p\to [\gamma]
\end{equation}
is a bijection,
which we use to identify $L^p(X,\mu,E)$
with the quotient vector space $\cL^p(X,\mu,E)/{N_p}$. Similarly, for $E$ a locally convex space we let
$N_{rc}\sub \cL^\infty_{rc}(X,\mu,E)$ be the vector space
of all $\gamma\in \cL^\infty_{rc}(X,\mu,E)$ such that
$\gamma(x)=0$ for $\mu$-almost all $x\in X$.
We use the map
\[
\cL^\infty_{rc}(X,\mu,E)/{N_{rc}}
\to L^\infty_{rc}(X,\mu,E),\quad \gamma+N_{rc}\to [\gamma]
\]
to identify the quotient vector space
$\cL^\infty_{rc}(X,\mu,E)/{N_{rc}}$
with
$L^\infty_{rc}(X,\mu,E):=\{[\gamma]\colon \gamma\in \cL^\infty_{rc}(X,\mu,E)\}$.
\end{numba}
\begin{numba}
If $E$ is a Fr\'{e}chet space,
we obtain a seminorm $\|.\|_{L^p,q}$ on $L^p(X,\mu,E)$ via
$\| [\gamma] \|_{L^p,q}:=\|\gamma\|_{\cL^p,q}$,
for each continuous seminorm~$q$ on~$E$.
We give $L^p(X,\mu,E)$ the locally convex vector topology defined by the
seminorms $\|.\|_{L^p,q}$, for $q\in P(E)$.
Likewise, for $E$ a locally convex space, we make $L^\infty_{rc}(X,\mu,E)$
a locally convex space
using the seminorms $\|.\|_{L^\infty,q}$
defined via
$\| [\gamma] \|_{L^\infty,q}:=\|\gamma\|_{\cL^\infty,q}$.
The latter topologies coincide with the quotient topologies on the quotient
spaces
$L^p(X,\mu,E)=\cL^p(X,\mu,E)/{N_p}$ and $L^\infty_{rc}(X,\mu,E)
=\cL^\infty_{rc}(X,\mu,E)/{N_{rc}}$, respectively.
\end{numba}
\begin{numba}
If $E$ is a Fr\'{e}chet space, then
\[
L^\infty_{rc}(X,\mu,E)\sub L^\infty(X,\mu,E)
\]
and the above topology on $L^\infty_{rc}(X,\mu,E)$
coincides with the induced topology.
If $\mu$ is a finite measure
and $E$ a Fr\'{e}chet space,
then
\[
L^\infty(X,\mu,E)\sub L^p(X,\mu,E)
\]
for all $p\in [1,\infty[$,
and the inclusion map is linear and continuous
as
$\|[\gamma] \|_{L^p,q}\leq
\sqrt[p]{\mu(X)}\|[\gamma]\|_{L^\infty,q}$.
We also have
\[
L^{p_1}(X,\mu,E)\sub L^{p_2}(X,\mu,E)
\]
for all $p_1\geq p_2$ in $[1,\infty[$,
and the inclusion map is continuous linear.\footnote{In fact,
$q(\gamma(x))^{p_2}\leq \max\{1,r^{p_1}q(\gamma(x))^{p_1}\}$
and thus
$\|[\gamma]\|_{L^{p_2},q}
\leq (\|1+(q\circ \gamma)\|_{L^{p_1}})^{p_1/p_2}
\leq (\|1\|_{L^{p_1}}+\|q\circ \gamma\|_{L^{p_1}})^{p_1/p_2}
=(\sqrt[p_1]{\mu(X)}+\|q\circ \gamma\|_{L^{p_1}})^{p_1/p_2}$.
Set $C:= (\sqrt[p_1]{\mu(X)}+1)^{p_1/p_2}$.
By the preceding.
$\|[\gamma]\|_{L^{p_2},q}\leq C$
for all $[\gamma]\in L^{p_1}(X,\mu,E)$
such that
$\|[\gamma]\|_{L^{p_1},q}\leq 1$.
Hence
$\|[\gamma]\|_{L^{p_2},q}\leq C\|[\gamma]\|_{L^{p_1},q}$
for all $[\gamma]\in L^{p_1}(X,\mu,E)$.}
\end{numba}
\begin{numba}\label{defweakint}
We recall the concept of weak integral.
Let $(X,\Sigma,\mu)$ be a measure space, $E$ be a locally convex space,
$E'$ be the space of all continuous linear functionals $\lambda\colon E\to\R$
and $\gamma\colon X\to E$ be a function
such that $\lambda\circ \gamma\in \cL^1(X,\mu,\R)$
for each
$\lambda\in E'$. An element $w\in E$ is called
the \emph{weak integral of $\gamma$ with respect to~$\mu$}
(and denoted $\int_X\gamma(x)\,d\mu(x):=w$)
if
\[
(\forall \lambda\in E')\quad \lambda(w)=\int_X\lambda(\gamma(x))\, d\mu(x)\,.
\]
\end{numba}
\begin{numba}\label{basicweaki}
The following is clear, as $\lambda\circ \alpha\in E'$
for each $\lambda\in F'$:\\[2.5mm]
If $\gamma\colon X\to E$ has a weak integral
$\int_X\gamma(x)\,d\mu(x)$ in the situation of \ref{defweakint},
and $\alpha\colon E\to F$ is a continuous
linear map to a locally convex space~$F$,
then $\alpha\circ \gamma\colon X\to F$ has a weak integral in~$F$;
in fact,
\[
\int_X\alpha(\gamma(x))\,d\mu(x)=\alpha\left(
\int_X\gamma(x)\,d\mu(x)\right),
\]
as the right-hand side satisfies the property
which defines the weak integral on the left.
\end{numba}
\begin{numba}\label{pardep}
We shall use continuity of parameter-dependent integrals (see, e.g.,
\cite[Proposition~3.5]{Bil}, or \cite{GaN}):\\[2.3mm]
\emph{Let $E$ be a locally convex space, $X$ a topological
space, $a,b\in \R$ such that $a<b$
and $f\colon X \times [a,b]\to E$ be a continuous function such that
the weak integral}
\[
g(x):=\int_a^bf(x,t)\,dt
\]
\emph{exists in~$E$ for each $x\in X$. Then $g\colon X\to E$
is continuous.}
\end{numba}
The next lemma compiles essential basic properties
of the spaces from~\ref{thespaces1}.
\begin{la}\label{labase1}
Let $(X,\Sigma,\mu)$ be a measure space and $E$ be a Fr\'{e}chet space
$($resp, a Banach space$)$. Let $p\in [1,\infty]$. Then
$L^p(X,\mu,E)$ and $L^\infty_{rc}(X,\mu,E)$
are Fr\'{e}chet spaces $($resp., Banach spaces$)$.
Moreover,
each $\gamma\in \cL^1(X,\mu,E)$ admits
a weak integral $\int_X \gamma(x)\, d\mu(x)$,
and the map
\[
I\colon L^1(X,\mu,E)\to E\,,\quad [\gamma]\mto\int_X\gamma(x)\,d\mu(x)
\]
is well-defined and continuous linear, with $q(I(\gamma))\leq \|[\gamma]\|_{L^1,q}$
for each continuous seminorm~$q$ on~$E$.
\end{la}
\begin{numba}
We say that a locally convex space $E$ is \emph{integral
complete} if each continuous curve $\gamma\colon [0,1]\to E$
has a weak integral $\int_a^b\gamma(t)\, dt\in E$.
It is known that $E$ is integral complete
if and only if $E$ has the \emph{metric convex compactness
property} (metric CCP) in the sense that the closed
convex hull
\[
\wb{\conv(K)}\sub E
\]
is compact for each compact metrizable subset $K\sub E$
(see \cite{Wei}; cf.\ \cite{Voi}).
\end{numba}
\begin{la}\label{indeedmetr}
If $E$ is a locally convex space
and $K\sub E$ a metrizable compact subset
such that $\wb{\conv(K)}\sub E$
is compact, then
$\wb{\conv(K)}$
is metrizable.
\end{la}
\begin{numba}
If $\mu(X)<\infty$, we can define $\|[\gamma]\|_{L^1,q}:=\|\gamma\|_{\cL^1,q}\in [0,\infty[$
as in \ref{thespaces1}\,(a)
also for $E$ an arbitrary locally convex space
and $\gamma\in \cL^\infty_{rc}(X,\mu,E)$.
\end{numba}
\begin{la}\label{labaserc}
Let $(X,\Sigma,\mu)$ be a measure space
and $E$ be an integral complete
locally convex space.
If $\mu(X)<\infty$,
then each $\gamma\in \cL^\infty_{rc}(X,\mu,E)$ admits
a weak integral $\int_X \gamma(x)\, d\mu(x)$,
and the map
\[
I\colon L^\infty_{rc}(X,\mu,E)\to E\,,\quad
[\gamma]\mto\int_X\gamma(x)\,d\mu(x)
\]
is well-defined and continuous linear, with $q(I(\gamma))\leq \|[\gamma]\|_{L^1,q}
\leq \mu(X)\|[\gamma] \|_{L^\infty,q}$
for each continuous seminorm~$q$ on~$E$.
\end{la}
\begin{numba}
If $J\sub \R$ is an interval and $\mu$ the restriction
of the $1$-dimensional Lebesgue-Borel measure
$\lambda_1$ to $\cB(J)$, we omit mention of $\mu$
and simply write $\cL^p(J,E)$ instead of $\cL^p(J,\mu,E)$, etc.
If, moreover, $J=[a,b]$ with reals $a\leq b$,
we write $\int_a^b\gamma(t)\,dt$
in place of $\int_{[a,b]}\gamma(t)\, d\mu(t)$, for $\gamma\in \cL^1([a,b],E)$.
If $a>b$, we define
$\int_a^b\gamma(t)\,dt:=-\int_b^a\gamma(t)\,dt$ for $\gamma\in \cL^1([b,a],E)$.
Likewise for $\cL^\infty_{rc}(J,E)$
if~$E$ is an integral complete
locally convex space.
\end{numba}
\begin{numba}\label{noconcern}
Because the map $C(J,E)\to L^\infty(J,E)$, $\gamma\mto [\gamma]$
is injective, we can identify $\gamma\in C(J,E)$
with its equivalence class $[\gamma]$ in $L^\infty_{rc}(J,E)$,
for $E$ an arbitrary locally convex space
and $J\sub \R$ an interval.
Then
\[
\|\gamma\|_{\cL^\infty,q}=\sup_{t\in J}q(\gamma(t)),
\]
for each $\gamma\in C(J,E)$
and each continuous seminorm~$q$ on~$E$.
Likewise, we can identify $\gamma\in C(J,E)\cap \cL^p(J,E)$
with it coset $[\gamma]$ in $L^p(J,E)$,
for each Fr\'{e}chet space~$E$,
$\gamma\in C(J,E)$ and $p\in [1,\infty]$.
\end{numba}
\begin{rem}
In this section, we clearly
distinguish between a measurable
function $\gamma$ and its equivalence class
$[\gamma]$ under equality $\mu$-almost everywhere.
Later on, when the meaning is clear from
the context, we shall sometimes ignore the distinction.
\end{rem}
\begin{numba}
Let $J\sub \R$ be  non-degenerate interval, $E$ be a locally convex space,
$\eta\colon J\to E$ be a mapping and $t\in J$.
As usual, we say that $\eta$ is \emph{differentiable at~$t$}
if the limit $\eta'(t)=\lim_{s\to t}\frac{\eta(s)-\eta(t)}{s-t}$
exists in~$E$. Then $\eta$ is continuous at~$t$ in particular.\footnote{Indeed,
$\eta(s)=\eta(t)+(s-t)\frac{\eta(s)-\eta(t)}{s-t}\to\eta(t)+0\eta'(t)=\eta(t)$ as $s\to t$.}
\end{numba}
The following version of the Fundamental Theorem of Calculus is essential.
Compare \cite[25.16]{HBK} if $E$ is a Banach space.
\begin{la}\label{funda}
Let $J\sub \R$ be an interval, $E$ be a Fr\'{e}chet space and $t_0\in J$.
If $\gamma\in \cL^1(J,E)$, then the weak integrals needed to define
\[
\eta\colon J\to E\,,\quad t\mto \int_{t_0}^t\gamma(s)\, ds
\]
exist, and $\eta$ is a continuous function
which is differentiable $\lambda_1$-almost everywhere, with $\eta'=[\gamma]$.
\end{la}
For general locally convex spaces~$E$, we still have the following.
\begin{la}\label{funda2}
Let $J\sub \R$ be an interval, $E$ be an integral complete
locally convex space and $t_0\in J$.
If $\gamma\in \cL^\infty_{rc}(J,E)$, then the weak integrals needed to define
\[
\eta\colon J\to E\,,\quad t\mto \int_{t_0}^t\gamma(s)\, ds
\]
exist, and $\eta$ is a continuous function.
If also $\gamma_1\in \cL^\infty_{rc}(J,E)$
is a map such that $\eta(t)=\int_{t_0}^t\gamma_1(s)\, ds$
for all $t\in J$, then $[\gamma]=[\gamma_1]$.
\end{la}
\begin{numba}
In the situation of Lemmas~\ref{funda} and \ref{funda2},
respectively, we shall write $\eta':=[\gamma]$.
Since $[\gamma]$ is uniquely determined by $\eta$
in Lemma~\ref{funda2} (and $\gamma(t)=\eta'(t)$
$\lambda_1$-almost everywhere in Lemma~\ref{funda}),
we see that $\eta'\in L^\infty_{rc}(J,E)$ (resp.,
$\eta'\in L^1(J,E)$) is well-defined.
\end{numba}
If $E$ is a Fr\'{e}chet space,
then a function $\gamma \colon X\to E$
is in $\cL^\infty_{rc}(X,\mu,E)$
if and only if there exists a sequence
$(\gamma_n)_{n\in \N}$
of measurable
functions $\gamma_n\colon \colon X\to E$
with finite image such that $\gamma_n\to \gamma$
uniformly (see \cite[Corollary 3.19]{MEA}).\footnote{If $\gamma_n$ exist,
then even
$\gamma_n$ with $\im\,\gamma_n\sub \im\,\gamma$ exist,
as we may choose $v_1,\ldots, v_{m(n)}\in \im(\gamma)$
in the proof of \cite[Proposition 3.18]{MEA}.
We shall not use this fact.}
%
%
Passing to more restrictive
functions $\gamma_n$ (the step functions),
we arrive at the notion
of regulated functions.
\begin{numba}
Let $E$ be a locally convex space and $a,b\in \R$ with $a<b$.
Let $\cT([a,b],E)$
be the set of all functions
$\gamma\colon [a,b]\to E$
for which there exists a partition
\[
a=t_0<t_1<\cdots< t_{n-1}<t_n=b
\]
of $[a,b]$
such that $\gamma|_{]t_{j-1},t_j[}$ is constant
for all $j\in \{1,\ldots, n\}$.
We let
\[
\cR([a,b],E)\sub \cL^\infty_{rc}([a,b],E)
\]
be the space of functions $\gamma \colon [a,b]\to E$
which are the uniform limit of a sequence
$(\gamma_n)_{n\in \N}$ in $\cT([a,b],E)$.
Such functions are called \emph{regulated
functions} from $[a,b]$ to~$E$.
Then $\cR([a,b],E)$
is a vector subspace of $\cL^\infty_{rc}([a,b],E)$
and $R([a,b],E):=\{[\gamma]\colon \gamma\in \cR([a,b],E)\}$
is a vector subspace of $L^\infty_{rc}([a,b],E)$.
We endow both vector subspaces with the induced topology.
\end{numba}
We mention:
%
%
%
\begin{la}\label{stepswithin}
If $E$ is a locally convex space, $a<b$ in $\R$
and $\gamma\in \cR([a,b],E)$,
then the sequence $(\gamma_n)_{n\in \N}$ in $\cT([a,b],E)$
such that $\gamma_n\to \gamma$ uniformly can be chosen such that
$\gamma_n([a,b])\sub \gamma([a,b])$ for all
$n\in \N$.
\end{la}
\begin{numba}\label{regulatedfrech}
If $E$ is a Fr\'{e}chet space, then
$R([a,b],E)$ is the closure
of (equivalence classes of) $\cT([a,b],E)$ in $L^\infty_{rc}([a,b],E)$.
Thus $R([a,b],E)$
is a Fr\'{e}chet space.
\end{numba}
\begin{numba}\label{deflinpfwd}
Let $(X,\Sigma,\mu)$ be a measure space.
If $\lambda\colon E\to F$ is a continuous
linear map between Fr\'{e}chet spaces,
then
\[
L^p(X,\mu,\lambda)\colon L^p(X,\mu,E)\to L^p(X,\mu,F),\quad
[\gamma]\mto[\lambda\circ \gamma]
\]
(and the map
$\cL^p(X,\mu,\lambda)\colon \cL^p(X,\mu,E)\to \cL^p(X,\mu,F)$,
$\gamma\mto\lambda\circ \gamma$)
is continuous linear, for each $p\in [1,\infty]$.
In fact, if $q$ is a continuous seminorm
on $F$, then $q\circ\lambda$ is a continuous seminorm
on~$E$ and $\|\lambda\circ \gamma\|_{\cL^p,q}=\|\gamma\|_{L^p,q\circ\lambda}
\leq \|\gamma\|_{L^p,q\circ\lambda}$ for all $\gamma\in \cL^p(X,\mu,E)$,
entailing that the linear maps $\cL^p(X,\mu,\lambda)$
and $L^p(X,\mu,\lambda)$ are continuous.
Similarly,
the linear maps
\[
L^\infty_{rc}(X,\mu,\lambda)\colon L^\infty_{rc}(X,\mu,E)\to L^\infty_{rc}(X,\mu,F),\quad
[\gamma]\mto[\lambda\circ \gamma]
\]
and
$\cL^\infty_{rc}(X,\mu,\lambda)\colon \cL^\infty_{rc}(X,\mu,E)\to \cL^\infty_{rc}(X,\mu,F)$,
$\gamma\mto\lambda\circ \gamma$
are continuous linear for each linear map $\lambda\colon E\to F$
between locally convex spaces,
and are the maps
\[
R([a,b],\lambda)\colon R([a,b],E)\to R([a,b],F),\quad
[\gamma]\mto[\lambda\circ \gamma]
\]
and $\cR([a,b],E)\to \cR([a,b],F)$, $\gamma\mto \lambda\circ\gamma$.
\end{numba}
\begin{numba}\label{Lebspprod}
As a consequence of \ref{deflinpfwd},
we have that
\[
L^p(X,\mu,E_1\times E_2)\cong L^p(X,\mu,E_1)\times L^p(X,\mu, E_2)
\]
for all $p\in [1,\infty]$ and Fr\'{e}chet spaces
$E_1$ and $E_2$.
Similarly,
\[
L^\infty_{rc}(X,\mu,E_1\times E_2)\cong L^\infty_{rc}(X,\mu,E_1)\times
L^\infty_{rc}(X,\mu, E_2)
\]
for all locally convex spaces $E_1$ and $E_2$ and
$R([a,b],E_1\times E_2)\cong R([a,b],E_1)\times R([a,b], E_2)$.
\end{numba}
The following two lemmas will help us to deal with locally convex
direct sums and locally convex direct limits.
\begin{la}\label{psemiondsum}
Let $(E_n)_{n\in \N}$ be a sequence of locally convex spaces.
Fix $p\in [0,\infty]$.
Then \begin{equation}\label{formulasemn}
\bigoplus_{n\in\N}E_n\to[0,\infty[,\quad (x_n)_{n\in\N}\mto\|(q_n(x_n))_{n\in\N}\|_{\ell^p}
\end{equation}
is a continuous seminorm on the locally convex direct sum,
for each sequence $(q_n)_{n\in \N}$ of continuous seminorms $q_n\colon E_n\to[0,\infty[$.
The locally convex direct sum topology on $\bigoplus_{n\in\N}E_n$ is defined
by the set of seminorms of the form \emph{(\ref{formulasemn})}. \end{la}
More explicitly, the seminorms are given by
\[
\left(\sum_{n=1}^\infty(q_n(x_n))^p\right)^{1/p}
\]
if $p<\infty$. For $p=\infty$, they are given by
\[
\sup\{q_n(x_n)\colon n\in\N\}.
\]
\begin{la}\label{pbasisindl}
Let $E_1\sub E_2\sub\cdots$ be an ascending sequence of locally convex spaces
such that all inclusion maps $E_n\to E_{n+1}$ are continuous linear.
Endow $E:=\bigcup_{n\in \N}E_n$ with the $($not necessarily Hausdorff$)$
locally convex direct limit topology.
Fix $p\in [1,\infty]$.
Then the sets
\[
\left\{ \sum_{n=1}^\infty v_n\colon (v_n)_{n\in\N}\in \bigoplus_{n\in\N}E_n\;\mbox{such that}\;
\|(q_n(v_n))_{n\in\N}\|_{\ell^p}<1
\right\}
\]
%
form a basis of $0$-neighbourhoods for $E$,
if we let $(q_n)_{n\in \N}$ run through the set
of all sequences of continuous seminorms $q_n$ on $E_n$.
\end{la}
Also beyond metric spaces, let us say that a topological
space~$X$ is \emph{separable}
if it contains a countable dense subset.
\begin{defn}\label{deffep}
Let us say that a locally convex space $E$ has the
\emph{Fr\'{e}chet exhaustion property} (FEP)
if every separable closed vector subspace $S\sub E$
can be written as a union $S=\bigcup_{n\in \N}F_n$
of vector subspaces $F_1\sub F_2\sub\cdots$ of~$E$
which are Fr\'{e}chet spaces in the induced topology.
In this case, we also say that $E$ is a (FEP)-\emph{space}.
\end{defn}
For example, every Fr\'{e}chet space
has the Fr\'{e}chet exhaustion property
(as we can take $F_n:=S$ for all $n\in \N$).\\[2.3mm]
By the next lemma, the Fr\'{e}chet spaces $F_n$ in Definition~\ref{deffep}
are separable.
\begin{la}\label{metrinsep}
Let $E$ be a locally convex space
and $F\sub E$ be a vector subspace, endowed with the induced topology.
If $E$ is separable and $F$ is metrizable,
then also $F$ is separable.
\end{la}
\begin{numba}\label{Lpfepvsp}
If a locally convex space~$E$
has the Fr\'{e}chet exhaustion property
and $(X,\Sigma,\mu)$ is a measure space,
then the measurable functions $\gamma\colon X\to E$
with separable image and
\[
\|\gamma\|_{\cL^p,q}:=\|q\circ \gamma\|_{\cL^p}<\infty
\]
for all $q\in P(E)$ form a vector space
$\cL^p(X,\mu,E)$, giving rise to a Hausdorff
locally convex space $L^p(X,\mu,E)$
of equivalence classes
(as we prove in Appendix~\ref{appsecprel}).\footnote{If $E$ does
not have the (FEP),
then there is no reason why $\cL^p(X,\mu,E)$
should be closed under sums (as sums might not be measurable); it might not be a vector
space.} As in the Fr\'{e}chet case, we write
$\|[\gamma]\|_{L^p,q}:=\|\gamma\|_{\cL^p,q}$.
\end{numba}
\begin{la}\label{lafep}
\begin{itemize}
\item[\rm(a)]
Every locally convex direct
sum $E:=\bigoplus_{j\in J}E_j$
of sequentially complete \emph{(FEP)}-spaces\footnote{This condition is satisfied, e.g.,
if each $E_j$ is a Fr\'{e}chet space.}
is sequentially complete
and has the \emph{(FEP)}.
In this case,
\[
L^p(X, \mu,E)=\bigoplus_{j\in J}L^p(X,\mu,E_j)
\]
as a vector space,
for each $p\in [1,\infty]$ and measure space $(X,\Sigma,\mu)$.
Moreover,
\[
L^1(X,\mu,E)=\bigoplus_{j\in J}L^1(X,\mu,E_j)
\]
as a locally convex space.
%
%
%
\item[\rm(b)]
If a locally convex space $E$ has the \emph{(FEP)},
then every closed vector subspace $F\sub E$ has the \emph{(FEP)}.
\item[\rm(c)]
Every
strict $($LF$)$-space $E=\dl\,E_n$\vspace{-1.3mm}
has the \emph{(FEP)}.
In this case,
\[
L^p(X,\mu,E)=\bigcup_{n\in \N} L^p(X,\mu,E_n)
\]
as a vector space,
for each $p\in [1,\infty]$ and measure space $(X,\Sigma,\mu)$.
\item[\rm(d)]
Let $k\in \N_0\cup\{\infty\}$ and $\Gamma_c^{C^k}(V)$ be the space of compactly supported
$C^k$-sections in a
vector bundle $V$ over a paracompact
finite-dimensional smooth manifold~$M$,
whose typical fibre~$F$ is a Fr\'{e}chet space.\footnote{The topology
on this space is recalled in Appendix~\ref{appsecprel},
see \ref{deftosec}.}
Then $\Gamma_c^{C^k}(V)$ has the \emph{(FEP)},
and
\[
L^p(X,\mu,\Gamma_c^{C^k}(V))=\bigcup_{K} L^p(X,\mu,\Gamma_K^{C^k}(V))
\]
for each $p\in [1,\infty]$ and measure space $(X,\Sigma,\mu)$,
where $K$ ranges through the set of compact
subsets of~$M$ and
\[
\Gamma_K^{C^k}(V):=\{\sigma\in C^{C^k}_c(M)\colon
(\forall x\in M\setminus K)\;\sigma(x)=0\}.
\]
In particular:
\item[\rm(e)]
The space $\cX_c(M)=\Gamma_c(TM)$ of compactly
supported smooth vector fields on a paracompact finite-dimensional
smooth manifold~$M$ has the \emph{(FEP)}
$($which is the modelling space
of the Lie group $\Diff_c(M))$.
\end{itemize}
\end{la}
\begin{rem}
Lemma~\ref{lafep}\,(c) will be strengthened
further in Proposition~\ref{Mujileb}
(a certain analogue of Mujica's Theorem
on spaces of continuous functions
to locally convex direct limits).
\end{rem}
\begin{la}\label{fepweakint}
Let $E$ be a locally convex space.
If $E$ is sequentially complete
and has the Fr\'{e}chet exhaustion property,
then the weak integral
\[
\int_X \gamma \, d\mu
\]
exists in~$E$ for each measure space $(X,\Sigma,\mu)$
and each $\gamma\in \cL^1(X,\Sigma,\mu)$.
\end{la}
Given a measurable space $(X,\Sigma)$ and a locally convex space
$E$, we write $\cF(X,E)$ for the space of all measurable
functions $\gamma\colon X\to E$ with finite image.
The following result is needed in Section~\ref{secDL}.
\begin{la}\label{simpappfep}
Let $E$ be a locally convex space with the \emph{(FEP)}
and $(X,\Sigma,\mu)$ be a measure space.
Then $\cF(X,E)\cap \cL^p(X,\mu,E)$ is dense in $\cL^p(X,\mu,E)$,
for each $p\in [1,\infty[$.
We even have
\begin{equation}\label{strabi}
\gamma\in\overline{\{\eta\in\cF(X,E)\cap\cL^p(X,\mu, E)\colon
\eta(X)\sub \gamma(X)\cup\{0\}\}}
\end{equation}
for each $\gamma\in\cL^p(X,\mu,E)$.
\end{la}
\begin{numba}\label{banquot}
If $E$ is a locally convex space and $q\in P(E)$
a continuous seminorm on $E$, we set
\[
E_q:=E/q^{-1}(\{0\})
\]
and abbreviate $[x]:=x+q^{-1}(\{0\})$ for $x\in E$.
It is well-known that $E_q$ is a normed space with the norm
\[
\|.\|_q\colon E_q\to [0,\infty[,\quad \|[x]\|_q:=q(x).
\]
We let $\wt{E}_q$ be a completion of $E_q$ such that
$E_q\sub \wt{E}_q$, and write $\|.\|_q$ also for the extension
of the given norm on $E_q$ to a norm on $\wt{E}_q$.
Thus $\wt{E}_q$ is a Banach space, and
\[
\pi_q\colon E\to \wt{E}_q,\quad x\mto [x]
\]
is a continuous linear map with dense image such that $\|.\|_q\circ \pi_q=q$.
\end{numba}
\begin{la}\label{funda3}
Let $J\sub \R$ be an interval, $E$ be a sequentially complete
locally convex space which has the Fr\'{e}chet exhaustion property,
and $t_0\in J$.
If $\gamma\in \cL^1(J,E)$,
then the weak integrals needed to define
\[
\eta\colon J\to E\,,\quad t\mto \int_{t_0}^t\gamma(s)\, ds
\]
exist, and $\eta$ is a continuous function
which uniquely determines $[\gamma]$.
\end{la}
Again, we can therefore write $\eta':=[\gamma]$.
\begin{la}\label{henceemba}
Let $E$ be a Fr\'{e}chet space
and $\gamma\in \cL^1([a,b],E)$;
or let $E$ be an integral complete locally
convex space and $\gamma\in \cL^\infty_{rc}([a,b],E)$.
Let $F\sub E$ be a closed vector subspace and
$\eta\colon [a,b]\to E$ be the map given by
\[
\eta(t):=\int_a^t \gamma(s)\,ds\quad\mbox{for $t\in [a,b]$.}
\]
Then $\eta([a,b])\sub F$
if and only if $[\gamma]=[\gamma_1]$
for some $\gamma_1\in \cL^1([a,b],F)$
$($resp., $\gamma_1\in \cL^\infty_{rc}([a,b], F))$.
Likewise, we find $\gamma_1\in \cL^1([a,b],F)$ with $[\gamma]=[\gamma_1]$
if $E$ is a strict \emph{(LF)}-space,
$F\sub E$ a vector subspace which is a Fr\'{e}chet space
in the induced topology, and $\gamma\in \cL^1([a,b],E)$
with $\int_a^t\gamma(s)\,ds\in F$ for each $t\in [a,b]$.
\end{la}
The author does not know whether
the conclusion of Lemma~\ref{henceemba} would
hold for $E$ an arbitrary sequentially complete
(FEP)-space and
$\gamma\in \cL^1([a,b],E)$.
\begin{numba}
If $E$ is a Fr\'{e}chet space or sequentially complete (FEP)-space
and $\gamma\in \cL^1([a,b],E)$
(resp., $\gamma\in \cL^\infty_{rc}([a,b],E)$ with $E$ an
integral complete locally convex space),
we define
\[
\int_{t_0}^t[\gamma]:=\int_{t_0}^t\gamma(s)\,ds
\]
for $t,t_0\in [a,b]$.
The result is well-defined and only depends on~$[\gamma]$.
For $\gamma_1\in [\gamma]$,
we define $\int_{t_0}^t\gamma_1(s)\,ds:=\int_{t_0}^t[\gamma]=\int_{t_0}^t\gamma(s)\,ds$.
\end{numba}
\begin{numba}\label{defCk}
Let $E$ and $F$ be real locally convex spaces,
$U\sub E$ be open and $f\colon U\to F$ be a map.
If $f$ is continuous, we say that $f$ is~$C^0$.
We say that $f$ is $C^1$ if $f$ is continuous,
the directional derivative
\[
df(x,y):=(D_yf)(x):=\lim_{t\to 0}\frac{1}{t}(f(x+ty)-f(x))
\]
(with $0\not=t\in \R$)
exists in $F$ for all $(x,y)\in U\times E$,
and $df\colon U\times E\to F$ is continuous.
Recursively, for $k\in \N$ we say that $f$ is $C^k$
if $f$ is $C^1$ and $df\colon U\times E\to F$ is~$C^{k-1}$.
We say that $f$ is $C^\infty$ (or smooth)
if $f$ is $C^k$ for each $k\in \N_0$.
\end{numba}
\begin{numba}
Let $r\in \N_0\cup\{\infty\}$.
It can be shown that a map $f\colon U\to F$ as before is $C^r$
if and only if it is continuous, the iterated directional
derivatives
\[
d^{(k)}f(x,y_1,\ldots, y_k):=(D_{y_k}\cdots D_{y_1}f)(x)
\]
exist for all $k\in \N_0$ with $k\leq r$, $x\in U$ and $y_1,\ldots, y_k\in E$,
and the maps $d^{(k)}f\colon U\times E^k\to F$ so defined are continuous
(see, e.g., \cite{RES} or \cite{GaN}).
\end{numba}
\begin{numba}\label{BGN}
If $E$ is a locally convex space
and $U\sub E$ an open set, then
\[
U^{[1]}:=\{(x,y,t)\in U\times E\times \R\colon x+ty\in U\}
\]
is an open subset of $E\times E\times\R$
which contains
\[
U^{]1[}:=\{(x,y,t)\in U^{[1]}\colon t\not=0\}
\]
as an open dense subset. Moreover,
\[
U^{[1]}\;=\;U^{]1[} \cup (U\times E).
\]
If $f\colon U\to F$ is a $C^1$-map to a locally
convex space~$F$, then
\[
f^{[1]}\colon U^{[1]}\to F,\quad
(x,y,t)\mto
\left\{
\begin{array}{cl}
\frac{f(x+ty)-f(x)}{t} &\mbox{if $t\not=0$;}\\
df(x,y) & \mbox{if $t=0$}
\end{array}
\right.
\]
is continuous (see \cite{BGN} or \cite{GaN}).
In two cases, we shall find this very useful.\footnote{Conversely,
if there exists a continuous
map $f^{[1]}\colon U^{[1]}\to F$ such that
$f^{[1]}(x,y,t)=\frac{1}{t}(f(x+ty)-f(x))$
for all $(x,y,t)\in U^{]1[}$,
then $f$ is~$C^1$ (see \cite{BGN} of \cite{GaN}).}
\end{numba}
\begin{numba}\label{mvt}
We shall use the Mean Value Theorem in integral form (see \cite{GaN},
\cite{RES}):\\[2.3mm]
\emph{Let $E$ and $F$ be locally convex spaces, $U\sub E$ be an open subset,
$f\colon U\to F$ a $C^1$-map and $x,y\in U$ such that
the line segment $\{tx+(1-t)y\colon t\in [0,1]\}$
is contained in~$U$. Then}
\[
f(y)-f(x)=\int_0^1df(x+sy,y)\,ds.
\]
\emph{In particular, the preceding weak integral exists in~$F$.}
\end{numba}
\begin{numba}
Since compositions of $C^k$-maps are $C^k$
one can define $C^k$-manifolds and (smooth) Lie groups modelled
on a real locally convex space~$E$
in the expected way,
replacing the modelling space~$\R^n$ with $E$ in the
classical definitions of manifolds
and Lie groups
(see, e.g., \cite{RES} and \cite{GaN}
for streamlined expositions;
cf.\ also \cite{BGN}, \cite{Mil}
and \cite{Nee}).
If we speak of manifolds, then these are modelled
on a locally convex space (and thus need not be finite-dimensional).
Likewise, the Lie groups we consider
are modelled on arbitrary locally convex spaces.
As usual, $TM$ denotes the tangent bundle
of a smooth manifold~$E$ modelled on a locally convex
space, $T_xM$ the tangent space at $x\in M$, $\pi_{TM}\colon TM\to M$
the bundle projection sending $v\in T_xM$ to~$x$,
and $f\colon TM\to TN$ the tangent map
of a smooth map $f\colon M\to N$ between manifolds.
If $U$ is an open subset of a locally convex space~$E$,
we identify $TU$ with $U\times E$.
If $f\colon M\to E$ is a smooth map from a smooth manifold
to a locally convex space~$E$, we write $df$
for the second component of $Tf\colon TM\to TE=E\times E$.
\end{numba}
\begin{numba}
As usual, an element $\alpha=(\alpha_1,\ldots,\alpha_n)\in \N_0^n$
is called a multi-index and $|\alpha|:=\alpha_1+\cdots+\alpha_n$.
Let $e_j\in \R^n$ be the standard basis vector with $i$-th component $\delta_{i,j}$
in terms of Kronecker's delta.
If $U\sub \R^n$ is open, $k\in \N_0\cup\{\infty\}$ and $f\colon U\to E$ a $C^k$-map
to a locally convex space, we use the short-hand
\[
\partial^\alpha f(x):=\frac{\partial^\alpha f}{\partial x^\alpha}(x):=
((D_{e_1})^{\alpha_1}\cdots (D_{e_n})^{\alpha_n}f)(x)
\]
for the partial derivatives of $f$ at $x\in U$,
for $\alpha\in \N_0^n$ such that $|\alpha|\leq k$.
\end{numba}
We shall also use certain $C^{r,s}$-maps $(x,y)\mto f(x,y)$
on products with different degrees
of differentiability in~$x$ and~$y$
(see \cite{Alz} and \cite{AaS}):
\begin{defn}
Let $E_1$, $E_2$ and $F$ be real locally convex spaces,
$U\sub E_1$ and $V\sub E_2$ be open subsets,
$r,s\in \N_0\cup\{\infty\}$
and $f\colon U\times V\to F$ be a map.
Assume that the iterated directional derivatives
\[
d^{(i,j)}f(x,w_1,\ldots, w_i,y_1,\ldots, y_j):=
(D_{(w_i,0)}\cdots D_{(w_1,0)}D_{(0,y_j)}\cdots D_{(0,y_1)} f)(x)
\]
exist for all $i,j\in \N_0$ such that $i\leq r$ and
$j\leq s$, and all $w_1,\ldots, w_i\in E_1$
and $y_1,\ldots, y_j\in E_2$.
Moreover, assume that the mappings
\[
d^{(i,j)}f\colon U\times V\times E_1^i\times E_2^j\to F
\]
so obtained are continuous.
Then $f$ is called a $C^{r,s}$-map.
\end{defn}
$C^{r,s,t}$-maps on products $U\times V\times W$
are defined analogously (see \cite{Alz} for details).
We shall use the following simple
facts:
\begin{numba}\label{basicCrs}
Let $U$, $V$, $W$, $P$ and $Q$
be open subsets of locally convex spaces,
and $E$ and $F$ be locally convex spaces.
Let $r,s,t,\tau\in \N_0\cup\{\infty\}$.
Then the following holds:
\begin{itemize}
\item[(a)]
$f\colon U\times V\to F$ is $C^\infty$
if and only if $f$ is $C^{\infty,\infty}$
(see \cite[Remark~3.16]{AaS}).
\item[(b)]
If $f\colon E\times V\to F$ is $C^{(0,s)}$
for some $s\in \N_0\cup\{\infty\}$
and $f(.,y)\colon E\to F$ is linear for each $y\in V$,
then $f$ is $C^{\infty,s}$
(cf.\ \cite[Lemma~3.14]{AaS}).
\item[(c)]
If $f\colon U\times V\to F$ is $C^{(r,s)}$,
$g_1\colon P\to U$ is $C^r$ and $g_2\colon Q\to V$ is $C^s$,
then $g\circ (g_1\times g_2)\colon P\times Q\to F$
is $C^{(r,s)}$ (see \cite[Lemma~3.17]{AaS}).
\item[(d)]
If $g_1\colon P\times Q\to U$ is $C^{r,s}$,
$g_2\colon W\to V$ a $C^t$-map
and $f\colon U\times V\to F$ a $C^\tau$-map with $\tau\geq r+s+t$,
then $f\circ (g_1,g_2)\colon P\times Q\times W\to F$
is $C^{r,s,t}$
(as a special case of \cite[Lemma~81]{Alz}).
\item[(e)]
If $f\colon U\times P\times Q\to F$ is $C^{r,k,\ell}$
for all $k,\ell\in \N_0\cup\{\infty\}$ such that $k+\ell\leq s$,
then $f\colon U\times (P\times Q)\to F$ is $C^{r,s}$
(see \cite[Lemma~77]{Alz}). \item[(f)]
If $s<\infty$ and $s\geq 1$, then a map $f\colon U\times V\to F$
with $V$ open in~$E$
is $C^{(r,s)}$ if and only if
$f$ is $C^{(r,0)}$ and $C^{(0,1)}$
and the partial differential
$d_2f\colon U\times (V\times E)\to F$
is $C^{(r,s-1)}$
(see \cite[Lemma~3.11]{AaS}).
\end{itemize}
\end{numba}
We need a version of the Chain Rule
for curves in manifolds which are differentiable
at a point.
\begin{la}\label{chainpw}
Let $E$ and $F$ be locally convex spaces, $U\sub E$ be open
and $f\colon U\to F$ be a $C^1$-map.
If $J\sub \R$ is a non-degenerate interval,
$\eta\colon J\to U$
a continuous function and $t_0\in J$ such that $\eta'(t_0)$
exists, then also $(f\circ\eta)'(t_0)$ exists, and
$(f\circ\eta)'(t_0)=df(\eta(t_0),\eta'(t_0))$.
\end{la}
\begin{numba}
Let $M$ be a $C^1$-manifold modelled on a locally convex space~$E$,
$J\sub \R$ be a non-degenerate interval and $t_0\in J$.
We say that a continuous curve $\eta\colon J\to M$ is \emph{differentiable at~$t_0$}
if $(\phi\circ \gamma)'(t_0)$ exists in~$E$ for some chart $\phi\colon U\to V$
of~$M$ around~$\gamma(t_0)$ (which takes the open neighbourhood $U$ of $\gamma(t_0)$
in $M$ diffeomorphically onto an open subset $V\sub E$).
In this case, set
\[
\gamma'(t_0):=T\phi^{-1}(\phi(\gamma(t_0)),(\phi\circ\gamma)'(t_0))\in T_{\gamma(t_0)}M.
\]
This definition is independent of the chosen chart.\footnote{If also
$\psi$ is a chart around $\gamma(t_0)$, then
$\psi\circ\gamma=(\psi\circ\phi^{-1})\circ (\phi\circ \gamma)$
on a neighbourhood of~$t_0$.
By Lemma~\ref{chainpw}, $(\psi\circ\gamma)'(t_0)$ exists
and equals $d(\psi\circ\phi)^{-1}(\phi(\gamma(t_0)),(\phi\circ\gamma)'(t_0))$.
Hence $T\phi^{-1}(\phi(\gamma(t_0)),(\phi\!\circ\!\gamma)'(t_0))
=T\psi^{-1}T(\psi\!\circ\!\phi^{-1})(\phi(\gamma(t_0)),(\phi\!\circ\!\gamma)'(t_0))
=(\psi(\gamma(t_0)),(\psi\!\circ\!\gamma)'(t_0))$.}
\end{numba}
\begin{la}\label{locLip}
Let $E$ and $F$ be locally convex spaces,
$V\sub E$ be open, $f\colon V\to F$ be a $C^1$-map, $x\in V$
and $p\in P(F)$ be a continuous
seminorm.
Then there exists an open neighbourhood
$V_1\sub V$ of~$x$
and a continuous seminorm $q\in P(E)$ such that
\[
p(f(z)-f(y))\leq q(z-x)\quad\mbox{for all $z,y\in V_1$.}
\]
\end{la}
Now a simple compactness argument (cf.\ proof or Lemma~\ref{locLip2})
shows:
%
%
%
\begin{la}\label{locLipK}
Let $E$ and $F$ be locally convex spaces,
$V\sub E$ be open, $f\colon V\to F$ be a $C^1$-map, $K\sub V$
a compact set
and $p\in P(F)$ a continuous
seminorm.
Then there exists an open subset
$V_1\sub V$ with $K\sub V_1$
and a continuous seminorm $q\in P(E)$ such that
\[
p(f(z)-f(y))\leq q(z-x)\quad\mbox{for all $z,y\in V_1$.}
\]
\end{la}
\begin{la}\label{locLip2}
Let $E_1$, $E_2$ and $F$ be locally convex spaces,
$V\sub E_1$ be open
and $f\colon V\times E_2\to F$ be a $C^1$-map
such that $f(x,.)\colon E_2\to F$ is linear for all $x\in V$.
Let $K\sub V$ be a compact set
and $p\in P(F)$ be a continuous
seminorm.
Then there exist continuous seminorms
$q_1\in P(E_1)$ and $q_2\in P(E_2)$ such that
$K+B^{q_1}_1(0)\sub V$ and
\begin{equation}\label{easireqat}
p(f(z,v)-f(y,w))
\leq  q_2(v-w)+q_1(z-y)q_2(w)
\end{equation}
for all $z,y\in K+B_1^{q_1}(0)$ and $v,w\in E_2$.
\end{la}
Note that, in contrast to Lemma~\ref{funda},
we cannot ensure differentiability almost everywhere
in Lemma~\ref{funda2}.
The next lemma sometimes helps
us to get around this difficulty.
We use notation as in~\ref{banquot}.
\begin{la}\label{banana}
Let $E$ and $F$ be locally convex spaces,
$V\sub E$ be open, $f\colon V\to F$ be a $C^2$-map, $x\in V$
and $p\in P(F)$ be a continuous
seminorm.
Then there exists a continuous seminorm $q\in P(E)$
such that $B^q_1(x)\sub V$
and a $C^1$-map
\[
g\colon B^{\|.\|_q}_1(\pi_q(x))
\to \wt{F}_p
\]
on the ball $B^{\|.\|_q}_1(\pi_q(x))$ in $\wt{E}_q$ such that
\[
\pi_p\circ f|_{B^q_1(x)}=g\circ \pi_q|_{B^q_1(x)}^{B^{\|.\|_q}_1(\pi_q(x))}.
\]
\end{la}
Also the following generalization of Lemma~\ref{banana}
will be used, which follows by a simple
compactness argument.
%
%
%
\begin{la}\label{bananaK}
Let $E$ and $F$ be locally convex spaces,
$V\sub E$ be open, $f\colon V\to F$ be a $C^2$-map, $K\sub V$
a compact set
and $p\in P(F)$ be a continuous
seminorm.
Then there exists an open set $V_1\sub V$ with $K\sub V_1$,
a continuous linear map $\lambda\colon E\to Y$
to a Banach space~$Y$
and a $C^1$-map $g\colon W\to \wt{F}_q$
on an open subset $W\sub Y$ with $\lambda(V_1)\sub W$
such that
\[
\pi_q\circ f|_{V_1}=g\circ \lambda|_{V_1}^W.
\]
\end{la}

\begin{numba}
Recall that a mapping $p\colon E\to F$ between
complex locally convex spaces is called
a \emph{continuous homogeneous polynomial} of degree $n\in \N_0$
if there exists a continuous complex $n$-linear map $\beta\colon E^n\to F$
such that
\[
p(x)=\beta(x,x,\ldots,x)\quad\mbox{for all $x\in E$.}
\]
\end{numba}
In some cases, we can even establish analyticity
(rather than mere smoothness) of evolution maps.
Analyticity is understood in the following sense.
\begin{numba}\label{defcana}
Let $E$ and $F$ be complex locally convex spaces
and $U\sub E$ be an open set. Following \cite{BaS}
(see also \cite{RES} and \cite{GaN}), a
mapping $f\colon U\to F$ is called \emph{complex
analytic} if $f$ is continuous and
each $z\in U$ has an open neighbourhood $W\sub U$
such that
\[
f(w)=\sum_{n=0}^\infty p_n(w-z)\quad \mbox{for all $w\in W$},
\]
with pointwise convergence, for some sequence of continuous homogeneous
polynomials $p_n\colon E \to F$ of degree~$n$.
\end{numba}
\begin{numba}
For $E$, $F$, $U$ as in \ref{defcana},
a map $f\colon U\to F$ is complex analytic
if and only if $f$ is $C^\infty$
on the underlying real locally convex spaces,
with complex linear differentials
\[
df(x,.)\colon E\to F
\]
for all $x\in U$ (see \cite{RES}).
If $F$ is integral complete (or only Mackey complete),
then complex analyticity of~$f$ follows if $f$ is $C^1$
with complex linear differentials~\cite{GaN}.
\end{numba}
If $E$ is a real locally convex space, then the direct
product $E_\C:=E\times E$ becomes a complex locally convex space
if we define $(x+iy)(v,w):=(xv-yw,yv+xw)$
for $x,y\in \R$, $v,w\in E$.
Identifying $E$ with $E\times \{0\}\sub E_\C$,
we have $E_\C=E\oplus i E$ and $(x+iy)(v+iw)=
(xv-yw)+i(yv+xw)$.\\[2.3mm]
Real analytic maps are defined via existence
of complex analytic extensions.
\begin{numba}
Let $E$ and $F$ be real locally convex spaces
and $U\sub E$ be an open set.
Following \cite{RES} and \cite{BGN}
(cf.\ also \cite{Mil}),
a map $f\colon U\to F$ is called \emph{real analytic}
if there exists a complex analytic map
\[
\wt{f}\colon \wt{U}\to F_\C
\]
on an open subset $\wt{U}\sub E_\C$ with $U\sub \wt{U}$,
such that $f=\wt{f}|_U$.
\end{numba}
Compositions of complex analytic maps are complex
analytic; compositions of real analytic mappings are real
analytic (see, e.g., \cite{RES} or \cite{GaN}).
%
%
%
%
%
%
%
%
%
%
%
%
%
%
%
%
%
\section{Mappings between Lebesgue spaces}\label{mpbelbg}
We establish continuity and differentiability
properties for certain non-linear
maps between Lebesgue spaces as considered in Section~\ref{prels},
with parameters.
\begin{la}\label{operders}
Let
$E_1$, $E_2$ and $F$ be locally convex spaces,
$V\sub E_1$ be an open set
and $f\colon V\times E_2\to F$ be a
continuous map
such that $f(x,.)\colon E_2\to F$ is linear for each $x\in V$.
Assume the following:
\begin{itemize}
\item[\rm(a)]
$E_2$ and $F$ are integral complete and $\cE$
denotes $\cL^\infty_{rc}$ or $\cR$; or
\item[\rm(b)]
$E_2$ and $F$ are Fr\'{e}chet spaces
or \emph{(FEP)}-spaces
and $\cE$ denotes~$\cL^p$ for some $p\in [1,\infty]$.
\end{itemize}
Let $a,b\in \R$ such that $a<b$,
$\eta\colon [a,b]\to V$ be a continuous function,
and $\gamma\in \cE([a,b],E_2)$.
Then
\[
\theta:=f\circ (\eta,\gamma)\in \cE([a,b],F).
\]
\end{la}
\begin{prop}\label{oponinfty}
If $E_2$ and $F$ are integral complete
locally convex spaces
in the situation of Lemma~{\em\ref{operders}},
$\cE$ is $L^\infty_{rc}$ or $R$
and the map $f$ is $C^k$ for some
$k\in \N_0\cup\{\infty\}$, then also the map
\begin{equation}\label{thempstar}
\wt{f}\colon C([a,b],V)\times \cE([a,b],E)\to \cE([a,b],F),\;\;
(\eta,[\gamma])\mto [f\circ (\eta,\gamma)]
\end{equation}
is $C^k$.
\end{prop}
\begin{prop}\label{oponLp}
If $E_2$ and $F$ are Fr\'{e}chet spaces
or \emph{(FEP)}-spaces
in the situation of Lemma~{\em\ref{operders}},
$\cE=L^p$ with $p\in [1,\infty]$
and $f$ is $C^{k+1}$ for some
$k\in \N_0\cup\{\infty\}$, then the map
\begin{equation}\label{thempstar2}
\wt{f}\colon C([a,b],V)\times \cE([a,b],E)\to \cE([a,b],F),\;\;
(\eta,[\gamma])\mto [f\circ (\eta,\gamma)]
\end{equation}
is $C^k$.
\end{prop}
{\bf Proof of Lemma~\ref{operders}.}
Since $\eta([a,b])$ metrizable and compact
and hence second countable, we
have $\cB(\eta([a,b])\times E_2)=\cB(\eta([a,b]))\tensor
\cB(E_2)$ (see \ref{basicsmeas}\,(f)).
As~$f$ is continuous and hence Borel measurable,
we deduce that $f\circ (\eta,\gamma)$ is measurable.
Let $D_1\sub \eta([a,b])$ and $D_2\sub \gamma([a,b])$
be dense countable subsets.
Since $f$ is continuous, the countable set $f(D_1\times D_2)$
is dense in $\im(\theta)$.
Hence $\theta\colon [a,b]\to F$
is measurable and has separable image.
Let $q\in P(F)$ be
a continuous seminorm.
The map
\[
h\colon \eta([a,b])\times E_2\to F,\quad h(t,v):=f(\eta(t),v)
\]
is continuous and $h([a,b]\times \{0\})=0$.
Using the Wallace Lemma (see \ref{Wallace}),
we find an open subset $V_1\sub V$ such that $\eta([a,b])\sub V_1$
and an open $0$-neighbourhood
$W\sub E_2$ such that
$h(V_1\times W)\sub B^q_1(0)$.
We find a continuous seminorm $Q\in P(E_2)$ such that
$B^Q_1(0)\sub W$,
and we may assume that $V_1=\eta([a,b])+B^P_1(0)$
for a continuous seminorm $P$ on~$E_1$.
Thus
$q(v,w)\leq 1$ for all $v\in V_1$ and $w\in B^Q_1(0)$
and hence
\begin{equation}\label{will profit}
q(f(v,w))\leq Q(w)\quad\mbox{for all $v\in V_1$ and $w\in E_2$.}
\end{equation}
If $\cE=\cL^p$ with $p\in [1,\infty[$, we
have $q(\theta(t))=q(f(\eta(t),\gamma(t)))\leq Q(\gamma(t))$
and thus
\[
\|\theta\|_{\cL^p,q}=\sqrt[p]{\int_a^bq(\theta(t))^p\,dt}
\leq\sqrt[p]{\int_a^bQ(\gamma(t))^p\,dt}=\|\gamma\|_{\cL^p,Q}<\infty.
\]
Hence $\theta\in \cL^p([a,b],F)$.

If $\cE=\cL^\infty$, we have
have $q(\theta(t))\leq Q(\gamma(t))$
and thus $\sup q(\theta([a,b])\leq \sup Q(\gamma([a,b])<\infty$,
entailing that $\theta([a,b])\sub F$ is bounded.
Thus $\theta\in \cL^\infty([a,b],F)$.
Moreover, $\|\theta\|_{\cL^\infty,q}\leq \|\gamma\|_{\cL^\infty,Q}$
by the preceding.

If $\cE=\cL^\infty_{rc}$,
then the set $f(\eta([0,1]))\times \wb{\im(\gamma)})$
is compact and metrizable (see Lemma~\ref{quotcpmet}),
entailing that $\theta\in \cL^\infty_{rc}([a,b],F)$.

If $\cE=\cR$, we choose $\gamma_n\in \cT([a,b],E_2)$
such that $\gamma_n\to \gamma$ uniformly.
Since $C([a,b],E_1)\sub \cR([a,b],E_1)$,
we also find $\eta_n\in \cT([a,b],E_1)$
such that $\eta_n\to \eta$ uniformly.
Then $(\eta_n,\gamma_n)\in \cT([a,b],E_1\times E_2)$.
There is a continuous seminorm $P\in P(E_1)$ such that
$\eta([a,b])+B^P_1(0)\sub V$.
After passing to a subsequence,
we may assume that $\sup_{t\in [a,b]}P(\eta(t)-\eta_n(t))<1$,
whence $\eta_n(t)\in V$ for all $t\in [0,1]$
and thus
\[
f\circ (\eta_n,\gamma_n)\in \cT([a,b],F)
\]
for all $n\in \N$.
Given $q\in P(F)$, we choose $Q\in P(E_2)$
as above. Let $K:=\wb{\im(\gamma)}$.
We consider the continuous function
\[
g\colon [a,b]\times E_2\times B^P_1(0) \to F,\quad
g(t,y,v):=f(\eta(t)+v,y)-f(\eta(t),y).
\]
Since $g([a,b]\times K\times \{0\}=\{0\}\sub B^q_1(0)$,
the Wallace Lemma implies that there
is $S\in P(E_1)$ with $S\geq P$ such that
$g([a,b]\times K \times B^S_1(0))\sub B^q_1(0)$.
We find
$N\in \N$ such that $\sup_{t\in [a,b]}Q(\gamma(t)-\gamma_n(t))<1$
and $\sup_{t\in [a,b]}R(\eta(t)-\eta_n(t))<1$
for all $n\geq N$.
For $n\geq N$ and $t\in [a,b]$, we obtain
\begin{eqnarray*}
\lefteqn{q(f(\eta(t),\gamma(t))-f(\eta_n(t),\gamma_n(t)))}\quad\quad\\
&\leq &
q(f(\eta(t),\gamma(t))-f(\eta_n(t),\gamma(t))+q(f(\eta_n(t),\gamma(t)-\gamma_n(t)))
\leq 2,
\end{eqnarray*}
showing that $f\circ (\eta_n,\gamma_n)\to f\circ (\eta,\gamma)=\theta$ uniformly.
Thus $\theta\in \cR([a,b],F)$.\,\vspace{2.3mm}\Punkt

\noindent
{\bf Proof of Propositions \ref{oponinfty} and \ref{oponLp}.}
To see that $\wt{f}$ is $C^k$ if $f$ is $C^k$ (resp., $C^{k+1}$),
we may assume that $k\in \N_0$ and proceed by induction.
The case $k=0$ is a special case of the following lemma.
The induction step will be completed
once Lemma~\ref{operderspar}
is available.
\begin{la}\label{operderspar}
Let
$E_1$, $E_2$ and $F$ be locally convex spaces,
$V\sub E_1$ be an open set, $\Lambda$ be a topological space,
$a,b\in \R$ such that $a<b$ and
\[
f\colon \Lambda\times V\times E_2\to F
\]
be a map such that $f(\lambda,x,.)\colon E_2\to F$
is linear for all $(\lambda,x)\in \Lambda\times V$.
Assume the following:
\begin{itemize}
\item[\rm(a)]
$E_2$ and $F$ are integral complete, $f$ is continuous
and $\cE$ denotes $\cL^\infty_{rc}$ or $\cR$; or
\item[\rm(b)]
$E_2$ and $F$ are Fr\'{e}chet spaces or \emph{(FEP)}-spaces,
$\Lambda$ is an open
subset of a locally convex space $E_0$, the map $f$ is $C^1$
and $\cE$ denotes~$\cL^p$ for some $p\in [1,\infty]$.
\end{itemize}
Then the map
\begin{equation}\label{thempstarpar}
\wt{f}\colon \Lambda \times C([a,b],V)\times \cE([a,b],E)\to \cE([a,b],F),\;\;
(\lambda, \eta,\gamma)\mto f(\lambda,.) \circ (\eta,\gamma)
\end{equation}
is continuous.
\end{la}
\begin{proof}
Let us show that $\wt{f}$ is continuous
at each $(\lambda,\eta,\gamma)$.

Let $q\in P(F)$ be
a continuous seminorm.
If $\cE=\cL^\infty_{rc}$ or $\cE=\cR$,
we proceed as follows:
The map
\[
h\colon \Lambda\times \eta([a,b])\times E_2\to F,\quad h(t,v):=f(\eta(t),v)
\]
is continuous and $h(\Lambda\times [a,b]\times \{0\})=0$.
Using the Wallace Lemma (see \ref{Wallace}),
we find an open neighbourhood $V_0\sub \Lambda$ of~$\lambda$,
an open $0$-neighbourhood
$W\sub E$
and an open set $V_1\sub V$ such that $\eta([a,b])\sub V_1$
and
$h(V_0\times V_1\times W)\sub B^q_1(0)$.
We find a continuous seminorm $Q\in P(E_2)$ such that
$B^Q_1(0)\sub W$,
and we may assume that $V_1=\eta([a,b])+B^P_1(0)$
for a continuous seminorm $P$ on~$E_1$.
Thus
$q(\mu,v,w)\leq 1$ for all $\mu\in V_0$, $v\in V_1$ and $w\in B^Q_1(0)$
and hence
\begin{equation}\label{will profit2}
q(f(\mu,v,w))\leq Q(w)\quad\mbox{for all $\mu\in V_0$, $v\in V_1$ and $w\in E_2$.}
\end{equation}
Let $K:=\wb{\im(\gamma)}$.
We consider the continuous function
\[
g\colon V_0\times [a,b]\times E_2\times B^P_1(0) \to F,\quad
g(\mu,t,y,v):=f(\mu,\eta(t)+v,y)-f(\lambda,\eta(t),y).
\]
Since $g(V_0\times [a,b]\times K\times \{0\}=\{0\}\sub B^q_1(0)$,
the Wallace Lemma implies that,
after shrinking~$V_0$ if necessary,
there is
$S\in P(E_1)$ with $S\geq P$ such that
$g(V_0\times [a,b]\times K \times B^S_1(0))\sub B^q_1(0)$.
Then
\begin{eqnarray*}
\lefteqn{q(f(\bar{\lambda},\bar{\eta}(t),\bar{\gamma}(t)
-f(\lambda,\eta(t),\gamma(t)))}\qquad\\
&\leq&
q(f(\bar{\lambda},\bar{\eta}(t),\bar{\gamma}(t)-\gamma(t)))
+q(f(\bar{\lambda},\bar{\eta}(t),\gamma(t))-f(\lambda,\eta(t),\gamma(t)))\\
&\leq& Q(\bar{\gamma}(t)-\gamma(t))+ 1\leq 2
\end{eqnarray*}
for $\lambda_1$-almost all $t\in [a,b]$,
for all $\bar{\lambda}\in V_0$,
$\bar{\eta}\in \eta+C([a,b],B_1^P(0))$
and $\bar{\gamma}\in \gamma+ \Omega$ with
the $0$-neighbourhood
$\Omega\sub \cE([a,b],E)$
consisting of all $\zeta\in \cE([a,b],E)$
such that $\zeta(t)\in B_1^Q(0)\cap B_1^S(0)$
for $\lambda_1$-almost all $t\in [a,b]$.
Hence $\wt{f}$ is continuous at $(\lambda,\eta,\gamma)$.

In the situation of (b),
we have $\cE=\cL^p$ with $p\in [1,\infty]$.
Moreover, $f$ is $C^1$.
We apply Lemma~\ref{locLip2}
with $K:=\{\lambda\}\times \eta([a,b])\sub \Lambda\times V\sub E_0\times E_1$.
Given a continuous seminorm
$q\in P(F)$, it provides
continuous seminorms $Q\in P(E_0\times E_1)$ and $q_2\in P(E_2)$
such that $K+B^Q_1(0)\sub \Lambda\times V$ and
\[
q(f(\sigma,z,v)-f(\tau,y,w))
\leq
q_2(v-w)+ Q((\sigma,z)-(\tau,y))q_2(w)
\]
for all $(\sigma,y)$, $(\tau,z)\in K+B^Q_1(0)$
and $v,w\in E_2$.
After increasing $Q$, we may assume that
$Q(\sigma,z)=\max\{q_0(\sigma),q_1(z)\}$
for continuous seminorms $q_0\in P(E_0)$ and $q_1\in P(E_1)$.
Thus
\[
q(f(\sigma,z,v)-f(\tau,y,w))
\leq
q_2(v-w)+ \max\{q_0(\sigma-\tau),q_1(z-y)\}q_2(w)
\]
for all $\sigma,\tau\in B^{q_0}_1(\lambda)$,
$y,z \in \im(\eta)+B^{q_1}_1(0)$
and $v,w\in E_2$.
Hence, for all $\bar{\lambda}\in B^{q_0}_1(\lambda)$,
$\bar{\eta}\in \eta+C([a,b],B^{q_1}_1(0))$
and $\gamma,\bar{\gamma}$ in $\cL^p([a,b],E_2)$
we have
\begin{eqnarray*}
\lefteqn{q(f(\bar{\lambda}\bar{\eta}(t),\bar{\gamma}(t))-f(\lambda,\eta(t),
\gamma(t)))}\qquad\\
&\leq&
q_2(\bar{\gamma}(t)-\gamma(t))+ \max\{q_0(\bar{\lambda}-\lambda),
q_1(\bar{\eta}(t)-\eta(t))\}q_2(\gamma(t))
\end{eqnarray*}
for all $t\in [a,b]$. If $p=\infty$, we deduce that
\[
\|\wt{f}(\bar{\lambda},\bar{\eta},\bar{\gamma})-\wt{f}(\lambda,\eta,\gamma)\|_{\cL^\infty,q}
\|\bar{\gamma}-\gamma\|_{\cL^\infty,q_2}+\max\{q_0(\bar{\lambda}-\lambda),
\|\bar{\eta}-\eta\|_{\cL^\infty,q_1}\}\|\gamma\|_{\cL^\infty,q_2},
\]
which tends to $0$ if $(\bar{\lambda},\bar{\eta},\bar{\gamma})\to
(\lambda,\eta,\gamma)$.
If $p\in [1,\infty[$,
we deduce that
\begin{eqnarray*}
\lefteqn{\|\wt{f}(\bar{\lambda},\bar{\eta},\bar{\gamma})-\wt{f}(\lambda,\eta,\gamma)
\|_{\cL^p,q}}\qquad\\
&=&\|q\circ(\wt{f}(\bar{\lambda},\bar{\eta},\bar{\gamma})-
\wt{f}(\lambda,\eta,\gamma))\|_{\cL^p}\\
&\leq&
\|q_2\circ (\bar{\gamma}-\gamma)\|_{\cL^p}
+ \|\max\{q_0(\bar{\lambda}-\lambda),
(q_1\circ (\bar{\eta}-\eta))\}\cdot (q_2\circ \gamma)\|_{\cL^p}\\
&\leq&
\|\bar{\gamma}-\gamma\|_{\cL^p,q_2}+
\max\{q_0(\bar{\lambda}-\lambda),\|\bar{\eta}-\eta\|_{\cL^\infty,q_1}\}
\|\gamma\|_{L^p,q_2}.
\end{eqnarray*}
As the right hand side tends to $0$
as $(\bar{\lambda},\bar{\eta},\bar{\gamma})\to (\lambda,\eta,\gamma)$
in $C([a,b],E_1)\times \cL^p([a,b],E_2)$,
we see that $\wt{f}$ is continuous at
$(\lambda,\eta,\gamma)$.
\end{proof}
{\bf Proof of Propositions \ref{oponinfty} and \ref{oponLp}, completed.}
Let $k\in \N_0$ and assume that the assertion
holds for $k$ (i.e., $\wt{f}$ is $C^k$).
If $\cE$ is $L^\infty_{rc}$ or~$R$,
assume that $f$ is $C^{k+1}$;
if $\cE=L^p$ for some $p\in [1,\infty]$,
assume that~$f$ is $C^{k+2}$.
We show that $\wt{f}$ is $C^{k+1}$
with
\begin{equation}\label{givesdf}
d\wt{f}(\eta,[\gamma],\eta_1,[\gamma_1])=[df\circ (\eta,\gamma,\eta_1,\gamma_1)]
\end{equation}
for all $\eta\in C([a,b],V)$, $\eta_1\in C([a,b],E_1)$
and $[\gamma],[\gamma_1]\in \cE([a,b],E_2)$.
Consider the open set
\[
\hspace*{-25mm}(V\times E_2)^{[1]}\hspace*{90mm}
\]
\[
:= \{(x,v,y,w,t)\in V\times E_2\times E_1\times E_2 \times \R\colon
(x+ty,v+tw)\in V\times E_2\}
\]
in $E_1\times E_2\times E_1\times E_2 \times \R$
(cf.\ \ref{BGN}).
By~\ref{BGN},
there is a unique continuous
map
\[
f^{[1]}\colon (V\times E_2)^{[1]}\to F
\]
such that
\[
f^{[1]}(x,v,y,w,t)=\frac{f(x+ty,v+tw)-f(x,v)}{t}
\]
for all
$(x,v,y,w,t)\in (V\times E_2)^{[1]}$ such that $t\not=0$.
Then $f^{[1]}(x,v,y,w,0)=df((x,v),(y,w))$ for all
$(x,y)\in V\times E_1$.
If $f$ is $C^{k+1}$,
then $f^{[1]}$ is $C^k$;
if $f$ is $C^{k+2}$,
then $f^{[1]}$ is $C^{k+1}$ (see \cite{BGN}).
Since $\eta_1([a,b])$
is compact, there exists an open $0$-neighbourhood
$U\sub E_1$ such that $\eta_1([a,b])+U\sub V$.
Choose an open balanced $0$-neighbourhood
$W\sub E_1$ such that $W+W\sub U$.
There is $\ve>0$ such that $\im(\eta_1)\sub \ve^{-1}W$.
Then
\[
(\eta_1([a,b]+W)\times E_2\times \ve^{-1}W\times E_2\times \;]{-\ve},\ve[\;\sub
(V\times E_2)^{[1]}.
\]
Consider the function
$g\colon \;]{-\ve},\ve[\,\times ((\eta([a,b])+W)\times \ve^{-1}W)
\times (E_2\times E_2)\to F$,
\[
g(t,(x,y),(v,w)):=f^{[1]}(x,v,y,w,t).
\]
Lemma~\ref{operderspar} entails that the function
\[
\wt{g}\colon \;]{-\ve},\ve[\; \times C([a,b],(\eta([a,b)+W)\times \ve^{-1}W)
\times \cE([a,b],E_2\times E_2)\to \cE([a,b],F),
\]
\[\wt{g}(t,(\bar{\eta},\bar{\eta}_1),[(\bar{\gamma},\bar{\gamma}_1)])
:=
[g(t,.)\circ (\bar{\eta},\bar{\eta}_1,\bar{\gamma},\bar{\gamma}_1)]
=[f^{[1]}(.,t)\circ (\bar{\eta},\bar{\gamma},\bar{\eta}_1,\bar{\gamma}_1)]
\]
is continuous. Therefore the map
\[
]{-\ve},\ve[\;\to \cE([a,b],F),\quad
t\mto \Delta_t:=\wt{g}(t,\eta,\eta_1, [(\gamma,\gamma_1)])
\]
is continuous. Letting $t\in \;]{-\ve},\ve[\;\setminus\{0\}$
tend to~$0$, we deduce that
\[
\frac{[f\circ (\eta+t\eta_1,\gamma+t\gamma_1)]-[f\circ (\eta,\gamma)]}{t}
=\Delta_t\to \Delta_0=[df\circ (\eta,\gamma,\eta_1,\gamma_1)].
\]
Thus $d\wt{f}((\eta,[\gamma]),(\eta_1,[\gamma_1]))$
exists and
is given by
\begin{equation}\label{showsit}
d\wt{f}((\eta,[\gamma]),(\eta_1,[\gamma_1]))=
df\circ (\eta,[\gamma],\eta_1,[\gamma_1])
=\wt{h}((\eta,\eta_1),[(\gamma,\gamma_1)])
\end{equation}
with $h\colon (V\times E_1)\times (E_2\times E_2)\to F$,
$h((x,y),(v,w)):=df((x,v),(y,w))$.
Note that, for fixed $(x,y,t)\in V^{[1]}$ with $t\not=0$,
\[
f^{[1]}((x,v),(y,w),t)
=\frac{f(x+ty,v+tw)-f(x,v)}{t}
\]
is linear in $(v,w)$. Letting $t\to0$, we see that also
\[
h((x,y),(v,w))=df((x,v),(y,w))=f^{[0]}((x,v),(y,w),0)
\]
is linear in $(v,w)\in E_2\times E_2$
for fixed $(x,y)\in V\times E_1$.
Note that $h$ is $C^k$ if $\cE$ is $L^\infty_{rc}$ or $\cR$;
if $\cE$ is $L^p$,
then $h$ is $C^{k+1}$.
Hence $\wt{h}$ is $C^k$ by induction, and hence
$d\wt{f}$ is $C^k$, by (\ref{showsit}).
In particular, $d\wt{f}$ is continuous,
and thus $\wt{f}$ is~$C^1$.
Since $\wt{f}$ is~$C^1$ and
$d\wt{f}$ is $C^k$, the map $\wt{f}$ is $C^{k+1}$.\,\Punkt
\section{The spaces {\boldmath$AC_\cE([a,b],E)$} and mappings\\
between them}\label{secAC}
In this section, we define spaces of absolutely
continuous functions $\eta\colon [a,b]\to E$
with values in integral complete
locally convex spaces~$E$.
For general~$E$, we wish to distinguish the cases
that $\eta'\in L^\infty_{rc}([a,b],E)$
and $\eta'\in R([a,b],E)$ (a regulated function), respectively.
And if $E$ is a Fr\'{e}chet space (or a sequentially complete (FEP)-space),
we also wish to distinguish the cases
that $\eta'\in L^p([a,b],E)$ for some $p\in [1,\infty]$.
To enable a uniform treatment
of all of these cases, we find it convenient
to assume that locally convex spaces
\[
\cE([a,b],E)\sub L^\infty_{rc}([a,b],E)
\]
(resp., $\cE([a,b],E)\sub L^1([a,b],E)$)
have been selected for all $a,b\in \R$ such that $a<b$
and all integral complete locally convex spaces~$E$
(resp., all Fr\'{e}chet spaces~$E$, resp.,
all sequentially complete (FEP)-spaces)
in a reasonable way;
we then speak of a \emph{bifunctor}
on integral complete locally convex spaces
(resp., on Fr\'{e}chet spaces, resp., on sequentially complete (FEP)-spaces).
Given such a bifunctor,
we define and study certain locally convex spaces $AC_\cE([a,b],E)$
of absolutely continuous functions $\eta\colon [a,b]\to E$
with $\eta'\in \cE([a,b],E)$.
Natural additional axioms are worked out which enable
$AC_\cE([a,b],M)$ to be defined also for $M$ a manifold
modelled on~$E$; they are satisfied by all of $L^p$, $L^\infty_{rc}$, and~$R$.
Later, we shall associate a Lie group $AC_\cE([0,1],G)$
to each Lie group~$G$ modelled on~$E$.
The Lie group $AC_\cE([a,b],G)$ is needed to define
the notion of $\cE$-regularity for the Lie group~$G$.
\begin{defn}\label{bifunctFrech}
Assume that, for each Fr\'{e}chet space~$E$
and $a,b\in \R$ such that $a<b$,
a vector subspace $\cE([a,b],E)$ of $L^1([a,b],E)$
has been assigned, together with a locally
convex vector topology on~$\cE([a,b],E)$ such that
the inclusion map
\[
\cE([a,b],E)\to L^1([a,b],E),
\quad [\gamma]\mto[\gamma]
\]
is continuous.
Consider the following conditions:
\begin{itemize}
\item[(B1)]
For each continuous linear map $\lambda\colon E_1\to E_2$
between Fr\'{e}chet spaces,
we have $[\lambda\circ \gamma]\in \cE([a,b],E_2)$
for all $a,b\in \R$ such that $a<b$ and $[\gamma]\in \cE([a,b],E_1)$,
and the linear map
\[
\cE([a,b],\lambda)\colon \cE([a,b],E_1)\to \cE([a,b],E_2),\quad
[\gamma]\mto[\lambda\circ \gamma]
\]
is continuous.
\item[(B2)]
For each Fr\'{e}chet space~$E$, real numbers
$a,b,\alpha,\beta,c,d$ with $a\leq \alpha<\beta\leq b$ and $c<d$,
consider the map $f\colon [c,d]\to [a,b]$
given by
\begin{equation}\label{fafflin}
f(t):=\alpha+\frac{t-c}{d-c}(\beta-\alpha)\quad\mbox{for $t\in [c,d]$.}
\end{equation}
We require that
$[\gamma\circ f]\in \cE([c,d],E)$ for each $[\gamma]\in \cE([a,b],E)$
and that the linear map
\[
\cE(f,E)\colon \cE([a,b],E)\to \cE([c,d],E),\quad [\gamma]\mto[\gamma\circ f]
\]
is continuous.\footnote{In other words, we can pull back
along affine-linear maps.}
\end{itemize}
We call $\cE$ a \emph{bifunctor} on Fr\'{e}chet spaces
if (B1) and (B2) are satisfied.
\end{defn}
\begin{rem}
In particular, (B2) requires that
the linear map
\[
\rho\colon \cE([a,b],E)\to \cE([\alpha,\beta],E),\quad [\gamma]\mto [\gamma|_{[\alpha,\beta]}]
\]
is continuous, for all $a\leq \alpha<\beta\leq b$.
\end{rem}
\begin{numba}
If Fr\'{e}chet spaces are replaced with sequentially complete
(FEP)-spaces in Definition~\ref{bifunctFrech},
then we speak of a \emph{bifunctor on
sequentially complete} (FEP)-\emph{spaces}.
\end{numba}
\begin{defn}\label{bifunctic}
Assume that, for each integral complete locally convex space~$E$
and $a,b\in \R$ such that $a<b$,
a vector subspace $\cE([a,b],E)$ of
$L^\infty_{rc}([a,b],E)$
has been assigned, together with a locally
convex vector topology on~$\cE([a,b],E)$ such that
the inclusion map
$\cE([a,b],E)\to L^\infty_{rc}([a,b],E)$
is continuous.
Consider the following conditions:
\begin{itemize}
\item[(B1)]
For each continuous linear map $\lambda\colon E_1\to E_2$
between integral complete locally convex spaces,
we have $[\lambda\circ \gamma]\in \cE([a,b],E_2)$
for all $a,b\in \R$ such that $a<b$ and $[\gamma]\in \cE([a,b],E_1)$,
and the linear map
\[
\cE([a,b],\lambda)\colon \cE([a,b],E_1)\to \cE([a,b],E_2),\quad
[\gamma]\mto[\lambda\circ \gamma]
\]
is continuous.
\item[(B2)]
For each integral complete locally convex space~$E$, real numbers
$a$, $b$, $\alpha$, $\beta$, $c$, $d$ with $a\leq \alpha<\beta\leq b$ and $c<d$,
let $f\colon [c,d]\to [a,b]$
be as in~(\ref{fafflin}).
We require that
$[\gamma\circ f]\in \cE([c,d],E)$ for each $[\gamma]\in \cE([a,b],E)$
and that the linear map
\[
\cE(f,E)\colon \cE([a,b],E)\to \cE([c,d],E),\quad [\gamma]\mto[\gamma\circ f]
\]
is continuous.
\end{itemize}
We call $\cE$ a \emph{bifunctor} on integral complete
locally convex spaces
if (B1) and (B2) are satisfied.
\end{defn}
\begin{rem}\label{willgvcx}
The condition (B1) entails that
\[
\cE([a,b],E_1\times E_2)\cong \cE([a,b],E_1)\times \cE([a,b],E_2)
\]
as locally convex spaces, for all $a,b\in \R$ such that $a<b$
and Fr\'{e}chet spaces (resp., sequentially complete (FEP)-spaces,
resp., integral complete
locally convex spaces) $E_1$ and $E_2$.
To see this, let $\pi_j\colon E_1\times E_2\to E_j$
be the projection onto the $j$-th component, for $j\in \{1,2\}$.
We then see as in the proof of \ref{Lebspprod}
that
$(\cE([a,b],\pi_1),\cE([a,b],\pi_2))$
is an isomorphism of locally convex spaces.
\end{rem}
\begin{defn}\label{mostbaac}
Let $\cE$ be a bifunctor on Fr\'{e}chet spaces
(resp., a bifunctor on sequentially complete (FEP)-spaces,
resp., a bifunctor on integral complete locally convex spaces).
Let $a<b$ be real numbers and $E$ be a Fr\'{e}chet space
(resp., a sequentially complete (FEP)-space,
resp., an integral complete locally convex space).
Let $t_0\in [a,b]$.
We define $AC_\cE([a,b], E)\sub C([a,b],E)$ as the space
of all continuous functions $\eta\colon [a,b]\to E$
for which there exists a $[\gamma]\in \cE([a,b],E)$
such that
\[
(\forall t\in [a,b])\quad \eta(t)=\eta(t_0)+\int_{t_0}^t\gamma(s)\, ds\,.
\]
Lemma~\ref{funda} (resp., Lemma~\ref{funda2})
implies that $[\gamma]=\eta'$ is unique,
and that the map
\[
\Phi\colon AC_\cE([a,b],E)\to E \times \cE([a,b],E),\quad \eta\mto (\eta(t_0),\eta')
\]
is an isomorphism of vector spaces
(with $\Phi^{-1}(v,[\gamma])(t):=v+\int_{t_0}^t\gamma(s)\,ds$).
We give $AC_\cE([a,b],E)$ the Hausdorff
locally convex vector topology which makes~$\Phi$
an isomorphism of topological vector spaces.
\end{defn}
We shall see in Remark~\ref{closedemb2}
that both the definition of $AC_{\cE}([a,b],E)$
and its topology are independent of the choice
of~$t_0$.
\begin{rem}
\begin{itemize}
\item[(a)]
All of $L^p$ for $p\in [1,\infty]$
define bifunctors on Fr\'{e}chet spaces,
as well as $L^\infty_{rc}$ and $R$.
Therefore, we obtain function spaces
$AC_{L^p}([a,b],E)$ with $p\in [1,\infty]$;
$AC_{L^\infty_{rc}}([a,b],E)$,
and $AC_R([a,b],E)$ such that
\begin{eqnarray*}
AC_R([a,b],E)& \sub & AC_{L^\infty_{rc}}([a,b],E)\sub
AC_{L^\infty}([a,b],E)\\
&\sub & AC_{L^p}([a,b],E)\sub
AC_{L^p}([a,b],E)\sub AC_{L^1}([a,b],E)
\end{eqnarray*}
with continuous inclusion maps,
whenever $\infty\geq p\geq q\geq 1$.
Likewise for Fr\'{e}chet spaces replaced
with sequentially complete (FEP)-spaces.
\item[(b)]
$L^\infty_{rc}$ and $R$
define bifunctors on
integral complete locally convex spaces.
\end{itemize}
\end{rem}
\begin{rem}
If $\eta\colon \! [a,b]\to\R$,
then $\eta(t)=\eta(a)+\int_a^t\gamma(s)\,ds$ for some $\gamma\!\in\! \cL^1([0,1],\R)$
if and only if $\eta$ is absolutely continuous
in the sense that, for each $\ve>0$, there exists $\delta>0$
such that
\[
\sum_{j=1}^n|\eta(b_j)-\eta(a_j)|<\ve
\]
for all $n\in \N$ and disjoint subintervals
$]a_1,b_1[$, $\ldots$, $]a_n,b_n[$ of $[a,b]$
of total length $\sum_{j=1}^n(b_j-a_j)<\delta$.\\[2.3mm]
[In fact, if $\eta$ is absolutely continuous in the latter sense,
then $\eta$ is a primitive
of an $\cL^1$-function by \cite[Theorem 7.20]{Ru1}.
If, conversely, $\eta(t)=\eta(a)+\int_a^t\gamma(s)\,ds$
for some $\gamma\in \cL^1([a,b],\R)$,
consider the auxiliary function
\[
\zeta\colon [a,b]\to\R,\quad \zeta(t):=\int_a^t|\gamma(s)|\,ds.
\]
Then $\zeta$ is differentiable $\lambda_1$-almost
everywhere and $\zeta'(t)$ coincides
with $\gamma(t)$ for $\lambda_1$-almost every $t\in [a,b]$
(cf.\ \cite[Theorem 7.11]{Ru1}).
Now (c)$\impl$(a) in \cite[Theorem 7.18]{Ru1}
shows that $\zeta$ is absolutely continuous. Hence, given $\ve>0$,
we find $\delta>0$ such that
\[
\sum_{j=1}^n|\zeta(b_j)-\zeta(a_j)|<\ve
\]
for all $n\in \N$ and disjoint subintervals
$]a_1,b_1[$, $\ldots$, $]a_n,b_n[$ of $[a,b]$
of total length $\sum_{j=1}^n(b_j-a_j)<\delta$.
Since
\[
|\eta(b_j)-\eta(a_j)|=\left|\int_{a_j}^{b_j}\gamma(s)\,ds\right|
\leq \int_{a_j}^{b_j}|\gamma(s)|\,ds=\zeta(b_j)-\zeta(a_j)=|\zeta(b_j)-\zeta(a_j)|,
\]
we deduce that
\[
\sum_{j=1}^n|\eta(b_j)-\eta(a_j)|\leq
\sum_{j=1}^n|\zeta(b_j)-\zeta(a_j)|
<\ve
\]
for all intervals as before. Hence $\eta$ is absolutely continuous.]
\end{rem}
Give $C([a,b],E)$ the topology of uniform convergence
(defined by the seminorms $\|.\|_{\cL^\infty,q}$
with $q\in P(E)$).
Let $\cE$ be a bifunctor on Fr\'{e}chet spaces
(resp., a bifunctor on sequentially complete (FEP)-spaces, resp.,
a bifunctor on integral complete
locally convex spaces)
and $E$ be a Fr\'{e}chet space (resp.,
a sequentially complete (FEP)-space, resp., an integral
complete locally convex space).
It is useful to note:
%
\begin{la}\label{closedemb}
The map $\Psi\colon AC_\cE([a,b],E)\to
C([a,b],E)\times \cE([a,b],E)$, $\eta\mto (\eta,\eta')$
is a linear topological embedding with closed image.
\end{la}
\begin{proof}
The linearity is clear. Let $q\in P(E)$.
Since $\eta(t)=\eta(t_0)+\int_{t_0}^t \eta'(s)\,ds$, we have
\[
q(\eta(t))\leq q(\eta(t_0))+q\left(\int_{t_0}^t\eta'(s)\,ds\right)
\leq q(\eta(t_0))+\|\eta'\|_{L^1,q}
\]
for all $t\in [a,b]$ and thus $\|\eta\|_{\cL^\infty,q}\leq q(\eta(t_0))+
\|\eta'\|_{L^1,q}$, As a function of $\eta$, this is a continuous
seminorm on $AC_{L^1}([a,b],E)$ (resp., on
$AC_{L^\infty_{rc}}([a,b],E)$)\footnote{Recalling that
$\|.\|_{L^1,q}\leq (b-a)\|.\|_{L^\infty,q}$ on $L^\infty_{rc}([a,b],E)$.}
and hence
on $AC_{\cE}([a,b],E)$. Thus $\Psi$ is continuous,
entailing that the initial topology $\cO_\Psi$ on $AC_{\cE}([a,b],E)$
with respect to~$\Psi$ is coarser than the initial topology $\cO_\Phi$
with respect to $\Phi$.
The evaluation map $\ev_{t_0}\colon C([a,b],E)\to E$
is continuous linear. Since $\Phi=(\ev_{t_0}\times \id)\circ \Psi$,
we get the converse inclusion $\cO_\Phi\sub \cO_\Psi$
and hence equality, $\cO_\Phi=\cO_\Psi$.
Hence~$\Psi$ is a topological embedding.
Now consider a net $(\eta_\alpha,\eta_\alpha')_{\alpha\in A}$
in $\im(\Psi)$
such that $(\eta_\alpha\to \eta_\alpha')\to (\eta,\gamma)$
in $C([a,b],E)\times \cE([a,b],E)$.
Define $\zeta=\Phi^{-1}(\eta(t_0),\gamma)\in AC_\cE([a,b],E)$.
This is the function $[a,b]\to E$,
\[
t\mto \eta(t_0)+\int_{t_0}\gamma(s)\,ds.
\]
For $q\in P(E)$, we have
\begin{eqnarray*}
q(\eta_\alpha(t)-\zeta(t)
&=&q\left(\ev_{t_0}(\eta_\alpha-\eta(t_0))+
\int_{t_0}^t(\eta_\alpha'(s)-\gamma'(s)\,ds\right)\\
&\leq& \|\eta_\alpha-\eta\|_{\cL^\infty,q}+\|\eta_\alpha'-\gamma\|_{L^1,q}
\end{eqnarray*}
for each $t\in [a,b]$ and hence
\[
\|\eta_\alpha-\zeta\|_{\cL^\infty,q}\leq
\|\eta_\alpha-\eta\|_{\cL^\infty,q}+\|\eta_\alpha'-\gamma\|_{L^1,q}\to 0.
\]
As the left hand side converges to $\|\eta-\zeta\|_{\cL^\infty,q}$,
we see that $\eta=\zeta$, whence $\eta\in AC_\cE([a,b],E)$
with $\eta'=\gamma$ and thus $(\eta,\gamma)=\Psi(\eta)\in \im(\Psi)$.
\end{proof}
\begin{rem}\label{closedemb2}
If $t_1\in [a,b]$ and $\eta\in AC_\cE([a,b],E)$ with respect to $t_0\in [a,b]$,
with $\eta'=[\gamma]$, then
\[
\eta(t_1)=\eta(t_0)+\int_{t_0}^{t_1}\gamma(s)\,ds,
\]
whence $\eta(t)$ equals
\[
\eta(t_0)+\int_{t_0}^t\gamma(s)\,ds
=\eta(t_1)+\eta(t_0)-\eta(t_1)+ \int_{t_0}^t\gamma(s)\,ds
=\eta(t_1)+ \int_{t_1}^t\gamma(s)\,ds.
\]
We deduce that $\eta$ is also in $AC_{\cE}([a,b],E)$ with respect to~$t_1$.
Since the map $AC_{\cE}([a,b],E)\to E$, $\eta\mto
\eta(t_1)=\ev_{t_1}\circ \pr_1\circ \Psi$
(with $\Psi$ as in Lemma~\ref{closedemb2})
is continuous, we deduce that the topology on $AC_{\cE}([a,b],E)$
with respect to~$t_1$ is coarser than that with respect to~$t_0$,
and reversing the roles of~$t_0$ and~$t_1$
we deduce that both topologies are equal.
\end{rem}
\begin{rem}\label{intoC0}
By Lemma~\ref{closedemb}, the inclusion map
\[
j\colon AC_\cE([a,b],E)\to C([a,b],E)
\]
is continuous.
If $V\sub E$ is an open subset, then $C([a,b],V)$ is open
in $C([a,b],E)$. Thus
\[
AC_\cE([a,b],V)
:=\{\eta\in AC_\cE([a,b],E)\colon \eta([a,b])\sub V\}
\]
is open in $AC_\cE([a,b],E)$ (as $AC_\cE)[a,b],V)=j^{-1}(C([a,b],V))$).
\end{rem}
\begin{numba}
If $a\leq \alpha<\beta \leq b$ and $[\gamma]\in \cE([a,b],E)$,
then $[\gamma|_{[\alpha,\beta]}] \in \cE([\alpha,\beta],E)$
by axiom (B2), taking $c:=\alpha$ and $d:=\beta$
there. As a consequence,
if $\eta\in AC_\cE([a,b],E)$,
then $\eta|_{[\alpha,\beta]}\in AC_\cE([\alpha,\beta],E)$.
\end{numba}
Conversely,
we'd like to check the $\cE$-property
on subintervals.
\begin{defn}\label{locality}
Let $\cE$ be a bifunctor on Fr\'{e}chet spaces
(resp., on sequentially complete (FEP)-spaces,
resp., on integral complete locally convex spaces).
We say that $\cE$ satisfies the ``locality
axiom'' if the following condition
is satisfied:
\begin{itemize}
\item[(Loc)]
For each Fr\'{e}chet space (resp., sequentially complete (FEP)-space,
resp., integral complete
locally convex space)~$E$, $n\in \N$ and real numbers
$a=t_0<t_1<\cdots<t_n=b$,
the map
\[
\cE([a,b],E)\to \prod_{j=1}^n\cE([t_{j-1},t_j],E),\quad
[\gamma]\mto ([\gamma|_{[t_{j-1},t_j]}])_{j=1,\ldots, n}
\]
is an isomorphism of topological vector spaces.
\end{itemize}
\end{defn}
\begin{rem}\label{oldloc}
The locality axiom immediately implies
the following useful fact:
If $\gamma\colon [a,b]\to E$
is a measurable map and $a=t_0<t_1<\cdots< t_n=b$,
then $[\gamma]\in \cE([a,b],E)$
if and only if $[\gamma|_{[t_{j-1},t_j]}]\in \cE([t_{j-1},t_j],E)$
for all $j\in \{1,\ldots, n\}$.
\end{rem}
\
\begin{numba}
If $\cE$ is any of $L^p$ ($p\in [1,\infty]$),
$L^\infty_{rc}$ or $R$,
then $\cE$ satisfies the locality axiom
(for trivial reasons).
\end{numba}
\begin{numba}
The locality axiom implies
that if $\eta\colon [a,b]\to E$ is continuous,
$a=t_0<t_1<\cdots<t_n=b$
and $\eta|_{[t_{j-1},t_j]}\in AC_\cE([t_{j-1},t_j],E)$
for all $j\in \{1,\ldots, n\}$,
then $\eta\in AC_\cE([a,b],E)$.
\end{numba}
\begin{defn}\label{ckact}
Let $\cE$ be a bifunctor on Fr\'{e}chet spaces
(resp., on sequentially complete (FEP)-spaces,
resp., on integral complete locally convex spaces)
and $k\in \N\cup\{\infty\}$.
We say that \emph{$C^k$-functions act on~$AC_\cE$}
if the following condition is satisfied:
\begin{itemize}
\item[($\mbox{A}_k$)]
For all Fr\'{e}chet spaces (resp., sequentially complete (FEP)-spaces,
resp., integral complete
locally convex spaces) $E$ and $F$,
each $C^k$-map $f\colon V\to F$
on an open subset $V\sub E$,
all $a<b$ in $\R$
and each $\eta\in AC_\cE([a,b],E)$
such that $\eta([a,b])\sub V$, we have
\[
f\circ \eta\in AC_\cE([a,b],F).
\]
\end{itemize}
\end{defn}
\begin{la}\label{C2acts}
\begin{itemize}
\item[\rm(a)]
Let $\cE$ be $L^p$ for $p\in [1,\infty]$,
$L^\infty_{rc}$ or $R$,
considered as bifunctors on Fr\'{e}chet spaces
$($or on sequentially complete \emph{(FEP)}-spaces$)$.
Then $C^1$-functions act on~$A_\cE$.
\item[\rm(b)]
$C^2$-functions act on $AC_{L^\infty_{rc}}$
and on~$AC_R$, if $L^\infty_{rc}$ and $R$
are considered as bifunctors on
integral complete locally convex spaces.
\end{itemize}
\end{la}
\begin{proof}
(a)
Let $E$ and $F$ be Fr\'{e}chet spaces
(or sequentially complete (FEP)-spaces),
$a,b\in \R$ with $a<b$
and $\eta\in AC_\cE([a,b],E)$.
Write $\eta'=[\gamma]$.
Let $V\sub E$ be an open subset
such that $\eta([a,b])\sub V$ and $f\colon V\to F$
be a $C^1$-map.

Step~1.
The map $f\circ \eta\colon [a,b]\to F$ is
continuous.
Also the map
\[
df\colon V\times E\to F
\]
is continuous, and thus
$\theta:=df\circ (\eta,\gamma)\in \cE([a,b],F)$,
by Lemma~\ref{operders}.
We can therefore form a map
\[
\zeta\colon [a,b]\to F,\quad \zeta(t):=f(\eta(a))+\int_a^t df(\eta(s),\gamma(s))\,ds.
\]
Then $\zeta\in AC_\cE([a,b],F)$, $\zeta'(t)=[\theta]$
and $\zeta(a)=f(\eta(a))$.
If we can show that $\lambda\circ\zeta=\lambda\circ f\circ \eta$ for each $\lambda\in F'$,
then $f\circ\eta=\zeta\in AC_\cE([a,b],F)$ (by the Hahn-Banach Separation Theorem).
We may therefore assume now that~$F=\R$.

Step~2.
We claim that $f\circ \eta\in AC_{L^1}([a,b],\R)$.
If this is true, then
\[
(f\circ \eta)'(t)=df(\eta(t),\gamma(t))=\zeta'(t)
\]
for almost all $t\in [a,b]$ (by Lemmas~\ref{funda}
and \ref{chainpw})
and hence $f\circ \eta=\zeta$ since
both $f\circ \eta$ and $\zeta$
are absolutely continuous, $(f\circ \eta)'=\zeta'$
and $f(\eta(a))=\zeta(a)$.\\[2.3mm]
To prove the claim,
we use Lemma~\ref{locLipK}
to find a continuous seminorm $q\in P(E)$
and an open set $V_1\sub V$ with $\eta([a,b])\sub V_1$
such that
\[
|f(z)-f(y)|\leq q(z-y)\quad\mbox{for all $z,y\in V_1$.}
\]
We have $q\circ \gamma\in \cL^1([a,b],\R)$,
whence
\[
\sigma\colon [a,b]\to \R,\quad \sigma(t):=\int_a^tq(\gamma(s))\,ds
\]
is monotonically increasing and absolutely continuous.
If $\ve>0$, let $\delta\in \,]0,\rho]$
such that
\[
\sum_{j=1}^n|\sigma(b_j)-\sigma(a_j)|<\ve
\]
for each $n\in \N$ and disjoint intervals
$]a_1,b_1[$, $\ldots$, $]a_n,b_n[$ in $[a,b]$
of total length $\sum_{j=1}^n(b_j-a_j)<\delta$
(see \cite[Theorem 7.18]{Ru1}).
For each $j\in \{1,\ldots, n\}$,
we have
\[
\eta(b_j)-\eta(a_j)=\int_{a_j}^{b_j}\gamma(t)\,dt
\]
and hence
\[
q(\eta(b_j)-\eta(a_j))\leq \int_{a_j}^{b_j}q(\gamma(t))\,dt=\sigma(b_j)-\sigma(a_j)
=|\sigma(b_j)-\sigma(a_j)|.
\]
Thus
\[
\sum_{j=1}^n|f(\eta(b_j))-f(\eta(a_j))|
\leq
\sum_{j=1}^nq(\eta(b_j)-\eta(a_j))\leq
\sum_{j=1}^n|\sigma(b_j)-\sigma(a_j)|<\ve.
\]
Hence $f\circ\eta$ is absolutely continuous
(by \cite[Theorem 7.18]{Ru1}),
as required.

(b)
Let $E$ and $F$ be integral complete locally convex spaces,
$a,b\in \R$ with $a<b$
and $\eta\in AC_\cE([a,b],E)$.
Write $\eta'=[\gamma]$.
Let $V\sub E$ be an open subset
such that $\eta([a,b])\sub V$ and $f\colon V\to F$
be a $C^2$-map.
Step~1 of the proof of~(a)
applies without changes.
We may therefore assume now that $F=\R$.
By Lemma~\ref{bananaK},
we find an open subset $V_1\sub V$ with $\eta([a,b])\sub V_1$,
a continuous linear map $\lambda\colon E\to Y$
to a Banach space~$Y$
and a $C^1$-function $g\colon W\to \R$
on an open subset $W\sub Y$ with $\lambda(V_1)\sub W$
such that
\[
f|_{V_1}=g\circ \lambda|_{V_1}.
\]
Now $\lambda\circ \eta\in AC_\cE([a,b],Y)$
since
\[
\lambda(\eta(t))=\lambda\left(\int_a^t\gamma(s)\,ds\right)
=\int_a^t\lambda(\gamma(s))\,ds
\]
for each $t\in [a,b]$ (by \ref{basicweaki}),
where $[\lambda\circ\gamma]=\cE([a,b],\lambda)([\gamma])\in \cE([a,b],Y)$
by axiom (B1).
Thus $f\circ \eta=g\circ (\lambda\circ \eta)\in AC_\cF([a,b],\R)=AC_\cE([a,b],\R)$
by (a), if we set $\cF([a,b],H):=\cE([a,b],H)$
for each Fr\'{e}chet space~$H$.
\end{proof}
\begin{rem}\label{dercomp}
The preceding proof shows that
\[
(f\circ \eta)'=[f\circ (\eta,\gamma)]
\]
if $f\colon E\supseteq V\to F$ is $C^1$ (resp., $C^2$),
$\eta\in AC_\cE([a,b],V)$ and $\eta'=[\gamma]$.
\end{rem}
\begin{defn}\label{ACintoM}
Let $\cE$ be a bifunctor on Fr\'{e}chet spaces
(resp., on sequentially complete (FEP)-spaces,
resp., on integral complete locally convex spaces)
such that $C^k$-functions act on~$\cE$
for some $k\in \N$
and the locality axiom (Loc)
from Definition~\ref{locality}
is satisfied.
Let $M$ be a $C^k$-manifold modelled on
a Fr\'{e}chet space
(resp., a sequentially complete (FEP)-space, resp., an
integral complete locally convex space)~$E$.
If $a,b\in \R$ with $a<b$,
we let $AC_\cE([a,b],M)$ be the set of all
functions
\[
\eta\colon [a,b]\to M
\]
such that $\eta$ is continuous and
there is a partition $a=t_0<t_1<\cdots<t_n=b$
of $[a,b]$ such that, for each $j\in \{1,\ldots, n\}$,
there exists a chart $\phi_j\colon U_j\to V_j\sub E$
of~$M$ such that $\eta([t_{j-1},t_j])\sub U_j$
and $\phi_j\circ \gamma|_{[t_{j-1},t_j]}\in AC_\cE([t_{j-1},t_j],E)$.
\end{defn}
We imposed the axioms ($\mbox{A}_k$)
and (Loc) to ensure the independence
of the $AC_\cE$-property from
the chosen partition:
\begin{la}\label{goodinM}
Let $\cE$ be a bifunctor on Fr\'{e}chet spaces
$($resp., on sequentially complete \emph{(FEP)}-spaces,
resp., on integral complete locally convex spaces$)$
such that $C^k$-functions act on~$A_\cE$
for some $k\in \N\cup\{\infty\}$
and $\cE$ satisfies the locality axiom \emph{(Loc)}
from Definition~\emph{\ref{locality}}.
Let
$E$ be a Fr\'{e}chet space $($resp., a sequentially complete
\emph{(FEP)}-space, resp.,
an integral complete locally convex space$)$.
Let $a,b\in \R$ such that $a<b$
and $M$ be a $C^k$-manifold modelled on~$E$.
If $\eta\in AC_\cE([a,b],M)$,
then
\[
\phi\circ \eta|_{[\alpha,\beta]}\in AC_\cE([\alpha,\beta],M)
\]
for each chart $\phi\colon U\to V\sub E$ of~$M$
and all $\alpha<\beta$ in $J$ such that
$\eta([\alpha,\beta])\sub U$.
In particular, $AC_{\cE}([a,b],E)$
for~$E$ as a vector space
coincides as a set with $AC_\cE([a,b],E)$
for $E$ considered as a manifold
modelled on~$E$ with the maximal $C^k$-atlas
associated with the global chart $\id_ E$.
\end{la}
\begin{proof}
Let $a=t_0<t_1<\cdots< t_n=b$ and charts $\phi_j\colon U_j\to V_j\sub E$
be as in Definition~\ref{ACintoM}.
Thus
\[
\phi_j\circ \eta|_{[t_{j-1},t_j]}\in AC_{\cE}([t_{j-1},t_j],E)\quad\mbox{for all
$j\in \{1,\ldots, n\}$.}
\]
Using $\phi_j$ as the chart for the new points,
we may add additional
points inside $[t_{j-1},t_j]$.
We may therefore assume that $\alpha=t_k$ and $\beta=t_\ell$
for certain $k,\ell\in \{1,\ldots,n\}$ with $k<\ell$.
Then $\eta([t_{j-1},t_j])\sub U_j\cap U$
for all $j\in \{k+1,\ldots,\ell\}$ and the transition map
\[
\tau_j:=\phi\circ \phi_j^{-1}\colon \phi_j(U_j\cap U)\to \phi(U_j\cap U)
\]
is a $C^1$-dffeomorphism between open subsets
of~$E$. We have to show that $\phi\circ \eta|_{[\alpha,\beta]}\in AC_\cE([\alpha,\beta],E)$.
By the locality axiom, it suffices to show that
$\phi\circ\eta|_{[t_{j-1},t_j]}
\in AC_\cE([t_{j-1},t_j],E)$ for $j\in \{k+1,\ldots, \ell\}$.
But this is true,
since
$\phi\circ \eta|_{[\alpha,\beta]}
=\tau_j\circ (\phi_j\circ \eta|_{[t_{j-1},t_j]})$
and $C^k$-functions (like the $\tau_j$) act on~$AC_\cE$.
\end{proof}
\begin{rem}\label{reallylocal}
Let $\eta\colon [a,b]\to M$
be a continuous function such that, for
each $s\in [a,b]$, there is $\ve_s>0$
such that $\eta([a,b]\cap [s-\ve_s,t+\ve_s])\sub U_s$
for some chart $\phi_s\colon U_s\to V_s\sub E$ of~$M$
and
\[
\phi_s\circ \eta|_{[a,b]\cap[s-\ve_s,t+\ve_s]}\in AC_\cE([a,b]\cap[s-\ve_s,s+\ve_s],E).
\]
Then $\eta\in AC_\cE([a,b],M)$.\\[2.3mm]
In fact, $([a,b]\cap \;]s-\ve_s,s+\ve_s[)_{s\in [a,b]}$
is an open cover of the compact metric space $[a,b]$.
Lebesgue's Lemma provides a Lebesgue number
$\delta>0$ for this open cover.
Thus, for each $t\in [a,b]$, there exists $s(t)\in [a,b]$ such that
$[a,b]\cap \;]t-\delta,t+\delta[\;\sub
[a,b]\cap\;
]s(t)-\ve_{s(t)},s(t)+\ve_{s(t)}[$.
Choose $a=t_0<t_1<\cdots<t_n=b$
such that $t_j-t_{j-1}<\delta$ for all $j\in \{1,\ldots,n\}$.
Then $[t_{j-1},t_{j+1}]\sub [a,b]\cap\;
]s(t_{j-1})-\ve_{s(t_{j-1})},s(t_{j-1})+\ve_{s(t_{j-1})}[$
and hence $\phi_{s(t_{j-1})}\circ \eta|_{[t_{j-1},t_j]}\in AC_\cE([t_{j-1},t_j],E)$
for all $j\in\{1,\ldots, n\}$.
Thus $\eta\in AC_\cE([a,b],E)$.
\end{rem}
\begin{defn}
Let $\cE$ be a bifunctor on Fr\'{e}chet spaces
(resp., sequentially complete (FEP)-spaces, resp.,
integral complete locally convex spaces)
which satisfies
the locality axiom (Loc).
Let $E$ be a Fr\'{e}chet space
(resp., a sequentially complete (FEP)-spaces, resp.,
an integral complete locally convex space)
and $I\sub \R$ be a non-degenerate interval.
We let $AC_\cE(I,E)$ be the vector
space of all continuous
mappings $\eta\colon I\to E$ such that $\eta|_{[a,b]}\in
AC_\cE([a,b],E)$ for all compact intervals
$[a,b]\sub I$.
If smooth maps act on $AC_\cE$ and
$M$ is a smooth manifold modelled on~$E$,
we let
$AC_\cE(I,M)$ be the set of all continuous
mappings $\eta\colon I\to M$ such that $\eta|_{[a,b]}\in
AC_\cE([a,b],M)$ for all compact intervals
$[a,b]\sub I$.
\end{defn}
\begin{la}\label{actonelts}
Let $\cE$ be a bifunctor on Fr\'{e}chet spaces
$($resp., on sequentially complete \emph{(FEP)}-spaces, resp.,
on integral complete locally convex spaces$)$
such that $C^k$-functions act on~$A_\cE$
for some $k\in \N\cup\{\infty\}$
and $\cE$ satisfies the locality axiom \emph{(Loc)}
from Definition~\emph{\ref{locality}}.
Let
$E$ and $F$ be Fr\'{e}chet spaces $($resp.,
sequentially complete \emph{(FEP)}-spaces, resp.,
integral complete locally convex spaces$)$,
$a,b\in \R$ such that $a<b$
and
\[
f\colon M\to N
\]
be a $C^k$-map,
where $M$ be a $C^k$-manifold modelled on~$E$
and $N$ be a $C^k$-manifold modelled on~$F$.
Then $f\circ\eta \in AC_\cE([a,b],N)$
for all $\eta\in AC_\cE([a,b],M)$, enabling us
to define a map
\[
AC_\cE([a,b],f)\colon AC_\cE([a,b],M)\to AC_\cE([a,b],N),\quad
\eta\mto f\circ \eta.
\]
\end{la}
\begin{proof}
For each $t\in [a,b]$, there
is a chart $\psi_t\colon P_t\to Q_t\sub F$ of $N$
such that $f(\eta(t))\in P_t$
and a chart
$\phi_t\colon U_t\to V_t\sub E$ of~$M$
such that $\eta(t)\in U_t$.
After shrinking $U_t$, we may assume that
$f(U_t)\sub P_t$.
There is $\ve>0$ such that
$[a,b]\cap [t-\ve,t+\ve]
\sub \eta^{-1}(U_t)$.
Then $\phi_t\circ \eta|_{[a,b]\cap [t-\ve,t+\ve]}
\in AC_\cE([a,b]\cap  [t-\ve,t+\ve],E)$
(by Lemma~\ref{goodinM}).
Since
\[
\psi_t\circ f\circ \eta|_{[a,b]\cap [t-\ve,t+\ve]}
=(\psi_t\circ f\circ (\phi_t)^{-1})|_{V_t}
\circ (\phi_t\circ \eta|_{[a,b]\cap [t-\ve,t+\ve]})
\]
and $C^k$-functions (like $(\psi_t\circ f\circ (\phi_t)^{-1})|_{V_t}$)
act on $AC_\cE$, we see that
$\psi_t\circ f\circ \eta|_{[a,b]\cap [t-\ve,t+\ve]}
\in AC_\cE([a,b]\cap [t-\ve,t+\ve],F)$.
Hence $fg\circ \eta\in AC_\cE([a,b],N)$,
by Remark~\ref{reallylocal}.
\end{proof}
%
%
%
%
%
%
\begin{defn}\label{defpfwax}
Let $\cE$ be a bifunctor on Fr\'{e}chet spaces (resp.,
sequentially complete (FEP)-spaces, resp.,
integral complete
locally convex spaces).
We say that \emph{smooth functions act smoothly
on~$AC_\cE$} if $C^\infty$-functions act on $AC_\cE$
and the following holds:
\begin{itemize}
\item[(S)]
The map
\[
AC_\cE([a,b],f)\colon AC_\cE([a,b],V)\to AC_\cE([a,b],F),\quad
\eta\mto f\circ \eta
\]
is $C^\infty$,
for all $a<b$ in $\R$, Fr\'{e}chet
spaces (resp., sequentially complete (FEP)-spaces, resp.,
integral complete locally convex spaces)
$E$ and $F$, each open subset $V\sub E$ and each smooth map
$f\colon V \to F$.
\end{itemize}
\end{defn}
\begin{la}\label{explderv}
In the situation of Definition~\emph{\ref{defpfwax}},
the map
\[
h:=AC_\cE([a,b],f)\colon AC_\cE([a,b],V)\to AC_\cE([a,b],F)
\]
satisfies
\begin{equation}\label{givesdf3}
d h=AC_\cE([a,b],df),
\end{equation}
identifying $AC_\cE([a,b],E)^2$ with $AC_\cE([a,b],E\times E)$.
\end{la}
\begin{proof}
It is well known that the map
\[
g:=C([a,b],f)\colon C([a,b],V)\to C([a,b],F),\quad \eta\mto f\circ \eta
\]
is $C^\infty$ and
\[
dg=C([a,b],df)\colon C([a,b],V\times E)\to C([a,b],F),\quad
(\eta_1,\eta_2)\mto df\circ(\eta_1,\eta_2)
\]
if we identify $C([a,b],V)\times C([a,b],E)$
with $C([a,b],V\times E)$ (cf.\ \cite{GCX} or \cite{GaN}).
As we assume that smooth functions
act smoothly on $AC_\cE$, we know that $h$ is
smooth.
Let $j_E\colon AC_\cE([a,b],E)\to C([a,b],E)$
and $j_F\colon AC_\cE([a,b],F)\to C([a,b],F)$
be the inclusion maps, which are continuous linear
(see Remark~\ref{intoC0}).
Then $j_F\circ h=g\circ j_E|_{AC_\cE([a,b],V)}$
and the Chain Rule yields:
\[
j_F\circ dh=dg\circ (j_E|_{AC_\cE([a,b],V)}\times j_E).
\]
Thus $dh(\eta_1,\eta_2)=j_F(dh(\eta_1,\eta_2))=dg(\eta_1,\eta_2)=df\circ (\eta_1,\eta_2)$
for all $\eta_1\in AC_\cE([a,b],V)$ and $\eta_2\in AC_\cE([a,b],E)$.
Thus (\ref{givesdf3}) is valid and the proof is complete.
\end{proof}
\begin{la}
Smooth functions act smoothly
on $AC_{L^p}$ for
$L^p$ as a bifunctor on Fr\'{e}chet spaces
$($or sequentially complete \emph{(FEP)}-spaces$)$,
for each $p\in [1,\infty]$.
Moreover, smooth functions
act smoothly on $AC_{L^\infty_{rc}}$ and $AC_R$,
for $L^\infty_{rc}$ and $R$ considered
as bifunctors on integral complete
locally convex spaces.
\end{la}
This follows from the following more detailed result:
\begin{la}\label{oursbetter}
Let $\cE$ be $L^p$ for some $p\in [1,\infty]$,
and $E$ as well as $F$
be sequentially complete \emph{(FEP)}-spaces
$($e.g., Fr\'{e}chet spaces$)$.
Or let $\cE$ be $L^\infty_{rc}$ or $R$,
and let $E$ as well as~$F$ be integral
complete locally convex spaces.
Let $a<b$ be real numbers, $V\sub E$ be open and
$f\colon V\to F$ be a $C^{k+2}$-map
for some $k\in \N_0\cup\{\infty\}$.
Then the map
\[
AC_\cE([a,b],f)\colon AC_\cE([a,b],V)\to AC_\cE([a,b],F),\quad \eta\mto f\circ \eta
\]
is $C^k$. If $k\geq 1$, then
\begin{equation}\label{givesdf2}
d g=AC_\cE([a,b],df),
\end{equation}
identifying $AC_\cE([a,b],E)^2$ with $AC_\cE([a,b],E\times E)$.
\end{la}
\begin{proof}
Note first that $[f\circ \eta]\in AC_\cE([a,b],F)$
by Lemma~\ref{C2acts},
since $f$ is at least~$C^2$.
Let $j_E\colon AC_\cE([a,b],E)\to C([a,b],E)$
and $j_F\colon AC_\cE([a,b],F)\to C([a,b],F)$
be the inclusion maps, which are continuous linear
(see Remark~\ref{intoC0}).
Because $\Phi\colon AC_\cE([a,b],F)\to C([a,b],F)\times \cE([a,b],F)$
is a topological embedding with closed
image, \cite[Lemmas~10.1 and 10.2]{BGN} show that
$g$ will be $C^k$ if we can prove that
\begin{equation}\label{twwa}
j_F\circ AC_\cE([a,b],f)=C([a,b],f)\circ j_E|_{AC_\cE([a,b],V)}
\end{equation}
and $D_F\circ AC_\cE([a,b],f)$ are $C^k$, where
\[
D_F\colon AC_\cE([a,b],F)\to \cE([a,b],F),\quad \eta\mto \eta'
\]
is the differentiation operator (which is continuous
linear).
Since
\[
C([a,b],f)\colon C([a,b],V)\to C([a,b],F), \quad
\eta\mto f\circ \eta
\]
is $C^{k+2}$ and hence $C^k$ (see \cite{GaN}, cf.\ \cite{GCX}),
we deduce from (\ref{twwa}) that $j\circ AC_{\cE}([a,b],f)$ is~$C^k$.
Let $\eta'=[\gamma]$.
Then $D_F(f\circ \eta)=[df\circ (\eta,\gamma)]$
by Remark~\ref{dercomp}, and thus
\[
D_F\circ AC_\cE([a,b], f)=\wt{df}
\]
with $\wt{f}\colon C([a,b],V)\times \cE([a,b],E)\to \cE([a,b],F)$, $(\eta,[\gamma])\mto
[df\circ (\eta,\gamma)]$.
Since $df$ is $C^{k+1}$,
Proposition~\ref{oponLp} (resp., Proposition~\ref{oponinfty})
show that $\wt{f}$ is~$C^k$.\\[2.3mm]
To verify the validity of (\ref{givesdf2}),
abbreviate $h:=C([a,b],f)$.
Then $j_F\circ g=h\circ j_E|_{AC_\cE([a,b],V)}$
(by (\ref{twwa}))
and the Chain Rule yields:
\[
j_F\circ dg=dh\circ (j_E|_{AC_\cE([a,b],V)}\times j_E).
\]
Since $dh=C([a,b],df)$ (cf.\ \cite{GCX}),
we see that $dg(\eta_1,\eta_2)=dh(\eta_1,\eta_2)=df\circ (\eta_1,\eta_2)$
for all $\eta_1\in AC_\cE([a,b],V)$ and $\eta_2\in AC_\cE([a,b],E)$.
Thus (\ref{givesdf2}) is valid and the proof is complete.
\end{proof}
\begin{la}\label{analsook}
Let $\cE$ be a bifunctor on Fr\'{e}chet spaces $($resp.,
sequentially complete \emph{(FEP)}-spaces, resp.,
integral complete
locally convex spaces$)$
which satisfies the locality axiom and such that
smooth functions act smoothly
on~$AC_\cE$.
Let
$E$ and $F$ be Fr\'{e}chet
spaces $($resp., sequentially complete \emph{(FEP)}-spaces, resp.,
integral complete locally convex spaces$)$
over $\K\in \{\R,\C\}$,
$V\sub E$ be an open set and
$f\colon V \to F$ be a $\K$-analytic map.
Then also the map
\[
AC_\cE([a,b],f)\colon AC_\cE([a,b],V)\to AC_\cE([a,b],F),\quad
\eta\mto f\circ \eta
\]
is $\K$-analytic,
for all $a<b$ in $\R$.
\end{la}
\begin{proof}
Assume $\K=\C$ first.
As we assume that smooth functions
act smoothly on $AC_\cE$, we know that $h:=AC_\cE([a,b],f)$ is
smooth.
By Lemma~\ref{explderv}, we have
\[
dh=AC_\cE([a,b],df),\quad (\eta_1,\eta_2)\mto df\circ (\eta_1,\eta_2)
\]
if we identify $\cE([a,b],E\times E)$ with $\cE([a,b],E)^2$.
Since $df(x,.)\colon E\to F$ is complex linear
for each~$x$, it follows that $dh(\eta_1,\eta_2)$ is complex
linear in~$\eta_2$, for each $\eta_1\in AC_\cE([a,b],E)$.
This implies that the smooth map $h$ is complex analytic
(see \cite{RES} of \cite{GaN}).

If $\K=\R$, let $\tilde{f}\colon \tilde{V}\to F_\C$
be a complex analytic extension of $f$
to an open subset $\tilde{V}\sub E_\C$ which contains~$V$.
Note that both $E_\C$ and $F_\C$ are Fr\'{e}chet spaces
(resp., sequentially complete (FEP)-spaces,
resp., integral complete locally convex spaces).
Hence $AC_\cE([a,b],\tilde{f})\colon AC_\cE([a,b],\tilde{V})\to AC_\cE([a,b], F_\C)$
is complex analytic.
Here $AC_\cE([a,b],F_\C)=AC_\cE([a,b],F)_\C$ (cf.\ Remark~\ref{willgvcx})
and $AC_\cE([a,b],E_\C)=AC_\cE([a,b],E)_\C$.
Thus $AC_\cE([a,b],\tilde{f})$
is a complex analytic extension of $AC_\cE([a,b],f)$
and hence $AC_\cE([a,b],f)$ is real analytic.
\end{proof}
\begin{la}\label{auchnch}
Let $\cE$ be a bifunctor on Fr\'{e}chet spaces
$($resp., sequentially complete \emph{(FEP)}-spaces, resp.,
integral complete locally convex spaces$)$
which satisfies the locality axiom \emph{(Loc)},
and such that smooth functions act smoothly on~$A_\cE$.
Let $E_1$, $E_2$ and $F$ be Fr\'{e}chet spaces
$($resp., sequentially complete \emph{(FEP)}-spaces,
resp., integral complete locally convex spaces$)$
over $\K\in \{\R,\C\}$.
If $\K=\R$, let $r\in \{\infty,\omega\}$;
if $\K=\C$, let $r=\omega$.
Let $a<b$ be real numbers, $M$ be a $C^r_\K$-manifold
modelled on~$E_1$, $V\sub E_2$ be open and
$f\colon M\times V\to F$ be a $C^r_\K$-map
Also, let $\zeta\in AC_\cE([a,b],M)$.
Then the map
\[
AC_\cE([a,b],V)\to AC_\cE([a,b],F),\quad \eta\mto f\circ (\zeta,\eta)
\]
is $C^r_\K$.
\end{la}
\begin{proof}
We find $a=t_0<t_1<\cdots<t_n=b$
such that $\zeta([t_{j-1},t_j])\sub U_j$
for some chart $\phi_j\colon U_j\to V_j\sub E_1$ of~$M$,
for all $j\in \{1,\ldots, n\}$.
Then
\[
f_j:=f\circ ((\phi_j)^{-1}\times \id_{E_2})\colon V_j\times E_2\to F
\]
is a $C^r_\K$-map for $j\in \{1,\ldots,n\}$.
Moreover, $\zeta_j:=\phi_j\circ \zeta|_{[t_{j-1},t_j]}\in AC_\cE([t_{j-1},t_j],V_j)$
and
\begin{equation}\label{ingredi1}
(f\circ (\zeta,\eta))|_{[t_{j-1},t_j]}=f_j\circ(\zeta_j,\eta|_{[t_{j-1},t_j]}).
\end{equation}
Since $\cE$ satisfies the locality axiom, the map
\[
AC_\cE([a,b],F)\to
\prod_{j=1}^n AC_\cE([t_{j-1},t_j],F),
\quad \eta\mto (\eta|_{[t_{j-1},t_j]})_{j\in \{1,\ldots,n\} }
\]
is a linear topological embedding with closed
image.\footnote{The image consists of all
$(\eta)_{j\in\{1,\ldots, n\}}\in \prod_{j=1}^n AC_\cE([t_{j-1},t_j],F)$
such that $\eta_j(t_j)=\eta_{j+1}(t_j)$
for all $j\in\{1,\ldots, n-1\}$.}
Hence, by Lemmas 10.1 and 10.2 in \cite{BGN}
(and analogous lemmas for analytic maps in \cite{GaN}),
we need only show that the maps
\[
h_j\colon AC_\cE([a,b],V)\to AC_\cE([a,b],F),\quad \eta\mto (f\circ (\zeta,\eta))|_{[t_{j-1},t_j]}
\]
are $C^r_\K$ for all $j\in \{1,\ldots, n\}$.
The mappings
\[
AC_\cE([a,b],E_2)\to AC_\cE([t_{j-1},t_j],E_2),\quad \eta\mto\eta|_{[t_{j-1},t_j]}
\]
are continuous $\K$-linear for $j\in \{1,\ldots,n\}$,
whence the maps
\[
\rho_j\colon AC_\cE([a,b],V)\to AC_\cE([t_{j-1},t_j],V),\quad \eta\mto\eta|_{[t_{j-1},t_j]}
\]
are $C^r_\K$.
The map
\[
AC_\cE([t_{j-1},t_j],f_j)\colon AC_\cE([t_{j-1},t_j],V_j\times V)
\to AC_\cE([t_{j-1},t_j,F)
\]
is $C^r_\K$ as we assume that smooth functions act smoothly on $AC_\cE$
(see Lemma~\ref{analsook} if $r=\omega$).
Identifying $AC_\cE([t_{j-1},t_j],V_j\times V)$ with
\[
AC_\cE([t_{j-1},t_j],V_j)\times AC_\cE([t_{j-1},t_j],V),
\]
we have
\[
h_j(\eta)=AC_\cE([t_{j-1},t_j],f_j)
(\zeta_j,\rho_j(\eta))
\]
for all $\eta\in AC_\cE([a,b],V)$ (exploiting (\ref{ingredi1})).
Hence $h_j$ is $C^r_\K$, which completes the proof.
\end{proof}
\section{The Lie groups {\boldmath$AC_\cE([0,1],G)$}}\label{secACG}
In this section,
$\K\in \{\R,\C\}$.
If $\K=\R$, we let $r\in \{\infty,\omega\}$;
if $\K=\C$, we let $r:=\omega$.
We recall the local description of Lie group structures.
\begin{numba}\label{localdesc}
Let $U$ be a group and $U\sub G$ be a subset
such that $U= U^{-1}$ and $e\in U$.
Assume that $U$ is endowed with a $C^r_\K$-manifold structure modelled on
a locally convex topological $\K$-vector space~$E$
such that
\[
U\to U,\quad x\mto x^{-1}
\]
is $C^r_\K$, $D_U:=\{(x,y)\in U\times U\colon xy\in U\}$
is open in $U\times U$ and the map
\[
U\times U\to U,\quad (x,y)\mto xy
\]
is $C^r_\K$.
Also, assume that for each $g\in G$, there is an open identity
neighbourhood $W\sub U$ such that $gWg^{-1}\sub U$,
and that the map
\[
W\to U,\quad x\mto gxg^{-1}
\]
is $C^r_\K$.
Then there is a unique $C^r_\K$-manifold structure on~$G$
modelled on~$E$ which makes $G$ a $C^r_\K$-Lie group
and such that $G$ has $U$ (with its original manifold structure)
as an open $C^r_\K$-submanifold.
\end{numba}
\begin{prop}\label{propACLie}
Let $\cE$ be a bifunctor on Fr\'{e}chet spaces
$($resp., sequentially complete \emph{(FEP)}-spaces,
resp., integral complete locally convex spaces$)$
which satisfies the locality axiom \emph{(Loc)}
and such that smooth functions act
smoothly on $AC_\cE$.
Let $G$ be a $C^r_\K$-Lie group modelled
on a Fr\'{e}chet space $($resp.,
on a sequentially complete \emph{(FEP)}-space, resp.,
on an integral complete locally convex space$)$~$E$
over~$\K$.
Then
$AC_\cE([0,1],G)$ is a group under pointwise multiplication
and there is a unique $C^r_\K$-Lie group structure on
$AC_\cE([0,1],G)$ such that
\[
AC_\cE([0,1],U):=\{\eta\in AC_\cE([0,1],G)\colon \eta([0,1])\sub U\}
\]
is open in $AC_\cE([0,1],G)$ and
\[
AC_\cE([0,1],\phi)\colon AC_\cE([0,1],U)\to AC_\cE([0,1],V)
\]
is a $C^r_\K$-diffeomorphism for each
chart $\phi\colon U\to V$ of~$G$
such that $e\in U$ and $U=U^{-1}$.
\end{prop}
\begin{proof}
\emph{Existence of a Lie group structure}.
Let $m_G\colon G\times G\to G$ be the group multiplication
and $j_G\colon G\to G$ be the inversion map, $x\mto x^{-1}$.
Using Lemma~\ref{actonelts}, we see that
$\wt{G}:=AC_\cE([0,1],G)$ becomes a group if we define
the multiplication $m_{\wt{G}}$ via
\[
\eta_1\eta_2:=AC_\cE([0,1],m_G)(\eta_1,\eta_2)
\]
(identifying $AC_\cE([0,1],G)\times AC_\cE([0,1],G)$
with $AC_\cE([0,1],G\times G)$)
and the inversion $j_{\wt{G}}$ via
\[
\eta^{-1}:=AC_\cE([0,1],j_G)(\eta)
\]
for $\eta,\eta_1,\eta_2\in AC_\cE([0,1],G)$.
Pick a chart
$\phi\colon U\to V$ of~$G$
such that $e\in U$ and $U=U^{-1}$.
Then $\wt{V}:=AC_\cE([0,1],V)$ is an open subset
of $AC_\cE([0,1],E)$.
By Lemma~\ref{goodinM}, the map
\[
\wt{\phi}:=AC_\cE([0,1],\phi)
\colon AC_\cE([0,1],U)
\to AC_\cE([0,1],V),\;\,
\eta\mto \phi\circ \eta
\]
is a bijection.
We endow $\wt{U}:=AC_\cE([0,1],U)$
with the $C^r_\K$-manifold structure which turns
$AC_\cE([0,1],\phi)$ into a global chart.
Then
\[
D_U:=\{(x,y)\in U\times U\colon xy\in U\}
\]
is open in~$U\times U$, and hence
$D_V:=(\phi\times \phi)(D_U)$ is open
in $V\times V$.
The mappings
\[
j\colon V\to V,\quad j(x):=\phi((\phi^{-1}(x))^{-1})
\]
and
\[
m \colon D_V\to V,\quad (x,y)\mto \phi(\phi^{-1}(x)\phi^{-1}(y))
\]
(which express the group inversion and
group multiplication in the local chart) are~$C^r_\K$.
Now $AC_\cE([0,1],V)$ is an open subset of $AC_\cE([0,1],E)$
and $AC_\cE([0,1],D_V)$ is an open subset of $AC_\cE([0,1],E\times E)
\cong AC_\cE([0,1],E)\times AC_\cE([0,1],E)$.
Identifying $AC_\cE([0,1],G\times G)$ with
the direct product\linebreak
$AC_\cE([0,1],G)\times AC_\cE([0,1],G)$
as a set, we have
\begin{eqnarray*}
\lefteqn{AC_\cE([0,1],D_U)}\qquad\qquad\\
&=&
\!\!\{(\eta_1,\eta_2)\in AC_\cE([0,1],G\times G)\colon \!(\forall t\in [0,1])
\,
(\eta_1(t),\eta_2(t))\in D_U\}\\
&=& \!\! \{(\eta_1,\eta_2)\in \wt{U}\times \wt{U}\colon
\eta_1\eta_2 \in \wt{U}\}\;=: \; D_{\wt{U}}.
\end{eqnarray*}
Since smooth functions act smoothly on $AC_\cE$
(and appealing to Lemma~\ref{analsook} if $r=\omega$),
the map
\[
j_{\wt{G}}=(\wt{\phi})^{-1}\circ AC_\cE([0,1],j)\circ \wt{\phi}
\]
is $C^r_\K$. Likewise,
\[
m_{\wt{G}}=(\wt{\phi})^{-1}\circ AC_\cE([0,1],m)\circ (\wt{\phi}\times \wt{\phi})|_{D_{\wt{U}}}
\]
is $C^r_\K$.
If $\zeta\in \wt{G}$, then $K:=\zeta([0,1])\sub G$ is compact.
Because the map
\[
h\colon G\times G\to G,\quad h(x,y):=xyx^{-1}
\]
is continuous and $h(K\times \{e\})=\{e\}\sub U$,
the Wallace Lemma (see \ref{Wallace})
provides open sets $P,W\sub G$
such that $K\times \{e\}\sub P\times W\sub h^{-1}(U)$.
We may assume that $P\sub U$.
As a consequence,
\[
\zeta\eta\zeta^{-1}=h\circ (\zeta,\eta)\in AC_\cE([0,1],U)=\wt{U}
\]
for all $\eta\in AC_\cE([0,1],W):=\wt{W}$.
The map
\[
f\colon P\times W\to V,\quad f(x,y):=\phi(h(x,\phi^{-1}(y)))
\]
is smooth, and the map $\wt{W}\to \wt{U}$,
\[
\eta\mto \zeta\eta\zeta^{-1}= h\circ (\zeta,\eta)
=\wt{\phi}^{-1}(f\circ (\zeta,\wt{\phi}(\eta)))
\]
is $C^r_\K$ by Lemma~\ref{auchnch}.
Hence \ref{localdesc} provides a unique
$C^r_\K$-Lie group structure on
$\wt{G}$ modelled on $AC_\cE([a,b],E)$ such
that $\wt{U}$ is an open submanifold.

\emph{Uniqueness of the Lie group structure}.
If also
$\phi_1\colon U_1\to V_1$ is a chart of~$G$
such that $e\in U_1$ and $U_1=(U_1)^{-1}$,
then $U\cap U_1$ is open in $U$ and hence
$AC_\cE([0,1],\phi(U\cap U_1))$ is open
in $\wt{U}$, entailing that
$AC_\cE([0,1],U\cap U_1)$ is open
in $\wt{U}$.
Likewise,
$AC_\cE([0,1],U\cap U_1)$ is open
in $\wt{U}_1:=AC_\cE([0,1],U_1)$,
endowed with the $C^r_\K$-manifold structure
making
\[
\wt{\phi}_1:=AC_\cE([0,1],\phi_1)\colon
AC_\cE([0,1],U_1)\to AC_\cE([0,1],V_1)
\]
a global chart for $\wt{U}_1$.
Write $\wt{G}_1$ for $AC_\cE([0,1],G)$,
endowed with the unique $C^r_\K$-Lie group structure
making $\wt{U}_1$ an open submanifold.
Then $\id\colon \wt{G}\to \wt{G}_1$ is a group
homomorphism which is $C^r_\K$ on the open identity neighbourhood
$\wt{U}\cap\wt{U}_1$ (as we
presently verify) and hence $C^r_\K$.
Likewise, $\id\colon \wt{G}_1\to\wt{G}$ is
$C^r_\K$, and thus $\wt{G}=\wt{G}_1$
as a $C^r_\K$-Lie group.
Here, we used that
\[
\id|_{\wt{U}\cap\wt{U}_1}
=(\wt{\phi}_1)^{-1}\circ AC_\cE([0,1], \phi_1\circ \phi^{-1}|_{\phi(\wt{U}\cap\wt{U}_1)})\circ
\wt{\phi}|_{\wt{U}\cap\wt{U}_1}
\]
is $C^r_\K$.
\end{proof}
%
%
%
%
%
%
%
\begin{rem}\label{intoC}
In the situation of Proposition~\ref{propACLie},
the inclusion map
\[
\ve\colon AC_\cE([0,1],G)\to C([0,1],G)
\]
is a group homomorphism and $C^r_\K$.\\[2.3mm]
In fact, for $\phi$, $\wt{\phi}$ and $\wt{U}$ as in the preceding proof,
the map
\[
C([0,1],\phi)\colon
C([0,1],U)\to C([0,1],V)\sub C([0,1],E)
\]
is a chart for $C([0,1],G)$.
The inclusion map~$\ve$
is $C^r_\K$ on the open identity
neighbourhood $\wt{U}$ since
\[
C([0,1],\phi)\circ \ve|_{\wt{U}}\circ (\wt{\phi})^{-1}
\]
is the inclusion map $AC_\cE([0,1],V)\to C([0,1],V)$,
which is a restriction of the continuous linear inclusion map
$AC_\cE([0,1],E)\to C([0,1],E)$ and hence~$C^r_\K$.
Since $\ve$ is a group homomorphism, it follows
that $\ve$ is~$C^r_\K$.
\end{rem}
\begin{rem}
It is well known that the
evaluation map $\ev_t\colon C([0,1],G)\to G$, $\eta\mto \eta(t)$
is a group homomorphism and $C^r_\K$
for each $t\in [0,1]$.
By the previous remark, also
the evaluation map
$\ev_t\circ \ve\colon AC_\cE([0,1],G)\to G$, $\eta\mto\eta(t)$
is a group homomorphism and $C^r_\K$, for each $t\in [0,1]$.
\end{rem}
\begin{numba}
A local $C^r_\K$-Lie group
is a quintuple $(U,e,D_U,m,j)$,
where $U$ is a $C^r_K$ manifold
modelled on a locally convex topological $\K$-vector space~$E$,
$e\in U$, $D_U$ an open subset of $U\times U$
such that $(U\times\{e\})\cup(\{e\}\times U)\sub D_U$
and
\[
m\colon D_U\to U,\quad (x,y)\mto m(x,y)=:xy
\]
and $j\colon U\to U$, $x\mto j(x)=:x^{-1}$
are $C^r_\K$-maps satisfying axioms
as in \cite{GaN} or \cite{SUR}.
Then $T_e(U)$ is a Lie algebra.
\end{numba}
For example, every open identity symmetric\footnote{That is, $U=U^{-1}$.}
neighbourhood
$U$ in a $C^r_\K$-Lie group~$G$ is a
local $C^r_\K$-Lie group with $D_U:=\{(x,y)\in U\times U\colon xy\in U\}$.
\begin{defn}
Let $\cE$ be a bifunctor on Fr\'{e}chet spaces
(resp., sequentially complete (FEP)-spaces,
resp., integral complete locally convex spaces)
which satisfies the locality axiom (Loc)
and such that smooth functions act
smoothly on $AC_\cE$.
Let
$(U,e,D_U,m,j)$ be a local $C^r_\K$-Lie group
which is modelled
on a Fr\'{e}chet space (resp.,
on a sequentially complete (FEP)-space, resp.,
on an integral complete locally convex space)~$E$
over~$\K$
and admits a global chart
$\phi\colon U\to V\sub E$.
We then consider $\wt{U}:=AC_\cE([a,b],U)$
as a local $C^r_\K$-Lie group
with the global chart $AC_\cE([a,b],\phi)$,
\[
D_{\wt{U}}:=AC_\cE([a,b],D_U)
\]
(identifying $AC_\cE([a,b], E^2)$ with $AC([a,b],E)^2$),
multiplication $AC_\cE([a,b], m)$ and the inversion map
$AC_\cE([a,b],j)$.
\end{defn}
\begin{defn}
If $(U_1,e_1,D_{U_1},m_1,j_1)$
and $(U_2,e_2,D_{U_2},m_2,j_2)$
are local $C^r_\K$-Lie groups,
then we call a map $f\colon U_1\to U_2$
a \emph{local group homomorphism}
if $(f\times f)(D_{U_1})\sub D_{U_2}$,
$m_2\circ (f\times f)|_{D_{U_1}}=f\circ m_1$
and $j_2\circ f=f\circ j_1$.
If, moreover, $f$ is $C^r_\K$, we call $f$
a \emph{$C^r_\K$-homomorphism between local $C^r_\K$-Lie groups}.
\end{defn}
%
\begin{la}\label{pbckgp}
Let $\cE$ be a bifunctor on
Fr\'{e}chet spaces $($resp., sequentially complete \emph{(FEP)}-spaces, resp., integral
complete locally convex spaces$)$
which satisfies the locality axiom and such that
smooth functions act smoothly on $AC_\cE$.
Let $E$ be Fr\'{e}chet space $($resp., sequentially complete \emph{(FEP)}-space, resp., integral
complete locally convex space$)$
and
$G$ be a $C^r_\K$-Lie group $($or a local $C^r_\K$-Lie group admitting a global
chart$)$
modelled on~$E$.
Then the following holds:
\begin{itemize}
\item[\rm(a)]
If
$a$, $b$, $\alpha$, $\beta$, $c$ and $d$ are real numbers
with $a\leq \alpha<\beta\leq b$ and $c<d$,
let $f\colon [c,d]\to [a,b]$
be as in~\emph{(\ref{fafflin})}.
Then
$\eta\circ f\in AC_\cE([c,d],G)$ for each $\eta\in AC_\cE([a,b],G)$
and the $($local$)$ group homomorphism
\[
AC_\cE(f,G)\colon AC_\cE([a,b],G)\to AC_\cE([c,d],G),\quad \eta \mto \eta\circ f
\]
is $C^r_\K$.
\item[\rm(b)]
Let $n\in \N$ and $a=t_0<t_1<\cdots<t_n=b$ be real numbers.
If $G$ is a Lie group,
then the map
\[
\Phi\colon AC_\cE([a,b],G)\to\prod_{j=1}^nAC_\cE([t_{j-1},t_j],G),\quad
\eta\mto(\eta|_{[t_{j-1},t_j]})_{j\in\{1,\ldots, n\}}
\]
is a $C^r_\K$-homomorphism. Moreover, $\Phi$ is
a $C^r_\K$-diffeomorphism onto a $C^r_\K$-submanifold
of $\prod_{j=1}^nAC_\cE([t_{j-1},t_j],G)$.
If $G$ is a local $C^r_\K$-Lie group admitting a global chart,
then $\Phi$ is
a $C^r_\K$-homomorphism between $($local$)$ $C^r_\K$-Lie groups
and a $C^r_\K$-diffeomorphism onto a $C^r_\K$-submanifold
of $\prod_{j=1}^nAC_\cE([t_{j-1},t_j],G)$.
\item[\rm(c)]
For each $t_0\in [a,b]$,
\[
\ev_{t_0}\colon AC_\cE([0,1],G)\to G,\quad \eta\mto \eta(t_0)
\]
is a $C^r_\K$-homomorphism. Identifying the Lie algebra~$\ch$
of $AC_\cE([a,b],G)$ with $AC_\cE([a,b],\cg)$
by means of $d (AC_\cE([a,b],\phi))|_{\ch}$ $($where
$\phi\colon U_\phi\to V_\phi$ is a chart of~$G$ with $e\in U_\phi)$,
the tangent map $L(\ev_{t_0})$ is the point evaluation
$\ve_{t_0}\colon AC_\cE([a,b],\cg)\to \cg$, $\eta\mto \eta(t_0)$.
\end{itemize}
\end{la}
\begin{proof}
(a)
Using Lemma~\ref{closedemb},
we see that the
map
\begin{equation}\label{mightbla}
AC_\cE(f,E)\colon AC_\cE([a,b],E)\to AC_\cE([c,d],E), \;\;
\eta\mto \eta\circ f
\end{equation}
is continuous linear (since $\cE(f,E)$ is continuous linear by (B2)
and also $C(f,E)$ is continuous linear, see, e.g., \cite{GCX} or \cite{GaN}).
Let $\eta\in AC_\cE([a,b],G)$.
If $t\in [c,d]$, then there exist $\ve>0$
and a chart $\psi\colon U_\psi\to V_\psi$ for~$G$
such that $\eta([f(t)-\ve,f(t)+\ve]\cap [a,b])\sub U_\psi$
and
\[
\psi\circ \eta|_{[f(t)-\ve,f(t)+\ve]\cap [a,b]}\in
AC_\cE([f(t)-\ve,f(t)+\ve]\cap [a,b],E).
\]
We have $f([t-\delta,t+\delta]\cap[c,d])\sub [f(t)-\ve,f(t)+\ve]
\cap [a,b]$ for some $\delta>0$. Write $f_t$ for the restriction of $f$ to a map
$[t-\delta,t+\delta]\cap[c,d]\to [f(t)-\ve,f(t)+\ve]$.
Then
\[
\psi\circ (\eta\circ f)|_{[t-\delta,t+\delta]\cap[c,d]}=
(\psi\circ \eta|_{[f(t)-\ve,f(t)+\ve]\cap [a,b]})\circ f_t
\]
is an element of $AC_\cE([t-\delta,t+\delta]\cap[c,d],E)$
and hence $\eta\circ f\in AC_\cE([c,d],G)$, by
Remark~\ref{reallylocal}.\\[2.3mm]
If $\phi\colon U\to V\sub E$ is a chart for~$G$ defined
on an open symmetric identity neighbourhood $U\sub G$,
then
$AC_\cE([a,b],\phi)$ is chart for $AC_\cE([a,b],G)$
around~$e$ and
$AC_\cE([c,d],\phi)$ is chart for $AC_\cE([c,d],G)$
around~$e$.
Since
\[
AC_\cE([c,d],\phi)\circ AC_\cE(f,G)\circ AC_\cE([a,b],\phi)^{-1}
=AC_\cE(f,E)|_{AC_\cE([a,b],V)}
\]
is $C^r_\K$,
the (local) group homomorphism
$AC_\cE(f,G)$ is $C^r_\K$.

(b) The image of $\Phi$ is the set
\[
\{
\{(\eta_j)_{j\in\{1,\ldots, n\}}\in \prod_{j=1}^nAC_\cE([t_{j-1},t_j],G)
\colon (\forall j\in \{2,\ldots,n\})
\eta_{j-1}(t_j)=\eta_j(t_j)\}.
\]
Let $\phi\colon U\to V\sub E$ be a chart for~$G$
such that $U\sub G$ is a symmetric open identity neighbourhood.
Then
\[
\psi:=
\prod_{j=1}^n AC_\cE([t_{j-1},t_j],\phi)
\colon
\prod_{j=1}^nAC_\cE([t_{j-1},t_j],U)
\to
\prod_{j=1}^nAC_\cE([t_{j-1},t_j],V)
\]
is a chart for $\prod_{j=1}^nAC_\cE([t_{j-1},t_j],G)$.
Since
\[
F:=\{(\eta_j)_{j\in\{1,\ldots, n\}}\in\prod_{j=1}^nAC_\cE([t_{j-1},t_j],E)
\colon (\forall j\in \{2,\ldots,n\})
\eta_{j-1}(t_j)=\eta_j(t_j)\}
\]
is a closed vector subspace of $\prod_{j+1}^nAC_\cE([t_{j-1},t_j],E)$
and $\psi$ takes
\[
\im(\Phi)\cap \prod_{j=1}^nAC_\cE([t_{j-1},t_j],U)
\]
onto $F\cap \prod_{j=1}^nAC_\cE([t_{j-1},t_j],V)$,
we have found a submanifold chart for $\im(\Phi)$ around~$e$.
If $G$ is a local Lie group
with a global chart,
then, by the preceding,
$\im(\phi)$ is a $C^r_\K$-submanifold
of the direct product, if we choose $\phi$ as the global chart.
If $G$ is a Lie group,
then $\im(\Phi)$ is a subgroup
and translates of $\psi$ provide submanifold charts
around each point in $\im(\Phi)$,
whence $\im(\Phi)$ is a submanfold
of the product modelled on~$F$.
In either case,
since $\psi|_{\im(\Phi)}\circ \Phi\circ AC_\cE([a,b],\phi^{-1})$ is the restriction
to the open set $AC_\cE([a,b],V)$ of the isomorphism of topological
vector spaces
\[
AC_\cE([a,b],E)\to F,
\quad
\eta\mto(
\eta|_{[t_{j-1},t_j]})_{j\in\{1,\ldots,n\}},
\]
we deduce that $\Phi$ is a $C^r_\K$-diffeomorphism
from $AC_\cE([a,b],G)$ to the $C^r_\K$-submanifold $\im(\Phi)$.

(c) Follows from $\phi\circ \ev_{t_0}\circ AC_\cE([a,b],\phi)^{-1}=\ve_{t_0}$.
\end{proof}
\begin{la}\label{strgp}
Let $\cE$ be a bifunctor on
Fr\'{e}chet spaces $($resp., sequentially complete \emph{(FEP)}-spaces, resp., integral
complete locally convex spaces$)$
which satisfies the locality axiom and such that
smooth functions act smoothly on $AC_\cE$.
Let $E$ be Fr\'{e}chet space $($resp., sequentially complete \emph{(FEP)}-space,
resp., integral
complete locally convex space$)$
and
$G$ be a $C^r_\K$-Lie group
modelled on~$E$. Let $a<b$ be real numbers.
Then $AC_\cE([a,b],G)_*:=\{\eta\in AC_\cE([a,b],G)\colon \eta(a)=e\}$
is a $C^r_\K$-Lie subgroup of $AC_\cE([a,b],G)$ and the map
\[
\Psi\colon AC_\cE([0,1],G)_*\times G\to AC_\cE([0,1],G),\quad (\eta,g)\mto (t\mto \eta(t)g)
\]
is a $C^r_\K$-diffeomorphism.
\end{la}
\begin{proof}
Let $U\sub G$ be a symmetric open identity
neighbourhood on which a chart $\phi\colon U\to V\sub E$ of~$G$
is defined.
Then $AC_\cE([a,b],\phi)$ is chart of $AC_\cE([a,b],G)$ which maps the set
\[
AC_\cE([a,b],U)\cap AC_\cE([a,b],G)_*
\]
onto $AC_\cE([a,b],V)\cap AC_\cE([a,b],E)_*$.
Hence $AC_\cE([a,b],G)_*$ is a $C^r_\K$-submanifold
of $AC_\cE([a,b],G)$ modelled on A$C_\cE([a,b],E)_*$.\\[2.3mm]
The group homomorphism $j_G\colon G\to AC_\cE([a,b],G)$, $j(g)(t):=g$
is $C^r_\K$ since, for each chart $\phi\colon U\to V\sub E$
of $G$ with $e\in U$, we have that
\[
AC_\cE([a,b],\phi)\circ j_G\circ \phi^{-1}=j_E|_V
\]
with the linear map $j_E\colon E\to AC_\cE([a,b],E)$, $j_E(v)(t):=v$.
The linear map $j_E$
is continuous (as $(\Phi\circ j_E)(v)=(v,0)$ for
$\Phi$ as in Definition~\ref{mostbaac}
with $t_0:=a$). Thus $j_G$ is $C^r_\K$
and hence $\Psi$ is $C^r_\K$,
since the group multiplication $\mu$ of $AC_\cE([a,b],G)$
is $C^r_\K$ and $\Psi(\eta,g)=\eta j_G(g)=\mu(\eta,j_G(g))$.
The evaluation map $\ev_a\colon AC_\cE([a,b],G)\to G$, $\eta\mto \eta(a)$
is $C^r_\K$ by Lemma~\ref{pbckgp}\,(c).
The map
\[
\Psi^{-1}\colon AC_\cE([a,b],G)\!\to\! AC_\cE([a,b],G)_*\times G,\;
\eta\mto (\mu((j_G(\ev_a(\eta)))^{-1},\eta),\,\ev_a(\eta))
\]
is $C^r_\K$ as a map to $AC_\cE([a,b],G)\times G$
and hence also $C^r_\K$ as a map to the $C^r_\K$-submanifold
$AC_\cE([a,b],G)_*\times G$.
Thus $\Psi$ is a $C^r_\K$-diffeomorphism.
\end{proof}
\begin{defn}
Let $\cE$ be a bifunctor
on
Fr\'{e}chet spaces (resp., on sequentially complete (FEP)-spaces,
resp., on integral complete
locally convex space). We say that $\cE$ satisfies the \emph{embedding axiom}
if the following holds:
\begin{itemize}
\item[(E)]
For real numbers $a<b$ and each Fr\'{e}chet space (resp., sequentially complete (FEP)-spaces,
resp., integral complete locally convex space) $E$
and closed vector subspace~$F$,
we have
\begin{equation}\label{eqgvclo}
AC_\cE([a,b],F)=\{\eta\in AC_\cE([a,b],E)\colon \eta([a,b])\sub F\},
\end{equation}
and the linear map
\[
\cE([a,b],j)\colon \cE([a,b],F)\to \cE([a,b], E),\quad [\gamma]\mto[j\circ\gamma]
\]
induced by the inclusion map $j\colon F\to E$
is a topological embedding.
\end{itemize}
As the evaluation maps $AC_\cE([a,b],E)\to E$, $\eta\mto \eta(t)$
are continuous, (\ref{eqgvclo}) implies that
$AC_\cE([a,b],F)$ is a \emph{closed} vector subspace
of $AC_\cE([a,b],E)$.
\end{defn}
\begin{rem}
Lemma~\ref{henceemba}
implies that
$L^p$ as a bifunctor on Fr\'{e}chet spaces
(for any $p\in [1,\infty]$)
and $L^\infty_{rc}$ as a bifunctor on integral complete locally convex
spaces satisfy the embedding
axiom.\footnote{For $R$, the author
would not expect this. For $L^p$ as a bifunctor on
sequentially complete (FEP)-spaces,
the author did not succeed to prove the property.}
\end{rem}
\begin{la}\label{intsubmf}
Let $\cE$ be a bifunctor
on Fr\'{e}chet spaces $($resp., on sequentially complete \emph{(FEP)}-spaces,
resp., on integral complete
locally convex space$)$
which satisfies the locality axiom
and
such that smooth functions act on~$AC_\cE$.
Let $M$ be a $C^r_\K$-manifold
modelled on such a space~$E$
and $N\sub M$ be a $C^r_\K$-submanifold
modelled on a closed vector subspace
$F\sub E$. Let $a<b$ be real numbers.
If $F$ is complemented in~$E$ or
$\cE$ satisfies the embedding axiom,
then a map $\eta\colon[a,b]\to N$
is in $AC_\cE([a,b],N)$ if and only if it is in
$AC_\cE([a,b],M)$.
\end{la}
\begin{proof}
Let $j\colon N\to M$ be the inclusion map. If $\eta\in AC_\cE([a,b],N)$,
then $j\circ \eta\in AC_\cE([a,b],M)$, by Lemma~\ref{actonelts}.

If, conversely, $\eta\in AC_\cE([a,b],M)$
and $t\in [a,b]$, we find a chart $\phi\colon U\to V\sub E$
of~$M$ such that $\phi(U\cap N)=V\cap F$.
There is $\delta>0$ such that $\eta([a,b]\cap [t-\delta,t+\delta])\sub U$.
Then $\psi:=\phi|_{U\cap N}\colon U\cap N\to V\cap F$
is chart for~$N$,
and
\[
\psi\circ\eta|_{[a,b]\cap [t-\delta,t+\delta]}
\]
is a mapping to~$F$ which is in
$AC_\cE([a,b]\cap [t-\delta,t+\delta],E)$
(as it coincides with $\phi\circ\eta|_{[a,b]\cap [t-\delta,t+\delta]}$
as a mapping to~$E$).
If $\cE$ satisfies the embedding axiom, this implies that
\begin{equation}\label{alsowithoute}
\psi\circ\eta|_{[a,b]\cap [t-\delta,t+\delta]}\in
AC_\cE([a,b]\cap [t-\delta,t+\delta],F).
\end{equation}
If $F$ is complemented in~$E$, then we find a continuous linear
map $\lambda\colon E\to F$ such that $\lambda|_F=\id_F$.
Again,
\begin{eqnarray*}
\psi\circ\eta|_{[a,b]\cap [t-\delta,t+\delta]}
&=& \lambda\circ \phi\circ\eta|_{[a,b]\cap [t-\delta,t+\delta]}\\
&=& AC_\cE([a,b]\cap [t-\delta,t+\delta],\lambda)
(\phi\circ\eta|_{[a,b]\cap [t-\delta,t+\delta]})\\
&\in&
AC_\cE([a,b]\cap [t-\delta,t+\delta],F).
\end{eqnarray*}
Hence $\eta\in AC_\cE([a,b],N)$, by Remark~\ref{reallylocal}.
\end{proof}
\begin{la}\label{ACHsubACG}
Let $\cE$ be a bifunctor
on Fr\'{e}chet spaces $($resp., on sequentially complete \emph{(FEP)}-spaces,
resp., on integral complete
locally convex space$)$
which satisfies the locality axiom
and
such that smooth functions act on~$AC_\cE$.
Let $G$ be a $C^r_\K$-Lie group
modelled on such a space~$E$
and $H\sub G$ be a $C^r_\K$-Lie subgroup
modelled on a closed vector subspace
$F\sub E$. Let $a<b$ be real numbers.
If $F$ is complemented in~$E$
or $\cE$
satisfies the embedding axiom,
then $AC_\cE([a,b],H)$ is a $C^r_\K$-submanifold
of $AC_\cE([a,b],G)$.
\end{la}
\begin{proof}
Let $\phi\colon U\to V\sub E$ be a chart
for $G$ defined on a symmetric open identity neighbourhood
$U\sub G$ such that
\begin{equation}\label{lttl1}
\phi(U\cap H)=V\cap F.
\end{equation}
Since $F$ is complemented in $E$
(in which case $AC_\cE([a,b],)=AC_\cE([a,b],F)\oplus AC_\cE([a,b],Y)$
if $E=F\oplus Y$) or $\cE$ satisfies the embedding axiom,
we have that
\begin{equation}\label{lttl2}
AC_\cE([a,b], F)=\{\eta\in AC_\cE([a,b], E)\colon \eta([a,b])\sub F\}
\end{equation}
is a closed vector subspace of $AC_\cE([a,b],E)$ and that the inclusion map
$AC_\cE([a,b],F)\to AC_\cE([a,b],E)$ is a topological embedding.
As the chart\linebreak
$AC_\cE([a,b],\phi)$ of $AC_\cE([a,b], G)$
takes
\[
AC_\cE([a,b],U)\cap AC_\cE([a,b], H)
\]
onto the set $AC_\cE([a,b],V)\cap AC_\cE([a,b], F)$ by (\ref{lttl1}) and (\ref{lttl2}),
we deduce that the subgroup $AC_\cE([a,b], H)$ is a $C^r_\K$-submanifold
of $AC_\cE([a,b], G)$ modelled on
$AC_\cE([a,b],F)$.
Since $AC_\cE([a,b],\phi)$ restricts to the chart
$AC_\cE([a,b], \phi|_{U\cap H}^{V\cap F})$ of $AC_\cE([a,b], H)$,
the given $C^r_\K$-Lie group structure
on $AC_\cE([a,b],H)$ coincides with the manifold structure
as a $C^r_\K$-submanifold of $AC_\cE([a,b], G)$.
\end{proof}
\begin{rem}\label{L1inpart}
Consider a strict (LF)-space $E$ and
a vector subspace $F\sub E$ which is a Fr\'{e}chet space in the induced
topology.
Then the conclusion of Lemma~\ref{intsubmf}
remains valid for $\cE=L^1$
because (\ref{alsowithoute})
is satisfied by Lemma~\ref{henceemba}.
As a consequence, also
the conclusion of Lemma~\ref{ACHsubACG}
remains valid for $\cE=L^1$ whenever $E$ is a strict (LF)-space
and $F\sub E$ a Fr\'{e}chet subspace.
\end{rem}
\section{{\boldmath$\cE$}-regularity and local
{\boldmath$\cE$}-regularity}\label{secintroreg}
\begin{defn}\label{defetadot}
Let $\cE$ be a bifunctor on Fr\'{e}chet spaces
(resp., on sequentially complete (FEP)-spaces,
resp., on integral complete locally convex spaces)
which satisfies the locality axiom, and such that smooth functions
act on $AC_\cE$.
Let $E$ be a Fr\'{e}chet space (resp., a sequentially complete (FEP)-space,
resp., an integral complete locally
convex space) and
$M$ be a smooth manifold modelled on~$E$.
Let $a<b$ be real numbers and $\eta\in AC_\cE([a,b],M)$.
Let $a=t_0<t_1<\cdots<t_n=b$ such that
$\eta([t_{j-1},t_j])\sub U_j$ for a chart $\phi_j\colon U_j\to V_j\sub E$
of~$M$. Then
\[
\eta_j:=\phi_j\circ \eta|_{[t_{j-1},t_j]}\in AC_\cE([t_{j-1},t_j],E)
\]
for all $j\in \{1,\ldots, n\}$,
enabling us to define
\[
\eta_j'\in \cE([a,b],E).
\]
Write $\eta_j'=[\gamma_j]$ with $\gamma_j\in \cL^1([t_{j-1},t_j],E)$
(resp., $\gamma_j\in \cL^\infty_{rc}([t_{j-1},t_j],E)$).
We define
\[
\gamma\colon [a,b]\to TM
\]
via $\gamma(t):=T(\phi_j)^{-1}(\eta_j(t),\gamma_j(t))$ if $t\in [t_{j-1},t_j[$
with $j\in \{1,\ldots, n\}$,
and $\gamma(b)=T(\phi_n)^{-1}(\gamma_n(b))$.
Then $\gamma$ is measurable
and we define
\[
\dot{\eta}:=[\gamma].
\]
\end{defn}
\begin{rem}
\begin{itemize}
\item[(a)]
If $\pi_{TM}\colon TM\to M$ is the bundle projection
taking $v\in T_xM$ to~$x$, then $\pi_{TM}\circ \gamma=\eta$
is a continuous map (this property
enters into Lemma~\ref{maybeuse}).
\item[(b)]
If $f\colon M\to N$ is a smooth map between smooth manifolds
in the preceding situation, then $f\circ \eta\in AC_\cE([a,b],N)$
for each $\eta\in AC_\cE([a,b],M)$ and
\begin{equation}\label{chainpth}
(f\circ \eta)^{.}=[Tf\circ \gamma]\quad \mbox{if $\; \dot{\eta}=[\gamma]$;}
\end{equation}
this follows from Lemma~\ref{actonelts}
and Remark~\ref{dercomp}.
\end{itemize}
\end{rem}
\begin{numba}
Let $G$ be a Lie group,
with multiplication $m_G\colon G\times G\to G$ and inversion
$j_G\colon G\to G$.
Let $TG$ be the tangent bundle,
considered as a Lie group
with multiplication $T(m_G)$ (identifying $T(G\times G)$ with
$TG\times TG$) and inversion $T(i_G)$.
We identify $g\in G$ with $0_g\in T_g(G)$.
Then $0_e\in T_e(G)=:L(G)=:\cg$
is the neutral element for~$TG$.
If $v\in T_gG$ and $w\in T_h(G)$ with $g,h\in G$, then
\begin{equation}\label{prepro}
vw=gw+vh,
\end{equation}
where $gw=T_h\lambda_g(w)$ and $vh=T_g\rho_h(v)$
with the left translation $\lambda_g\colon G\to G$, $x\mto gx$
and the right translation $\rho_h\colon G\to G$, $x\mto xh$.
In the following, we consider
the smooth $\cg$-valued $1$-forms
\[
\omega_\ell\colon TG\to \cg,\quad v\mto (\pi_{TG}(v))^{-1}(v)
\]
and
\[
\omega_r \colon TG\to \cg,\quad v\mto v(\pi_{TG}(v))^{-1}.
\]
Likewise if $G$ is a local Lie group.
\end{numba}
\begin{la}\label{prorulingp}
Let $\cE$ be a bifunctor
on Fr\'{e}chet spaces $($resp., on sequentially complete \emph{(FEP)}-spaces,
resp., on integral complete
locally convex spaces$)$.
Assume that $\cE$ satisfies the locality axiom
and that smooth functions
act on $AC_\cE$.
Let $G$ be a Lie group $($or local Lie group$)$
modelled on a Fr\'{e}chet space
$($resp., a sequentially complete \emph{(FEP)}-space,
resp., an integral complete locally convex space$)$ $E$,
and $a<b$.
Let $\eta,\eta_1,\eta_2\in AC_\cE([a,b],G)$
and write $\eta^{-1}$ for the map
$[a,b]\to G$, $t\mto (\eta(t))^{-1}$.
Write $\dot{\eta}=[\gamma]$,
$\dot{\eta}_1=[\gamma_1]$,
and $\dot{\eta}_2=[\gamma_2]$
with measurable functions $\gamma,\gamma_1,\gamma_2\colon [a,b]\to TG$.
Then
\begin{eqnarray}
(\eta_1\eta_2)^{.}&=&
[t\mto \gamma_1(t)\eta_2(t)+\eta_1(t)\gamma_2(t)]\;\mbox{and}\label{eqnproru}\\
(\eta^{-1})^{.} &=& [t\mto -\eta(t)^{-1}\gamma(t)\eta(t)^{-1}],\label{eqnquoru}
\end{eqnarray}
assuming that $\eta_1\eta_2$ is defined
in the case of a local Lie group~$G$.
\end{la}
\begin{proof}
This follows from (\ref{prepro})
and Remark~\ref{dercomp}.
\end{proof}
\begin{la}\label{maybeuse}
Let $G$ be a Lie group $($or local Lie group$)$
with Lie algebra $\cg$
and
$\gamma\colon [a,b]\to TG$ be a map
such that $\pi_{TG}\circ \gamma$ is continuous.
Then $\gamma$ is measurable if and only if
$\omega_r\circ \gamma\colon [a,b]\to \cg$
is measurable, if and only if $\omega_\ell\circ \gamma\colon [a,b]\to \cg$
is measurable.
\end{la}
\begin{proof}
If $\gamma$ is measurable, then also
$\omega_r\circ \gamma$ and $\omega_\ell\circ\gamma$,
as $\omega_r$ and $\omega_\ell$ are smooth mappings,
hence continuous and hence Borel measurable.
The map $\sigma\colon G\times \cg\to TG$, $(g,v)\mto gv$
obtained from multiplication in $TG$
is smooth and hence Borel measurable.
If $\omega_\ell(\gamma)$ is measurable
and $\eta:=\pi_{TG}\circ\gamma$ is continuous,
then $\eta([a,b])\sub G$ is compact and metrizable
(see Lemma~\ref{quotcpmet}).
By \ref{basicsmeas} (e) and (f),
the map $(\eta,\omega_\ell\circ \gamma)\colon [a,b]\to G\times TG$ is Borel
measurable. Hence so is $\gamma=\sigma\circ(\eta,\omega_\ell\circ\gamma)$.
If $\omega_r\circ \gamma$ is measurable,
we can argue in the same way, using
$\sigma \colon G\times \cg\to TG$, $\sigma(g,v):=vg$
instead.
\end{proof}
\begin{defn}
Let $\cE$ be a bifunctor on Fr\'{e}chet spaces
(resp., on sequentially complete (FEP)-spaces, resp.,
on integral complete
locally convex spaces) such that
$\cE$ satisfies the locality axiom~(Loc)
and smooth functions act on~$AC_\cE$.
Let $E$ be a Fr\'{e}chet space (resp., a sequentially complete
(FEP)-space, resp., an integral
complete locally convex space)
and $G$ be a Lie group (or local Lie group)
modelled on~$E$.
If $a<b$ are real numbers and
$\eta\in AC_\cE([a,b],G)$,
we define the \emph{left logarithmic derivative}
of~$\eta$ via
\[
\delta(\eta):=\delta^\ell(\eta):=[\omega_\ell \circ \gamma],
\]
where $\dot{\eta}=[\gamma]$
with a measurable function $\gamma\colon [a,b]\to TG$.
Similarly, we
define the \emph{right logarithmic derivative}
of~$\eta$ via
\[
\delta^r(\eta):=[\omega_r \circ \gamma],
\]
where $\dot{\eta}=[\gamma]$
with a measurable function $\gamma\colon [a,b]\to TG$.
\end{defn}
\begin{defn}\label{defpfwax2}
Let $\cE$ be a bifunctor on Fr\'{e}chet spaces
(resp., on sequentially complete (FEP)-spaces, resp.,
on integral complete
locally convex spaces). We say that $\cE$ satisfies the
\emph{pushforward axioms}
if the following holds:
\begin{itemize}
\item[(P1)]
Let
$a<b$ be real numbers,
$E_1$ be a locally convex space, $V\sub E_1$ be an open subset
and
\[
f\colon V\times E_2\to F
\]
be a continuous map which is linear in the second argument,
where $E_2$ and $F$ are
Fr\'{e}chet spaces (resp., sequentially complete (FEP)-spaces, resp.,
integral complete
locally convex spaces). Then
\[
\wt{f}(\eta,[\gamma]):=[f\circ(\eta,\gamma)]\in \cE([a,b],F)
\]
for all $[\gamma]\in \cE([a,b],E_2)$.
\item[(P2)]
If $f$ is smooth in the situation of (P1),
then also the map
\[
\wt{f}\colon C([a,b],V)\times \cE([a,b],E_2)\to \cE([a,b],F)
\]
is smooth.
\end{itemize}
\end{defn}
By Lemma~\ref{operders}
and Propositions \ref{oponinfty}
and \ref{oponLp}, we have:
\begin{la}\label{lebhavePA}
The bifunctors
$L^p$ on Fr\'{e}chet spaces $($or sequentially complete \emph{(FEP)}-space$)$
satisfy the pushforward axioms
for all $p\in [1,\infty]$.
Moreover, the bifunctors
$L^\infty_{rc}$ and $R$ on integral complete
locally convex spaces
satisfy the pushforward axioms.\,\Punkt
\end{la}
\begin{la}\label{fordpf2}
In the situation of Definition~\emph{\ref{defpfwax2}\,(P2)},
we have
\begin{equation}\label{dpwfd2}
d\wt{f}((\eta_1,[\gamma_1]),(\eta_2,[\gamma_2]))
=[df\circ (\eta_1,\gamma_1,\eta_2,\gamma_2)]
\end{equation}
for all $\eta_1\in C([a,b],V)$, $\eta_2\in C([a,b], E_1)$
and $[\gamma_1]$, $[\gamma_2]\in \cE([a,b],E_2)$.
\end{la}
\begin{proof}
In the case of a bifunctor
on Fr\'{e}chet spaces or sequentially complete
(FEP)-spaces, set $\cF:=L^1$;
in the case of a bifunctor on integral
complete locally convex spaces,
set $\cF:=L^\infty_{rc}$.
By Propositions \ref{oponinfty} and \ref{oponLp}, the map
\[
g\colon C([a,b],V)\times \cF([a,b],E_2)\to\cF([a,b], F),\quad
(\eta,[\gamma])\mto [f\circ (\eta,\gamma)]
\]
is smooth; moreover,
\[
dg((\eta_1,[\gamma_1]),(\eta_2,[\gamma_2]))
=[df\circ (\eta_1,\gamma_1,\eta_2,\gamma_2)]
\]
for all $\eta_1\in C([a,b],V)$, $\eta_2\in C([a,b], E_1)$
and $[\gamma_1]$, $[\gamma_2]\in \cF([a,b],E_2)$,
by (\ref{givesdf}).
Let $j_F\colon \cE([a,b],F)\to \cF([a,b], F)$
be the inclusion map and define $j_{E_2}$
analogously. Let $j_{E_1}$ be the identity map of $C([a,b], E_1)$.
Then
\[
j_F\circ \wt{f}=g\circ (j_{E_1}|_{C([a,b],V)}\times j_{E_2})
\]
and thus, using the Chain Rule,
\[
j_F\circ d\wt{f}=dg\circ (j_{E_1}|_{C([a,b],V)}\times j_{E_2}\times j_{E_1}\times j_{E_2}).
\]
Hence, for
$\eta_1\in C([a,b],V)$, $\eta_2\in C([a,b], E_1)$
and $[\gamma_1]$, $[\gamma_2]\in \cE([a,b],E_2)$,
\begin{eqnarray*}
d\wt{f}(\eta_1,[\gamma_1], \eta_2,[\gamma_2])&=&
j_F(d\wt{f}(\eta_1,[\gamma_1], \eta_2,[\gamma_2]))
=dg(\eta_1,[\gamma_1], \eta_2,[\gamma_2])\\
&=&[df\circ(\eta_1,\gamma_1,\eta_2,\gamma_2)],
\end{eqnarray*}
establishing (\ref{dpwfd2}).
\end{proof}
\begin{la}\label{globpfw}
Let $\cE$ be a bifunctor on Fr\'{e}chet spaces
$($resp., on sequentially complete \emph{(FEP)}-spaces, resp.,
on integral complete
locally convex spaces$)$
which satisfies the locality axiom and
the pushforward axioms.
Let
$a<b$ be real numbers,
$E_2$ and $F$ be
Fr\'{e}chet spaces $($resp., sequentially complete \emph{(FEP)}-spaces, resp.,
integral complete
locally convex spaces$)$,
$M$ be a manifold modelled on
a locally convex space~$E_1$
and
\[
f\colon M\times E_2\to F
\]
be a smooth map which is linear in the second argument.
Let $\eta\in C([a,b],M)$.
Then
\[
\wt{f}(\eta,[\gamma]):=[f\circ(\eta,\gamma)]\in \cE([a,b],F)
\]
for all $[\gamma]\in \cE([a,b],E_2)$,
and the map
\[
\cE([a,b],E_2)\to\cE([a,b],F),\quad [\gamma]\mto\wt{f}(\eta,[\gamma])
\]
is continuous linear. If $M=G$ is a Lie group, then
\[
\wt{f}\colon C([a,b],G)\times \cE([a,b],E_2)\to \cE([a,b], F)
\]
is smooth.
\end{la}
\begin{proof}
Fix $\eta\in C([a,b],M)$.
Using a compactness argument, we find
$a=t_0<t_1<\cdots<t_n=b$ such that $\eta([t_{j-1},t_j])\sub U_j$
for a chart $\phi_j\colon U_j\to V_j\sub E_1$ of~$M$.
Then
\[
f_j\colon V_j\times E_2\to F,\quad f_j(x,y):=f(\phi_j^{-1}(x),y)
\]
is a smooth map which is linear in its second argument and thus
\[
\wt{f_j}\colon C([t_{j-1},t_j],V_j)\times \cE([t_{j-1},t_j],E_2)
\to \cE([t_{j-1},t_j],F),\quad
(\sigma,[\tau])\mto[f_j\circ(\sigma,\tau)]
\]
is smooth by the pushforward axiom (P2).
For $[\gamma]\in \cE([a,b],E_2)$,
we have $[\gamma|_{[t_{j-1},t_j]}]\in \cE([t_{j-1},t_j],E_2)$
by (B2) and hence
\[
[f\circ (\eta,\gamma)|_{[t_{j-1},t_j]}]=
\wt{f_j}(\phi_j\circ\eta|_{[t_{j-1},t_j]},\gamma|_{[t_{j-1},t_j]})\in
\cE([t_{j-1},t_j],F)
\]
for all $j\in \{1,\ldots, n\}$.
By the locality axiom,
we get
\[
\wt{f}(\eta,\gamma)=[f\circ(\eta,\gamma)]\in \cE([a,b],F).
\]
Again by the locality axiom, the linear map $\wt{f}(\eta,.)$
will be continuous if
we can show that the map
\[
\cE([a,b],E_2)\to \cE([t_{j-1},t_j],F),\quad
[\gamma]\mto [f\circ (\eta,\gamma)|_{[t_{j-1},t_j]}]
\]
is continuous for each $j\in \{1,\ldots, n\}$.
But this map is the composition of the continuous map
$\wt{f_j}$
and the map
\[
\cE([a,b],E_2)\to C([t_{j-1},t_j],V)\times \cE([t_{j-1},t_j],E_2),
\]
$[\gamma]\mto (\phi_j\circ \eta|_{[t_{j-1},t_j]},\, [\gamma|_{[t_{j-1},t_j]}])$,
which is continuous by (B2).\\[2.3mm]
If $M=G$ is a Lie group and $\eta\in C([a,b],G)$, let us show that
$\wt{f}$ is smooth on $P\times \cE([a,b],E_2)$
for some open neighbourhood $P$ of $\eta$ in $C([a,b],G)$.
Let $U\sub G$ be a symmetric open identity neighbourhood
on which a chart $\phi\colon U\to V\sub E_1$ of~$G$ is defined.
Let $W\sub G$ be an open identity neighbourhood
such that $WW\sub U$.
We may assume that $a=t_0<\cdots<t_n=b$
has been chosen such that
\[
\eta([t_{j-1},t_j])\sub \eta(t_{j-1})W.
\]
Then
\[
P:=\{\zeta\in C([0,1],G)\colon (\eta^{-1}\zeta)([a,b])\sub W
\]
is an open neighbourhood of~$\eta$ in~$C([a,b],G)$.
For $\zeta\in P$ and $t\in [t_{j-1},t_j]$, we have
\[
\zeta(t)=\eta(t)\eta(t)^{-1}\zeta(t)\in \eta(t_{j-1})WW\sub \eta(t_{j-1})U.
\]
Now
\[
\psi_j\colon \eta(t_{j-1})C([t_{j-1},t_j],U)
\to C([t_{j-1},t_j],V),\quad
\tau\mto \phi\circ (\eta(t_{j-1})^{-1}\tau)
\]
is a chart for $C([t_{j-1},t_j],G)$ around $\eta|_{[t_{j-1},t_j]}$
such that
\[
\zeta|_{[t_{j-1},t_j]}
\]
is in the domain $\eta(t_{j-1})C([t_{j-1},t_j],U)$ of $\psi_j$
for each $\eta\in P$.
The restriction map
\[
\rho_j\colon C([a,b],G)\to C([t_{j-1},t_j],G),\quad \tau\mto \tau|_{[t_{j-1},t_j]}
\]
is a smooth group homomorphism (cf.\ \cite{GCX}).
The map
\[
r_j\colon \cE([a,b],E_2)\to\cE([t_{j-1},t_j],E_2),\quad
[\tau]\mto[\tau|_{[t_{j-1},t_j]}]
\]
is continuous linear (and hence smooth),
by (B2).
The map
\[
g_j\colon V\times E_2\to F,\quad g_j(x,y):=f(\eta(t_{j-1})\phi^{-1}(x),y)
\]
is linear in its second argument and smooth.
Hence, by the pushforward axiom (P2), the map
\[
\wt{g_j}\colon C([t_{j-1},t_j], V)\times \cE([t_{j-1},t_j],E_2)
\to\cE([t_{j-1},t_j],F),\quad (\sigma,[\tau])\mto
[g_j\circ (\sigma,\tau)]
\]
is smooth.
By the locaility axiom,
the map $\wt{f}$ will be smooth on $P\times \cE([a,b], E_2)$
if we can show that the map
\[
h_j\colon P\times \cE([a,b],E_2)\to \cE([t_{j-1},t_j],F),\quad
(\zeta,[\gamma])\mto [f\circ (\zeta,\gamma)|_{[t_{j-1},t_j]}]
\]
is smooth for all $j\in \{1,\ldots, n\}$. But $h_j$
is the map
\[
(\zeta,[\gamma])\mto
\wt{g_j}(\psi_j(\rho_j(\zeta)),\, r_j([\gamma]))
\]
and hence $h_j$ is smooth as a composition of smooth maps.
\end{proof}
\begin{la}\label{deltanic}
Let $\cE$ be a bifunctor
on Fr\'{e}chet spaces $($resp., on sequentially complete \emph{(FEP)}-spaces,
resp., on integral complete
locally convex spaces$)$
which satisfies the locality axiom,
the pushforward axiom \emph{(P1)},
and such that smooth functions act on~$AC_\cE$.
Let $G$ be a Lie group $($or local Lie group$)$
modelled on a Fr\'{e}chet space
$($resp., a sequentially complete \emph{(FEP)}-space,
resp., an integral complete locally convex space$)$,
and $a<b$.
If $\eta\in AC_\cE([a,b],G)$,
then $\delta^\ell(\eta),\delta^r(\eta)\in \cE([a,b],\cg)$.
\end{la}
\begin{proof}
Let $E$ be the modelling space of~$G$.
With $M:=G$,
let $a=t_0<\cdots<t_n=b$,
$\phi_j$, $\eta_j$, $\gamma_j$ and $\gamma$
be as in Definition~\ref{defetadot}.
For each $j\in \{1,\ldots,n\}$,
\[
f_j:=\omega_\ell|_{TU_j}\circ T\phi^{-1}_j\colon V_j\times E\to \cg
\]
is a $C^\infty$-map and linear in its second argument.
Now $\eta_j\in
AC_\cE([t_{j-1},t_j],E)$ and
$\eta'_j=[\gamma_j]\in \cE([a,b],E)$.
By definition,
\[
\delta^\ell(\eta)=[\omega_\ell\circ \gamma]
\]
where $\omega_\ell(\gamma(t))=
\omega_\ell(T\phi_j^{-1}(\eta_j(t),\gamma_j(t)))=
f_j(\eta_j(t),\gamma_j(t))$ for $t\in [t_{j-1},t_j[$.
By the pushforward axiom,
$[\omega_\ell\circ \gamma|_{[t_{j-1},t_j]}]
=[f_j\circ (\eta_j,\gamma_j)]\in \cE([t_{j-1},t_j],\cg)$.
Hence $\delta^\ell(\eta)=[\omega_\ell\circ\gamma]\in \cE([a,b],\cg)$,
by the locality axiom. The proof for $\delta^r(\eta)$ is similar.
\end{proof}
\begin{la}\label{logarules}
Let $\cE$ be a bifunctor on Fr\'{e}chet spaces
$($resp., on sequentially complete \emph{(FEP)}-spaces, resp.,
on integral complete
locally convex spaces$)$.
Assume that
$\cE$ satisfies the locality axiom~\emph{(Loc)},
the pushforward axiom \emph{(P1)},
and that smooth functions act on~$AC_\cE$.
Let $E$ be a Fr\'{e}chet space $($resp., a sequentially complete \emph{(FEP)}-space,
an integral
complete locally convex space$)$
and $G$ be a Lie group
modelled on~$E$.
If $a<b$ are real numbers and
$\eta,\eta_1,\eta_2\in AC_\cE([a,b],G)$, then
\begin{equation}\label{eqlodif1}
\delta^\ell(\eta_1\eta_2^{-1})=[t\mto
\Ad(\eta_2(t))(\gamma_1-\gamma_2)]
\end{equation}
with $\delta^\ell(\eta_j)=[\gamma_j]$ for $j\in \{1,2\}$ and
\begin{equation}\label{eqlodif2}
\delta^r(\eta_1^{-1}\eta_2)=[t\mto
\Ad(\eta_1(t))^{-1}(\zeta_2-\zeta_1)]
\end{equation}
with $\delta^r(\eta_j)=[\zeta_j]$ for $j\in \{1,2\}$.
Also,
\begin{equation}\label{nicef1}
\delta^\ell(\eta_1\eta_2)=[t\mto \Ad(\eta_2(t))^{-1}(\gamma_1(t))+\gamma_2(t)]
\end{equation}
and
\begin{equation}\label{nicef2}
\delta^\ell(\eta^{-1})=-\delta^r(\eta).
\end{equation}
If $\delta^\ell(\eta)=0$ or $\delta^r(\eta)=0$, then $\eta$ is constant.
Moreover, $\delta^\ell(\eta_1)=\delta^\ell(\eta_2)$
if and only if $\eta_2=g\eta_1$ for some $g\in G$.
Likewise,
$\delta^r(\eta_1)=\delta^r(\eta_2)$
if and only if $\eta_2=\eta_1g$ for some $g\in G$.\\[2.3mm]
If $G$ is a local Lie group modelled on~$E$,
then \emph{(\ref{nicef2})} always holds
while \emph{(\ref{eqlodif1})},
\emph{(\ref{eqlodif2})} and \emph{(\ref{nicef1})}
hold whenever
$\eta_1\eta_2^{-1}$,
$\eta_1^{-1}\eta_2$
and $\eta_1\eta_2$, respectively, are defined.
If $\delta^\ell(\eta_1)=\delta^\ell(\eta_2)$
$($or $\delta^r(\eta_1)=\delta^r(\eta_2))$
and $\eta_1(t_0)=\eta_2(t_0)$ for some
$t_0\in [a,b]$, then $\eta_1=\eta_2$.
\end{la}
\begin{proof}
Assume first that $G$ is a Lie group.
(\ref{eqlodif1}), (\ref{eqlodif2}),
(\ref{nicef1}) and (\ref{nicef2})
follow immediately from
(\ref{prorulingp})
and the definition of logarithmic derivatives.\\[2.3mm]
If $\delta^\ell(\eta)=0$,
then $\eta_j'=[\gamma_j]=0$ for all $j\in \{1,\ldots,n\}$
in the proof of Lemma~\ref{deltanic},
whence $\eta_j$ and $\eta|_{[t_{j-1},t_j]}=\phi_j^{-1}\circ \eta_j$
are constant; thus~$\eta$ is constant.\\[2.3mm]
If $\delta^\ell(\eta_1)=\delta^\ell(\eta_2)$,
then $\delta^\ell(\eta_1\eta_2^{-1})=0$ by (\ref{eqlodif1})
and thus $\eta_1\eta_2^{-1}$ is constant,
taking the value $g\in G$, say.
Thus $\eta_1=g\eta_2$ and $\eta_2=g^{-1}\eta_1$.
The proof for right logarithmic derivatives is analogous.\\[2.3mm]
If $G$ is a local Lie group, we can establish
(\ref{eqlodif1})--(\ref{nicef2})
as before if all expressions are defined.
If $\delta^\ell(\eta_1)=\delta^\ell(\eta_2)$
and $\eta_1(t_0)=\eta_2(t_0)$ for some $t_0\in [a,b]$,
then
\[
A:=\{t\in [a,b]\colon \eta_1(t)=\eta_2(t)\}
\]
is a non-empty, closed subset of $[a,b]$.
If we can show that~$A$ is also open, then $A=[a,b]$
(as $[a,b]$ is connected) and hence $\eta_1=\eta_2$.
If $t_1\in A$, we find $\delta>0$ such that
$\theta(t):=\eta_1(t)\eta_2(t)^{-1}$ is defined for
all $t\in [t_1-\delta, t_1+\delta]\cap [a,b]$
and
\[
\theta(t)\eta_2(t)=\eta_1(t).
\]
Then $\delta^\ell(\theta)=0$ by (\ref{eqlodif1}),
whence $\theta$ is constant (as in the group case).
Since $\theta(t_1)=e$, we deduce that $\theta(t)=e$
for all $t\in [t_1-\delta, t_1+\delta]\cap [a,b]$,
whence $\eta_1(t)=\eta_2(t)$ and $[t_1-\delta, t_1+\delta]\cap [a,b]\sub A$.
The proof for right logarithmic derivatives is similar.
\end{proof}
\begin{rem}
Identifying measurable functions and their equivalence
classes,
(\ref{eqlodif1}),
(\ref{eqlodif2}) and (\ref{nicef1})
can be rewritten as
\begin{eqnarray*}
\delta^\ell(\eta_1\eta_2^{-1})(t)&=&
\Ad(\eta_2(t)).(\delta^\ell(\eta_1)(t)-\delta^\ell(\eta_2)(t)),\\
\delta^r(\eta_1^{-1}\eta_2)(t)&=&
\Ad(\eta_1(t))^{-1}(\delta^r(\eta_2)(t)-\delta^r(\eta_1)(t))\;\;\mbox{and}\\
\delta^\ell(\eta_1\eta_2)(t)&=&\Ad(\eta_2(t))^{-1}(\delta^\ell(\eta_1)(t))+\delta^\ell(\eta_2)(t)
\end{eqnarray*}
for $\lambda_1$-almost all $t\in [a,b]$.
\end{rem}
\begin{defn}
Let $\cE$ be a bifunctor
on Fr\'{e}chet spaces (resp., on sequentially complete (FEP)-spaces,
resp., on integral complete
locally convex spaces)
which satisfies the locality axiom.
Let $E$ be a Fr\'{e}chet space
(resp., a sequentially complete (FEP)-space, resp.,
an integral complete locally convex space),
$W \sub \R\times E$ be a subset
and $f\colon W \to E$ be a map.
We say that a continuous
function $\eta\colon I\to E$ on a non-degenerate
interval $I\sub \R$ is an $AC_\cE$-\emph{Carath\'{e}odory solution}
to the differential equation
\[
y'=f(t,y)
\]
if $(t,\eta(t))\in W$ for all $t\in I$, the map
\[
t\mto f(t,\eta(t))
\]
is in $\cE(I,E)$, and
\[
(\forall t_1,t_2\in I)\quad \eta(t_2)-\eta(t_1)
=\int_{t_1}^{t_2} f(s,\eta(s))\,ds.
\]
Or equivalently, if $\eta\in AC_\cE(I,E)$
with $\graph(\eta)\sub W$
and
\[
\eta'=[t\mto f(t,\eta(t))].
\]
If $(t_0,y_0)\in W$ and $\eta$ is a before
with $t_0\in I$ and $\eta(t_0)=y_0$,
then we call $\eta$
an $AC_\cE$-\emph{Carath\'{e}odory solution}
to the initial
value problem\footnote{Compare \cite{HBK} for
the case that~$E$ is a Banach space.}
\[
\left\{
\begin{array}{rcl}
y'&=& f(t,y)\\
y(t_0)&=&y_0.
\end{array}
\right.
\]
Equivalently, $\eta \colon I\to E$
is a continuous
function on a non-degenerate
interval $I\sub \R$
with $t_0\in I$ such that $\eta(t_0)=y_0$,
$(t,\eta(t))\in W$ for all $t\in I$, the map
\[
t\mto f(t,\eta(t))
\]
is in $\cE(I,E)$, and
\[
(\forall t\in I)\quad \eta(t)=y_0+\int_{t_0}^t f(s,\eta(s))\,ds.
\]
Or equivalently, if $\eta\in AC_\cE(I,E)$
with $\graph(\eta)\sub W$ such that $t_0\in I$,
$\eta(t_0)=y_0$ and
\[
\eta'=[t\mto f(t,\eta(t))].
\]
\end{defn}
\begin{defn}
Let $\cE$ be a bifunctor
on Fr\'{e}chet spaces (resp., on sequentially complete
(FEP)-spaces, resp., on integral complete
locally convex spaces)
which satisfies the locality axiom
and such that smooth functions act on~$AC_\cE$.
Let $E$ be a Fr\'{e}chet space
(resp., a sequentially complete (FEP)-space, resp.,
an integral complete locally convex spaces),
$M$ be a smooth manifold
modelled on~$E$,
$W \sub \R\times M$ be a subset
and $f\colon W\to TM$ be a map
such that $f(t,y)\in T_y(M)$ for all
$(t,y)\in W$.
Let $(t_0,y_0)\in W$.
An $AC_\cE$-\emph{Carath\'{e}odory solution}
to
\[
\left\{
\begin{array}{rcl}
y'&=& f(t,y)\\
y(t_0)&=&y_0
\end{array}
\right.
\]
is a map $\eta\in AC_\cE(I,M)$
on a non-degenerate interval $I\sub \R$ with $t_0\in I$
such that $\graph(\eta)\sub W$, $\eta(t_0)=y_0$
and such that, for each $t\in I$,
there exists $\ve>0$ such that $\eta(I\cap \;]t-\ve,t+\ve[)\sub U$
for some chart $\phi\colon U\to V\sub E$ of~$M$
and $\zeta:=\phi\circ \eta|_{I\cap \;]t-\ve,t+\ve[)}$
is an $AC_\cE$-Carath\'{e}odory solution to
$y'=g(t,y)$
with
\[
g\colon (\id_\R\times \phi)(W\cap(\R\times U))\to E,\quad
g(t,y):=d\phi(f(t,\phi^{-1}(y))).
\]
\end{defn}
%
%
%
\begin{defn}
Let $\cE$ be a bifunctor
on Fr\'{e}chet spaces (resp., on sequentially complete (FEP)-spaces,
resp., on integral complete
locally convex space)
which satisfies the locality axiom,
the pushforward axioms,
and such that smooth functions act on~$AC_\cE$.
Let
$G$ be a Lie group modelled on
a Fr\'{e}chet space (resp., a sequentially complete (FEP)-space,
resp., an integral complete locally convex space),
with Lie algebra~$\cg$ and neutral element~$e$.
We say that $G$ is \emph{$\cE$-semiregular}
if for each
$\gamma\in \cE([0,1],\cg)$,
there exists $\eta\in AC_\cE([0,1],G)$
such that
\begin{equation}\label{detsevo0}
\delta^\ell(\eta)=\gamma\quad\mbox{and}\quad
\eta(0)=e.
\end{equation}
If it exists, then $\Evol(\gamma):=\eta$
is uniquely determined
by (\ref{detsevo0}) (see Lemma~\ref{logarules}).
If, moreover, smooth functions
act smoothly on $AC_\cE$,
then we say that $G$ is \emph{$\cE$-regular}
if $G$ is $\cE$-semiregular
and the map
\[
\Evol \colon
\cE([a,b],\cg)\to AC_\cE([a,b],G)
\]
is smooth.
\end{defn}
\begin{rem}
Write $\gamma=[\zeta]\in \cE([0,1],\cg)$
in the preceding definition.
Then (\ref{detsevo0}) is satisfied
if and only if $\eta\colon [0,1]\to G$
is a Carath\'{e}odory solution
to the initial value problem
\[
y'=f(t,y),\qquad y(0)=e
\]
with $f\colon [0,1]\times G\to TG$, $f(t,y):=y\zeta(t)$
(using the left action $G\times TG\to TG$
given by $gv=T\lambda_g(v)$).
%
%
\end{rem}
\begin{rem}
Lemma~\ref{pbckgp}\,(c) shows that if a Lie group
$G$ over $\K\in \{\R,\C\}$ is $\cE$-regular and $\Evol\colon \cE([0,1],\cg)\to AC_\cE([0,1],G)$
is smooth (resp., $\K$-analytic),
then also
\[
\evol\colon \cE([0,1],\cg)\to G, \;\; \gamma\mto\Evol(\gamma)(1)
\]
is smooth (resp., $\K$-analytic).
In contrast to the case of $C^k$-regularity, no exponential
laws are available for spaces of measurable
maps, whence we cannot deduce continuity or
differentiability properties of $\Evol$ from
such of $\evol$ in the current situation.
In the case of measurable regularity properties,
$\Evol$ (rather than $\evol$)
is the key object to work with,
whose study provides the largest amount of information.
\end{rem}
\begin{defn}
Let $\cE$ be a bifunctor
on Fr\'{e}chet spaces (resp., on sequentially complete (FEP)-spaces,
resp., on integral complete
locally convex space)
which satisfies the locality axiom,
the pushforward axioms,
and such that smooth functions act on~$AC_\cE$.
Let
$G$ be a local Lie group modelled on
a Fr\'{e}chet space (resp., a sequentially complete (FEP)-space,
resp., an integral complete locally convex space),
with Lie algebra~$\cg$ and neutral element~$e$.
We say that $G$ is \emph{locally $\cE$-semiregular}
if there exists an open $0$-neighbourhood
\[
\Omega\sub \cE([0,1],\cg)
\]
such that, for each $\gamma\in \Omega$,
there exists $\eta\in AC_\cE([0,1],G)$
such that
\begin{equation}\label{detsevo}
\delta^\ell(\eta)=\gamma\quad\mbox{and}\quad
\eta(0)=e.
\end{equation}
If it exists, then $\Evol(\gamma):=\eta$
is uniquely determined
by (\ref{detsevo}) (Lemma~\ref{logarules}).
If, moreover, $G$ has a global chart and smooth functions
act smoothly on $AC_\cE$,
then we say that $G$ is \emph{locally $\cE$-regular}
if $G$ is locally $\cE$-semiregular
and $\Omega$ can be chosen such that
\[
\Evol \colon
\Omega \to AC_\cE([a,b],G)
\]
is smooth.
\end{defn}
\begin{prop}\label{givsordr}
Let $\cE$ be a bifunctor
on Fr\'{e}chet spaces $($resp., on sequentially complete \emph{(FEP)}-spaces,
resp., on integral complete
locally convex space$)$
which satisfies the locality axiom,
the pushforward axioms, and
such that smooth functions act smoothly on~$AC_\cE$.
Let
$G$ be a $\cE$-semiregular Lie group
modelled on
a Fr\'{e}chet space $($resp., a sequentially complete \emph{(FEP)}-space,
resp., an integral complete locally convex space$)$ $E$,
with Lie algebra~$\cg$ and neutral element~$e$.
Then the map
\[
\Evol\colon \cE([0,1],\cg)\to AC_\cE([0,1],G)
\]
is smooth
if and only if $\Evol$ is smooth as a map
\[
\cE([0,1],\cg)\to C([0,1],G).
\]
If $G$ is a locally $\cE$-semiregular
local Lie group modelled on~$E$ admitting a global chart
and $\Omega\sub \cE([0,1],\cg)$
an open $0$-neighbourhood on which $\Evol$ is defined,
then
\[
\Evol\colon \Omega\to AC_\cE([0,1],G)
\]
is smooth if and only if $\Evol$ is smooth as
a map $\Omega\to C([0,1],G)$.
\end{prop}
\begin{proof}
If $\Evol$ is smooth to $AC_\cE([0,1],G)$, then also to $C([0,1],G)$,
since the inclusion map $AC_\cE([0,1],G)\to C([0,1],G)$ is smooth
(cf.\ Remark~\ref{intoC}).\\[2.3mm]
Conversely, write $\Evol_C$ for $\Evol$ as a map to $C([0,1],G)$
and assume that $\Evol_C$ is smooth. Let $U\sub G$ be a symmetric identity neighbourhood
on which a chart $\phi\colon U\to V\sub E$ is defined.
Let $\Omega\sub \cE([0,1],\cg)$ be an open $0$-neighbourhood
such that $\Evol_C(\Omega)\sub C([0,1],U)$
and $\Evol_C|_\Omega$ is smooth. Let us show that
$\Evol|_\Omega\colon \Omega\to AC_\cE([0,1],G)$
is smooth,
or equivalently, that
\[
g:=AC_\cE([0,1],\phi)\circ \Evol|_\Omega\colon \Omega\to AC_\cE([0,1],E)
\]
is smooth.
By Lemma~\ref{closedemb}
and \cite[Lemma~10.1]{BGN},
the latter will hold
if we can show that $g$ is smooth as a map to $C([0,1],E)$
(which is the case $C([0,1],\phi)\circ \Evol_C$
is smooth) and the map
\[
h\colon \Omega\to \cE([0,1],E),\quad \gamma\mto (g(\gamma))'
\]
is smooth.
To see that $h$ is smooth,
consider $U$ is a local Lie group with
$D_U:=\{(x,y)\in U\times U\colon xy\in U\}$
and make $V\sub E$ a local Lie group
such that $\phi\colon U\to V$
is an isomorphism of local Lie groups.
Then
\[
\phi\circ \Evol_G=\Evol_V\circ \cE([0,1],L(\phi)).
\]
Let $\mu\colon D_V\to V$ be the local group
multiplication and
\[
\nu\colon V\times L(V)\to TV,\quad (x,v)\mto T\mu(0_x,v).
\]
Since $\eta'=\nu\circ(\eta,\delta^\ell\eta)$, we have
\begin{eqnarray}
(g(\gamma))'&=&(\Evol_V(L(\phi)\circ \gamma))'
=\nu\circ (\Evol_V(L(\phi)\circ\gamma),L(\phi)\circ\gamma)\notag\\
&=& \nu\circ (\phi\circ \Evol_C(\gamma),L(\phi)\circ \gamma)\notag\\
&=& \wt{\nu}\Big(C([0,1],\phi)(\Evol_C(\gamma)),\; \cE([0,1],L(\phi))(\gamma)\Big)
\label{lngwdd}
\end{eqnarray}
with $\wt{\nu}$ as in the pushforward axiom~(P2).
As the map $\cE([0,1],L(\phi))$ is continuous linear by
the bifunctor axiom (B1)
and the map $C([0,1],\phi)\colon C([0,1],U)\to C([0,1],V)$
is a $C^\infty$-diffeomorphism
(being a chart of $C([0,1],G)$),
we deduce from (\ref{lngwdd})
with the pushforward axiom (P2)
that~$h$
is smooth.
In the case of a local Lie group,
this completes the proof. If $G$ is a Lie group,
let us show that $\Evol$ is smooth on an open neighbourhood
of each $\gamma\in \cE([0,1],\cg)$.
Write $\eta:=\Evol(\gamma)$.
Let $W\sub U$ be an open identity neighbourhood
such that
\[
W^{-1}WW\sub U.
\]
Using a compactness argument, we can find
%
%
$0=t_0<t_1<\cdots<t_n=1$
with
\[
\eta(t_{j-1})^{-1}\eta([t_{j-1},t_j])\sub W
\]
for all $j\in \{1,\ldots, n\}$.
Then
\[
Q:=\{\zeta \in C([0,1], G)\colon \eta^{-1}\zeta\in C([0,1],W) \}
\]
is an open neighbourhood of $\eta$ in $C([0,1],G)$
and hence
\[
P:=(\Evol_C)^{-1}(Q)
\]
is an open neighbourhood of $\gamma$ in $\cE([0,1],\cg)$.
By Lemma~\ref{pbckgp}\,(b), $\Evol|_P$
will be smooth if we can show that the map
\[
P\to \prod_{j=1}^n AC_\cE([t_{j-1},t_j],G),\quad
\tau\mto (\Evol(\tau)|_{[t_{j-1},t_j]})_{j\in\{1,\ldots,n\}}
\]
is smooth.
The latter holds if, for each $j\in \{1,\ldots, n\}$,
the component
\[
f_j\colon P\to AC_\cE([t_{j-1},t_j],G),\quad \tau\mto \Evol(\tau)|_{[t_{j-1},t_j]}
\]
is smooth.
It is well-known that the evaluation map
$\ve_{t_{j-1}}\colon C([a,b],G)\to G$,
$\zeta\mto \zeta(t_{j-1})$ is smooth (cf.\ \cite{GCX}).
By Lemma~\ref{strgp}, $g_j$
is smooth if we can show that the map
\[
P\to G,\quad \tau\mto \Evol(\tau)(t_{j-1})=\ve_{t_{j-1}}(\Evol_C(\tau))
\]
is smooth (which is the case by smoothness of $\Evol_C$ and $\ve_{t_{j-1}}$)
and the map
\[
g_j\colon P\to AC_\cE([t_{j-1},t_j],G)_*,\quad \tau\mto (f_j(\tau)(t_{j-1}))^{-1}f_j(\tau)
\]
is smooth.
Note that, if $\zeta\in Q$ and $t\in [t_{j-1},t_j]$, then
\begin{eqnarray*}
g_j(\zeta)(t) &=&
\Evol(\zeta)(t_{j-1})^{-1}\Evol(\zeta)(t)\\
&=&
\big(\eta(t_{j-1})^{-1}\Evol(\zeta)(t_{j-1})\big)^{-1}
\eta(t_{j-1})^{-1} \eta(t) (\eta(t))^{-1}\Evol(\zeta)(t)\\
&\in&
W^{-1}WW\sub U.
\end{eqnarray*}
We therefore only need to show that
\[
AC_\cE([t_{j-1},t_j],\phi)\circ g_j \colon P\to AC_\cE([t_{j-1},t_j],V)
\]
is smooth. As a map $G_j$ to $C([t_{j-1},t_j],E)$, this map is smooth
by smoothness of $\Evol_C$. Hence, by
Lemma~\ref{closedemb}
and \cite[Lemma~10.1]{BGN},
it only remains to show that the map
\[
h_j\colon P\to \cE([t_{j-1},t_j],E),\quad
\tau\mto (\phi\circ g_j(\tau))'
\]
is smooth. But
\[
h_j(\tau)=\wt{\nu}(G_j(\tau), \cE([t_{j-1},t_j],L(\phi))(\tau))
\]
with the smooth map
\[
\wt{\nu}\colon C([t_{j-1},t_j],V)\times \cE([t_{j-1},t_j],E))\to
\cE([t_{j-1},t_j],E),\quad \wt{\nu}(\sigma,\tau):=\nu\circ (\sigma,\tau)
\]
(as in the pushforward axiom~(P2)),
since
\[
(\phi\circ g_j(\tau))'=\nu\circ (\phi\circ g_j(\tau),
\delta^\ell(\phi\circ g_j(\tau)))
\]
with
$\delta^\ell(\phi\circ g_j(\tau)))
=L(\phi)\circ \delta^\ell(
\Evol(\tau)(t_{j-1})^{-1}\Evol(\tau))
=L(\phi)\circ \delta^\ell(\Evol(\tau))=L(\phi)\circ\tau$.
Hence $h_j$ is smooth and hence $\Evol|_P$
is smooth.
\end{proof}
We now obtain Theorem~A
as a special case of the following corollary:
\begin{cor}\label{corsimA}
Let
$G$ be a $\cE$-semiregular Lie group
modelled on a Fr\'{e}chet space or
a sequentially complete \emph{(FEP)}-space. Let $p\in [1,\infty]$.
Then we have the following implications:
\[
\mbox{\emph{$G$ is $L^p$-regular $\impl$ $G$ is $L^q$-regular for all
$q\geq p$};}\vspace{-1mm}
\]
\[
\mbox{\emph{$G$ is $L^\infty$-regular $\impl$ $G$ is $L^\infty_{rc}$-regular
$\impl$ $G$ is $R$-regular};}
\]
\[
\mbox{\emph{$G$ is $R$-regular $\impl$ $G$ is $C^0$-regular}.}
\]
\end{cor}
\begin{proof}
Let $\cg:=L(G)$.
Since
\begin{equation}\label{chaini}
C^0([0,1],\cG)\sub
R([0,1],\cg)\sub L^\infty_{rc}([0,1],\cg)\sub L^q([0,1],\cg)\sub L^p([0,1],\cg)
\end{equation}
with continuous linear inclusion maps,
the $\cE$-semiregularity with respect to a class~$\cE$ of spaces
further on the right in the chain~(\ref{chaini}) of inclusions
implies $\cF$-semiregularity with respect to each class
$\cF$ of spaces further on the left.
Let us write $\Evol_\cE$ and $\Evol_\cF$ for the respective
evolution map and $j_{\cE,\cF}$ for the inclusion map
$\cF([0,1],\cg)\to \cE([0,1],\cg)$.
We know that both $\cE$ and $\cF$
satisfy the locality axiom,
the pushforward axioms, and
that smooth functions act smoothly on~$AC_\cE$ and on $AC_\cF$.
Now $\Evol_\cF=\Evol_\cE\circ j_{\cE,\cF}$
as mappings to $C([0,1],G)$.
If $G$ is $\cE$-regular,
then $\Evol_\cE$ is smooth
(see Proposition~\ref{givsordr})
and hence also
$\Evol_\cF=\Evol_\cE\circ j_{\cE,\cF}$ is smooth.
Thus, again by Proposition~\ref{givsordr},
$G$ is $\cF$-regular.
\end{proof}
A similar argument shows:
\begin{cor}
Let $G$ be a Lie group modelled
on an integral complete locally convex space. If $G$ is $L^\infty_{rc}$-regular,
then $G$ is also $R$-regular.\,\Punkt
\end{cor}
\begin{rem}\label{implalsolc}
Analogous implications are available
for local Lie groups.
\end{rem}
\begin{defn}
Let $\cE$ be a bifunctor on Fr\'{e}chet spaces (resp., on sequentially complete
(FEP)-spaces, resp., on integral complete locally convex spaces).
Let $E$ be a Fr\'{e}chet space (resp., a sequentially complete (FEP)-space, resp.,
an integral complete locally convex space) and
$[\gamma]\in \cE([0,1],E)$. For $n\in \N$ and
$k\in \{0,1,\ldots, n-1\}$, define
\[
\gamma_{n,k}\colon [0,1]\to E,\quad
\gamma_{n,k}(t):=\frac{1}{n}\gamma((k+t)/n).
\]
Then $[\gamma_{n,k}]\in \cE([0,1],E)$
for all $n\in \N$ and $k\in \{0,1,\ldots,n-1\}$
(by axiom (B2)).
We say that $\cE$ has the \emph{subdivision property}
if the following holds
for each Fr\'{e}chet space (resp., sequentially complete (FEP)-space, resp.,
each integral complete locally convex space)~$E$:\\[2.3mm]
For each $\gamma\in \cE([0,1],E)$ and continuous
seminorm $q$ on~$\cE([0,1],E)$, we have that
\[
\sup_{k\in \{0,1,\ldots, n-1\}}q(\gamma_{n,k})\;\to\; 0\quad\mbox{as $\,n\to\infty$.}
\]
\end{defn}
\begin{prop}\label{lctoglb}
Let $\cE$ be a bifunctor on Fr\'{e}chet spaces $($resp., on sequentially complete
\emph{(FEP)}-spaces, resp., on integral complete locally convex spaces$)$
which satisfies the locality axiom,
the pushforward axioms,
and such that smooth functions act smoothly on~$AC_\cE$.
If, moreover, $\cE$ has the subdivision property,
then the following conditions are equivalent for
each Lie group $G$ modelled
on a Fr\'{e}chet space $($resp., a sequentially complete \emph{(FEP)}-space,
resp., an integral complete locally convex space$)$:
\begin{itemize}
\item[\rm(a)]
$G$ is $\cE$-regular;
\item[\rm(b)]
$G$ is locally $\cE$-regular.
\end{itemize}
Also the following conditions are equivalent if $\cE$
has the subdivision property:
\begin{itemize}
\item[\rm(c)]
$G$ is $\cE$-semiregular;
\item[\rm(d)]
$G$ is locally $\cE$-semiregular.
\end{itemize}
\end{prop}
\begin{proof}
Let $\cg:=L(G)$.
If $G$ is $\cE$-semiregular (resp., $\cE$-regular),
then trivially $G$ is also locally $\cE$-semiregular (resp.,
locally $\cE$-regular).

Now assume that $G$ is locally $\cE$-regular (resp., $\cE$-semiregular).
Thus, there exists an open $0$-neighbourhood
$\Omega\sub \cE([0,1],\cg)$ such that $\Evol(\gamma)\in AC_\cE([0,1],G)$
exists for each $\gamma\in \Omega$ (resp., moreover
$\Evol\colon \Omega\to AC_{\cE}([0,1],G)$ is smooth).
After shrinking $\Omega$, we may assume that
$\Omega=B^q_1(0)$ for a continuous seminorm
$q\colon \cE([0,1],\cg)\to[0,\infty[$.
Now let $\gamma\in \cE([0,1],\cg)$.
By the subdivision property,
we find $n\in \N$ such that
\[
\gamma_{n,k}\in \Omega\quad\mbox{for all $k\in \{0,1,\ldots,n-1\}$.}
\]
By Axiom (B2), the linear map
\[
\alpha_k\colon \cE([0,1],\cg)\to \cE([0,1],\cg),\quad \eta\mto \eta_{n,k}
\]
is continuous for all $k\in \{0,1,\ldots, n-1\}$.
Hence, we find an open neighbourhood $W\sub \cE([0,1],\cg)$ of $\gamma$
such that
\[
\eta_{n,k}\in \Omega\quad\mbox{for all $\eta\in W$ and all $k\in \{0,1,\ldots, n-1\}$.}
\]
For $\eta\in W$, define
$\Evol(\eta)\colon [0,1]\to G$
via \[
\Evol(\eta)(t):=\Evol(\eta_{n,0})(nt)\quad\mbox{if $t\in [0,\frac{1}{n}]$}
\]
and
\[
\Evol(\eta)(t):=
\Evol(\eta_{n,0})(1)\cdots\Evol(\eta_{n,k-1}(1))\Evol(\eta_{n,k})(nt-k)
\]
if $t\in [\frac{k}{n},\frac{k+1}{n}]$ with $k\in \{1,\ldots, n-1\}$.
The map $\Evol(\eta)$ is continuous
and $\Evol(\eta)|_{[k/n,(k+1)/n]}$ is in $AC_\cE([k/n,(k+1)/n],G)$
(since $\Evol(\eta_{n,k})\in AC_\cE([0,1],G)$),
as a consequence of axiom (B2).
So $\Evol(\eta)\in AC_\cE([0,1],G)$.
Since $\Evol(\eta)(0)=e$ and
$\Evol(\eta)$ has left logarithmic derivative
$\eta$ by construction,
indeed $\Evol(\eta)$ is a left evolution for
$\eta$. Notably, $\Evol(\gamma)$ is a left evolution for~$\gamma$
and hence $G$ is $\cE$-semiregular.
If $\Evol\colon \Omega\to AC_\cE([0,1],G)$
is smooth, then also
\[
\Evol\colon W\to AC_\cE([0,1],G),\;\, \eta\mto\Evol(\eta)
\]
as just defined is smooth. To see this,
we re-use that the map
\[
AC_\cE([0,1],G)\to\prod_{k=0}^{n-1}AC_\cE([k/n,(k+1)/n],G),\;\,
\zeta\mto (\zeta|_{[k/n,(k+1)/n]})_{k=0,1,\ldots, n-1}
\]
is an isomorphism of Lie groups onto a Lie subgroup.
We therefore only need to show that
\begin{equation}\label{retgleich}
W\to AC_\cE([k/n,(k+1)/n],G),\quad \eta\mto \Evol(\eta)|_{[k/n, (k+1)/n]}
\end{equation}
is smooth. For fixed $k$,
the Lie group $AC_\cE([k/n, (k+1)/n],G)$ is isomorphic
to $AC_\cE([0,1],G)$ via $\zeta\mto (t\mto \zeta((k+t)/n))$
and the composition of this isomorphism and $\Evol_W$
is the map
\[
W\to AC_\cE([0,1],G),\quad \eta\mto\evol(\eta_{n,0})\cdots\evol(\eta_{n,k-1})\Evol(\eta_{n,k}).
\]
This map is smooth as it is the product of compositions of
$\Evol\colon \Omega\to AC_\cE([0,1],G)$
(or $\evol\colon \Omega\to G\sub AC_\cE([0,1],G)$)
and the continuous linear maps $\eta\mto \eta_{n,j}$
for $j\in \{0,\ldots, k\}$.
We deduce that the map in (\ref{retgleich})
(and hence also $\Evol\colon W\to AC_\cE([0,1],G)$)
is smooth.
\end{proof}
\begin{la}\label{setsubdiv}
$L^p$ as a bifunctor on Fr\'{e}chet spaces $($or sequentially
complete \emph{(FEP)}-spaces$)$
has the subdivision property,
for all $p\in [1,\infty]$. Moreover,
$L^\infty_{rc}$ and $R$
have the subdivision property
as bifunctors on integral complete locally convex spaces.
\end{la}
\begin{proof}
In the case when $\cE$ is $L^\infty$, $L^\infty_{rc}$
or $R$, we have
\begin{eqnarray*}
\|\gamma_{n,k}\|_{L^\infty,q}&=&\esssup_{t\in [0,1]} \frac{1}{n}q(\gamma((k+t)/n))\\
&\leq& \frac{1}{n}\esssup_{t\in [0,1]} \frac{1}{n}q(\gamma(t))
=\frac{1}{n}\|\gamma\|_{L^\infty,q}
\end{eqnarray*}
for all $[\gamma]\in \cE([0,1],E)$
$q\in P(E)$, $n\in \N$, $q\in P(E)$
and $k\in \{0,1,\ldots, n-1\}$.
Hence
\[
\max_{k\in \{0,1,\ldots, n-1\}}\|\gamma_{n,k}\|_{L^\infty,q}\leq \frac{1}{n}\|\gamma\|_{L^\infty,q}
\to 0
\]
as $n\to\infty$, showing that the subdivision property is satisfied.

Now assume that $\cE=L^p$ with $p\in [1,\infty[$.
Let $\gamma\in \cL^p([0,1],E)$.
Substituting $s=(k+t)/n$, we see that
\begin{eqnarray*}
\|\gamma_{n,k}\|_{\cL^p,q}&=&
\sqrt[p]{\int_0^1q(\gamma((k+t)/n))^p\frac{dt}{n^p}}\\
&\leq& \sqrt[p]{\int_0^1q(\gamma((k+t)/n))^p\frac{dt}{n}}
=
\sqrt[p]{\int_{k/n}^{(k+1)/n}q(\gamma(s))\,ds}
\leq \|\gamma\|_{\cL^p,q}
\end{eqnarray*}
for each continuous seminorm~$q$ on~$E$.
Thus $\|\gamma_{n,k}\|_{\cL^p,q}\leq \|\gamma\|_{\cL^p,q}$
for all $n\in \N$ and $k\in \{0,1,\ldots,n-1\}$.\\[2.3mm]
Let $\ve>0$.
For $m\in \N$, define
\[
A_m:=\{t\in [0,1]\colon q(\gamma(t))>m\}.
\]
Then each $A_n$ is a Borel set in $[0,1]$,
we have $A_1\subseteq A_2\supseteq\cdots$,
and $\bigcap_{m\in \N}A_m=\emptyset$.
Thus $\one_{A_m}\to 0$ holds
pointwise for the characteristic functions
of the set~$A_m$.
Hence, by dominated convergence,
\[
\int_{A_m}q(\gamma(t))^p\,dt=
\int_0^1q(\gamma(t))^p\,\one_{A_m}\,dt\to 0
\]
as $m\to\infty$. We therefore find $m\in \N$ such that
\[
\int_{A_m}q(\gamma(t))^p\,dt\, \leq\, \ve^p/2
\]
Choose $n_0\in \N$ so large that $m/{n_0}\leq\ve/\sqrt[p]{2}$.
Given $n\geq n_0$, define
\[
A_{n,k}:=\{t\in [0,1]\colon (k+t)/n\in A_m\}
\]
for $k\in \{0,1,\ldots, n-1\}$.
Note that, if $t\in [0,1]\setminus A_{n,k}$,
then $(k+1)/n\in [0,1]\setminus A_m$ and hence
\[
q(\gamma_{n,k}(t))^p=
\left(\frac{1}{n}q(\gamma((k+t)/n))\right)^p
\leq (m/n)^p \leq (m/{n_0})^p\leq \ve^p/2.
\]
Moreover, substituting $s=(k+t)/n$,
\[
\int_{A_{n,k}}q(\gamma_{n,k}(t))^p\,dt
=\frac{n}{n^p}\int_{A_m\cap [k/n,(k+1)/n]}q(\gamma(s))^p\,ds\leq\ve^p/2.
\]
Thus
\begin{eqnarray*}
\|\gamma_{n,k}\|_{\cL^p,q} &=& \sqrt[p]{\int_0^1q(\gamma_{n,k}(t))\,dt}\\
&=& \sqrt[p]{\int_{A_n,k}q(\gamma_{n,k}(t))^p\,dt
+\int_{[0,1]\setminus A_{n,k}}\underbrace{q(\gamma_{n,k}(t))^p}_{\leq\ve^p/2}\, dt}\\
&\leq& \sqrt[p]{\ve^p/2+\ve^p/2}=\ve
\end{eqnarray*}
and thus
\[
\max_{k\in \{0,1,\ldots, n-1\}}\|\gamma_{n,k}\|_{\cL^p,q}\leq \ve
\]
for all $n\geq n_0$. Therefore, the subdivision property
is satisfied.
\end{proof}
\begin{prop}\label{equarg}
Let $\cE$ be a bifunctor
on Fr\'{e}chet spaces $($resp., on sequentially complete \emph{(FEP)}-spaces,
resp., on integral complete
locally convex space$)$
which satisfies the locality axiom,
the pushforward axioms, and
such that smooth functions act on~$AC_\cE$.
Let
$G$ be a Lie group
modelled on
a Fr\'{e}chet space
$($resp., a sequentially complete \emph{(FEP)}-space,
resp., an integral complete locally convex space$)$ $E$,
with Lie algebra~$\cg$ and neutral element~$e$.
Let $H\sub G$ be a subgroup
which is a submanifold of~$G$,
modelled on a closed vector subspace $F\sub E$.
Assume that
\begin{equation}\label{thisgivessub}
H=\{x\in G \colon (\forall j\in J) \,\alpha_j(x)=\beta_j(x)\}
\end{equation}
for Lie groups $H_j$
modelled on Fr\'{e}chet spaces
$($resp., sequentially complete \emph{(FEP)}-spaces,
resp., integral complete locally convex spaces$)$
and smooth homomorphisms $\alpha_j,\beta_j\colon G\to H_j$.
Also, assume that $F\sub E$ is complemented
or that $\cE$ satisfies the embedding axiom.
Then the following holds:
\begin{itemize}
\item[\rm(a)]
If $G$ is $\cE$-semiregular,
then also $H$ is $\cE$-semiregular.
\item[\rm(b)]
If smooth functions act smoothly on $AC_\cE$
and $G$ is $\cE$-regular, then also $H$ is $\cE$-regular.
\end{itemize}
\end{prop}
\begin{proof}
(a) Let $\cg:=L(G)$ and $\ch:=L(H)\sub \cg$.
If $\gamma\in \cE([0,1],\ch)$, then $\gamma\in \cE([0,1],\cg)$
and thus $\eta:=\Evol(\gamma)\in AC_\cE([0,1],G)$
exists. For each $j\in J$,
\[
\alpha_j\circ \eta\quad\mbox{and}\quad \beta_j\circ \eta
\]
are elements of $AC_\cE([0,1], H_j)$.
If $\lambda \colon H\to G$ is the inclusion map,
then
\[
\alpha_j\circ\lambda=\beta_j\circ \lambda,
\]
entailing that $L(\alpha_j)\circ L(\lambda)=L(\beta_j)\circ L(\lambda)$
and hence
\[
L(\alpha_j)|_\ch=L(\beta_j)|_\ch.
\]
Hence
\begin{eqnarray*}
\delta^\ell(\alpha_j\circ \eta)&=& L(\alpha_j)\circ \delta^\ell(\eta)
=L(\alpha_j)\circ \gamma\\
&=&
L(\alpha_j)|_\ch\circ \gamma
=L(\beta_j)|_\ch\circ \gamma=\delta^\ell(\beta_j\circ \eta)
\end{eqnarray*}
and thus $\alpha_j\circ \eta=\beta_j\circ \eta$.
As $j$ was arbitrary, we deduce that
$\eta([0,1])\sub H$. Since $H$ is a submanifold of~$G$
and $\eta\in AC_\cE([0,1],G)$ with $\eta([0,1])\sub H$,
we obtain $\eta\in AC_\cE([0,1],H)$
with Lemma~\ref{intsubmf},
as the assume that~$F$ is complemented in~$E$
or $\cE$ satisfies the embedding axiom.
By construction, $\delta^\ell\eta=\gamma$
and thus $\eta=\Evol(\gamma)\in AC_\cE([0,1],H)$.

(b) By (a), $H$ is $\cE$-semiregular and
\begin{equation}\label{givsit}
\lambda\circ \Evol_H=\Evol_G|_{\cE([0,1],\ch)},
\end{equation}
if $\lambda\colon H\to G$ is the inclusion map
and $\Evol_G\colon \cE([0,1],\cg)\to AC_\cE([0,1],G)$
as well as $\Evol_H\colon \cE([0,1],\ch)\to AC_\cE([0,1],H)$
are the respective evolution maps.
Since $\Evol_G|_{\cE([0,1],\ch)}$ is smooth
and $AC_\cE([0,1],H)\sub AC_\cE([0,1],G)$ is a submanifold
(Lemma~\ref{ACHsubACG}),
we deduce from (\ref{givsit})
that $\Evol_H$ is smooth.
\end{proof}
\begin{rem}\label{equalc}
Assume that $G$ is replaced with a local Lie group
in the preceding proposition, with domain $D_G$ for the
multiplication.
Assume that $H_j$ is a submanifold of~$G$
and that $\alpha_j$, $\beta_j$
are smooth homomorphisms of local groups
from $G$ to local Lie groups $H_j$
such that (\ref{thisgivessub})
holds.
Also, assume that $G$ admits a global chart which restricts to a global chart for~$H$.
If the modelling space of~$H$ is complemented
in that of~$G$ or $\cE$ satisfies
the embedding axiom, then $H$ with
$D_H:=\{(x,y)\in D_G\cap (H\times H)\colon xy\in H\}$
is a local Lie group
and we have:
\begin{itemize}
\item[\rm(a)]
If $G$ is locally $\cE$-semiregular,
then also $H$ is locally $\cE$-semiregular.
\item[\rm(b)]
If smooth functions act smoothly on $AC_\cE$
and $G$ is locally $\cE$-regular, then also $H$ is locally $\cE$-regular.
\end{itemize}
The proof follows the same lines.
\end{rem}
\begin{la}\label{evoldetais}
Let $\cE$ be a bifunctor
on Fr\'{e}chet spaces $($resp., on sequentially complete \emph{(FEP)}-spaces,
resp., on integral complete
locally convex space$)$
which satisfies the locality axiom,
the pushforward axioms, and
such that smooth functions act smoothly on~$AC_\cE$.
Let $G$ be a Lie group $($or local Lie group
with global chart$)$ modelled on
a Fr\'{e}chet space $($resp., a sequentially complete \emph{(FEP)}-space,
resp., an integral complete locally convex space$)$,
with Lie algebra~$\cg$ and neutral element~$e$.
Then the map
\[
\delta^\ell\colon AC_\cE([0,1],G)
\to \cE([0,1],\cg),\quad \eta\mto\delta^\ell(\eta)
\]
is smooth and
\begin{equation}\label{deroflgdr}
d(\delta^\ell)(\eta)=\eta'\quad
\mbox{for all $\eta\in AC_\cE([0,1],\cg)$}
\end{equation}
if we identify the Lie algebra
$T_e AC_\cE([0,1],G)$
with $AC_\cE([0,1],\cg)$ by means of the isomorphism
$dAC_\cE([0,1],\phi)|_{T_eAC_\cE([0,1],G)}$, for $U\sub G$
an open symmetric identity neighbourhood and
$\phi\colon U\to V\sub \cg$ a $C^\infty$-diffeomorphism such that
$d\phi|_\cg=\id_\cg$.
\end{la}
\begin{proof}
Step 1.
Assume that we can show that $\delta^\ell|_W$ is smooth
for some identity neighbourhood $W\sub AC_\cE([0,1],G)$.
Then also $\delta^\ell|_{W\eta}$ is smooth, for each $\eta\in AC_\cE([0,1],G)$
(as we now verify) and thus $\delta^\ell$ is smooth.
In fact, for each $\zeta\in W\eta$ we
have
\[
\delta^\ell(\zeta)=\delta^\ell((\zeta \eta^{-1})\eta)
=\Ad(\eta^{-1})(\delta^\ell|_W(\zeta\eta^{-1}))+\delta^\ell(\eta).
\]
As the second summand is constant (i.e., independent of $\zeta$)
and the map\linebreak
$W\eta\to W$, $\zeta\mto \zeta\eta^{-1}$ is smooth,
it only remains to show that the map
\[
h \colon \cE([0,1],\cg)\to \cE([0,1],\cg),\quad
[\gamma]\mto [t\mto \Ad(\eta(t)^{-1})(\gamma(t))]
\]
is smooth.
Now
\[
f\colon G\times \cg\to \cg,\quad f(x,y):=\Ad_x(y)
\]
is a smooth map which is linear in its second argument.
Since $h=\wt{f}(\eta^{-1},.)$
with
\[
\wt{f}\colon C([0,1],G)\times \cE([0,1],\cg)\to \cE([0,1],\cg),\quad
(\tau,[\sigma])\mto [f\circ (\tau,\sigma)],
\]
we deduce with Lemma~\ref{globpfw}
that $h$ is continuous linear and hence smooth.\\[2.3mm]
Step 2. Let $\phi\colon U\to V\sub E$ be a chart of~$G$ with $\phi(e)=0$,
defined an an open symmetric identity neighbourhood~$U$.
Then $U$ is a local Lie group with $D_U:=\{(x,y)\in U\times U\colon xy\in U\}$.
We give $V$ the local Lie group structure
which makes $\phi\colon U\to V$ an isomorphism of local Lie groups.
Consider the smooth map
\[
\omega \colon TV\to L(V),\quad \omega (v)=T\mu(0_{\pi_{TV}(x)^{-1}},v).
\]
groups.
If we identify $TV$ with $V\times E$ and $L(V)=T_0V=\{0\}\times E$
with $E$ (via $(0,y)\mto y$), then $\omega$ becomes the map
\[
\omega\colon V\times E\to E,\quad \omega(x,y)=
d\mu(x^{-1},x;0,y).
\]
For $\eta\in AC_\cE([0,1],V)$ with $\eta'=[\gamma]$,
the left logarithmic derivative is
\[
\delta^\ell_V(\eta)=[\omega \circ (\eta, \gamma)].
\]
Since $\omega$ is linear in its second argument and smooth,
the map
\[
\wt{\omega} \colon C([0,1],V)\times \cE([0,1],E)\to \cE([0,1],E),\quad
\wt{\omega}(\tau,[\sigma]):=[\omega\circ (\tau,\sigma)]
\]
is smooth, by the pushforward axiom (P2).
Hence
\[
\delta^\ell_V\colon AC_\cE([0,1],V)\to \cE([0,1],E),\quad
\eta\mto \wt{\omega}(\eta,\eta')
\]
is smooth.
Now, for $s\not=0$ close to~$0$:
\[
\frac{\delta^\ell_V(s\eta)-\overbrace{\delta^\ell_V(0)}^{=0}}{s}
=\wt{\omega}(s\eta,[\gamma])\to \wt{\omega}(0,[\gamma])
\]
as $s\to 0$.
Since $d\mu(0,0,0,y)=y$ for all $y\in E$, we deduce that
\[
d(\delta^\ell)(0,\eta)=\wt{\omega}(0,[\gamma])=[t\mto d\mu(0,0,0,\gamma(t))]
=[\gamma]=\eta'.
\]
To complete the proof, let us write $\delta^\ell_U$
for the restriction of the map\linebreak
$\delta^\ell\colon AC_\cE([0,1],G)\to \cE([0,1],\cg)$
to $AC_\cE([0,1],V)$.
Then
\[
L(\phi)\circ \delta^\ell_U(\eta)=\delta^\ell_V(\phi\circ \eta)
\]
for all $\eta\in AC_\cE([0,1],U)$ and thus
\[
\cE([0,1],L(\phi))\circ \delta^\ell_U=\delta^\ell_V\circ AC_\cE([0,1],\phi).
\]
Taking the differential at $\eta=e$, we obtain
\[
\cE([0,1],L(\phi))
\circ d\delta^\ell_U
=\delta^\ell_V\circ T_e AC_\cE([0,1],\phi)
\]
on $T_eAC_\cE([0,1],G)$.
Composing with $T_e AC_\cE([0,1],\phi)^{-1}$ on the right
and with $\cE([0,1],L(\phi))^{-1}$ on the left,
(\ref{deroflgdr}) follows.
\end{proof}
\begin{rem}\label{drEvo}
Let $\cE$ be a bifunctor
on Fr\'{e}chet spaces (resp., on sequentially complete (FEP)-spaces,
resp., on integral complete
locally convex spaces)
which satisfies the locality axiom,
the pushforward axioms, and
such that smooth functions act smoothly on~$AC_\cE$.
Let $G$ be a Lie group (or local Lie group
with global chart) modelled on
such a space,
with Lie algebra~$\cg$ and neutral element~$e$.
Since $\delta^\ell(\Evol(\gamma))=\gamma$,
Lemma~\ref{evoldetais}
and the Chain Rule
entail:
\begin{itemize}
\item[(a)]
If $G$ is $\cE$-regular (resp., locally $\cE$ regular),
then
\[
T_0\Evol(\gamma)(t)=\int_0^t \gamma(s)\,ds\quad\mbox{for all $t\in [0,1]$,}
\]
for all $\gamma\in \cE([0,1],\cg)\cong \{0\}\times \cE([0,1],\cg)=T_0\cE([0,1],\cg)$,
identifying $T_e AC_\cE([0,1],G)$ with $AC_\cE([0,1],\cg)$
as in Lemma~\ref{evoldetais}. More generally:
\item[(b)]
If $F\sub\cE([0,1],E)$ is a vector subspace
and $\Omega\sub F$ an open $0$-neighbour\-hood
for some locally convex vector topology on~$F$
such that $\Evol(\gamma)\in AC_\cE([0,1],G)$
exists for all $\gamma\in \Omega$ and
\[
\Evol_\Omega \colon \Omega\to AC_\cE([0,1],G),\quad \gamma\mto\Evol(\gamma)
\]
is $C^1$, then $T_0(\Evol_\Omega)(\gamma)(t)=\int_0^t\gamma(s)\,ds$
for all $t\in [0,1]$ and $\gamma\in F\cong \{0\}\times F=T_0F$.
More generally:
\item[(c)]
If $\gamma\in \cE([0,1],\cg)$
such that $\Evol(r\gamma)\in AC_\cE([0,1],G)$
exists for all $r$ in a non-degenerate
interval $J\sub \R$ with $0\in J$
and
\[
\frac{d}{dr}\Big|_{r=0}\Evol(r\gamma)
\]
exists,\footnote{We require that the limit exists
in some chart for $AC_\cE([0,1],G)$ around~$e$
(and hence in every chart around~$e$, by Lemma~\ref{chainpw}).}
then
\[
\left(\frac{d}{dr}\Big|_{r=0}\Evol(r\gamma)\right)(t)=\int_0^t\gamma(s)\,ds,
\]
identifying $T_eAC_\cE([0,1],G)$ with $AC_\cE([0,1],\cg)$ as above.\\[2mm]
[Proof: Abbreviate $\eta:=\frac{d}{dt}\Big|_{r=0}\Evol(r\gamma)$.
Since $\delta^\ell(\Evol(r\gamma))=r\gamma$
Lemma~\ref{chainpw} and (\ref{deroflgdr}) entail that
\[
\gamma=\frac{d}{dr}\Big|_{r=0}\delta^\ell(\Evol(r\gamma))
=d(\delta^\ell)(\eta)=\eta'
\]
and thus $\eta(t)=\eta(0)+\int_0^t\gamma(s)\,ds$.
Since $\Evol(\gamma)\in AC_\cE([0,1],\cg)_*$,
we have $\eta\in T_0AC_\cE([0,1],\cg)_*=\{\zeta\in AC_\cE([0,1],\cg)\colon
\zeta(0)=0\}$ (using the above identification).
Thus $\eta(t)=\int_0^t\gamma(s)\,ds$.]
\end{itemize}
\end{rem}
\begin{numba}
Let $\K\in \{\R,\C\}$.
If $\K=\R$, let $r\in \N\cup\{\infty,\omega\}$;
if $\K=\C$, let $r=\omega$.
Following \cite{SUB},
a $C^r_\K$-map $f\colon M\to N$ between
$C^r_\K$-manifolds modelled on locally convex topological $\K$-vector spaces
$E$ and $F$ is called a \emph{$C^r_\K$-submersion}
if, for each $x\in M$, there exists
a chart $\phi\colon U_\phi\to V_\phi\sub E$ of~$M$
with $x\in U_\phi$
and a chart $\psi\colon U_\psi\to V_\psi\sub F$ of~$N$ such that
$f(U_\phi)\sub U_\psi$ and
\[
\psi\circ f\circ \phi^{-1}=\pi|_{V_\phi}
\]
for a continuous linear map $\pi\colon E\to F$ which admits a continuous
linear right inverse $\sigma\colon F\to E$ (i.e., $\pi\circ \sigma=\id_F$).
\end{numba}
\begin{numba}
Assume $r\in \{\infty,\omega\}$ if $\K=\R$ and $r=\omega$
if $\K=\C$. It is known that
a surjective $C^r_\K$-homomorphism $q\colon G\to Q$
between $C^r_\K$-Lie groups is a $C^r_\K$-submersion
if and only if $N:=\ker(q)$
is a $C^r_\K$-Lie subgroup of~$G$ and
$q\colon G\to Q$ is an $N$-principal
bundle of class $C^r_\K$,
i.e., it admits local $C^r_\K$-sections (cf.\ \cite{SUB}).
\end{numba}
Our next proposition subsumes Theorem~G.
\begin{prop}\label{regextensions}
Let $\cE$ be a bifunctor
on Fr\'{e}chet spaces $($resp., on sequentially complete \emph{(FEP)}-spaces,
resp., on integral complete
locally convex space$)$
which satisfies the locality axiom,
the pushforward axioms, the subdivision property,
and such that smooth functions act smoothly on~$AC_\cE$.
Let
$G$ be a Lie group.
Assume that $N\sub G$ is a normal
Lie subgroup
such that $G/N$ can be given
a smooth Lie group structure making
\[
q\colon G\to G/N,\quad x\mto xN
\]
a smooth submersion
such that both $N$ and $G/N$ are modelled
on a Fr\'{e}chet space $($resp., a sequentially complete \emph{(FEP)}-space,
resp., an integral complete locally convex space$)$
and both $N$ and $G/N$ are $\cE$-regular.
Then also $G$ is modelled on
a Fr\'{e}chet space $($resp., a sequentially complete \emph{(FEP)}-space,
resp., an integral complete locally convex space$)$
and $G$ is $\cE$-regular.
\end{prop}
\begin{proof}
Let $\cg:=L(G)$, $\cn:=L(N)$ and $\cq:=L(Q)$;
thus $L(q)\colon \cg\to \cq$ is a continuous
linear map with kernel $\cn$ admitting a continuous
linear right inverse, entailing that
\[
\cg\cong \cn\oplus \cq
\]
is a Fr\'{e}chet space (resp., a sequentially complete
(FEP)-space, resp., an integral complete locally
convex space).
Let
\[
\phi\colon U\to V
\]
be a chart for~$G$, defined
on an open symmetric identity neighbourhood
$U\sub G$.
Let $W\sub Q$ be an open symmetric identity neighbourhood
on which a smooth section $\sigma\colon W\to G$ is defined
(thus $q\circ \sigma=\id_W$).
After shrinking~$W$, we may assume that $\sigma(W)\sub U$
and that there is a chart $\psi\colon W\to W_1$
for $Q$ defined on~$W$.
By hypothesis, we have smooth evolution maps
\[
\Evol_H\colon \cE([0,1],\ch)\to AC_\cE([0,1],H)
\]
and $\Evol_Q\colon \cE([0,1],\cq)\to AC_\cE([0,1],Q)$.
Since $\Evol_Q$ is continuous,
there is an open $0$-neighbourhood $P\sub \cE([0,1],\cq)$ such that
\[
\Evol_Q(\gamma)\in AC_\cE([0,1],W)\quad\mbox{for all $\gamma\in P$.}
\]
Then
\[
\Omega:=\{\gamma \in \cE([0,1],\cg)\colon L(q)\circ \gamma\in P\}
\]
is an open $0$-neighbourhood in $\cE([0,1],\cg)$ such that
\[
\Evol_Q(L(q)\circ \gamma)\in AC_\cE([0,1],W)
\]
for all $\gamma\in \Omega$. Then
\[
\zeta:=\sigma\circ \Evol_Q(L(q)\circ \gamma)\in AC_\cE([0,1],U)\sub AC_\cE([0,1],G)
\]
and
\[
L(q)\circ \delta^\ell(\zeta)=
\delta^\ell (q\circ \zeta)=
\delta^\ell \Evol_Q(L(q)\circ\gamma)=L(q)\circ \gamma,
\]
entailing that the function
\[
\gamma-\delta^\ell(\zeta)
\]
takes its values in $\ker L(q)=\cn$.
As $\cn$ is complemented in~$\cg$ and
$\gamma-\delta^\ell(\zeta)\in \cE([0,1],\cg)$,
we obtain
\[
\tau:=\gamma-\delta^\ell(\zeta)\in \cE([0,1],\ch).
\]
For $\theta\in AC_\cE([0,1],\cn)$, we have
\begin{eqnarray}
\delta^\ell(\theta\zeta) =  \gamma &\Leftrightarrow&
\Ad(\zeta)^{-1}(\delta^\ell(\theta))+\delta^\ell(\zeta)=\gamma\notag\\
&\Leftrightarrow& \delta^\ell\theta=\Ad(\zeta)(\tau).\label{equiexta}
\end{eqnarray}
Note that $\Ad(g).\cn\sub \cn$ for each $g\in G$ since $N$
is a normal Lie subgroup of~$G$ (see \cite{GaN}) and that
\[
f\colon G\times \cn \to\cn,\quad f(g,v):=\Ad_g(v)
\]
is a smooth map which is linear in the second argument.
By the pushforward axiom (P2), the associated map
\[
\wt{f}\colon C([0,1],U)\times \cE([0,1],\cn)\to \cE([0,1],\cn),\quad
\wt{f}(\alpha,[\beta]):= [f\circ (\alpha,\beta)]
\]
is smooth.
In particular, we have
\[
\Ad(\zeta)(\tau)
=\wt{f}(\zeta,\tau)\in \cE([0,1],\cn),
\]
enabling us to define
\[
\theta:=\Evol_N(\Ad(\zeta)(\tau)).
\]
Then $\theta\zeta=\Evol_G(\gamma)$, by (\ref{equiexta}),
and thus $G$ is locally $\cE$-semiregular.
Note that the map
\[
g\colon \Omega\to AC_\cE([0,1], G),\quad
\gamma\mto \zeta=\sigma\circ \Evol_Q(\cE([0,1],L(q))(\gamma))
\]
is smooth. We use here the hypothesis that
smooth functions act smoothly on $AC_\cE$;
hence
\[
AC_\cE([0,1], \phi\colon \sigma\circ \psi^{-1})\colon
AC_\cE([0,1],W_1)\to AC_\cE([0,1],V)
\]
is smooth and hence also the map
$AC_\cE([0,1],W)\to AC_\cE([0,1],U)$, $\alpha\mto \sigma\circ \alpha$,
which is the following composition of smooth maps:
\[
AC_\cE([0,1],\phi)^{-1}\circ
AC_\cE([0,1], \phi\colon \sigma\circ \psi^{-1})\circ AC_\cE([0,1],\psi).
\]
The map
\[
\delta^\ell\colon AC_\cE([0,1],U)\to \cE([0,1],\cg)
\]
is smooth (see Lemma~\ref{evoldetais}), entailing that
\[
h \colon \Omega\to \cE([0,1],\cg),\quad \gamma\mto \tau:=\gamma-\delta^\ell(\zeta)
=\gamma-\delta^\ell(g(\gamma))
\]
is smooth. As $\cn$ is complemented in~$\cn$,
we may consider $g$ as a map to $\cE([0,1],\cn)$.
Now the formula
\[
\Evol(\gamma)=\theta\zeta=\Evol_N(\wt{f}(g(\gamma),h(\gamma)))g(\gamma)
\]
shows that $\Evol\colon \Omega\to AC_\cE([0,1],G)$
is smooth.
Hence $G$ is locally $\cE$-regular.
As we assume that $\cE$ satisfies the subdivision
property, we deduce with Proposition~\ref{lctoglb}
that $G$ is $\cE$-regular.
\end{proof}
As in the study of $C^k$-semiregularity \cite{SEM},
it is very useful for refined results
to have a group structure on $\cE([0,1],\cg)$
available if $G$ is $\cE$-semiregular.
\begin{defn}
Let $\cE$ be a bifunctor
on Fr\'{e}chet spaces (resp., sequentially complete (FEP)-spaces,
resp., integral complete
locally convex spaces)
which satisfies the locality axiom,
the pushforward axioms, and
such that smooth functions act smoothly on~$AC_\cE$.
Let $G$ be a Lie group modelled on
such a space,
with Lie algebra~$\cg$ and neutral element~$e$.
If $G$ is $\cE$-semiregular, then the map
\[
\delta^\ell\colon AC_\cE([0,1],G)_*\to \cE([0,1],\cg)
\]
is a bijection such that $\delta^\ell(e)=0$.
We can therefore make
$\cE([0,1],\cg)$ a group with group multiplication $\odot$
and neutral element~$0$ in such a way that
$\delta^\ell\colon AC_\cE([0,1],G)_*\to \cE([0,1],\cg)$ becomes
an isomorphism of groups.
Then also
\[
\Evol=(\delta^\ell)^{-1} \colon \cE([0,1],\cg) \to AC_\cE([0,1],G)_*
\]
is an isomorphism of groups.
We have
\[
[\gamma_1 ] \odot [\gamma_2] = [\Ad(\Evol([\gamma_2]))^{-1}\gamma_1 +
\gamma_2]
\]
\[
[\gamma_1]\odot[\gamma_2]^{-1} = [\Ad(\Evol([\gamma_2]))(\gamma_1-\gamma_2)]
\]
and
\[
[\gamma]^{-1} = -[\Ad(\Evol([\gamma]))(\gamma)]
\]
for all $[\gamma]$, $[\gamma_1]$, $[\gamma_2]\in\cE([0,1],\cg)$.
For fixed $[\gamma_2]\in \cE([0,1],\cg)$, abbreviate $\eta:=
\Evol([\gamma_2])^{-1}$.
Then right translation with $[\gamma_2]$ in $(\cE([0,1],\cg),\odot)$ is
the affine-linear map
$\rho_{[\gamma_2]}\colon \cE([0,1],\cg)\to \cE([0,1],\cg)$,
\[
[\gamma_1]\mto [\gamma_1]\odot [\gamma_2]=
[\Ad(\eta)\gamma_1]+[\gamma_2]
=\wt{f}(\eta,[\gamma_1])+[\gamma_2]
\]
with $f\colon G\times \cg\to \cg$, $f(x,y):=\Ad_x(y)$
and
\[
\wt{f}\colon C([0,1],G)\times \cE([0,1],\cg)\to\cE([0,1],\cg),\quad
(\sigma,[\tau])\mto [f\circ (\sigma,\tau)].
\]
The pushforward axioms and locality axiom
entail that $\wt{f}$ is smooth (see Lemma~\ref{globpfw})
and hence continuous.
As a consequence, the affine-linear map
$\rho_{[\gamma_2]}$ is continuous
and hence homeomorphism
(as $\rho_{[\gamma_2]^{-1}}$ is continuous by the same
argument).
If $G$ is a Lie group over $\K\in \{\R,\C\}$,
then the affine $\K$-linear map $\rho_{[\gamma_2]}$
is a $\K$-analytic diffeomorphism $\cE([0,1],\cg)\to\cE([0,1],\cg)$.
\end{defn}
\begin{defn}\label{odotloca}
Let $\cE$ be a bifunctor
on Fr\'{e}chet spaces (resp., on sequentially complete (FEP)-spaces,
resp., on integral complete
locally convex spaces)
which satisfies the locality axiom,
the pushforward axioms, and
such that smooth functions act smoothly on~$AC_\cE$.
Let $G$ be a local Lie group with global chart modelled on
such a space,
with Lie algebra~$\cg$ and neutral element~$e$.
Let $D_G$ be the domain
of the multiplication of~$G$
and $D_{\wt{G}}=AC_\cE([0,1],D_G)$
be the domain of the multiplication in
$AC_\cE([0,1],G)$.
If $G$ is locally $\cE$-semiregular, then
\[
\delta^\ell\colon AC_\cE([0,1],G)_*\to \cE([0,1],\cg)
\]
is an injective smooth map with $\delta^\ell(e)=0$
whose image contains an open $0$-neighbourhood $\Omega\sub \cE([0,1],\cg)$.
If $G$ is locally $\cE$-semiregular
and
\[
\Evol\colon \Omega\to AC_\cE([0,1],G)
\]
is continuous, then $W:=\{\eta\in AC_\cE([0,1],G)_*\colon
\delta^\ell(\eta), \delta^\ell(\eta^{-1})\in \Omega\}$
is an open identity neighbourhood
in $AC_\cE([0,1],G)$. After replacing $\Omega$
with its open subset $\delta^\ell(W)$,
we may assume that $\Omega=\delta^\ell(W)$
and thus
\[
\delta^\ell|_W \colon W\to\Omega
\]
is a homeomorphism. Consider $W$ as a local Lie group
with $D_W:=\{(\eta_1,\eta_2)\in D_{\wt{G}}\cap (W\times W)\colon \eta_1\eta_2
\in W\}$.
We give $\Omega$ the local topological group structure
which makes $\delta^\ell|_W$ an isomorphism
of local topological groups,
and write $\odot$ for the local multiplication
on~$\Omega$. Then $\odot$ is given by formulas
as in the Lie group case.
In particular, for each $[\gamma_2]\in \Omega$, we have
\[
\rho_{[\gamma_2]}([\gamma_1])=[\Ad(\Evol([\gamma_2]))^{-1}\gamma_1+\gamma_2]
\]
for $[\gamma_1]$ in some open $0$-neighbourhood in~$\Omega$,
and this is the restriction of an invertible affine-linear continuous
map (with inverse of analogous form).
If $G$ is locally $\cE$-regular, then we may assume that
$\Evol\colon \Omega\to W\sub AC_\cE([0,1],G)_*$ is
smooth. Thus $\delta^\ell\colon W\to \Omega$ is a $C^\infty$-diffeomorphism
and thus $\Omega$ (with $\odot$) is a smooth local Lie group.
\end{defn}
\begin{prop}\label{evocxanl}
Let $\cE$ be a bifunctor
on Fr\'{e}chet spaces $($resp., on sequentially complete \emph{(FEP)}-spaces,
resp., on integral complete
locally convex spaces$)$
which satisfies the locality axiom,
the pushforward axioms, and
such that smooth functions act smoothly on~$AC_\cE$.
Let
$G$ be a $\cE$-regular complex analytic Lie group
modelled on
a complex Fr\'{e}chet space $($resp., a sequentially complete complex \emph{(FEP)}-space,
resp., an integral complete complex locally convex space$)$~$E$,
with Lie algebra~$\cg$.
Then
\[
\Evol\colon \cE([0,1],\cg)\to AC_\cE([0,1],G)
\]
is complex analytic.
If $G$ is a locally $\cE$-regular complex analytic
local Lie group modelled on~$E$ with a global chart
and $\Omega\sub \cE([0,1],\cg)$ an open $0$-neighbourhood
such that $\Evol$ is defined on $\Omega$,
\[
\Evol\colon \Omega\to AC_\cE([0,1],G)
\]
is smooth and ${-}[\Ad(\Evol([\gamma]))(\gamma)]\in \Omega$
for each $[\gamma]\in\Omega$,
then $\Evol$ is complex analytic on~$\Omega$.
\end{prop}
\begin{proof}
If $G$ is a $\cE$-regular complex analytic Lie group,
set $\Omega:=\cE([0,1],G)$;
if $G$ is a locally $\cE$-regular complex analytic local Lie group,
let $\Omega\sub \cE([0,1],\cg)$ be an open $0$-neighbourhood
which is a smooth local Lie group with multiplication $\odot$
as in Definition~\ref{odotloca}.
Since $\Omega$
is open subset of the complex locally convex space
$\cE([0,1],\cg)$, we can consider it as an open complex
analytic submanifold. The tangent map $T_0\Evol$
is complex linear, as it corresponds to
the integration operator $\gamma\mto (t\mto\int_0^t\gamma(s)\,ds)$
(see Remark~\ref{drEvo}\,(a)).
Let $\rho_{[\gamma]^{-1}}$ be the right translation map
$[\gamma_1]\mto [\gamma_1]\otimes[\gamma]^{-1}$
with $[\gamma]\in \Omega$
and $r_\eta$ be the right translation map
of $W\sub aC_\cE([0,1],G)_*$ with $\eta:=\Evol([\gamma])$.
Since $\Evol$ is a smooth group homomorphism (resp.,
homomorphism of local Lie groups), we have
\[
\Evol=r_\eta\circ \Evol\circ \rho_{[\gamma]}^{-1}
\]
on some open neighbourhood of $[\gamma]$ and hence
\[
T_{[\gamma]}\Evol=T_e(r_\eta)\circ T_0(\Evol)\circ T_{[\gamma]}(\rho_{[\gamma]^{-1}})
\]
which is complex linear (recalling that $\rho_{[\gamma^{-1}]}$
is the restriction of a complex affine-linear
continuous map).
Thus $\Evol$ is a smooth map between
complex analytic manifolds such that
the tangent map $T_{[\gamma]}\Evol$ is complex linear at each
$[\gamma]$ in its domain.
Therefore, $\Evol$ is complex analytic (cf.\ \cite{RES}).
\end{proof}
\begin{la}\label{realvscxlcly}
Let $\cE$ be a bifunctor
on Fr\'{e}chet spaces $($resp., on sequentially complete \emph{(FEP)}-spaces,
resp., on integral complete
locally convex spaces$)$
which satisfies the locality axiom,
the pushforward axioms, and
such that smooth functions act smoothly on~$AC_\cE$.
Let
$G$ be a real analytic local Lie group
which is an open subset of a
Fr\'{e}chet space $($resp., sequentially complete \emph{(FEP)}-space,
resp., integral complete locally convex space$)$ $E$,
with Lie algebra $\cg$.
Let $\wt{G}\sub E_\C$ be an open subset
which is a complex
analytic local Lie group
with $G\sub \wt{G}$, such that the inclusion map
$G\to \wt{G}$ is a homomorphism of
real analytic local Lie groups.
Then the following conditions are equivalent:
\begin{itemize}
\item[\rm(a)]
$G$ is locally $\cE$-regular
and there is an open $0$-neighbourhood $\Omega$ in\linebreak
$\cE([0,1],\cg)$
such that each $[\gamma]\in \Omega$
has an evolution $\Evol_G(\gamma)$ in~$G$
and
\[
\Evol_G\colon \Omega\to AC_\cE([0,1],G)
\]
is real analytic.
\item[\rm(b)]
$\wt{G}$ is locally $\cE$-regular.
\end{itemize}
\end{la}
\begin{proof}
After shrinking $G$ and $\wt{G}$,
we may assume that complex conjugation $E_\C\to E_\C$,
$x+iy\mto x-iy$ restricts to an
antiholomorphic automorphism $\sigma\colon \wt{G}\to \wt{G}$
of the complex analytic local Lie group $\wt{G}$
and
\[
G=\{z\in \wt{G}\colon \sigma(z)=z\}=\wt{G}\cap E.
\]
Hence $G$ is a real analytic submanifold
of $\wt{G}$ and
\[
AC_\cE([0,1],G)=AC_\cE([0,1],\wt{G})
\cap AC_\cE([0,1],E)
\]
is a real analytic
submanifold of $AC_\cE([0,1],\wt{G})$.
If $\wt{G}$ is $\cE$-regular,
then we have a complex analytic evolution
\[
\Evol_{\wt{G}}\colon \wt{\Omega}\to AC_\cE([0,1],\wt{G})
\]
(see Proposition~\ref{evocxanl}).
Then $\Omega:=\wt{\Omega}\cap \cE([0,1],\cg)$
is an open $0$-neighbourhood in $\cE([0,1],\cg)$.
For each $[\gamma]\in \Omega$,
we have
\[
\Evol_G([\gamma])=\Evol_{\wt{G}}([\gamma])\in AC_\cE([0,1],G)
\]
(cf.\ Remark~\ref{equalc}).
hence $\Evol_G$ has the complex analytic extension $\Evol_{\wt{G}}$
and thus $\Evol_G$ is real analytic.\\[2.3mm]
If, conversely, $G$ is locally $\cE$-regular
with $\Evol_G\colon \Omega\to AC_\cE([0,1],G)$
real analytic, then
$\Evol_G$ has a complex analytic extension
\[
(\Evol_G)\wt{\;}
\colon \wt{\Omega}\to AC_\cE([0,1],\wt{G}).
\]
After shrinking $\Omega$ and $\wt{\Omega}$
if necessary, we then have that
$(\Evol_G)\wt{\;}([\gamma])=\Evol_{\wt{G}}([\gamma])$
for all $[\gamma]\in \wt{\Omega}$
(cf.\ \cite[Proposition~9.9]{SEM}
for an analogous discussion of
local $C^k$-regularity).
Hence $\wt{G}$ is locally $\cE$-regular.
%
%
\end{proof}
\begin{prop}\label{sammelsu}
Let $\cE$ be a bifunctor
on Fr\'{e}chet spaces $($resp., on sequentially complete \emph{(FEP)}-spaces,
resp., on integral complete
locally convex spaces$)$
which satisfies the locality axiom,
the pushforward axioms, and
such that smooth functions act smoothly on~$AC_\cE$.
Let $G$ be a Lie group modelled
on such a space, with Lie algebra~$\cg$.
If $G$ is $\cE$-semiregular,
then the following holds:
\begin{itemize}
\item[\rm(a)]
If $\Evol\colon \cE([0,1],\cg)\to AC_\cE([0,1],G)$
is continuous at~$0$, then $\Evol$
is continuous.
\item[\rm(b)]
If $\Evol\colon \cE([0,1],\cg)\to AC_\cE([0,1],G)$
is $C^1$ on some open $0$-neighbourhood,
then $\Evol$ is smooth and thus $G$ is $\cE$-regular.
\item[\rm(c)]
If $\Evol\colon \cE([0,1],\cg)\to AC_\cE([0,1],G)$
is continuous and the smooth
homomorphisms from $G$ to
$\cE$-regular Lie groups separate points on~$G$,
then $G$ is $\cE$-regular.
\item[\rm(d)]
If $G$ is a real analytic Lie group
and $\Evol\colon \cE([0,1],\cg)\to AC_\cE([0,1],G)$
is real analytic on some open $0$-neighbourhood,
then $\Evol$ is real analytic.
\item[\rm(e)]
If $\Evol\colon \cE([0,1],\cg)\to C([0,1],G)$
is continuous at~$0$, then the map\linebreak
$\Evol\colon \cE([0,1],\cg)\to C([0,1],G)$
is continuous and $(\cE([0,1],\cg),\odot)$
is a topological group.
\end{itemize}
\end{prop}
\begin{proof}
(a) For $[\gamma]\in \cE([0,1],\cg)$,
let $\rho_{[\gamma]^{-1}}$ be right translation with $[\gamma]^{-1}$
in the group $(\cE([0,1],\cg),\odot)$
and $r_\eta$ be right translation
with $\eta:=\Evol([\gamma])$ in the Lie group
$AC_\cE([0,1],G)$. Since $\Evol$ is a group homomorphism
from $(\cE([0,1], \cg),\odot)$ to $AC_\cE([0,1],G)$,
we have
\[
\Evol = r_\eta\circ \Evol \circ \rho_{[\gamma]^{1}}.
\]
Since $r_\eta$ and $\rho_{[\gamma]^{1}}$ are continuous,
we see that $\Evol$ will be continuous at $[\gamma]$
if $\Evol$ is continuous at~$0$
(cf.\ \cite[Theorem D]{SEM}
for an analogous result in the case of $C^k$-semiregularity).\\[2.3mm]
(b) Step~1. If $\Evol$ is $C^1$ on an open $0$-neighbourhood
$W\sub \cE([0,1],\cg)$ and $[\gamma]\in \cE([0,1], \cg)$,
then $W\otimes [\gamma]$ is an open neighbourhood
of $[\gamma]$ in $\cE([0,1],\cg)$
(since $\rho_{[\gamma]}$ is a $C^\infty$-diffeomorphism)
and the formula
\[
\Evol|_{W\odot [\gamma]} = r_\eta\circ \Evol|_W  \circ \rho_{[\gamma]^{1}}|_{W\odot [\gamma]}
\]
shows that $\Evol|_{W\odot[\gamma]}$ is $C^1$.
Hence $\Evol$ is $C^1$.\\[2.3mm]
Step 2. The map
\[
f\colon G\times \cg\to \cg,\quad g(x,y):=\Ad_x(y)
\]
is linear in its second argument and smooth.
Hence
\[
\wt{f}\colon C([0,1],G)\times \cE([0,1],\cg)\to \cE([0,1],\cg),\quad
\wt{f}(\eta,[\gamma]):=[f\circ (\eta,\gamma)]
\]
is smooth by Lemma~\ref{globpfw}.
As a consequence, the group multiplication
\[
AC_\cE([0,1],\cg)^2\to AC_\cE([0,1],\cg),\,
([\gamma_1],[\gamma_2])\mto \wt{f}(\Evol([\gamma_2])^{-1},[\gamma_1])+[\gamma_2]
\]
is $C^1$ and also the inversion map
\[
AC_\cE([0,1],\cg)\to AC_\cE([0,1],\cg),\quad
[\gamma] \mto \,- \wt{f}(\Evol([\gamma]),[\gamma]).
\]
An inductive argument as in the case of $C^k$-regularity
in~\cite[Theorem E]{SEM} now show that $\Evol$ is $C^k$
for each $k\in \N$ and hence smooth.\\[2.3mm]
%
%
%
(c) We can repeat the proof of \cite[Theorem F]{SEM}.\\[2.3mm]
(d) Replace `smooth' with `real analytic' in Step~1 from
the proof~of~(b).\\[2.3mm]
(e) Since $C([0,1],G)$ is a topological group
when endowed with the compact-open topology,
we can argue as in the proof of~(a).
The continuity of the group operations
of $\cE([0,1],\cg)$ follow with the Pushforward Axiom~(P2).
%
%
\end{proof}
\begin{prop}
Let $\cE$ be a bifunctor
on Fr\'{e}chet spaces $($resp., on sequentially complete \emph{(FEP)}-spaces,
resp., on integral complete
locally convex spaces$)$
which satisfies the locality axiom,
the pushforward axioms, has the subdivision property,
and
such that smooth functions act smoothly on~$AC_\cE$.
Let $G$ be a real analytic Lie group modelled
on such a space and $\wt{W}$ be a complex analytic
local Lie group with global chart,
such that $\wt{W}$ is a complexification
of some open symmetric identity neighbourhood
$W\sub G$ with global chart and the inclusion map
$W\to\wt{W}$ is a homomorphism of
real analytic local Lie groups.
Then the following conditions
are equivalent:
\begin{itemize}
\item[\rm(a)]
$G$ is $\cE$-regular and $\Evol\colon \cE([0,1],\cg)\to
AC_\cE([0,1],G)$ is real analytic;
\item[\rm(b)]
$\wt{W}$ is locally $\cE$-regular.
\end{itemize}
\end{prop}
\begin{proof}
If $G$ is $\cE$-regular and $\Evol\colon \cE([0,1], \cg)\to G$
is real analytic, then $W$ is locally $\cE$-regular and
$\Evol_W:=\Evol|_\Omega \colon \Omega\to AC_\cE([0,1],W)$
is real analytic for some open $0$-neighbourhood
$\Omega\sub \cE([0,1],\cg)$. Hence $\wt{W}$ is locally $\cE$-regular,
by Lemma~\ref{realvscxlcly}.

If, conversely, $\wt{W}$ is locally $\cE$-regular,
then $W$ is locally $\cE$-regular with real analytic
evolution $\Evol_W\colon \Omega\to AC_\cE([0,1],W)$
on an open $0$-neighbourhood $\Omega|sub \cE([0,1],\cg)$,
by Lemma~\ref{realvscxlcly}.
Hence $G$ is locally $\cE$-regular
As $\cE$ has the subdivision prperty,
we deduce with Proposition~\ref{lctoglb}
that $G$ is $\cE$-regular.
Let $\Evol\colon \cE([0,1],\cg)\to AC_\cE([0,1], G)$
be the evolution map.
Since $\Evol|_\Omega=\Evol_W$ is real analytic,
$\Evol$ is real analytic by Proposition~\ref{sammelsu}\,(d).
\end{proof}
\section{Banach-Lie groups are {\boldmath$L^1$}-regular}\label{secbana}
In this section, we prove Theorem~C and related results.
\begin{defn}
Let $P$ be a set, $(F,\|.\|)$ be a Banach space
and $U\sub F$ be a subset.
We say that a mapping
\[
f\colon P\times U\to F
\]
defines a uniform family of contractions
if there exists $\theta\in [0,1[$ such that
\[
(\forall p\in P)(\forall y,z\in U)\quad
\|f(p,z)-f(p,y)\|\leq\theta \|z-y\|.
\]
\end{defn}
We recall from~\cite{FIX}
(cf.\ also \cite{FRE}):
\begin{la}\label{PARFIX}
Let $E$ be a locally convex space,
$(F,\|.\|)$ be a Banach space, $P\sub E$ and $U\sub F$ be open sets,
$k\in \N_0\cup\{\infty\}$
and
\[
f\colon P\times U\to F
\]
be a $C^k$-map.
Let $Q\sub P$ be the set of all $p\in P$
such that $f_p:=f(p,.)\colon U\to F$, $y\mto f(p,y)$
has a fixed point~$x_p$.
Then $x_p$ $($if it exists$)$ is unique,
$Q$ is open in~$P$ and the map
\[
\psi\colon Q\to U,\quad p\mto x_p
\]
taking a parameter $p\in Q$ to the fixed point of $f_p$
is $C^k$.\,\Punkt
\end{la}
\begin{la}\label{prepabanreg}
Let $\K\in \{\R,\C\}$ and $G$ be a local $\K$-analytic
Lie group modelled on a Banach space,
which admits a global chart.
Let $\cg=L(G)$.
Then $G$ is locally $L^1$-regular
and there is an open $0$-neighbourhood
$\Omega\sub L^1([0,1],\cg)$
on which $\Evol$ is defined
and such that
\[
\Evol\colon \Omega\to AC_{L^1}([0,1],G)
\]
is $\K$-analytic.
\end{la}
\begin{proof}
We may assume that~$G$ is an open subset
of its modelling space~$E$ and $e=0$.
We identify $L(G)=T_0(G)=\{0\}\times E$ with~$E$.
Let $D_G\sub G\times G$ be the domain of the multiplication
\[
\mu\colon D_G\to G.
\]
We shall use
\[
g\colon G\times E\to E,\quad g(x,y):=d\mu(x,0;0,y)
\]
and the second differential
\[
d^{(2)}\mu\colon D_G\times (E\times E)\times (E\times E)\to E.
\]
Endow $E\times E$ with the maximum
norm, $(x,y)\mto\max\{\|x\|,\|y\|\}$. Since $d^{(2)}\mu$ is continuous, there is an open convex
$0$-neighbourhood $U\sub G$ and $r>0$ such that
\[
d^{(2)}\mu(U\times (\wb{B}_r^{E\times E}(0))^2)\sub \wb{B}_1^E(0).
\]
Therefore the continuous bilinear map $d^{(2)}\mu(x,0,.)\colon (E\times E)^2\to E$
has operator norm
\[
\|d^{(2)}\mu(x,0,.)\|_{op}\leq \frac{1}{r^2},
\]
for each $x\in U$.
Given $\theta\in \,]0,1[$, consider the open
ball
\[
P:=\{[\gamma]\in L^1([0,1],E)\colon \int_0^1\|\gamma(t)\|\,dt <
r^2\theta\}
\]
around~$0$ in $L^1([0,1],E)$
and the map
\[
f\colon L^1([0,1],E) \times C([0,1],U)
\to C([0,1],E),\,
f([\gamma],\eta)(t):=\int_0^t g(\eta(s),\gamma(s))\,ds.
\]
Note that
\[
J\colon L^1([0,1],E)\to C([0,1],E),\quad
J([\gamma])(t):=\int_0^t\gamma(s)\,ds
\]
is a continuous linear map with operator norm
$\|J\|_{op}\leq 1$.
Now
\[
\wt{g}\colon C([0,1],U)\times L^1([0,1],E)\to L^1([0,1],E),
\;\, (\eta,[\gamma])\mto [g\circ (\eta,\gamma)]
\]
is a smooth map (as $L^1$ satisfies the pushforward
axiom) and
\[
f=J\circ \wt{g}.
\]
Hence $f$ is smooth.
Moreover, $f|_{P\times C([0,1],U)}$ defines a uniform family of contractions.
In fact, if $[\gamma]\in P$ and $\eta_1,\eta_2\in C([0,1],U)$,
then
\begin{eqnarray*}
\lefteqn{\|g(\eta_2(s),\gamma(s))-g(\eta_1(s),\gamma(s))\|}\qquad\qquad\\
&= & \left\|\int_0^1
dg(\eta_1(s)+ \tau(\eta_2(s)-\eta_1(s)),\gamma(s),\eta_2(s)-\eta_1(s),0)\,d\tau\right\|\\
&\leq &\int_0^1
\|dg(\eta_1(s)+ \tau(\eta_2(s)-\eta_1(s)),\gamma(s),\eta_2(s)-\eta_1(s),0)\|\,d\tau\\
&\leq& \frac{1}{r^2}\|\gamma(s)\|\, \|\eta_2(s)-\eta_1(s)\|
\leq \frac{1}{r^2}\|\gamma(s)\|\, \|\eta_2-\eta_1\|_\infty
\end{eqnarray*}
for all $s\in [0,1]$,
using that
\begin{eqnarray*}
\|dg(x_\tau,\gamma(s),\eta_2(s)-\eta_1(s),0)\|
&=& \|d^{(2)}\mu(x_\tau,0,0,\gamma(s),\eta_2(s)-\eta_1(s),0)\|\\
&\leq&  \|d^{(2)}\mu(x_\tau,0,.)\|_{op} \|\gamma(s)\|\, \|\eta_2(s)-\eta_1(s)\|
\end{eqnarray*}
for each $\tau\in [0,1]$, with
with $x_\tau:=\eta_1(s)+ \tau(\eta_2(s)-\eta_1(s)$.
Hence
\begin{eqnarray*}
\|\wt{g}(\eta_2,[\gamma])-\wt{g}(\eta_1,[\gamma])\|_{L^1}
&=& \int_0^1 \|g(\eta_2(s),\gamma(s))-g(\eta_1(s),\gamma(s))\|\,ds\\
&\leq & \int_0^1\frac{1}{r^2} \|\gamma(s)\|\, \|\eta_2-\eta_1\|_\infty\,ds\\
&\leq & \frac{\|[\gamma]\|_{L^1}}{r^2} \|\eta_2-\eta_1\|_\infty
\leq \theta \|\eta_2-\eta_1\|_\infty
\end{eqnarray*}
and thus
\[
\|f([\gamma],\eta_2)-f([\gamma],\eta_1)\|
\leq
\|J\|_{op}\|\wt{g}(\eta_2,[\gamma])-\wt{g}(\eta_1,[\gamma])
\leq \theta\|\eta_2-\eta_1\|_\infty.
\]
Thus $f$ defines a uniform family of contractions.
By Lemma~\ref{PARFIX}, the set $\Omega$ of all $[\gamma]\in P$
for which $f([\gamma],.)\colon C([0,1],U)\to C([0,1],E)$
has a fixed point $\psi([\gamma]):=\eta$
is open in~$P$, and the map
\[
\psi\colon \Omega \to C([0,1],U),\quad [\gamma]\mto \psi([\gamma])
\]
is smooth.
Note that $f(0,0)=0$, i.e., $0\in C([0,1],U)$
is a fixed point of $f(0,.)$.
Hence $0\in \Omega$ and thus $\Omega$ is an open $0$-neighbourhood
in $L^1([0,1],E)$.
Note that, for each $[\gamma]\in \Omega$,
$\eta:=\psi([\gamma])$ satisfies
\begin{equation}\label{fppreex}
\eta=f([\gamma],\eta)=J(\wt{g}(\eta,[\gamma])),
\end{equation}
whence $\eta\in \im(J)\sub AC_{L^1}([0,1],E)$.
By (\ref{fppreex}), we have
\[
\eta(t)=\int_0^t g(\eta(s),\gamma(s))\,ds
\]
for all $t\in [0,1]$. Hence $\eta$ is a Caratheodory solution to
\[
\eta'(t)=g(\eta(s),\gamma(s)),\quad \eta(0)=0.
\]
In other words, $\eta$ is a Caratheodory solution
to
\[
\eta'(t)=d\mu(\eta(s),0,0,\gamma(s)),\quad \eta(0)=0
\]
and thus $\eta=\Evol([\gamma])$.
Thus $G$ is locally $L^1$-regular.\\[2.3mm]
After shrinking $G$, we may assume that the inclusion
map from $G$ into $\wt{G}$ is a homomorphism
of real analytic local Lie groups
for a complex analytic local Banach-Lie group $\wt{G}$
which is an open subset of $E_\C=E\oplus iE$.
By the preceding, $\wt{G}$ is locally $L^1$-regular.
Hence
\[
\Evol=\psi\colon \Omega\to G
\]
is real analytic (possibly after shrinking the
open $0$-neighbourhood $\Omega\sub \cE([0,1],E)$),
by Lemma~\ref{realvscxlcly}.
\end{proof}
We deduce the following result (which subsumes
Theorem~C):
\begin{prop}
Let $\K\in \{\R,\C\}$ and $G$ be a $\K$-analytic
Banach-Lie group with Lie algebra $\cg=L(G)$.
Then $G$ is $L^p$-regular
for each $p\in [1,\infty]$
and
\[
\Evol\colon L^p([0,1],\cg)\to AC_{L^p}([0,1],G)
\]
is $\K$-analytic.
\end{prop}
\begin{proof}
By Lemma~\ref{prepabanreg},
$G$ is locally $L^1$-regular,
whence $G$ is $L^1$-regular
(by Proposition~\ref{lctoglb})
and hence $L^p$-regular
(see Theorem~A). Let
\[
\Evol_G\colon L^p([0,1],\cg)\to AC_{L^p}([0,1],G)
\]
be the evolution map. To see that $\Evol_G$
is not only smooth, but real analytic,
let $U\sub G$ be an open symmetric
identity neighbourhood on which a chart is defined
and which injects into a complex analytic local
Lie group $\wt{U}$ with global chart,
such that $U$ is the fixed point set
of an antiholomorphic involution
and the inclusion map $U\to \wt{U}$
an isomorphism of local groups onto the latter.
Then $\wt{U}$ is locally $L^1$-regular
(by Lemma~\ref{prepabanreg})
and thus $\wt{U}$ is locally $L^p$-regular
(see Corollary~\ref{corsimA}
and Remark~\ref{implalsolc}).
As a consequence, $U$
(and hence also $G$)
is locally $L^p$-regular
and
\[
\Evol_G|_\Omega
\]
is real analytic on some
open $0$-neighbourhood $\Omega\sub L^p([0,1],\cg)$
(Lemma~\ref{realvscxlcly}).
Hence $\Evol_G$ is real analytic,
by Proposition~\ref{sammelsu}\,(d).
\end{proof}
\section{Measurable regularity for projective limits}\label{secPL}
We describe situations where measurable
regularity properties pass from the steps $G_n$
of a projective system
\[
\cdots \to G_2\to G_1
\]
of Lie groups to the projective limit $G=\pl\,G_n$.\vspace{-0.5mm}
As a tool, we first discuss projective limits
of Lebesgue spaces
and spaces of absolutely continuous
functions.
%
%
The proofs of Lemmas~\ref{compaPL}--\ref{compaPLrc2}
have been relegated to an appendix (Appendix~\ref{appeB}).
\begin{la}\label{compaPL}
Let $(X,\Sigma,\mu)$ be a measure space,
$((E_n)_{n\in\N}, (\phi_{n,m})_{n\leq m})$ be a projective
system of Fr\'{e}chet spaces $E_n$ and continuous linear mappings\linebreak
$\phi_{n,m}\colon E_m\to E_n$ for $n\leq m$.
Let $E$ be a Fr\'{e}chet space such that
$E=\pl\, E_n$,\vspace{-1.3mm} with the limit maps
$\phi_n\colon E\to E_n$.
Then
\[
L^p(X,\mu,E)=\pl\, L^p(X,\mu,E_n)
\]
for each $p\in [1,\infty]$,
with bonding maps $L^p(X,\mu,\phi_{n,m})$ and the limit maps $L^p(X,\mu,\phi_n)$.
\end{la}
\begin{la}\label{compaPLrc}
Let $(X,\Sigma,\mu)$ be a measure space,
$((E_n)_{n\in \N}, (\phi_{n,m})_{n\leq m})$ be a projective
system of locally convex spaces $E_n$ and continuous linear maps
$\phi_{n,m}\colon E_n\to E_m$ for $n\leq m$ in~$\N$.
Let $E$ be a locally convex space such that
$E=\pl\, E_n$,\vspace{-1.3mm} with the limit maps
$\phi_n\colon E\to E_n$.
Then
\[
L^\infty_{rc}(X,\mu,E)=\pl\, L^\infty_{rc}(X,\mu,E_n)
\]
with bonding maps $L^\infty_{rc}(X,\mu,\phi_{n,m})$
and the limit maps $L^\infty_{rc}(X,\mu,\phi_n)$.
\end{la}
\begin{la}\label{compaPLrc2}
Let $a<b$ be real numbers,
$((E_n)_{n\in \N}, (\phi_{n,m})_{n\leq m})$ be a projective
system of Fr\'{e}chet spaces~$E_n$ and continuous linear maps
$\phi_{n,m}\colon E_n\to E_m$ for $n\leq m$ in~$\N$.
Let $E$ be a Fr\'{e}chet space such that
$E=\pl\, E_n$,\vspace{-1.3mm} with the limit maps
$\phi_n\colon E\to E_n$.
If
$\phi_{n,m}(E_m)$ is
dense in~$E_n$ for all $n,m\in \N$ such that $n\leq m$,
then
$R([a,b],E)=\pl\, R([a,b],E_n)$\vspace{-0.5mm}
with bonding maps $R([a,b],\phi_{n,m})$
and the limit maps $R([a,b],\phi_n)$.
\end{la}
\begin{la}\label{plACE}
Let $\cE$ be a bifunctor on Fr\'{e}chet spaces
$($resp., integral complete
locally convex spaces$)$
and $a<b$ be real numbers.
Let $((E_n)_{n\in \N}, (\phi_{n,m})_{n\leq m})$ be a projective
system of Fr\'{e}chet spaces $($resp., integral complete
locally convex spaces$)$~$E_n$ and continuous linear maps
$\phi_{n,m}\colon E_n\to E_m$ for $n\leq m$ in~$\N$.
Let $E$ be a locally convex space such that
$E=\pl\, E_n$,\vspace{-0.5mm} with the limit maps
$\phi_n\colon E\to E_n$.
We assume that
\[
\cE([a,b],E)= \pl\, \cE([a,b],E_n)\vspace{-0.5mm}
\]
as a locally convex space,
with bonding maps $\cE([a,b],\phi_{n,m})$
and the limit maps $\cE([a,b],\phi_n)$.
Then
\[
AC_\cE([a,b],E)=\pl\,AC_\cE([a,b],E_n)\vspace{-.5mm}
\]
as a locally convex space,
with bonding maps $AC_\cE([a,b],\phi_{n,m})$
and the limit maps $AC_\cE([a,b],\phi_n)$.
\end{la}
\begin{proof}
For each $n\in\N$, the map
\[
\psi_n\colon AC_\cE([a,b],E_n)\to E_n\times \cE([a,b],E_n),\quad
\eta\mto (\eta(a),\eta')
\]
is an isomorphism of topological vector spaces and so is the corresponding map
$\psi\colon AC_\cE([a,b],E)\to E\times \cE([a,b],E)$.
Hence
\[
AC_\cE([a,b],E)\cong E\times \cE([a,b],E)=\pl\, (E_n\times \cE([a,b],E_n)
\cong \pl\, AC_\cE([a,b],E_n).
\]
In more detail, let us make the bonding maps
and limit maps explicit which are involved,
to ensure the final conclusion of the lemma.
First, note that the spaces
\[
E_n\times \cE([a,b],E_n)
\]
form a projective system
together with the bonding maps $\phi_{n,m}\times \cE([a,b], \phi_{n,m})$.
By the compatibility of projective limits and direct products,
we have that
\[
E\times \cE([a,b],E)=\pl\, (E_n\times \cE([a,b],E_n))
\]
for this system, with the limit maps $\phi_n\times \cE([a,b], \phi_n)$.
Next, note that the locally convex spaces
\[
AC_\cE([a,b],E_n)
\]
form a projective system with the bonding maps $AC_\cE([a,b],\phi_{n,m})$.
Since
\[
(\phi_{n,m}\times \cE([a,b], \phi_{n,m}))\circ \psi_m=\psi_n\circ
AC_\cE([a,b], \phi_{n,m})
\]
and thus
\[
\psi_n^{-1}\circ (\phi_{n,m}\times \cE([a,b], \phi_{n,m}))=
AC_\cE([a,b], \phi_{n,m})\circ \psi_m^{-1},
\]
we have
\[
E\times \cE([a,b],E)=\pl\, AC_\cE([a,b], E_n)
\]
for the preceding projective system, with the limit maps
$\psi_n^{-1}\circ (\phi_n\times \cE([a,b], \phi_n)$.
Since $\psi\colon AC_\cE([a,b],E)\to E\times \cE([a,b], E)$
is an isomorphism of topological vector spaces,
we conclude that
\[
AC_\cE([a,b],E)=\pl\,AC_\cE([a,b], E_n)
\]
for the preceding projective system,
with the desired limit maps
\[
\psi_n^{-1}\circ (\phi_n\times \cE([a,b], \phi_n)\circ \psi=AC_\cE([a,b],\phi_{n,m}).
\]
We used that $(\phi_n\times \cE([a,b],\phi_n))\circ \psi
=\psi_n\circ AC_\cE([a,b],\phi_{n,m})$
as both maps take $\eta\in AC_\cE([a,b], E)$
to $(\phi_n(\eta(a)),\phi_n\circ \eta')=
(\phi_n(\eta(a)),(\phi_n\circ \eta)')$.
\end{proof}
\begin{defn}\label{dplcha}
Let $((G_n)_{n\in \N}, (q_{n,m})_{n\leq m})$
be a projective system of Lie groups $G_n$ modelled
on locally convex spaces $E_n$ and smooth group
homomorphisms $q_{n,m}\colon G_m\to G_n$.
Let $G$ be a Lie group modelled on a locally convex space~$E$
such that
\[
G=\pl\,G_n\vspace{-.5mm}
\]
for the above projective system as a set, with limit maps $q_n\colon G\to G_n$
which are smooth group homomorphisms.
We say that a chart $\phi\colon U\to E$ of~$G$
with $e\in U$ is a \emph{projective limit chart}
if
\begin{itemize}
\item[(a)]
There exist continuous linear maps
\[
\alpha_n\colon E\to E_n\quad\mbox{and}\quad \alpha_{n,m}\colon E_m\to E_n
\]
such that $((E_n)_{n\in \N}, (\alpha_{n,m})_{n\leq m})$
is a projective system of locally convex spaces and
\[
E=\pl\,E_n\vspace{-.5mm}
\]
for this system as a locally convex space, with the limit maps $\alpha_n$; and
\item[(b)]
$G_n$ is modelled on $E_n$ and there exist charts $\phi_n\colon U_n\to V_n$
of $G_n$ such that
\[
\phi_{n,m}(U_m)\sub U_n\quad \mbox{and}\quad
\alpha_{n,m}(V_m)\sub V_n
\]
for all $n,m\in \N$ with $n\leq m$,
\[
U=\bigcap_{n\in \N}q_n^{-1}(U_n),\quad
V=\bigcap_{n\in \N}\alpha_n^{-1}(V_n),
\]
\[
\phi_n\circ q_{n,m}|_{U_m}=\alpha_{n,m}\circ \phi_m\quad\mbox{for all $n,m\in \N$
with $n\leq m$},
\]
\[
q_n(U)\sub U_n\quad\mbox{and}\quad \alpha_n(V)\sub V_m\quad
\mbox{for all $n\in \N$,}
\]
and
\[
\alpha_n \circ \phi=\phi_n\circ q_n|_U\quad\mbox{for all $n\in\N$.}
\]
\end{itemize}
\end{defn}
We mention that existence of projective limit
charts can also be characterized as follows
(see Appendix~\ref{appeB}
for the straightforward proof).
\begin{la}\label{reformug}
A Lie group $G=\dl\,G_n$\vspace{-.5mm}
as in Definition~\emph{\ref{dplcha}}
admits a projective limit chart if and only if
\[
L(G)=\pl\, L(G_n)\vspace{-.5mm}
\]
as a locally convex space with the bonding maps
$L(q_{n,m})$ and limit maps $L(q_n)$,
and there exist $C^\infty$-diffeomorphisms
\[
\psi\colon U\to W\quad\mbox{and}\quad \psi_n\colon U_n\to W_n
\]
from open identity neighbourhoods
$U\sub G$ and $U_n\sub G_n$ onto open sets
$W\sub L(G)$ and $W_n\sub L(G_n)$, respectively,
such that
\[
d\psi|_{L(G)}=\id_{L(G)},\quad d\psi_n|_{L(G_n)}=\id_{L(G_n)}\;\;\;\mbox{for all $n\in \N$,}
\]
\[
q_{n,m}(U_m)\sub U_n\quad\mbox{and}\quad L(q_{n,m})(W_m)\sub W_n
\]
for all $n,m\in \N$ with $n\leq m$,
\[
U=\bigcap_{n\in \N}q_n^{-1}(U_n),\quad
W=\bigcap_{n\in \N}L(q_n)^{-1}(W_n),
\]
\begin{equation}\label{compaby}
\psi_n\circ q_{n,m}|_{U_m}=L(q_{n,m})\circ \psi_m
\quad\mbox{for all $n,m\in \N$
with $n\leq m$},
\end{equation}
\begin{equation}\label{compby3}
q_n(U)\sub U_n\quad\mbox{and}\quad L(q_n)(W)\sub W_n\quad\mbox{for all
$n\in \N$}
\end{equation}
and
\begin{equation}\label{compby2}
L(q_n)\circ \psi=\psi_n\circ q_n|_U\quad\mbox{for all $n\in\N$.}
\end{equation}
\end{la}
\begin{rem}\label{maysymm}
After shrinking $U$ and the $U_n$,
we can always achieve that $U=U^{-1}$ and $U_n=U_n^{-1}$
are symmetric identity neighbourhoods for all
$n\in \N$ in Definition~\ref{dplcha}.\\[2.3mm]
[In fact, $q_n(U\cap U^{-1})=q_n(U)\cap q_n(U)^{-1}\sub U_n\cap U_n^{-1}$
and thus
\[
q_n(U')\sub U_n'
\]
for all $n\in \N$ if we define $U':=U\cap U^{-1}$ and $U_n':=U_n\cap U_n^{-1}$.
We also set $V':=\phi(U')$ and $V_n':=\phi_n(U_n')$.
Then
\begin{equation}\label{stillea}
\alpha_n(V')=\alpha_n(\phi(U'))=\phi_n(q_n(U'))
\sub \phi_n(U_n')=V_n'
\end{equation}
for $n\in \N$. Likewise,
\[
q_{n,m}(U_m')\sub q_{n,m}(U_m)\cap q_{n,m}(U_m)^{-1}\sub U_n\cap U_n^{-1}=U_n'
\]
for $n\leq m$ in $\N$
and thus
\[
\alpha_{n,m}(V_m')=\alpha_{n,m}(\phi_m(U_m'))=\phi_n(q_{n,m}(U_m'))
\sub \phi_n(U_n')=V_n'.
\]
Realizing the projective limit as a subgroup of
$\prod_{n\in \N}G_n$ as usual, we see that $U=\bigcap_{n\in \N}q_n^{-1}(U_n)$
entails
\[
U'=\bigcap_{n\in \N} q_n^{-1}(U_n').
\]
By (\ref{stillea}), we have
\[
V'\sub \bigcap_{n\in \N}\alpha_n^{-1}(V_n').
\]
To see that equality holds, let $y\in \bigcap_{n\in \N}\alpha_n^{-1}(V_n')$.
Then $y\in V$. Abbreviate $x:=\phi^{-1}(y)$.
For each $n\in \N$, we have that
\[
\phi_n(q_n(x))=\alpha_n(\phi(x))=\alpha_n(y)\in V_n'
\]
and hence $q_n(x)\in U_n'$. Thus $x\in \bigcap_{n\in \N}q_n^{-1}(U_n')=U'$
and $y=\phi(x)\in V'$.
\end{rem}
\begin{prop}\label{passtoPL}
Let $\cE$ be a bifunctor on Fr\'{e}chet spaces
$($resp., integral complete
locally convex spaces$)$
which satisfies the locality axiom, the pushforward
axioms, has the subdivision property,
and such that smooth functions act smoothly on $AC_\cE$.
Let $((G_n)_{n\in \N}, (q_{n,m})_{n\leq m})$
be a projective system of Lie groups $G_n$ modelled
on Fr\'{e}chet spaces $($resp.,
integral complete locally convex spaces$)$ $E_n$ and smooth group
homomorphisms $q_{n,m}\colon G_m\to G_n$.
Let $G$ be a Lie group modelled on a locally convex space~$E$
such that
\[
G=\pl\,G_n\vspace{-.5mm}
\]
for the above projective system as a set,
with limit maps $q_n\colon G\to G_n$
which are smooth group homomorphisms.
Assume that
\begin{itemize}
\item[\rm(a)]
$G$ admits a projective limit
chart in the sense of Definition~\emph{\ref{dplcha}};
\item[\rm(b)]
$G_n$ is $\cE$-regular for each $n\in\N$; and
\item[\rm(c)]
$\cE([0,1],L(G))=\pl\,\cE([0,1],L(G_n))$\vspace{-.5mm}
with respect to the bonding maps $\cE([0,1],L(q_{n,m}))$
and limit maps $\cE([0,1],L(q_n))$.
\item[\rm(d)]
$U_n$ in Lemma~\emph{\ref{reformug}} can be chosen such that $U_n=U_n^{-1}$ and
\[
\Omega:=\bigcap_{n\in \N}\cE([0,1], L(q_n))^{-1}(\Evol_{G_n}^{-1}(AC_\cE([0,1],U_n)))
\]
is a $0$-neighbourhood in $\cE([0,1],L(G))$.
\end{itemize}
Then also $G$ is $\cE$-regular.
\end{prop}
\begin{rem}\label{basfuauto}
By Lemma~\ref{compaPL}
and Lemma~\ref{compaPLrc},
condition (c) of Proposition~\ref{passtoPL}
is automatically satisfied if $\cE=L^p$
as a bifunctor on Fr\'{e}chet spaces
or $\cE=L^\infty_{rc}$ as a bifunctor
to integral complete locally convex spaces.
If $\cE=R$ as a bifunctor to Fr\'{e}chet spaces
and $L(q_{n,m})$ has dense image for all positive integers
$n\leq m$, then
condition (c) of Proposition~\ref{passtoPL}
is satisfied by Lemma~\ref{compaPLrc2}.
\end{rem}
Before proving Proposition~\ref{passtoPL},
let us spell out simple situations
in which condition (d) of the proposition
is satisfied:
\begin{la}\label{codaut}
Condition \emph{(d)} from Proposition \emph{\ref{passtoPL}}
is automatically satisfied in the following situations:
\begin{itemize}
\item[\rm(i)]
$L(q_{n,m})$ and $q_{n,m}$ are injective for all positive integers $n\leq m$
$($whence also $L(q_n)$ is injective on $L(G)=\pl\,L(G_n)$
and $q_n$ is injective on $G=\pl\, G_n)$. Or:\vspace{-.5mm}
\item[\rm(ii)]
$U_n=G_n$ for each $n\in \N$.
\end{itemize}
\end{la}
\begin{proof}
In the situation of (ii), we simply have $\Omega=\cE([0,1],L(G)$.
To prove (i),
after identifying $L(G_n)$ with a vector subspace
of $L(G_1)$ by means of the injective linear map $L(q_{1,n})$,
we may assume that
\[
L(G_1)\supseteq L(G_2)\supseteq \cdots
\]
and $L(q_{n,m})$ is the inclusion map for all integers $n\leq m$.
We identify $L(G)$ with the vector subspace $\bigcap_{n\in \N}L(G_n)$
of $L(G_1)$. Then also $L(q_n)$ becomes the inclusion map
for each $n\in \N$. In the situation of
Lemma~\ref{reformug}, we then have
$U_n\supseteq U_m$ for all positive integers $n\leq m$
and $U=\bigcap_{n\in \N}U_n$.

By the definition of the topology on the projective limit $L(G)$,
there is $k\in \N$ and a $0$-neighbourhood $P\sub L(G_k)$
such that $L(G)\cap P=L(q_k)^{-1}(P)\sub U$.
After passing to a cofinal subsequence,
we may assume that $k=1$.
Thus
\[
P\sub U\sub U_n\quad\mbox{for all $n\in \N$}
\]
and $P=P\cap L(G_n)=L(q_{1,n})^{-1}(P)$ is an open $0$-neighbourhood
in $L(G_n)$ for each $n\in \N$.
As a consequence, $Q_n:=\psi_n^{-1}(P)\sub U_n$ is an open identity neighbourhood
in $G_n$ and $Q:=\psi^{-1}(P)$ is an open identity neighbourhood
in~$G$.
We have
\[
\psi_n(q_n(Q))=L(q_n)(\psi(Q))=L(q_n)(P)=P
\]
and thus
\[
q_n(Q)=\psi_n^{-1}(P)=Q_n
\]
for each $n\in \N$.
Hence
\[
Q\sub \bigcap_{n\in \N}q_n^{-1}(Q_n).
\]
To see that
\begin{equation}\label{intsc}
Q = \bigcap_{n\in \N}q_n^{-1}(Q_n),
\end{equation}
let
$x\in \bigcap_{n\in \N}q_n^{-1}(Q_n)$.
Then $x\in \bigcap_{n\in \N}q_n^{-1}(U_n)=U$.
We have
\[
\psi(x)=L(q_n)(\psi(x))=\psi_n (q_n(x))\in \psi_n(Q_n)=P
\]
for each $n\in \N$ and thus $x\in \psi^{-1}(P)=Q$.
Likewise,
\[
\psi_n(q_{n,m}(Q_m))=L(q_{n,m})(\psi_m(Q_m))=L(q_{n,m})(P)=P
\]
and thus
\begin{equation}\label{prsrs}
q_{n,m}(Q_m)=Q_n
\end{equation}
for all positive integers $n\leq m$.
We claim that the open $0$-neighbourhood
\begin{eqnarray}
\Omega'&:=& \{[\gamma]\in \cE([0,1],L(G))\colon \Evol_{G_1}([L(q_1)\circ \gamma])\in
AC_\cE([0,1],Q_1)\}\notag\\
&=& \cE([0,1],L(q_1))^{-1}(\Evol_{G_1}^{-1}(AC_\cE([0,1],Q_1)))\label{RHs}
\end{eqnarray}
in $\cE([0,1],L(G))$ coincides with the subset
\[
\Omega':=\bigcap_{n\in \N}\cE([0,1], L(q_n))^{-1}(\Evol_{G_n}^{-1}(AC_\cE([0,1],Q_n)))
\]
of the intersection $\Omega$
from Proposition~\ref{passtoPL}\,(d).
If this is true, then indeed $\Omega$ is a $0$-neighbourhood.
To prove the claim,
note first that $\Omega'$ is a subset of~(\ref{RHs})
by definition.
To prove the converse inclusion,
let $[\gamma]\in \cE([0,1],L(G))$
such that
\[
\Evol_{G_1}([L(q_1)\circ \gamma])\in AC_\cE([0,1],Q_1).
\]
Then
\begin{eqnarray*}
q_{1,n}\circ \Evol_{G_n}([L(q_n)\circ \gamma])
&=& \Evol_{G_1}([L(q_{1,n})\circ L(q_n)\circ \gamma])\\
&=& \Evol_{G_1}(L(q_1)\circ \gamma])
\in AC_\cE([0,1],Q_1),
\end{eqnarray*}
showing that $\Evol_{G_n}([L(q_n)\circ\gamma])$
takes its values in
\[
q_{1,n}^{-1}(Q_1)=Q_n
\]
(using (\ref{prsrs}) and the injectivity of $q_{1,n}$).
Thus
\[
\Evol_{G_n}([L(q_n)\circ \gamma])\in AC_\cE([0,1],Q_n)
\sub AC_\cE([0,1],U_n)
\]
for each $n\in \N$ and thus
$[\gamma] \in \Omega'$.
\end{proof}
{\bf Proof of Proposition~\ref{passtoPL}.}
Let $\psi_n\colon U_n\to W_n$
and $\psi\colon U\to W$ be as in Lemma~\ref{reformug}.
We may assume that $U$ and each $U_n$ is a symmetric
identity neighbourhood.
Hence $AC_\cE([0,1],\psi)$ and
$AC_\cE([0,1],\psi_n)$ are charts for
$AC_\cE([0,1],G)$ and $AC_\cE([0,1], G_n)$,
respectively.
Since $G_n$ is $\cE$-regular, we have a smooth evolution map
\[
\Evol_{G_n}\colon \cE([0,1],L(G_n))\to AC_\cE([0,1],G_n).
\]
Since $AC_\cE([0,1],U_n)$ is an open identity neighbourhood
in $AC_\cE([0,1],G_n)$, we deduce that
\[
\Omega_n:=(\Evol_{G_n})^{-1}(AC_\cE([0,1],U_n))
\]
is an open $0$-neighbourhood in $\cE([0,1],L(G_n))$.
Since
\[
q_{n,m}\circ \Evol_{G_m}=\Evol_{G_m}\circ L(q_{n,m})
\]
and $q_{n,m}(U_m)\sub U_n$, we deduce that
\[
AC_\cE([0,1],L(q_{n,m}))(\Omega_m)\sub \Omega_n\quad\mbox{for all $n\leq m$ in $\N$.}
\]
Hypothesis (c) implies that
\begin{equation}\label{pretly}
AC_\cE([0,1],L(G))=\pl\,AC_\cE([0,1],L(G_n))\vspace{-.8mm}
\end{equation}
using the bonding maps $AC_\cE([0,1],L(q_{n,m}))$
and limit maps $AC_\cE([0,1],L(q_n))$
(see Lemma~\ref{plACE}).
Define
\[
\Omega:=\{[\gamma]\in \cE([0,1],L(G))\colon (\forall n\in \N)\,
\cE([0,1],L(q_n))([\gamma])\in \Omega_n\}.
\]
If $[\gamma]\in \Omega$, then
\[
\eta_n:=\Evol_{G_n}(\cE([0,1],L(q_n))([\gamma]))
=\Evol_{G_n}([L(q_n)\circ \gamma])\in AC_\cE([0,1],G_n)
\]
for each $n\in \N$.
Then
\[
q_{n,m}\circ \eta_m=\Evol_{G_n}([L(q_{n,m})\circ L(q_m)\circ \gamma])
=\Evol_{G_n}([L(q_n)\circ \gamma])=\eta_n
\]
for all $n\leq m$ in $\N$,
whence there is a unique map
\[
\eta\colon [0,1]\to U\sub G
\]
such that $q_n\circ \eta=\eta_n$ for all $n\in \N$.
Define
\[
\zeta:=\psi\circ \eta\quad\mbox{and}\quad \zeta_n:=\psi_n\circ \eta_n\;\;\mbox{for
$n\in \N$.}
\]
Then
\[
L(q_{n,m})\circ \zeta_m=L(q_{n,m})\circ \psi_m\circ \eta_m
=\psi_n\circ q_{n,m}\circ \eta_m=\psi_n\circ\eta_n=\zeta_n
\]
for all $n\leq m$ in $\N$.
By (\ref{pretly}),
there is $\theta\in AC_\cE([0,1],L(G))$ such that
\begin{equation}\label{prelll}
L(q_n)\circ \theta=\zeta_n\quad\mbox{for all $n\in \N$,}
\end{equation}
Now
\begin{equation}\label{pre222}
L(q_n)\circ \zeta =L(q_n)\circ \psi \circ \eta
=\psi_n\circ q_n \circ \eta =\psi_n\circ\eta_n=\zeta_n
\end{equation}
for all $n\in \N$.
As the maps $L(q_n)$ separate points on $L(G)=\pl\,L(G_n)$\vspace{-.5mm}
for $n\in \N$, we deduce from (\ref{prelll}) and (\ref{pre222})
that
\begin{equation}\label{hencequl}
\zeta=\theta\in AC_\cE([0,1],L(G)).
\end{equation}
Hence $\eta=\phi^{-1}\circ \zeta\in AC_\cE([0,1],G)$.
Since
\[
\cE([0,1], L(q_n))(\delta^\ell(\eta))
=\delta^\ell (q_n\circ \eta)=
\delta^\ell(\eta_n)=
\cE([0,1],L(q_n))([\gamma])
\]
for each $n\in \N$ and the
maps $\cE([0,1],L(q_n))$ separate points on
$\cE([0,1],L(G))=\pl\,\cE([0,1],L(G_n))$,\vspace{-.8mm}
we deduce that
\[
\delta^\ell(\eta)=[\gamma]
\]
and thus $\eta=\Evol([\gamma])$.
If $\Omega$ is a $0$-neighbourhood in $\cE([0,1],L(G))$,
then $G$ is locally $\cE$-regular by the preceding.
Moreover,
\begin{eqnarray*}
\lefteqn{AC_\cE([0,1],L(q_n))\circ
AC_\cE([0,1],\psi)\circ \Evol|_{\Omega^0}}\qquad\qquad\\
&=&AC_\cE([0,1],\psi_n)\circ \Evol_{G_n}\circ \cE([0,1],L(q_n))|_{\Omega^0}
\end{eqnarray*}
by (\ref{hencequl}) for each $n\in \N$, which is a smooth
map (where $\Omega^0$ denotes the interior of~$\Omega$).
As a consequence,
the map $AC_\cE([0,1],\psi)\circ \Evol|_{\Omega^0}$
to
\[
AC_\cE([0,1],L(G))=\pl\, AC_\cE([0,1], L(G_n))\vspace{-.8mm}
\]
is smooth and hence $\Evol|_{\Omega^0}$
is smooth. Thus $G$ is locally $\cE$-regular and hence
$\cE$-regular, by Proposition~\ref{lctoglb}.\,\Punkt
\begin{prop}\label{mapKbanreg}
Let $M$ be a finite-dimensional smooth manifold,
$K\sub M$ be a compact set and $H$ be a Banach-Lie group.
Then $C^\infty_K(M,H)$ is $L^1$-regular.
\end{prop}
\begin{proof}
Let $E$ be the modelling space of~$G$
and $\tau\colon U\to V$ be a chart for $G$
with $\tau(e)=0$, defined an a symmetric open
identity neighbourhood $U\sub G$.
Then
\[
\phi:=C^\infty_K(M,\tau)\colon C^\infty_K(M,U)\to C^\infty_K(M,V)\sub C^\infty_K(M,E)
\]
is a chart for $C^\infty_K(M,H)$ and
\[
\phi_n:=C^n_K(M,\tau)\colon C^n_K(M,U)\to C^n_K(M,V)\sub C^n_K(M,E)
\]
is a chart for $C^n_K(M,H)$, for each $n\in \N$ (cf.\ \cite{GCX}).
We have that
\[
C^\infty_K(M,H)=\bigcap_{n\in \N}C^n_K(M,H)=\pl\,C^n_K(M,H)\vspace{-1mm}
\]
as a set (with the respective inclusion maps as
the limit maps and bonding maps).
Using the inclusion maps $\alpha_n\colon C^\infty_K(M,E)\to C^n_K(M,E)$
for $n\in \N$
and $\alpha_{n,m}\colon C^m_K(M,E)\to C^n_K(M,E)$ for positive
integers $n\leq m$, we see that all conditions from Definition~\ref{dplcha}
are satisfied and thus $\phi$ is a projective limit chart.
Hence condition~(a) from
Proposition~\ref{passtoPL}
is satisfied and by Remark~\ref{basfuauto}
and Lemma~\ref{codaut}\,(i),
also conditions~(c) and~(d)
are satisfied.
Since every Banach-Lie group is $L^1$-regular
by Theorem~C, also condition~(b)
is satisfied and hence $C^\infty_K(M,H)$
is $L^1$-regular by Proposition~\ref{passtoPL}.
\end{proof}
\begin{rem}\label{varL1mapgp}
The same argument shows that $C^\infty(M,H)$ is
$L^1$-regular if $H$ is a Banach-Lie group
and $M$ a compact smooth manifold
(possibly with boundary or corners).
\end{rem}
\begin{prop}
Let $\pi\colon P\to M$ be a smooth
principal bundle over a compact
smooth manifold~$M$,
whose structure group is a Banach-Lie group~$H$.
Then the gauge group $\Gau(P)$
is $L^1$-regular.
\end{prop}
\begin{proof}
For some $m\in \N$,
we can cover~$M$ by the interiors
$M_j^0$ of
compact submanifolds $M_j$ of~$M$
with boundary for $j\in \{1,\ldots, m\}$
(of full dimension),
such that there is a smooth section $\sigma_j\colon M_j\to P$
for~$\pi$.
For all $i,j\in \{1,\ldots, m\}$,
there is a unique map $k_{i,j}\colon M_i\cap M_j\to H$
such that
\[
\sigma_i(x)k_{i,j}(x)=\sigma_j(x)\quad
\mbox{for all $x\in M_i\cap M_j$.}
\]
Then $\Gau(P)$ can be identified with the
Lie subgroup $S$ of
\[
G:=\prod_{j=1}^nC^\infty(M_j,H)
\]
consisting of all $\gamma=(\gamma_j)_{j=1,\ldots, m}\in G$
such that
\[
(\forall i,j)(\forall x\in M_i\cap M_j)\; \gamma_i(x)=
k_{i,j}(x)\gamma_j(x)k_{j,i}(x)\}
\]
(see \cite{Woc} or \cite{Jak}).
Each of the Lie groups $C^\infty(M_j,H)$
is $L^1$-regular (see Remark~\ref{varL1mapgp}),
whence also the finite direct product
$G=\prod_{j=1}^n C^\infty(M_j,H)$ is $L^1$-regular
(cf.\ Theorem G).
For $i,j\in \{1,\ldots, m\}$
and $x\in M_i\cap M_j$,
the mappings
\[
\alpha_{i,j,x}\colon G\to H,\quad \gamma\mto \gamma_i(x)
\]
and
\[
\beta_{i,j,k}\colon G\to H,\quad
\gamma\mto
k_{i,j}(x)\gamma_j(x)k_{j,i}(x)
\]
are smooth group homomorphisms.
Since
\[
S=\{\gamma\in G\colon (\forall i,j\in \{1,\ldots, m\}(\forall
x\in M_i\cap M_j)\; \alpha_{i,j,x}(\gamma)=\beta_{i,j,x}(\gamma)\},
\]
we deduce with
Proposition~\ref{equarg}
that $S$ (and hence also $\Gau(P)$) is $L^1$-regular.
\end{proof}
Also the following variant
of Proposition~\ref{passtoPL}
is useful.
\begin{prop}\label{regglocha}
Let $\cE$ be a bifunctor on Fr\'{e}chet spaces
$($resp., integral complete
locally convex spaces$)$
which satisfies the locality axiom, the pushforward
axioms, has the subdivision property,
and such that smooth functions act smoothly on $AC_\cE$.
Let $((G_n)_{n\in \N}, (q_{n,m})_{n\leq m})$
be a projective system of Lie groups $G_n$ modelled
on Fr\'{e}chet spaces $($resp.,
integral complete locally convex spaces$)$ $E_n$ and smooth group
homomorphisms $q_{n,m}\colon G_m\to G_n$.
Let $G$ be a Lie group modelled on a locally convex space~$E$
such that
\[
G=\pl\,G_n\vspace{-.5mm}
\]
for the above projective system as a set,
with limit maps $q_n\colon G\to G_n$
which are smooth group homomorphisms.
Assume that
\begin{itemize}
\item[\rm(a)]
$G$ admits a projective limit
chart $\phi\colon U\to V$ determined by charts
$\phi_n\colon U_n\to V_n$
as in Definition~\emph{\ref{dplcha}},
such that $U=G$ and $U_n=G_n$;
\item[\rm(b)]
$G_n$ is $\cE$-regular for each $n\in\N$; and
\item[\rm(c)]
$\cE([0,1],L(G))=\pl\,\cE([0,1],L(G_n))$\vspace{-.5mm}
with respect to the bonding maps $\cE([0,1],L(q_{n,m}))$
and limit maps $\cE([0,1],L(q_n))$.
\end{itemize}
Then $G$ is $\cE$-regular.
\end{prop}
\begin{proof}
After replacing $G$ with $V$
and $G_n$ with $V_n$,
we may assume that $G$ is an open subset of~$E$ and $G_n$ is an open subset of $E_n$
for each $n\in \N$.
Moreover, $q_n=\alpha_n|_G$ for each $n\in \N$
and $q_{n,m}=\alpha_{n,m}|_{G_m}$ for all positive integers $n\leq m$.
We identify $L(G)$ with $E$ and $L(G_n)$ with $E_n$.
Then $L(q_n)=\alpha_n$ and $L(q_{n,m})=\alpha_{n,m}$.
If $[\gamma]\in \cE([0,1],E)$,
we can form
\[
\eta_n:=\Evol_{G_n}([\alpha_n\circ \gamma])\in AC([0,1], G_n)\sub AC([0,1],E_n).
\]
Then $\alpha_{n,m}\circ \eta_m=q_{n,m}\circ \eta_m=\Evol_{G_n}([L(q_{n,m})\circ L(q_m)\circ\gamma])
=\eta_m$ for all positive integers $n\leq m$.
Since
\[
AC([0,1],E)=\pl\, AC_\cE([0,1],E_n)\vspace{-.8mm}
\]
with the limit maps $AC_\cE([0,1],\alpha_n)$ and bonding maps
$AC_\cE([0,1],\alpha_{n,m})$,
we see that there is a unique $\Psi([\gamma]):=\eta$ in $AC_\cE([0,1],E)$
such that
\[
\alpha_n\circ \eta=\eta_n
\]
for all $n\in \N$. Since
\[
AC_\cE([0,1],\alpha_n)\circ \Psi=\Evol_{G_n}\circ \cE([0,1],\alpha_n)
\]
is smooth for all $n\in \N$, we deduce that $\Psi$ is smooth.
In particular, $\Psi$ is continuous and since $\Psi(0)=e\in AC_\cE([0,1],G)$
and $AC_\cE([0,1],G)$ is open in $AC_\cE([0,1],E)$,
we deduce that
\[
\Omega:=\Psi^{-1}(AC_\cE([0,1],G))
\]
is an open $0$-neighbourhood in $\cE([0,1],E)$.
As in the proof of Proposition~\ref{passtoPL},
we see that $\Psi([\gamma])=\Evol([\gamma])$ for each $[\gamma]\in
\Omega$. Since $\Evol|_\Omega=\Psi|_\Omega$ is smooth,
$G$ is locally $\cE$-regular and hence $\cE$-regular,
as we assume that $\cE$ has the subdivision property.
\end{proof}
See, e.g., \cite{ALG}
for the notion of a continuous inverse algebra $A$
and the fact that its group $A^\times$ of invertible elements
is an analytic Lie group.
For the concept of locally $m$-convex
topological algebra, see \cite{EAM}.
If a locally $m$-convex
continuous inverse algebra is integral
complete, then $A^\times$ is $C^0$-regular~\cite{GN1}.
\begin{prop}\label{algsreg}
Let $A$ be a continuous inverse
algebra.
If $A$ is locally $m$-convex and
a Fr\'{e}chet space,
then its unit group $A^\times$
is $L^1$-regular.
\end{prop}
\begin{proof}
Like every locally $m$-convex Fr\'{e}chet algebra,
$A$ is a countable projective limit $A=\pl\,A_n$\vspace{-.5mm}
of Banach algebras~$A_n$. Since $A_n^\times$ is
$L^1$-regular by Theorem~C and the identity maps
on $A^\times$ and $A_n^\times$ are global charts,
we deduce with Proposition~\ref{regglocha}
that $A^\times$ is $L^1$-regular.
\end{proof}
Neeb and Wagemann \cite{NaW} constructed a regular
Lie group structure on the mapping group
$C^\infty([0,1],H)$, for each
regular Lie group~$H$.
\begin{prop}
$C^\infty(\R,H)$ is $L^1$-regular
for each Banach-Lie group~$H$.
\end{prop}
\begin{proof}
Since $C^\infty(\R,H)=C^\infty(\R,H)_*\rtimes H$
where $H$ is $L^1$-regular by Theorem~C,
we deduce with Theorem~G that $C^\infty(\R,H)$
will be $L^1$-regular if we can show
that
\[
C^\infty(\R,H)_* :=\{\eta\in C^\infty(\R,H)\colon \eta(0)=e\}
\]
is $L^1$-regular.
Abbreviate $\ch:=L(H)$. As shown in \cite{NaW}, the map
\[
\phi\colon C^\infty(\R,H)_*\to C^\infty(\R,\ch),\quad
\eta\mto \delta^\ell(\eta)
\]
is a global chart for $C^\infty(\R,H)$.
Likewise,
\[
\phi_n\colon
\phi\colon C^\infty([{-n}.n],H)_*\to C^\infty([{-n},n],\ch),\quad
\eta\mto \delta^\ell(\eta)
\]
is a global chart for the Lie subgroup
\[
C^\infty([{-n},n],H)_*:=\{\eta\in C^\infty([{-n},n],H)\colon \eta(0)=e\}
\]
of $C^\infty([{-n},n],H)$.
Now
\[
C^\infty([{-n},n],H)_*:=\{\eta\in C^\infty([{-n},n],H)\colon \alpha(\eta)=\beta(\eta)\}
\]
with the smooth homomorphisms $\alpha,\beta\colon C^\infty([{-n},n],H)\to H$,
$\alpha(\eta):=\eta(0)$, $\beta(\eta):=e$.
Since $C^\infty([{-n},n],H)$
is $L^1$-regular (see Remark~\ref{varL1mapgp}),
we deduce with Proposition~\ref{equarg}
that $C^\infty([{-n},n],H)$ is $L^1$-regular.
Note that
\[
C^\infty(\R,\ch)=\pl\, C^\infty([{-n},n],\ch)\vspace{-1mm}
\]
as a locally convex space.
Using Proposition~\ref{regglocha}, we find that
$C^\infty([0,1],H)_*$\linebreak
$=\pl\, C^\infty([{-n},n],H)$\vspace{-.5mm}
is $L^1$-regular.
\end{proof}
\section{Measurable regularity for weak direct\\
products and some direct limits}\label{secDL}
\begin{numba}
Let $(H_j)_{j\in J}$
be a family of Lie groups $H_j$, with modelling space $E_j$.
Let
\[
E:=\bigoplus_{j\in J}E_j:=\Big\{(x_j)_{j \in J}\in\prod_{j\in J}E_j\colon\mbox{$x_j=0$
for all but finitely many $j$}\Big\}
\]
be the direct sum of the given locally convex spaces,
endowed with the locally convex direct sum topology.
Then
\[
G:=\bigoplus_{j\in J}H_j:=\Big\{(x_j)_{j \in J}\in\prod_{j\in J}H_j\colon\mbox{$x_j=e$
for all but finitely many $j$}\Big\}
\]
is a group under pointwise multiplication.
If $\phi_j\colon U_j\to V_j$ is a chart for $H_j$ defined on
an open identity neighbourhood $U_j=U_j^{-1}$ in $H_j$ with $\phi(e)=0$,
then $G$ can be given a Lie group structure modelled on~$E$
such that
\[
\phi:=\oplus_{j\in J}\phi_j\colon \bigoplus_{j\in J}U_j\to \bigoplus_{j\in J}V_j,\quad (x_j)_{j\in J}\mto
(\phi_j(x_j))_{j\in J}
\]
is a chart around the identity element (cf.\ \cite{MEA}).
Here
\[
\bigoplus_{j\in J} U_j:=\bigoplus_{j\in J}H_j\cap \prod_{j\in J}U_j\;\;\;\mbox{and}
\;\;\;
\bigoplus_{j\in J} V_j:=\bigoplus_{j\in J}E_j\cap \prod_{j\in J}V_j.
\]
The Lie group $\bigoplus_{j\in J}H_j$ is called
the \emph{weak direct product} of the family $(H_j)_{j\in J}$ of Lie groups.
\end{numba}
\begin{prop}\label{mrgsum}
If $(H_j)_{j\in J}$
is a family of $L^1$-regular Lie groups $H_j$
modelled on sequentially complete
\emph{(FEP)}-spaces,
then also the weak direct product
$G:=\bigoplus_{j\in J}H_j$
is modelled on a sequentially complete
\emph{(FEP)}-space and $L^1$-regular.
\end{prop}
\begin{proof}
Let $E_j$ be the modelling space of $G_j$ and $E:=\bigoplus_{j\in J}E_j$,
which is a sequentially complete (FEP)-space by Lemma~\ref{lafep}\,(a).
Pick a chart $\phi_j\colon U_j\to V_j\sub E_j$ for $H_j$ around~$e$
with $U_j=U_j^{-1}$ and $\phi_j(e)=0$. Let $\phi:=\oplus_{j\in J}\phi_j\colon U\to V$
be the corresponding chart of $G$, with $U:=\bigoplus_{j\in J}U_j$
and $V:=\bigoplus_{j\in J}V_j\sub E$.
We identify $L(G)$ with $E$ using the isomorphism $d\phi|_{L(G)}$,
and $L(H_j)$ with $E_j$ using $d\phi_j|_{L(H_j)}$.
If $F\sub J$ is a finite set, we
consider
\[
G_F:=\prod_{j\in F}H_j
\]
as a Lie subgroup of~$G$
and
identify $L(G_F)$
with $\prod_{j\in F}E_j$.
Let $\iota_F\colon G_F\to G$ be the inclusion map.
With identifications as before, $L(\iota_F)$ is the inclusion map
\[
\prod_{j\in F}E_j\to E.
\]
If $[\gamma]\in L^1([0,1],E)$, after changing the representative if necessary we may assume that
$\gamma\in \cL^1([0,1],\prod_{j\in F}E_j)$
for some finite subset $F\sub J$ (see Lemma~\ref{lafep}\,(a)).
Hence $\eta:=\Evol_{G_F}([\gamma])$ is defined
and
\[
\delta^\ell(\iota_F\circ\eta)=L(\iota_F)\circ\delta^\ell(\eta)=[\gamma]
\]
entails that
\begin{equation}\label{sreada1}
\iota_F\circ \eta=\Evol_G([\gamma]).
\end{equation}
Let $\gamma_j$ be the $j$-th component of~$\gamma$.
Then
\begin{equation}\label{sreada}
\eta=(\Evol_{G_j}([\gamma_j]))_{j\in F}.
\end{equation}
Recall from Lemma~\ref{lafep}\,(a)
that the summation map
\[
\Sigma \colon \bigoplus_{j\in J}L^1([0,1], E_j)\to L^1([0,1],E)
\]
is an isomorphism of topological vector spaces; the inverse map is
\[
\Sigma^{-1}\colon L^1([0,1],E)\to\bigoplus_{j\in J}L^1([0,1],E_j),\;\;
([\gamma])_{j\in J}\mto ([\gamma_j])_{j\in J}.
\]
Consider the map
\[
\Phi\colon \bigoplus_{j\in J}C([0,1],H_j)\to C([0,1],G)
\]
taking $(\gamma_j)_{j\in J}$
to the function with components $\gamma_j$.
Then $\Phi$ is a group homomorphism and smooth,
because $\Phi$ takes the open set $\bigoplus_{j\in J}C([0,1],U_j)$
into $C([0,1], U)$ and
\[
C([0,1],\phi)\circ \Phi\circ \bigoplus_{j\in J}C([0,1],\phi_j)^{-1}
\]
is the restriction of the continuous linear
summation map
\[
\bigoplus_{j\in J} C([0,1],E_j)\to C([0,1],E)
\]
and hence smooth.
Since each of the maps $\Evol_{H_j}\colon L^1([0,1],E_j)\to C([0,1],H_j)$ is smooth,
also
\[
\oplus_{j\in J}\Evol_{H_j}\colon \bigoplus_{j\in J}L^1([0,1], E_j)\to \bigoplus_{j\in J}
C([0,1],H_j)
\]
is smooth (see \cite[Proposition 7.1]{MEA}).
By (\ref{sreada1}) and (\ref{sreada}), we have
\[
\Evol_G=\Phi\circ \Big(\oplus_{j\in J}\Evol_{H_j}\Big)\circ\Sigma^{-1}.
\]
Hence $\Evol_G\colon L^1([0,1], E)\to C([0,1], G)$
is smooth, being a composition of smooth maps.
As a consequence,
$\Evol\colon L^1([0,1],E)\to AC_{L^1}([0,1],G)$ is smooth
(see Proposition~\ref{givsordr}) and thus $G$ is $L^1$-regular.
\end{proof}
%
%
%
%
Theorem~D from the introduction now follows as a corollary.\\[2.3mm]
{\bf Proof of Theorem D.}
Let $(M_j)_{j\in J}$ be a locally finite family of
compact submanifolds $M_j\sub M$ with boundary,
of full dimension, such that the interiors $M_j^0$ cover~$M$.
Then
\[
G:=\bigoplus_{j\in J}C^k(M_j,H)
\]
is $L^1$-regular
(see Remark~\ref{varL1mapgp}
and Proposition~\ref{mrgsum}).
Let $E$ be the modelling space of~$H$.
Then
\[
S:=\{\gamma=(\gamma_j)_{j\in J}\in G\colon (\forall j_1,j_2\in J)
(\forall x\in M_{j_1}\cap M_{j_2})\;
\gamma_{j_1}(x)=\gamma_{j_2}(x)\}
\]
is a Lie subgroup of~$G$ modelled on
\[
F:=\{(\gamma_j)_{j\in J}\in \bigoplus_{j\in J} C^k(M_j,E)
\colon (\forall j_1,j_2\in J)(\forall x\in M_{j_1}\cap M_{j_2})\;
\gamma_{j_1}(x)=\gamma_{j_2}(x)\},
\]
which is a complemented vector subspace of $\bigoplus_{j\in J}C^k(M_j,E)$
(cf.\ Remark \ref{alsnee}).
To see this, let $\phi\colon U\to V\sub E$ be a chart for~$H$
defined on an open identity neighbourhood $U\sub H$ such that
$U=U^{-1}$.
Then
\[
\psi:=\oplus_{j\in J}C^k(M_j,\phi)\colon\bigoplus_{j\in J} C^k(M_j,U)\to\bigoplus_{j\in J}
C^k(M_j,V)
\]
is a chart for~$G$ such that
\[
\psi\left(S\cap \bigoplus_{j\in J}C^k(M_j,U)\right)=F\cap \bigoplus_{j\in J}C^k(M_j,V).
\]
As the maps
\[
G\to H,\quad \gamma\mto \gamma_j(x)
\]
are smooth group homomorphisms for all $j\in J$ and $x\in M_j$,
all hypotheses of Proposition~\ref{equarg} are satisfied and thus
$S$ is $L^1$-regular.
The map
\[
\Psi\colon C^k_c(M,E)\to F,\quad \gamma\mto(\gamma|_{M_j})_{j\in J}
\]
is an isomorphism of topological vector spaces.
Recall that
$C^k_c(M,\phi)$
is a chart for $C^k_c(M,H)$.
Moreover, $\psi$ restricts to a chart
\[
\psi_S\colon S\cap \bigoplus_{j\in J}C^k(M_j,U)\to
F\cap \bigoplus_{j\in J}C^k(M_j,V).
\]
It remains to observe that the map
\[
\Theta\colon C^k_c(M,H)\to S,\quad \gamma\mto (\gamma|_{M_j})_{j\in J}
\]
is an isomorphism of groups such that
\[
\Theta(C^k_c(M,U))=S\cap \bigoplus_{j\in J} C^k(M_j,U)
\]
and $\psi_S\circ \Theta\circ C^k_c(M,\phi)^{-1}$ is the restriction
of $\Psi$ to a $C^\infty$-diffeomorphism
\[
C^k_c(M,V)\to F\cap \bigoplus_{j\in J}C^k(M_j,E).
\]
Hence $\Theta$ is an isomorphism of Lie groups
and thus also the Lie group $C^k_c(M,H)$ is $L^1$-regular (being isomorphic
to the $L^1$-regular Lie group $S$).\,\vspace{2mm}\Punkt

\noindent
We record a second corollary to Proposition~\ref{mrgsum}.
\begin{cor}\label{gaucreg}
Let
$M$ be a finite-dimensional
paracompact smooth manifold,
$H$ be a Banach-Lie group
and $\pi\colon P\to M$ be a smooth principal
bundle with structure group~$H$.
Then $\Gau_c(P)$ is
$L^1$-regular.
\end{cor}
\begin{proof}
Let $(M_j)_{j\in J}$ be a locally finite family of compact submanifolds $M_j\sub M$ with boundary,
of full dimension, such that the interiors $M_j^0$ cover~$M$ for $j\in J$
and there is a smooth section $\sigma_j\colon M_j\to P$
for~$\pi$, for each $j\in J$.
For all $i,j\in J$,
there is a unique map $k_{i,j}\colon M_i\cap M_j\to H$
such that
\[
\sigma_i(x)k_{i,j}(x)=\sigma_j(x)\quad
\mbox{for all $x\in M_i\cap M_j$.}
\]
Then $\Gau_c(P)$ can be identified with the
Lie subgroup $S$ of
\[
G:=\bigoplus_{j\in J} C^\infty(M_j,H)
\]
consisting of all $\gamma=(\gamma_j)_{j\in J}\in G$
such that
\[
(\forall i,j\in J)(\forall x\in M_i\cap M_j)\; \gamma_i(x)=
k_{i,j}(x)\gamma_j(x)k_{j,i}(x)\}
\]
(see \cite{Jak} if $M$ is $\sigma$-compact).\footnote{We can use this identification
to define the Lie group structure on $\Gau_c(P)$.}
%
%
Each of the Lie groups $C^\infty(M_j,H)$
is $L^1$-regular (see Remark~\ref{varL1mapgp}),
whence also
$G=\bigoplus_{j\in J} C^\infty(M_j,H)$ is $L^1$-regular
(by Proposition~\ref{mrgsum}).
For $i,j\in J$
and $x\in M_i\cap M_j$,
the mappings
\[
\alpha_{i,j,x}\colon G\to H,\quad \gamma\mto \gamma_i(x)
\]
and
\[
\beta_{i,j,k}\colon G\to H,\quad
\gamma\mto
k_{i,j}(x)\gamma_j(x)k_{j,i}(x)
\]
are smooth group homomorphisms.
Since
\[
S=\{\gamma\in G\colon (\forall i,j\in \{1,\ldots, m\}(\forall
x\in M_i\cap M_j)\; \alpha_{i,j,x}(\gamma)=\beta_{i,j,x}(\gamma)\},
\]
we deduce with
Proposition~\ref{equarg}
that $S$ (and hence also $\Gau_c(P)$) is $L^1$-regular,
using that modelling space is complemented
(cf.\ Lemma~\ref{toviasum}).
%
%
%
\end{proof}
See \cite[Lemma~7.1]{SEM} for the following tool:
\begin{la}\label{getC1tool} Let $M$ and $N$ be $C^1$-manifolds modeled on locally convex spaces
and $f\colon M\to N$ be a map.
Then $f$ is $C^1$ if and only if there exists a continuous map $\omega \colon TM\to TN$ with
the following properties:
\begin{itemize}
\item[\rm(a)]
$\omega(T_xM)\sub T_{f(x)}N$ for each $x\in M$;
\item[\rm(b)]
If $\ve>0$ and $\gamma\colon \,]{-\ve},\ve[\,\to M$ is a $C^1$-map,
then $f\circ \gamma$ is $C^1$
with $(f\circ\gamma)'(0)=\omega(\gamma'(0))$.
\end{itemize}
In this case, $Tf=\omega$.
If $M$ is an open subset of a locally convex space~$X$,
it suffices to take paths of the form $\gamma(s)=x+sy$ in {\rm(b)},
for $x\in M$ and $y\in X$.\,\Punkt
\end{la}
The next lemma motivates the definition of $\omega$
in the proof of
Proposition~\ref{moresitu}
(and is used in the proof of Corollary~\ref{moredl}).
\begin{la}\label{ifthn}
Let $\cE$ be a bifunctor on Fr\'{e}chet spaces $($resp., on sequentially complete \emph{(FEP)}-spaces,
resp., on integral complete locally convex spaces$)$ which satisfies the locality axiom, the
pushforward axioms, and such that smooth functions act smoothly on $AC_\cE$.
Let $G$ be a Lie group modelled on such a space.
If $G$ is $\cE$-regular and $[\gamma],\eta\in \cE([0,1],\cg)$,
then
\[
\theta\colon \R\to AC_\cE([0,1],G),\quad s\mto \Evol(\eta+s[\gamma])
\]
is a smooth curve in $AC_\cE([0,1],G)$ with right logarithmic derivative
\begin{equation}\label{horrib}
\delta^r(\theta)(s)=I_\cg([\Ad(\Evol(\eta+s[\gamma])).\gamma])
=I_\cg((\eta+s[\gamma]+[\gamma])\odot (\eta+s[\gamma])^{-1})
\end{equation}
for $s\in\R$, where $I_\cg\colon \cE([0,1],\cg)\to AC_\cE([0,1],\cg)$, $[\zeta]\mto
(t\mto \int_0^t \zeta(\tau)\,d\tau)$.
\end{la}
\begin{proof}
Since $\Evol\colon (\cE([0,1],\cg),\odot )\to AC_\cE([0,1],G)$ is a
smooth homomorphism between Lie groups,
we have
\begin{eqnarray*}
(\delta^r\theta)(s)&=&L(\Evol)\delta^r(s\mto \eta+s[\gamma])\\
&=&I_\cg d\rho_{(\eta+s[\gamma])^{-1}}(\eta+s[\gamma],[\gamma])\\
&=&I_\cg([\Ad(\Evol(\eta+s[\gamma])).\gamma]).
\end{eqnarray*}
This establishes the first equality in (\ref{horrib})
and the second equality is merely a rewriting
with the help of (\ref{eqlodif1}) in Lemma~\ref{logarules}.
\end{proof}
\begin{prop}\label{moresitu}
Let $\cE$ be a bifunctor on Fr\'{e}chet spaces $($resp., on sequentially complete \emph{(FEP)}-spaces,
resp., on integral complete locally convex spaces$)$ which satisfies the locality axiom, the
pushforward axioms, and such that smooth functions act smoothly on $AC_\cE$.
Let $G$ be a $\cE$-semiregular Lie group modelled on such a space,
with Lie algebra $\cg$.
Define
\[
I_{\cg}\colon \cE([0,1],\cg)\to C([0,1],\cg), \quad \zeta\mto\left(t\mto
\int_0^t\zeta(\tau)\,d\tau\right).
\]
Assume that
\begin{itemize}
\item[\rm(a)]
$\Evol\colon \cE([0,1],\cg)\to C([0,1],G)$ is continuous;
\item[\rm(b)] The map
$\R\to C([0,1],G)$, $s\mto \Evol(\eta+s[\gamma])$
is $C^1$ for all $[\gamma],\eta\in \cE([0,1],\cg)$; and
\item[\rm(c)]
$\frac{d}{ds}\Big|_{s=0}\Evol(\eta+s[\gamma])
=
I_{\cg}([\Ad(\Evol(\eta))\gamma]).\Evol(\eta)$,
where the dot means multiplication in the tangent Lie group $T(C([0,1],G))$
and we identify
$T_1C([0,1],G)$ with $C([0,1],\cg)$.
\end{itemize}
Then $G$ is $\cE$-regular.
\end{prop}
\begin{proof}
Since $\Evol$ is continuous by assumption,
$(\cE([0,1],\cg),\odot)$ is a topological group by Proposition~\ref{sammelsu}\,(e).
Thus
$\omega\colon T\cE([0,1],\cg)\to T(C([0,1],G))$,
\begin{eqnarray*}
\omega(\eta,[\gamma])&:=&I_\cg(d\rho_{\eta^{-1}}(\eta,[\gamma])).\Evol(\eta)
=I_{\cg}([\Ad(\Evol(\eta))\gamma]).\Evol(\eta)\\
&=& I_\cg(([\gamma]+\eta)\odot \eta^{-1}).\Evol(\eta)
\end{eqnarray*}
is continuous.
Moreover, $\omega$ takes $T_\eta \cE([0,1],\cg)=\{\eta\}\times \cE([0,1],\cg)$
inside\linebreak
$T_{\Evol(\eta)}(C([0,1],G))$.
Hence, by Lemma~\ref{getC1tool}, $\Evol$ will be $C^1$ if we can show that,
for all $[\gamma],\eta\in \cE([0,1],\cg)$, the curve
\[
\xi\colon \R\to C([0,1],G),\quad \xi(s):=\Evol(\eta+s[\gamma])
\]
is $C^1$ and satisfies
\begin{equation}\label{willengh}
\xi'(0)=\omega(\eta,[\gamma]).
\end{equation}
But this is the case by hypotheses (b) and (c).
\end{proof}
\begin{cor}\label{moredl}
Let $\cE$ be a bifunctor on Fr\'{e}chet spaces $($resp., on sequentially complete \emph{(FEP)}-spaces,
resp., on integral complete locally convex spaces$)$ which satisfies the locality axiom, the
pushforward axioms, and such that smooth functions act smoothly on $AC_\cE$.
Let $(J,\leq)$ be a directed set,
\[
((G_j)_{j\in J}, (\alpha_{i,j})_{i\geq j})
\]
be a direct system of $\cE$-semiregular Lie groups $G_j$ modelled on spaces as just described and smooth homomorphisms
$\alpha_{i,j}\colon G_j\to G_i$.
Let $G$ be a Lie group modelled on a space as just described and $\alpha_j\colon G_j\to G$ be smooth homomorphisms
for $j\in J$ such that $\alpha_i\circ\alpha_{i,j}=\alpha_j$
for all $i,j\in J$ such that $i\geq j$. Let $\cg:=L(G)$ and $\cg_j:=L(G_j)$.
Assume that each $[\gamma]\in \cE([0,1],\cg)$
is of the form $[L(\alpha_j)\circ\zeta]$ for some $j\in J$ and $[\zeta]
\in \cE([0,1],\cg_j)$.
Then $G$ is $\cE$-semiregular. If, moreover, $\Evol\colon \cE([0,1],\cg)\to C([0,1],G)$
is continuous at~$0$ and each $G_j$ is $\cE$-regular, then $G$ is $\cE$-regular.
\end{cor}
\begin{proof}
If $[\gamma]\in \cE([0,1],\cg)$ and $[\zeta]\in \cE([0,1],\cg_j)$ such that
$[\gamma]=[L(\alpha_j)\circ \zeta]$, then $\alpha_j\circ  \Evol_{G_j}([\zeta])$
is the left evolution of~$[\gamma]$.
Hence $G$ is $\cE$-semiregular.
Now assume that
\[
\Evol\colon \cE([0,1],\cg)\to C([0,1], G)
\]
is continuous
at~$0$ and hence continuous, by Proposition~\ref{sammelsu}\,(e).
If $[\gamma],[\eta]\in \cE([0,1],\cg)$,
then there are $i,j\in J$ and $[\zeta]\in \cE([0,1],\cg_i)$,
$[\xi]\in \cE([0,1],\cg_j)$ such that $[\gamma]=[L(\alpha_i)\circ\zeta]$
and $[L(\alpha_j)\circ\xi]$. Since $J$ is directed,
there is $\ell\in J$ such that $\ell\geq i,j$.
Let $\bar{\zeta}:=L(\alpha_{\ell,i})\circ\zeta$
and $\bar{\xi}:=L(\alpha_{\ell,j})\circ\xi$.
Then $[\bar{\zeta}]$, $[\bar{\xi}]\in \cE([0,1],\cg_\ell)$
and $[L(\alpha_\ell)\circ \bar{\zeta}]=[L(\alpha_\ell)\circ L(\alpha_{\ell,i})\circ\zeta]
=[L(\alpha_\ell\circ\alpha_{\ell,i})\circ\zeta]=[L(\alpha_i)\circ\zeta]=[\gamma]$;
likewise, $[L(\alpha_\ell)\circ\bar{\xi}]=[\xi]$.
Note that
\[
\theta\colon \R\to AC_\cE([0,1],G_\ell),
\quad \theta(s):=\Evol_{G_\ell}([\bar{\zeta}]+s[\bar{\xi}])
\]
is a $C^1$-curve in $AC_\cE([0,1],G_\ell)$
with
\[
\theta'(s)=
I_{\cg_\ell}([\Ad(\Evol_{G_\ell}([\bar{\zeta}])\bar{\xi})]).\Evol_{G_\ell}([\bar{\zeta}]),
\]
by Lemma~\ref{ifthn}.
Hence $\alpha_\ell\circ\theta\colon \R\to C([0,1],G)$,
$s\mto\Evol([\zeta]+s[\xi])$
is a $C^1$-curve and
\begin{eqnarray*}
\frac{d}{ds}\Big|_{s=0}\Evol([\zeta]+s[\xi]) &=&
T(\alpha_\ell)(\theta'(0))\\
&=&
L(\alpha_\ell)(I_{\cg_\ell}([\Ad(\Evol_{G_\ell}([\bar{\zeta}])\bar{\xi})])).\alpha_\ell(\Evol_{G_\ell}([\bar{\zeta}]))\\
&=&
I_{\cg}([\Ad(\Evol([\zeta])\xi]).\Evol([\zeta]).
\end{eqnarray*}
Now apply Proposition~\ref{moresitu}.
\end{proof}
To get ahead, we need a better understanding of Lebesgue spaces with
values in a locally convex direct limit. Our next result was stimulated by
the following fact~\cite{Muj}:\\[3mm]
{\bf Mujica's Theorem.}
\emph{If $X$ is a compact topological space and a locally convex space
$E$ is the direct limit of an ascending sequence $E_1\sub E_2\subseteq \cdots$
of locally convex spaces $($with continuous inclusion maps$)$, then the natural map}
\[
\Phi\colon \dl\, C(X,E_n)\to C(X,E)
\]
\emph{induced by the inclusion maps is a topological embedding.}\footnote{If the locally convex direct
limit topology is used on the left-hand side.}\\[3mm]
Hence $\Phi$ is an isomorphism of topological vector spaces
if $E=\bigcup_{n\in \N}E_n$ is compact regular.
This important special case was first obtained by Schmets~\cite{Jea}
(and, earlier, by Mujica in the special case of (LB)-spaces).
Our analogue of Mujica's Theorem for Lebesgue spaces reads as follows.
\begin{prop}\label{Mujileb}
Let $E_1\sub E_2\sub\cdots$ be an ascending sequence of locally convex spaces
such that the inclusion maps $E_n\to E_{n+1}$ are continuous linear.
Endow $E=\bigcup_{n\in\N}E_n$ with the locally convex direct limit topology
and assume the latter is Hausdorff. Let $(X,\Sigma,\mu)$ be a measure space. Then we have:
\begin{itemize}
\item[\rm(a)]
The injective continuous linear map
\[
\Phi\colon \bigcup_{n\in \N}L^\infty_{rc}(X,\mu,E_n)\to L^\infty_{rc}(X,\mu,E)
\]
induced by the inclusion maps is a topological embedding with respect to the locally convex direct limit topology on the
union on the left.
If the direct limit $E=\dl\,E_n$\vspace{-.5mm} is compact regular,\footnote{It suffices that every compact \emph{metrizable} subset of $E$ is a compact subset of some $E_n$.}
then $\Phi$ is an isomorphism
of topological vector spaces.
\item[\rm(b)]
If $E$ and each $E_n$ has the \emph{(FEP)}, then for each $p\in [1,\infty[$ the injective
continuous linear map
\[
\Phi\colon \bigcup_{n\in \N}L^p(X,\mu,E_n)\to L^p(X,\mu,E)
\]
induced by the inclusion maps is a topological embedding with respect to the locally convex direct limit topology
on the union on the left.
If $E=\dl\,E_n$\vspace{-.5mm} is a strict \emph{(LF)}-space, then $\Phi$ is an isomorphism
of topological vector spaces.
\end{itemize}
\end{prop}
\begin{proof}
(a) By construction, $\Phi$ is injective, continuous and linear.
As a consequence, the locally convex direct limit topology on the left is Hausdorff.
To see that $\Phi$ is open onto its image,
it suffices to show that $\Phi(\overline{W})$ is a zero-neighbourhood
in $\im(\Phi)$ for $W$ in a basis of $0$-neighbourhoods in the locally convex direct limit
$\bigcup_{n\in\N}L^\infty_{rc}(X,\mu,E_n)$
(as every topological vector space is a regular
topological space, we still have a basis if we pass to the closures $\overline{W}$).
By Lemma~\ref{pbasisindl}, we may assume that
\[
W=\left\{\sum_{n=1}^\infty [\gamma_n]\colon ([\gamma_n])_{n\in\N}\in\bigoplus_{n\in\N}L^\infty_{rc}(X,\mu,E_n)\colon
\sup\{\|\gamma_n\|_{\cL^\infty,q_n}\colon n\in\N\}<1\right\}
\]
for some sequence $(q_n)_{n\in\N}$ of continuous seminorms $q_n\colon E_n\to[0,\infty[$.
By the same lemma,
\[
V:=\left\{\sum_{n=1}^\infty v_n\colon (v_n)_{n\in\N}\in\bigoplus_{n\in\N}E_n\colon
\sup\{q_n(v_n)\colon n\in\N\}<1\right\}
\]
is a $0$-neighbourhood in~$E$. Hence $L^\infty_{rc}(X,\mu,V^0)$
is an open $0$-neighbourhood in $L^\infty_{rc}(X,\mu,E)$
and it suffices to prove that
\[
\im(\Phi)\cap L^\infty_{rc}(X,\mu,V^0)\sub \Phi(\overline{W}).
\]
To this end, let $[\gamma] \in\im(\Phi)\cap L^\infty_{rc}(X,\mu,V^0)$.
We may assume that $\gamma\in \cL^\infty_{rc}(X,\mu,E_N)$ for some $n$
and (possibly after changing it to $0$ on a set of measure zero)
that $\wb{\gamma(X)}\sub V^0$. Thus $\gamma(X)\sub V$.
There is a sequence $(\gamma_n)_{n\in\N}$ of finitely-valued measurable functions
$\gamma_n\colon X\to E_N$ with $\gamma_n(X)\sub \gamma(X)\sub V$
such that $\gamma_n(x)\to\gamma(x)$ in $E_N$, uniformly in $x\in X$.
Hence $[\gamma_n]\to[\gamma]$ in $L^\infty_{rc}(X,\mu, E_N)$.
If we can show that
\[
[\gamma_n]\in W\;\;\mbox{for each $\,n\in\N$,}
\]
then $[\gamma]\in \wb{W}$ and hence $[\gamma]=\Phi([\gamma])\in\Phi(\wb{W})$.
Fix $n\in\N$. We can write
\[
\gamma_n=\sum_{i=1}^m c_i\one_{A_i}
\]
with some $m\in\N$, $c_i\in E_N$ for $i\in\{1,\ldots,m\}$ and disjoint, non-empty measurable sets $A_i\sub X$.
Since $c_i\in V$, we can find $k\in \N$ such that
\[
c_i=\sum_{j=1}^k v_{i,j}
\]
with $v_{i,j}\in E_j$ for $j\in\{1,\ldots, k\}$ and
\[
\max\{ q_j(v_{i,j})\colon j=1,\ldots, k\}<1.
\]
Then $\gamma_n=\sum_{i=1}^m\sum_{j=1}^k v_{i,j}\one_{A_i}
=\sum_{j=1}^k\eta_j$ with $\eta_j:=\sum_{i=1}^m v_{i,j}\one_{A_i}$.
We have $\eta_j\in\cL^\infty_{rc}(X,\mu,E_j)$ and
\[
\|\eta_j\|_{\cL^\infty,q_j}=\max\{q_j(v_{i,j})\colon i=1,\ldots,m\}<1
\]
(using that the sets $A_i$ are disjoint). As a consequence, $[\gamma_n]=\sum_{j=1}^k[\eta_j]\in W$.
This completes the proof that $\Phi$ is open onto its image.
If $E$ is compact regular, then $\Phi$ is surjective and hence (being also a linear topological embedding) an isomorphism of topological
vector spaces.

(b) Again, $\Phi$ is continuous, linear and injective.
To see that $\Phi$ is open onto its image,
it suffices to show that $\Phi(\overline{W})$ is a zero-neighbourhood
in $\im(\Phi)$ for $W$ in a basis of $0$-neighbourhoods in the locally convex direct limit
$\bigcup_{n\in\N}L^\infty_{rc}(X,\mu,E_n)$.
By Lemma~\ref{pbasisindl}, we may assume that
\[
W=\left\{\sum_{n=1}^\infty [\gamma_n]\colon ([\gamma_n])_{n\in\N}\in\bigoplus_{n\in\N}L^p(X,\mu,E_n)\colon
\Big(\sum_{n=1}^\infty (\|\gamma_n\|_{\cL^p,q_n})^p\Big)^{1/p}<1\right\}
\]
for some sequence $(q_n)_{n\in\N}$ of continuous seminorms $q_n\colon E_n\to[0,\infty[$.
Using (as in the proof of Lemma~\ref{pbasisindl})
that $E$ is a quotient of $\bigoplus_{n\in\N}E_n$, we deduce from
Lemma~\ref{psemiondsum}
that
\[
q(x):=\inf\left\{\|(q_n(x))_{n\in\N}\|_{\ell^p}\colon (x_n)_{n\in\N}\in\bigoplus_{n\in\N}E_n\;\mbox{with}\;
x=\sum_{n=1}^\infty x_n\right\}
\]
defines a continuous seminorm $q\colon E\to[0,\infty[$.
It therefore suffices
to prove that
\[
\im(\Phi)\cap \{[\gamma]\in L^p(X,\mu,E)\colon \|\gamma\|_{\cL^p,q}<1\}
\sub \Phi(\overline{W}).
\]
To this end, let $[\gamma] \in\im(\Phi)$ such that $\|\gamma\|_{\cL^p,q}<1$.
We may assume that $\gamma\in \cL^p(X,\mu,E_N)$ for some $n$
(possibly after changing $\gamma$ to $0$ on a set of measure zero).
There is a net $(\gamma_\alpha)_{\alpha\in A}$ of finitely-valued measurable functions
$\gamma_\alpha\colon X\to E_N$ with $\gamma_\alpha(X)\sub \gamma(X)\cup\{0\}$
and $\mu(\gamma_\alpha^{-1}(E_N\setminus\{0\}))<\infty$
such that
$\gamma_\alpha\to\gamma$ in $\cL^p(X,\mu, E_N)$
(see Lemma~\ref{simpappfep}).
If we can show that
\[
[\gamma_\alpha]\in W\;\;\mbox{for each $\,\alpha$,}
\]
then $[\gamma]\in \wb{W}$ and hence $[\gamma]=\Phi([\gamma])\in\Phi(\wb{W})$.
Fix $\alpha$. We can write
\[
\gamma_\alpha=\sum_{i=1}^m c_i\one_{A_i}
\]
with some $m\in\N$, $c_i\in E_N$ for $i\in\{1,\ldots,m\}$ and disjoint, non-empty measurable sets $A_i\sub X$.
Then
\[
\|\gamma_\alpha\|_{\cL^p,q}=\left(\sum_{i=1}^m (q(c_i))^p\mu(A_i)\right)^{1/p}<1
\]
and we can find $k\in \N$ such that
\[
c_i=\sum_{j=1}^k v_{i,j}
\]
with $v_{i,j}\in E_j$ for $j\in\{1,\ldots, k\}$ and
\begin{equation}\label{slightly}
\left(\sum_{i=1}^m \sum_{j=1}^k (q_j(v_{i,j}))^p\mu(A_i)\right)^{1/p}<1.
\end{equation}
Then $\gamma_\alpha=\sum_{i=1}^m\sum_{j=1}^k v_{i,j}\one_{A_i}
=\sum_{j=1}^k\eta_j$ with $\eta_j:=\sum_{i=1}^m v_{i,j}\one_{A_i}$.
We have $\eta_j\in\cL^p(X,\mu,E_j)$ and
\[
\|\eta_j\|_{\cL^p,q_j}=\left(\sum_{i=1}^m (q_j(v_{i,j}))^p\mu(A_i)\right)^{1/p}
\]
(using that the sets $A_i$ are disjoint). As a consequence, $[\gamma_\alpha]=\sum_{j=1}^k[\eta_j]$
with
\[
\left(\sum_{j=1}^k(\|\eta_j\|_{\cL^p,q_n})^p\right)^{1/p}
=
\left(\sum_{j=1}^k\sum_{i=1}^m (q_j(v_{i,j}))^p\mu(A_i)\right)^{1/p}<1,
\]
by (\ref{slightly}). Thus $\eta_j\in W$, completing the proof
that $\Phi$ is open onto its image.
If $E=\bigcup_{n\in\N}E_n$ is a strict (LF)-space,
let us show that $\Phi$ is surjective. We claim that for each $\gamma\in \cL^p(X,\mu,E)$,
there exists $N\in\N$ such that $\mu(\gamma^{-1}(E\setminus E_N))=0$.
If this is true, then $[\gamma]=[\gamma\one_B]\in L^p(X,\mu, E_N)$
with $B:=\gamma^{-1}(E_N)$ and the proof is complete.
To prove the claim, we assume it is wrong and deduce a contradiction.
Thus, suppose there is an element $\gamma\in\cL^p(X,\mu,E)$
such that $\mu(\gamma^{-1}(E\setminus E_N))>0$
for each $N\in\N$.
Since
\[
\mu(\gamma^{-1}(E\setminus E_N))=\sum_{n=N}^\infty
\mu(\gamma^{-1}(E_{n+1}\setminus E_n)),
\]
there exists $n\geq N$ such that $\mu(\gamma^{-1}(E_{n+1}\setminus E_n))>0$.
Thus, using a simple induction, we find a sequence
\[
n_1<n_2<\cdots
\]
of positive integers such that
$\mu(\gamma^{-1}(E_{n_j+1}\setminus E_{n_j}))>0$ for each $j\in\N$.
Let $\rho_j\colon E\to E/E_{n_j}$ be the canonical quotient map.
Since $\rho_j\circ \gamma\in \cL^p(X,\mu, E/E_{n_j})$ does not vanish
almost everywhere, we find a continuous seminorm $Q_j$ on $E/E_{n_j}$
such that $\rho_j\circ \gamma\|_{\cL^p,Q_j}>0$.
Then $q_j:=Q_j\circ \rho_j$ is a continuous seminorm on $E$ which vanishes on $E_{n_j}$.
After replacing $Q_j$ with a large multiple, we may assume that
\[
\|\gamma\|_{\cL^p,q_j}=\|\rho_j\circ\gamma\|_{\cL^p,Q_j}\geq j.
\]
Now $q:=\sum_{j=1}^\infty q_j$ coincides with the finite sum $\sum_{i=1}^{j-1} q_j$ on $E_{n_j}$
(as $q_i$ vanishes on $E_{n_i}$ and hence on $E_{n_j}$ for each $i\geq j$).
Hence $q|_{E_{n_j}}$ is a continuous seminorm for each $j$ and hence $q$ is a continuous seminorm on
$E=\dl\,E_{n_j}$.\vspace{-.5mm}
Since
\[
\|\gamma\|_{\cL^p,q}\geq \|\gamma\|_{\cL^p,q_j}\geq j
\]
for each $j\in\N$, we cannot have $\|\gamma\|_{\cL^p,q}<\infty$,
contradicting the hypothesis that $\gamma\in\cL^p(X,\mu,E)$.
\end{proof}
\begin{rem}
A result for $L^p$-spaces very similar to Proposition~\ref{Mujileb}
was already abtained by Mayoral et al.\
\cite{BDF}, using a different concept of
vector-valued $L^p$-space which makes sense for $\mu$ a Radon measure
on a $\sigma$-compact locally compact space~$X$
(defined on a $\sigma$-algebra $\Sigma$
containing the Borel $\sigma$-algebra
such that $(X,\Sigma,\mu)$ is a complete measure space).
In the cited paper, a mapping $\gamma\colon X\to E$
to a locally convex space~$E$ is called
\emph{$\mu$-measurable}\footnote{This property is also
known as \emph{Lusin-measurability}.}
if there is a sequence $(K_n)_{n\in\N}$ of compact subsets $K_n\sub X$
such that
\[
f|_{K_n}\colon K_n\to E
\]
is continuous and
\[
\mu\left(X\setminus \bigcup_{n\in \N}K_n\right)=0;
\]
$L^p(X,\mu,E)$
(denoted $L^p(\{E\})$ there) is defined as the space of equivalence
classes of $\mu$-measurable
mappings $\gamma\colon X\to E$
such that $q\circ \gamma\in \cL^p(X,\mu)$
for all continuous seminorms~$q$ on~$E$.
The special case of $\ell^p$-spaces was discussed earlier
in~\cite{FMP}.
For related results concerning $\ell^1$, compare also~\cite{Mel}.
\end{rem}
Our next main goal is the following result:
\begin{prop}\label{unionbanreg}
Let $G$ be a Lie group whose Lie algebra $\cg:=L(G)$
is an \emph{(LB)}-space.
Assume that there exists a projective system
\[
((G_n)_{n\in\N},(\psi_{n,m})_{n\geq m})
\]
of Banach-Lie groups $G_n$ with Lie algebras $\cg_n:=L(G_n)$ and smooth homomorphisms
$\psi_{n,m}\colon G_m\to G_n$ such that $L(\psi_{n,m})$
is injective, and smooth homomorphisms $\psi_n\colon G_n\to G$
such that $\psi_n\circ\psi_{n,m}=\psi_m$ for all positive integers $n\geq m$
and
\[
\cg=\dl\,\cg_n\vspace{-.8mm}
\]
with the limit maps $L(\psi_n)$ and bonding maps $L(\psi_{n,m})$.
Then $L(\psi_n)$ is injective for all $n$, enabling
$x\in \cg_n$
to be identified
with $L(\psi_n)(x)\in \cg$.
Now $\cg_m\subseteq \cg_n$ if $m\leq n$ and $L(\psi_{n,m})\colon\cg_m\to\cg_n$
becomes the inclusion map. We show:
\begin{itemize}
\item[\rm(a)]
If $\cg=\bigcup_{n\in\N}\,\cg_n$\vspace{-.5mm}
is compact regular, then
$G$ is $L^\infty_{rc}$-regular.
\item[\rm(b)]
If $\cg=\bigcup_{n\in\N}\,\cg_n$\vspace{-.5mm} is a strict direct limit,
then $\cg$ is an \emph{(FEP)}-space and $G$ is $L^1$-regular.
\end{itemize}
\end{prop}
\begin{rem}
More generally, $G$ as in Proposition~\ref{unionbanreg} is $L^1$-regular
whenever $\cg$ is an (FEP)-space and the natural map
\[
\dl\, L^1([0,1],\cg_n)\to L^1([0,1],\cg)
\]
induced by the maps $L^1([0,1],L(\psi_n))$
is surjective (only these properties are used in the proof of~(b)).
\end{rem}
The following special case is easier to remember.
Here $\cg:=L(G)$ and $\cg_n:=L(G_n)$, as usual.
\begin{cor}\label{unionbanreg2}
Let $G$ be a Lie group modelled on an \emph{(LB)}-space
and $(G_n)_{n\in\N}$ be a sequence of Banach-Lie groups
such that
\[
G_1\sub G_2\sub\cdots\sub G
\]
$($with each inclusion a smooth group homomorphism$)$ and
$\cg=\dl\,\cg_n$.\vspace{-.5mm}
\begin{itemize}
\item[\rm(a)]
If $\,\dl\,\cg_n$\vspace{-.5mm} is compactly regular,
then $G$ is $L^\infty_{rc}$-regular.
\item[\rm(b)]
If $\,\dl\,\cg_n$\vspace{-.5mm} is a strict \emph{(LB)}-space,
then $G$ is $L^1$-regular.\,\Punkt
\end{itemize}
\end{cor}
Some auxiliary concepts and simple lemmas will help us to
prove the preceding proposition. As usual,
if $E$ is a Banach space, write $(\cL(E),\|.\|_{op})$ for the
Banach algebra of bounded linear operators $\alpha \colon E\to E$
and $\GL(E):=\cL(E)^\times$ for the group of invertible
operators (which is an open subset of $\cL(E)$ and a Banach-Lie group).
\begin{defn}
Let $(E,\|.\|)$ be a Banach space. We say that a non-empty set $M\sub \GL(E)$
of invertible operators is \emph{uniformly expanding} if
\[
\sup\{\|\alpha^{-1}\|_{op}\colon \alpha\in M\}<\infty.
\]
\end{defn}
\begin{la}\label{minexpauni}
If $(E,\|.\|)$ is a Banach space and $M\sub\GL(E)$ a uniformly expanding set
of invertible operators, define
\[
s:=\sup\{\|\alpha^{-1}\|_{op}\colon \alpha\in M\}\,\in\;\,]0,\infty[\,.
\]
Then
\begin{equation}\label{inclubd}
\alpha(B^E_r(0))\; \supseteq \; B^E_{r/s}(0)
\end{equation}
for all $\alpha\in M$ and $r>0$. \end{la}
\begin{proof}
For $\alpha\in M$ and $r>0$, we have
\[
\alpha^{-1}(B^E_{r/s}(0))\, \sub \, B^E_{(r\|\alpha^{-1}\|_{op})/s}(0)\, \sub \, B^E_r(0).
\]
Thus $B^E_r(0)\supseteq \alpha^{-1}(B^E_{r/s}(0))$. Applying $\alpha$ to both sides of this inclusion, (\ref{inclubd}) follows. \end{proof}
\begin{defn}
Let $G$ be a Banach-Lie group.
We say that a subset $M\sub G$ is
\emph{product-exponential} (or, in short, a (PE)-\emph{subset}),
if there exists $n\in \N$ and non-empty bounded subsets $B_1,\ldots, B_n\sub L(G)$
such that
\[
M\sub \exp_G(B_1)\exp_G(B_2)\cdots\exp_G(B_n).
\]
\end{defn}
\begin{la}\label{expprodset}
Let $(E,\|.\|)$ be a Banach space and $M\sub\GL(E)$
be a \emph{(PE)}-subset. Then $M$ is uniformly expanding.
\end{la}
\begin{proof}
Let $B_1,\ldots, B_n\sub \cL(E)$ be non-empty bounded subsets
such that $M\sub\exp(B_1)\cdots\exp(B_n)$,
where $\exp(\alpha):=\sum_{k=0}^\infty\frac{1}{k!}\alpha^k$ for $\alpha\in\cL(E)$.
Since $\exp$ is continuous and $\exp(0)=\id_E$,
there is a $0$-neighbourhood $V\sub\cL(E)$
such that
\[
(\forall\alpha\in V)\quad \|\exp(\alpha)\|_{op}\leq 2.
\]
Since each of the sets $-B_1,\ldots,-B_n$ is bounded,
we find $m\in\N$ such that
\[
(\forall j\in\{1,\ldots, n\})\quad -B_j\sub mV
\]
and thus
\[
\|\exp(-B_j)\|_{op}\sub \|\exp(mV)\|_{op}\sub[0,2^m]
\]
for all $j\in \{1,\ldots, n\}$, exploiting that
\[
\|\exp(m\alpha)\|_{op}=\|\exp(\alpha)^m\|_{op}\leq (\|\exp(\alpha)\|_{op})^m\leq 2^m
\]
for each $\alpha\in V$.
If $\alpha_j\in B_j$ for $j\in\{1,\ldots,n\}$, then
\begin{eqnarray*}
\lefteqn{\|(\exp(\alpha_1)\exp(\alpha_2)\cdots\exp(\alpha_n))^{-1}\|_{op}}\qquad\quad\\
&=& \|\exp(-\alpha_n)\cdots\exp(-\alpha_2)\exp(-\alpha_1)\|_{op}\\
&\leq & \|\exp(-\alpha_n)\|_{op}\cdots\|\exp(-\alpha_2)\|_{op}\|\exp(-\alpha_1)\|_{op}\leq 2^m
\end{eqnarray*}
by the preceding, whence
\[
\sup\{\|\alpha^{-1}\|_{op}\colon\alpha\in M\} \leq 2^m<\infty.
\]
Thus $M$ is uniformly expanding.
\end{proof}
\begin{la}\label{pestable}
If $\psi\colon G\to H$ is a smooth homomorphism
between Banach-Lie groups
and $M$ a \emph{(PE)}-subset of~$G$, then
$\psi(M)$ is a \emph{(PE)}-subset of~$H$.
In particular, $\Ad_H(\psi(M))$ is a \emph{(PE)}-subset
of $\GL(\ch)$ with $\ch:=L(H)$.
\end{la}
\begin{proof}
If $M$ is a (PE)-subset of $G$, then $M\sub \exp_G(B_1)\cdots\exp_G(B_n)$
for suitable $n\in\N$ and bounded sets $B_1,\ldots, B_n\sub G$. Since $\psi\circ\exp_G=\exp_H\circ L(\psi)$,
we then have
$\psi(M)\sub\exp_H(L(\psi)(B_1))\cdots\exp_H(L(\psi)(B_n))$.
As $L(\psi)(B_j)\sub L(H)$ is a bounded set for each $j\in\{1,\ldots, n\}$,
we see that $\psi(M)$ is a (PE)-subset of~$H$. Since $\Ad_H\colon H\to \GL(\ch)$
and hence also $\Ad_H\circ \psi\colon G\to\GL(\ch)$ is a smooth homomorphism between Banach-Lie groups,
the final assertion is a special case of the first.
\end{proof}
In the next lemma, we endow $C([0,1],\cL(E))$ with the supremum norm $\|.\|_\infty$,
\[
\|\gamma\|_\infty:=\sup\{\|\gamma(t)\|_{op}\colon t\in[0,1]\}
\]
for $\gamma\in C([0,1],\cL(E))$.
\begin{la}\label{pwtofct}
Let $\cE$ be $L^p$ with $p\in [1,\infty]$
or $L^\infty_{rc}$.
Let $(E,\|.\|)$ be a Banach space
and
\[
m_\gamma(\eta)(t):=\gamma(t)(\eta(t))
\]
for $\gamma\in C([0,1],\cL(E))$ and $[\eta]\in\cE([0,1],E)$.
Then
\[
\psi(\gamma)([\eta]):=[m_\gamma(\eta)]\in\cE([0,1],E).
\]
Moreover, $\psi(\gamma)\in \cL(\cE([0,1],E))$
and the map
\[
\psi \colon C([0,1],\cL(E))\to \cL(\cE([0,1],E))
\]
so obtained
is a homomorphism of unital Banach algebras with $\|\psi\|_{op}\leq 1$.
If $M\sub C([0,1],\GL(E))$
is a non-empty subset such that
\[
\{\gamma(t)\colon \gamma\in M,\, t\in[0,1]\}
\]
is a uniformly expanding subset of $\GL(E)$ $($for example,
a \emph{(PE)}-subset$)$,
then $\psi(M)$ is a uniformly expanding subset of $\GL(\cE([0,1],E))$.
\end{la}
\begin{proof}
The evaluation map $\beta\colon \cL(E)\times E\to E$, $\beta(\alpha,x):=\alpha(x)$
is continuous bilinear, whence $m_\gamma([\eta]):=[\beta\circ(\gamma,\eta)]\in\cE([0,1],E)$
for all $\gamma\in C([0,1],\cL(E))$ and $[\eta]\in\cE([0,1],E)$ by the pushforward axiom (P2)
(see Lemma~\ref{lebhavePA}).
Since $\|\beta(\gamma(t),\eta(t))\|=\|\gamma(t)(\eta(t))\|\leq\|\gamma(t)\|_{op}\|\eta(t)\|
\leq \|\gamma\|_\infty\|\eta(t)\|$ almost everywhere, we see that
\[
\|\psi(\gamma)([\eta])\|_\cE\leq \|\gamma\|_\infty \|\eta\|_\cE,
\]
whence $\|\psi(\gamma)\|_{op}\leq \|\gamma\|_\infty$ and thus $\|\psi\|_{op}\leq 1$.
Moreover, apparently $\psi$ is linear, multiplicative and $\psi(1)=\id_{\cE([0,1],E)}$.
If $M$ is as described in the lemma,
then
\[
s:=\sup\{\|(\gamma(t))^{-1}\|_{op}\colon\gamma\in M,\,t\in[0,1]\}<\infty.
\]
Now
\[
\|(\psi(\gamma))^{-1}\|_{op}=\|\psi(\gamma^{-1})\|_{op}\leq
\|\gamma^{-1}\|_\infty=\sup\{\|(\gamma(t))^{-1}\|\colon t\in[0,1]\}\leq s
\]
for each $\gamma\in\psi(M)$ and thus
\[
\sup\{\|(\psi(\gamma))^{-1}\|_{op}\colon \gamma\in M\}\leq s<\infty.
\]
Thus $\psi(M)$ is uniformly expanding.
\end{proof}
{\bf Proof of Proposition~\ref{unionbanreg}.}
(a) Since $\Phi\colon\dl\, L^\infty_{rc}([0,1],\cg_n)\to L^\infty_{rc}([0,1],\cg)$
is an isomorphism (and hence surjective)
by Proposition~\ref{Mujileb}, Corollary~\ref{moredl} shows that~$G$ is
$L^\infty_{rc}$-semiregular.
Let $U\sub C([0,1],G)$ be an identity neighbourhood.
After shrinking $U$, we may assume that $U=C([0,1], U_0)$
for some identity neighbourhood $U_0\sub G$.
Pick identity neighbourhoods $U_n\sub G$ for $n\in\N$ such that
\[
U_nU_n\sub U_{n-1}\quad\mbox{for all $\,n\in\N$.}
\]
Then
\[
U_nU_{n-1}\cdots U_1\sub U_0\quad \mbox{for all $\,n\in\N$.}
\]
For each $n\in\N$, the exponential map $\exp_{G_n}\colon\cg_n\to G_n$
is a local diffeomorphism at~$0$, whence we find
an open $0$-neighbourhood $B_n\sub \cg_n$ such that $V_n:=\exp_{G_n}(B_n)$
is an open identity neighbourhood in $G_n$ and $\exp_{G_n}\colon B_n\to V_n$ a diffeomorphism.
After shrinking~$B_n$, we may assume that~$B_n$ is bounded and $V_n\sub U_n$; thus
$V_n$ is an open identity neighbourhood and a (PE)-set.
Now $C([0,1],V_n)$ is an open identity neighbourhood in $C([0,1],G_n)$.
Since~$G_n$ is $L^1$-regular, there is an open $0$-neighbourhood
$P_n\sub L^\infty_{rc}([0,1],\cg_n)$ such that $\Evol_{G_n}(P_n)\sub C([0,1], V_n)$.
We claim that
\[
P:=\bigcup_{n\in\N}(P_n\odot P_{n-1}\cdots\odot P_2\odot P_1)
\]
is a $0$-neighbourhood in $L^\infty_{rc}([0,1],\cg)$.
If this is true, then
\begin{eqnarray*}
\Evol(P_n\odot\cdots\odot P_1) &=& \Evol(P_n)\cdots \Evol(P_2)\Evol(P_1)\\
&=& \Evol_{G_n}(P_n)\cdots\Evol_{G_2}(P_2)\Evol_{G_1}(P_1)\\
&\sub& C([0,1],V_n)\cdots C([0,1], V_2)C([0,1],V_1)\\
&\sub & C([0,1], U_n)\cdots C([0,1],U_2)C([0,1],U_1)\\
&\sub&  C([0,1],U_n\cdots U_2U_1)\sub C([0,1], U_0) = U
\end{eqnarray*}
entails
\[
\Evol(P)=\bigcup_{n\in\N}\Evol(P_n\odot\cdots\odot P_1)\sub U.
\]
Hence $\Evol$ is continuous at~$0$
and thus $G$ is $L^\infty_{rc}$-regular, by Corollary~\ref{moredl}.
To establish the claim, it suffices to find a sequence $(Q_n)_{n\in\N}$
of open convex $0$-neighbourhoods $Q_n\sub L^\infty([0,1],\cg_n)$
such that, for all $n\in \N$,
\[
Q_1+\cdots+Q_n\sub P_n\odot\cdots\odot P_1.
\]
Then $Q:=\bigcup_{n\in\N}(Q_1+\cdots +Q_n)$ is a $0$-neighbourhood
in the locally convex direct limit $L^\infty_{rc}([0,1],\cg)=\dl\, L^\infty_{rc}([0,1],\cg_n)$
such that $Q\sub P$, and so~$P$ is a $0$-neighbourhood as well.
Let $F_n$ be the Banach space $L^\infty_{rc}([0,1],\cg_n)$;
we may assume that $P_n=B^{F_n}_{r_n}(0)$ for some $r_n>0$.
Set $Q_1:=P_1$. Let $n>1$.
For $k\in\{1,\ldots,n-1\}$ and $\gamma\in P_k$, we have
\begin{eqnarray*}
\Ad_{G_n}(\Evol_{G_n}(\gamma)(t))^{-1}
&=& \Ad_{G_n}(\Evol_{G_k}(\gamma)(t))^{-1}\sub \Ad_{G_n}(V_k)^{-1}\\
&=& \Ad_{G_n}(\exp_{G_k}(-B_k)).
\end{eqnarray*}
Thus
\begin{eqnarray*}
\lefteqn{\Ad_{G_n}(\Evol_{G_n}(\gamma_{n-1}\odot\cdots\odot\gamma_1)(t))^{-1}}\qquad\\
&=&\Ad_{G_n}((\Evol_{G_n}(\gamma_{n-1}(t)\cdots\Evol_{G_n}(\gamma_1)(t)))^{-1})\\
&=&\Ad_{G_n}(\Evol_{G_n}(\gamma_1(t))^{-1})\cdots
\Ad_{G_n}(\Evol_{G_n}(\gamma_{n-1})(t))^{-1})\\
&\sub&
\Ad_{G_n}(\exp_{G_1}(-B_1))\cdots\Ad_{G_n}(\exp_{G_{n-1}}(-B_{n-1})),
\end{eqnarray*}
from which we deduce (with Lemma~\ref{pestable}) that
\[
\{\Ad_{G_n}(\Evol_{G_n}(\eta)(t))^{-1}\colon \eta\in P_{n-1}\odot\cdots\odot P_1,\, t\in[0,1]\}
\]
is a (PE)-subset of $\GL(\cg_n)$ and hence uniformly expanding (by Lemma~\ref{expprodset}).
Thus
\[
M_n:=\{ [\zeta]\mto [ \Ad_{G_n}(\Evol_{G_n}(\eta))^{-1}.\zeta ] \colon\eta\in P_{n-1}\odot\cdots\odot P_1\}
\]
is a uniformly expanding subset of $\GL(L^\infty_{rc}([0,1],\cg_n))$, by Lemma~\ref{pwtofct}.
As a consequence,
\[
s_n:=\sup\{\|\alpha^{-1}\|_{op}\colon\alpha\in M_n\}<\infty
\]
and hence
\[
\Ad_{G_n}(\Evol_{G_n}(\eta))^{-1}.P_n=\Ad_{G_n}(\Evol_{G_n}(\eta))^{-1}.B^{F_n}_{r_n}(0)
\subseteq B^{F_n}_{r_n/s_n}(0)=:Q_n
\]
for each $\eta\in P_{n-1}\odot\cdots\odot P_1$. Since $[\gamma]\odot\eta=[\Ad_{G_n}(\Evol(\eta)(t))^{-1}\gamma(t)]+\eta$
for all $[\gamma]\in P_n$ and $\eta\in P_{n-1}\odot\cdots\odot P_1$,
we deduce that
\[
P_1\odot \eta\supseteq Q_n+\eta
\]
for each $\eta\in P_{n-1}\odot\cdots\odot P_1$.
Hence
\[
P_n\odot\cdots \odot P_1\supseteq Q_n+P_{n+1}\odot\cdots\odot P_1
\supseteq Q_n+Q_{n-1}+\cdots+Q_1,
\]
which completes the proof of~(a).

(b) Replace $L^\infty_{rc}$ with $L^1$ in the proof of~(a).
\,\vspace{2mm}\Punkt

\noindent
%
%
As a first consequence, we see that
direct limits of finite-dimensional Lie groups
\cite{DI2}
are $L^\infty_{rc}$-regular:\\[2.3mm]
{\bf Proof of Theorem~E.}
Since $L(G)=\dl\,L(G_n)$\vspace{-.5mm} is a strict (LB)-space,
Corollary~\ref{unionbanreg2} applies.\,\Punkt\vspace{2.3mm}

\noindent
Another application are Lie groups
of real analytic maps.
We first consider groups of germs of Lie group-valued
analytic maps (as in \cite{DGS}):
\begin{cor}\label{germslinfty}
Let $\K\in\{\R,\C\}$.
Let $M$ is a $\K$-analytic manifold
modelled on a Fr\'{e}chet space,
$K\sub M$ a non-empty compact set
and $H$ a $\K$-analytic Banach-Lie group;
if $\K=\R$, assume that the topological
space unerlying~$M$ is regular.
Then $\Germ_\K(K,M,H)$
is an $L^\infty_{rc}$-regular
$\K$-analytic Lie group.
If $H$ is finite-dimensional, then $G:=\Germ_\K(K,M,H)$
has a $\K$-analytic evolution map
\[
\Evol\colon L^\infty_{rc}([0,1],L(G))\to AC_{L^\infty_{rc}}([0,1],G).
\]
\end{cor}
\begin{proof}
The case $\K=\C$: Let $U_1\supseteq U_2\supseteq\cdots$ be a basis of
open neighbourhoods of~$K$ in $M$ such that
each connected component of $U_n$ meets~$M$.
Let $G:=\Germ_\C(K,M,H)$, $\ch:=L(H)$ and $\cg:=\Germ_\C(K,M,\ch)$.
Then $\cg_n:=\Hol_b(U_n,\ch)$ with the supremum norm
is a Banach-Lie algebra. Moreover, $\cg_n$ can be identified with a vector subspace of~$\cg$
(identifying holomorphic functions with their associated germs around~$K$),
and
\begin{equation}\label{postdli}
\cg=\dl\,\cg_n\vspace{-.5mm}
\end{equation}
as a locally convex space (by definition).
The direct limit (\ref{postdli}) is compactly regular~\cite{DGS}.
The Lie group $G$ has an exponential map, given by
\[
\exp_G\colon \cg\to G,\quad [\gamma]\mto [\exp_H\circ \gamma].
\]
By the construction of the Lie group structure of~$G$ in~\cite{DGS},
$\exp_G$ is a local diffeomorphism at~$0$. Moreover, there is an open $0$-neighbourhood $V\sub\cg$
such that the Baker-Campbell-Hausdorff series converges on $V\times V$
and makes it a complex analytic local Lie group, with $\exp_G(V)$ open and $\exp_G|_V$
both a diffeomorphism and a homomorphism of local Lie groups (see \cite{DGS}).
Then $V\cap\cg_n$ and hence also $\exp_G(V\cap\cg_n)$ is a local Lie group with Lie algebra~$\cg_n$.
As a consequence, the subgroup
\[
G_n:=\langle \exp_G(V\cap \cg_n)\rangle=\langle \exp_G(\cg_n)\rangle
\]
is a Banach-Lie group with $L(G_n)=\cg_n$ and $\exp_{G_n}=\exp_G|_{\cg_n}$.
Let $i_{n,m}\colon G_m\to G_n$, $j_{n,m}\colon \cg_m\to \cg_n$,
$i_m\colon G_n\to G$ and $j_m\colon \cg_n\to\cg$ be the respective inclusion maps (for $n\geq m$).
Since $j_n$ and $j_{n,m}$ are continuous linear and hence complex analytic,
$\exp_{G_m}$ is a local diffeomorphism at~$0$ and $\exp_G$ as well as $\exp_{G_n}$ are complex analytic,
we deduce from $i_m\circ\exp_{G_m}=\exp_G\circ j_m$ and $i_{n,m}\circ\exp_{G_m}=\exp_{G_n}\circ j_{n,m}$
that $i_n$ and $i_{n,m}$ are complex analytic homomorphisms.
Hence~$G$ is $L^\infty_{rc}$-regular, by Corollary~\ref{unionbanreg2}.

The case $\K=\R$:
Let $\wt{M}$ be a complexification for~$M$ admitting an antiholomorphic involution
$\tau\colon\wt{M}\to\wt{M}$
such that $M\sub\wt{M}$ and $M$ is the fixed point set of~$\tau$ (see \cite{DGS}).
If $H$ is finite-dimensional, then $H$ has a complexification $H_\C$ with $H\sub H_\C$
(see \cite{Boulie}). Then $\Germ_\C(K,\wt{M}, H_\C)$
is a complexification of $\Germ_\R(K,M,H)$ (cf.\ \cite{DGS}).
Since $\Germ_\C(K,\wt{M},H_\C)$ is $L^\infty_{rc}$-regular (by the complex case
already treated), we deduce with Lemma~\ref{realvscxlcly}
that $\Germ_\R(K,M,H)$ is $L^\infty_{rc}$-regular with real analytic evolution~$\Evol$.
It remains to show that $G:=\Germ_\R(K,M,H)$ is $L^\infty_{rc}$-regular
when~$H$ is an arbitrary real Banach Lie group.
To achieve this,
let $U_1\supseteq U_2\supseteq\cdots$ be a basis of open
neighbourhoods of~$K$ in $\wt{M}$ such that $U_n=\tau(U_n)$ for each $n$
and each connected component of $U_n$ meets~$M$.
Let $\ch_\C=\ch\oplus i\ch$ be a complexification of~$\ch:=L(H)$
and $\sigma\colon \ch_\C\to\ch_\C$, $(x+iy)\mto x-iy$ for $x,y\in\ch$
be complex conjugation on~$\ch_\C$.
Then $\ck_n:=\Hol_b(U_n,\ch_\C)$ is a complex Banach-Lie algebra and
\[
\cg_n:=\{\gamma\in\Hol_b(U_n,\ch_\C)\colon \gamma=\sigma\circ\gamma\circ\tau\}
\]
is a closed real Lie subalgebra of $\ck_n$ such that $\ck_n=(\cg_n)_\C$
(cf.\ \cite{DGS}). Identifying $\gamma\in\cg_n$ with the germ of $\gamma|_{U_n\cap M}\colon U_n\cap M\to\ch$,
we can consider $\cg_n$ as a Lie subalgebra of $\cg:=\Germ_\R(K,M,\ch)$.
Moreover,
\[
\cg=\dl\,\cg_n\vspace{-.5mm}
\]
(cf.\ \cite{DGS}). We can now complete the proof as in the case $\K=\C$,
replacing $\C$ with $\R$ there.
\end{proof}
Taking $\K=\R$ and $M=K$,
we get the following result as a special case,
which subsumes Theorem~F:
\begin{cor}\label{anmupreg}
If $M$ is a real analytic compact manifold and $H$ a Banach-Lie group,
then the Lie group $G:=C^\omega(M,H)$
of all $H$-valued real analytic mappings on~$M$
is $L^\infty_{rc}$-regular.
If $H$ is a finite-dimensional
Lie group, then $\Evol_G\colon L^\infty_{rc}([0,1],L(G))\to AC_{L^\infty_{rc}}([0,1],G)$
is real analytic.\,\Punkt
\end{cor}
\begin{rem} Real analyticity of $\Evol_G$
in Corollaries~\ref{germslinfty} and~\ref{anmupreg}
also holds for all Banach-Lie groups $G$
for which a complex Banach-Lie group $G_\C$
exists such that the inclusion map $G\sub G_\C$ is a real analytic
homomorphism which makes $G_\C$ a complexification of~$G$.
Only this property is used in the proof.
\end{rem}
\begin{cor}
The Lie group $C^\omega(\R,H)$ of all real analytic $H$-valued maps on~$\R$
is $L^\infty_{rc}$-regular, for each Banach-Lie group~$H$.
\end{cor}
\begin{proof}
We recall from \cite{DGS} that $C^\omega(\R,H)_*:=\{\gamma\in C^\omega(\R,H)\colon\gamma(0)=e\}$
is a Lie subgroup of $C^\omega(\R,H)$ and
\[
C^\omega(\R, H)=C^\omega(\R,H)_*\rtimes H
\]
as a Lie group. Since $H$ is $L^\infty_{rc}$-regular by Theorems~C and~A,
we need only prove $L^\infty_{rc}$-regularity for $C^\omega(\R,H)$;
then als $C^\omega(\R,H)$ will be $L^\infty_{rc}$-regular, by Theorem~G.
Since $\Germ([{-n},n],\R,H)$ is $L^\infty_{rc}$-regular by Corollary~\ref{germslinfty},
also its Lie subgroup
\[
\Germ([{-n},n],\R,H)_*:=\{[\gamma]\in\Germ([{-n},n],\R,H)\colon \gamma(0)=e\}
\]
is $L^\infty_{rc}$-regular, by Proposition~\ref{equarg} (as it is the kernel
of the point evaluation $\Germ([{-n},n],\R,H)\to H$, $[\gamma]\mto\gamma(e)$,
which is a smooth homomorphism to the $L^\infty_{rc}$-regular Lie group~$H$).
Now
\[
C^\omega(\R,H)_*\, =\,\pl\, \Germ([{-n},n],\R,H)_*
\]
and
\[
C^\omega(\R,H)_*\to C^\omega(\R,L(H)),\quad \gamma\mto \delta^\ell(\gamma)
\]
is a projective limit chart (see \cite{DGS}).
Thus Proposition~\ref{regglocha}
shows that $C^\omega(\R,H)_*$ (and hence also $C^\omega(\R,H)$)
is $L^\infty_{rc}$-regular.
\end{proof}
%
%
%
%
%
%
%
%
%
%
%
%
%
%
%
\section{Regularity properties of
{\boldmath$\Diff_c(\R^n)$},
{\boldmath$\Diff_K(\R^n)$} and {\boldmath$\Diff(\bS_1)$}}\label{secRS}
After a brief introduction
to the diffeomorphism groups $\Diff_c(\R^n)$ and
$\Diff_K(\R^n)$,
we prove the $L^1$-regularity of
$\Diff_K(\R^n)$, $\Diff(\bS_1)$ and $\Diff_c(\R^n)$.
\begin{numba}\label{funtops}
If $U\sub \R^n$ is open
and $E$ a locally convex space,
we endow $C(U,E)$ with the topology of uniform convergence
on compact sets, determined by the seminorms
\[
\|.\|_{L,q}\colon C(U,E)\to [0,\infty[,\quad \gamma\mto \sup_{x\in L}q(\gamma(x))
\]
for $q$ in the set of continuous seminorms on~$E$ and $L$
ranging through the compact subsets of~$U$.
If $E$ and $U$ are as before and $r\in \N_0\cup\{\infty\}$,
we endow
the space $C^r(U,E)$ of all $C^r$-functions $\gamma\colon U\to E$
with the compact-open $C^r$-topology,
i.e., the initial topology with respect to the maps
\[
C^r(U,E)\to C(U,E),\quad \gamma\mto \frac{\partial^\alpha}{\partial x^\alpha}
\]
for $\alpha\in \N_0^n$ with $|\alpha|\leq r$.
Given a compact subset $K\sub U$, we give
\[
C^r_K(U,E):=\{\gamma\in C^r(U,E)\colon \gamma|_{U\setminus K}=0\}
\]
the induced topology. It is the locally convex vector
topology given by the seminorms
\[
\|\gamma\|_{C^k,q}:=\max_{|\alpha|\leq k}\|\partial^\alpha \gamma\|_{\cL^\infty,q},
\]
for all $k\in \N_0$ such that $k\leq r$ and all
continuous seminorms~$q$ on~$E$ (with multi-indices $\alpha\in \N_0^n$).
As usual, $C^r_c(U,E)=\bigcup_K C^r_K(U,E)$
is the locally convex direct limit of the spaces $C^r_K(U,E)$.
If $E$ is a Fr\'{e}chet space, then also $C^r_K(U,E)$
is a Fr\'{e}chet space. If $E$ is a separable Fr\'{e}chet space,
then also $C^r_K(U,E)$ is separable.\footnote{Since $C^r_K(U,E)$
is isomorphic to $C^r_K(\R^n,E)$,
it suffices to show that the Fr\'{e}chet space $C^r(\R^n,E)$ is separable.
Let $J:=\{\alpha\in \N_0^n\colon |\alpha|\leq r\}$.
We claim that $C(\R^n,E)$ is separable.
If this is true, then the Fr\'{e}chet space
$C(\R^n,E)^J$ (with the product topology)
will be separable.
Since $C^r(\R^n,E)\to C(\R^n,E)^J$,
$\gamma\mto (\partial^\alpha\gamma)_{|\alpha|\leq r}$ is a topological
embedding, the separability
of $C^r(\R^n,E)$ follows.
To prove the claim, for $m\in \N$ let $(h_{m,k})_{k\in \N}$
be a partition of unity on~$\R^n$
subordinate to $(B_{1/m}(x))_{x\in \R^n}$
\cite[I.8.6, Satz~3]{Shu}
(using balls with respect to some
norm on~$\R^n$).
Let $D\sub E$ be a countable dense subset.
Then the countable set $\{a h_{m,k}\colon m,k\in \N,\,a\in D\}$
is easily seen to be total in $C(\R^n,E)$
and thus $C(\R^n,E)$ is separable.}
\end{numba}
\begin{numba}\label{basicdfmo1}
Let $C^\infty_c(\R^n,\R^n)$ be the space of
all compactly supported, $\R^n$-valued smooth functions on~$\R^n$
and $\Diff_c(\R^n)$ be the set of all diffeomorphisms $\phi\colon \R^n\to\R^n$
which are compactly supported in the sense that
\[
\phi-\id_{\R^n}\in C^\infty_c(\R^n,\R^n).
\]
It is known that $\Diff_c(\R^n)$ is a Lie group
under composition of diffeomorphisms, with neutral element $\id_{\R^n}$
(cf.\ \cite{Mic}, see~\cite{DRN}).
The set
\[
\Omega:=\{\phi-\id_{\R^n}\colon \phi\in \Diff_c(\R^n)\}
\]
is open in $\Diff_c(\R^n)$
and the map
\[
\Phi\colon \Diff_c(\R^n)\to \Omega,\quad \phi\mto\phi-\id_{\R^n}
\]
is a global chart; moreover,
\[
\{\gamma\in C^\infty_c(\R^n,\R^n)\colon \|\gamma'\|_{\cL^\infty,\|.\|_{op}}<1\}
\]
is an open subset of $\Omega$ (see \cite{DRN}).
Here $\|\alpha\|_{op}$ denotes the
operator norm of a linear map $\alpha\colon \R^n\to \R^n$
with respect to the maximum norm $\|.\|_\infty$ on $\R^n$.
We can make $\Omega$ a Lie group
in such a way that $\Phi$ becomes an isomorphism
of Lie groups:
Its group multiplication is given by
\[
\gamma*\eta:=\Phi(\Phi^{-1}(\gamma)\circ\Phi^{-1}(\eta))
=(\id_{\R^n}+\gamma)\circ(\id_{\R^n}+\eta)-\id_{\R^n}=
\eta+\gamma\circ (\id_{\R^n}+\eta)
\]
for $\gamma,\eta\in \Omega$, and
the constant function~$0$ is the neutral element.
Note that $\Phi$ takes $\Diff_K(\R^n)$
onto $\Omega\cap C^\infty_K(\R^n,\R^n)=:\Omega_K$.
Hence $\Diff_K(\R^n)$ is a Lie subgroup of $\Diff_c(\R^n)$
modelled on $C^\infty_K(\R^n,\R^n)$
which has
\[
\Phi_K\colon \Diff_K(\R^n)\to\Omega_K,\quad \phi\mto\phi-\id_{\R^n}
\]
as a global chart. Again, $\Omega_K$
can be made a Lie group isomorphic to $\Diff_K(\R^n)$
using the multiplication~$*$.
To see that $\Diff_c(\R^n)$ and $\Diff_K(\R^n)$
are $L^1$-regular,
we need only show that $\Omega$
and $\Omega_K$ are $L^1$-regular.
\end{numba}
\begin{numba}
As usual for tangent bundles of open subsets of locally convex spaces,
we have\footnote{Thus, we are using the addition
of the locally convex space to trivialize the tangent bundle,
\emph{not} left or right multiplication
in the Lie group $(\Omega,*)$.}
\[
T\Omega=\Omega\times C^\infty_c(\R^n,\R^n)\quad \mbox{and}\quad
T\Omega_K=\Omega_K\times C^\infty_K(\R^n,\R^n).
\]
For fixed $\eta\in \Omega$, right translation with $\eta$ is the map
\[
\rho_\eta\colon \Omega\to\Omega,\quad \gamma\mto \eta+\gamma\circ (\id_{\R^n}+\eta)
\]
which is the restriction of the affine linear map $C^\infty_c(\R,\R^n)\to
C^\infty_c(\R^n,\R^n)$ given by the same formula, which is continuous
(see \cite{DRN}).
Hence
\[
d\rho_\eta(\gamma,\gamma_1)=\gamma_1\circ(\id_{\R^n}+\eta)\quad
\mbox{for all $(\gamma,\gamma_1)\in \Omega\times C^\infty_c(\R^n,\R^n)$.}
\]
We identify the Lie
algebra $L(\Omega)=T_0\Omega=\{0\}\times C^\infty_c(\R^n,\R^n)$
with $C^\infty_c(\R^n,\R^n)$.
Let us calculate the product of
$(\gamma,\gamma_1)\in T\Omega$ and $\eta\in \Omega$
(identified with $0_\eta\in T_\eta\Omega$)
in the tangent group $T\Omega$. We have
\[
(\gamma,\gamma_1)\cdot \eta=T\rho_\eta(\gamma,\gamma_1)=
(\gamma*\eta,\gamma_1\circ (\id_{\R^n}+\eta)).
\]
Likewise for $\Omega_K$.
\end{numba}
\begin{numba}\label{firsnum}
Given $\gamma\in \cL^1([0,1],C^\infty_c(\R^n,\R^n))$,
we want to find a continuous function
$\eta\colon [0,1]\to \Omega$ which is a Carath\'{e}odory
solution to
\[
(\eta(t),\eta'(t))=(0,\gamma(t))\cdot \eta(t)=(\eta(t),\gamma(t)\circ (\id_{\R^n}+\eta(t))
\quad \mbox{($t\in [0,1]$)}
\]
in $T\Omega=\Omega\times C^\infty_c(\R^n,\R^n)$
with $\eta(0)=0$.
As a differential equation in a locally convex space,
this requires
\[
\eta'(t)=\gamma(t)\circ (\id_{\R^n}+\eta(t))
\]
and hence that
\begin{equation}\label{checkthis}
\eta(t)=\int_0^t\gamma(s)\circ(\id_{\R^n}+\eta(s))\,ds\quad
\mbox{for all $t\in [0,1]$}
\end{equation}
in $C^\infty_c(\R^n,\R^n)$.
We shall see that, as a function of~$s$, the integrand in (\ref{checkthis})
is
an element of $\cL^1([0,1],C^\infty_c(\R^n,\R^n))$
for each $\eta\in C([0,1],\Omega)$
(Lemma~\ref{nogap});
thus validity of (\ref{checkthis}) implies
that $\eta\in AC_{L^1}([0,1],\Omega)$
and $\eta=\Evol^r([\gamma])$.
\end{numba}
\begin{numba}\label{varL1}
Likewise, if $\gamma\in \cL^1([0,1],C^\infty_K(\R^n,\R^n))$,
we wish to find a continuous map
$\eta\colon [0,1]\to \Omega_K$
such that
(\ref{checkthis}) holds.
Then $\eta\in AC_{L^1}([0,1],\Omega_K)$ (see Lemma~\ref{nogap})
and $\eta=\Evol^r([\gamma])$.
\end{numba}
In \ref{varL1}, we need to make sure that the integrand
of~(\ref{checkthis})
is an element of
$\cL^1([0,1],C^\infty_K(\R^n,\R^n))$
as a function of~$s$.
Moreover, in both~\ref{firsnum} and~\ref{varL1},
the smooth dependence of $\Evol^r([\gamma])$ on~$[\gamma]$
remains to be shown.
To enable these tasks,
we now provide several preparatory lemmas
devoted to measurability and differentiability
properties in related situations.
The point evaluation $\ev_x\colon C^\infty_c(\R^n,\R^n)\to\R^n$,
$f\mto f(x)$ is continuous linear for each $x\in \R^n$,
and these point evaluations separate points on $C^\infty_c(\R^n,\R^n)$.
Hence, if the integrand in~(\ref{checkthis})
is an element of $\cL^1([0,1], C^\infty_c(\R^n,\R^n))$,
then~(\ref{checkthis}) holds if and only if the
continuous functions
$\eta_x:=\ev_x\circ\eta\colon [0,1]\to\R^n$
satisfy
\[
\eta_x(t)
=\int_0^t\gamma(s)(x+\eta_x(s))\,ds\quad\mbox{for all $t\in [0,1]$,}
\]
for all $x\in \R^n$.
Setting $\zeta_x(t):=x+\eta_x(t)$, the latter is equivalent to
\begin{equation}\label{niceinteqx}
\zeta_x(t)=x+\int_0^t\gamma(s)(\zeta_x(s))\,ds\quad\mbox{for all $t\in [0,1]$,}
\end{equation}
meaning that
$\zeta_x\colon [0,1]\to \R^n$ is a Carath\'{e}odory solution to
\[
\zeta_x'(t)=\gamma(t)(\zeta_x(t)),\quad \zeta_x(0)=x.
\]
Our strategy is to discuss
the solutions to~(\ref{niceinteqx}),
and their dependence on $(\gamma,x)$.
\begin{numba}\label{thereduyea}
By the preceding,
if
$\eta\colon [0,1]\to \Omega$
is continuous and the integrand of (\ref{checkthis})
is an element of $\cL^1([0,1],C^\infty_c(\R^n,\R^n))$,
then the validity of (\ref{niceinteqx})
for all $x\in \R^n$ implies the validity
of~(\ref{checkthis}).
Likewise, if
$\eta\colon [0,1]\to \Omega_K$
is continuous and the integrand of (\ref{checkthis})
is in $\cL^1([0,1],C^\infty_K(\R^n,\R^n))$,
then the validity of~(\ref{niceinteqx})
for all $x\in \R^n$ implies the validity
of~(\ref{checkthis}).
\end{numba}
\begin{la}\label{nogap}
If $\gamma\in \cL^1([0,1],C^\infty_c(\R^n,\R^n))$
and $\eta\in C([0,1],C^\infty_c(\R^n,\R^n))$,
then
\[
(s\mto \gamma(s)\circ(\id_{\R^n}+\eta(s)))\in \cL^1([0,1],C^\infty_c(\R^n,\R^n)).
\]
If $\gamma\in \cL^1([0,1],C^\infty_K(\R^n,\R^n))$
and $\eta\in C([0,1],C^\infty_K(\R^n,\R^n))$,
then
\[
(s\mto \gamma(s)\circ(\id_{\R^n}+\eta(s)))\in \cL^1([0,1],C^\infty_K(\R^n,\R^n)).
\]
\end{la}
\begin{proof}
The map
\[
f\colon C^\infty_c(\R^n,\R^n)\times C^\infty_c(\R^n,\R^n)\to C^\infty_c(\R^n,\R^n),\quad
f(\sigma,\tau):=\tau\circ (\id_{\R^n}+\sigma)
\]
is smooth (see~\cite{DRN}),
and $f(\sigma,\sbull)$ is linear for each $\sigma\in C^\infty_c(\R^n,\R^n)$.
Hence, by Lemma~\ref{operders}\,(b),
\[
f\circ (\eta,\gamma)\in \cL^1([0,1],C^\infty_c(\R^n,\R^n))
\]
for all $\gamma\in C([0,1],C^\infty_c(\R^n,\R^n))$
and $\eta\in \cL^1([0,1],C^\infty_c(\R^n,\R^n))$,
where
\[
(f\circ (\eta,\gamma))(s)=f(\eta(s),\gamma(s))=\gamma(s)\circ (\id_{\R^n}+\eta(s)).
\]
Given a compact set $K\sub \R^n$, the map $f$ restricts to
a smooth map
\[
f_K\colon C^\infty_K(\R^n,\R^n)\times C^\infty_K(\R^n,\R^n)\to C^\infty_K(\R^n,\R^n),\quad
(\sigma,\tau)\mto \tau\circ (\id_{\R^n}+\sigma),
\]
and again Lemma~\ref{operders}\,(b)
can be applied.
\end{proof}
{\bf{\boldmath$L^1$}-regularity of {\boldmath$\Diff_K(\R^n)$}}
\begin{la}\label{givsLon}
Let $K\sub \R^n$ be compact, $m\in \N_0$ and
$\gamma\in \cL^1([0,1],C^\infty_K(\R^n,\R^m))$.
Then
\[
\wh{\gamma}\colon [0,1]\times\R^n\to \R^m,\quad\wh{\gamma}(t,x):=\gamma(t)(x)
\]
is measurable. For any measurable function
$\zeta\colon [0,1]\to \R^n$, define a function
\[
(\wh{\gamma})_*(\zeta)\colon [0,1]\to \R^m
\]
via $(\wh{\gamma})_*(\zeta)(t):=\wh{\gamma}(t,\zeta(t))=\gamma(t)(\zeta(t))$.
Then $(\wh{\gamma})_* (\zeta)\in \cL^1([0,1],\R^m)$.
\end{la}
\begin{proof}
Since $\R^n$ is second countable, the Borel $\sigma$-algebra
$\cB(C^\infty_K(\R^n,\R^m)\times \R^n)$ coincides with the product $\sigma$-algebra
$\cB(C^\infty_K(\R^n,\R^m))\tensor \cB(\R^n)$ (see \ref{basicsmeas}\,(f)).
Therefore the map
\[
\gamma\times\id_{\R^n}\colon [0,1]\times\R^n\to C^\infty_K(\R^n,\R^m)\times\R^n,\;\;
(t,x)\mto (\gamma(t),x)
\]
is Borel measurable.
The evaluation map
\[
\ve\colon C^\infty_K(\R^n,\R^m)\times\R^n\to\R^m,\quad\ve(f,x):=f(x)
\]
is $C^\infty$ (see, e.g., \cite{ZOO} or \cite{GaN}),
hence continuous and hence measurable.
Thus $\wh{\gamma}=\ve\circ (\gamma\times \id_{\R^n})$ is measurable.
As $(\id_{[0,1]},\zeta)\colon [0,1]\to [0,1]\times \R^n$,
$t\mto (t,\zeta(t))$
is measurable,
also the composition $(\wh{\gamma})_*(\zeta)=\wh{\gamma}\circ (\id_{[0,1]},\zeta)$
is measurable. On $\R^m$, we use the maximum-norm~$\|.\|_\infty$,
giving rise to a continuous norm $q:=\|.\|_{\cL^\infty,\|.\|_\infty}$
on $C^\infty_K(\R^n,\R^m)\sub \cL^\infty(\R^n,\R^m)$,
which in turn gives rise to a continuous seminorm $\|.\|_{\cL^1,q}$
on $\cL^1([0,1],C^\infty_K(\R^n,\R^m))$.
Now
\begin{equation}\label{immed2}
\|(\wh{\gamma})_*(\zeta)(t)\|_\infty=\|\gamma(t)(\zeta(t))\|_\infty
\leq \|\gamma(t)\|_{\cL^\infty,\|.\|_\infty}=q(\gamma(t)),
\end{equation}
whence $\int_0^1\|(\wh{\gamma})_*(\zeta)(t)\|_\infty\,dt\leq
\int_0^1q(\gamma(t))\,dt=\|\gamma\|_{\cL^1,q}<\infty$
and hence $(\wh{\gamma})_*(\zeta)\in \cL^1([0,1],\R^m)$,
with $\|(\wh{\gamma})_*(\zeta)\|_{\cL^1,\|.\|_\infty}\leq \|\gamma\|_{\cL^1,q}$.
\end{proof}
With notation as in the preceding lemma, we have:
\begin{la}\label{phimglat}
For each $m\in \N$ and each compact subset
$K\sub \R^n$,
the map
\[
\Phi_m\colon
L^1([0,1],C^\infty_K(\R^n,\R^m))\times C([0,1],\R^n)\to L^1([0,1],\R^m),
\]
$\Phi_m([\gamma],\zeta):=[(\wh{\gamma})_*(\zeta)]$, is smooth.
\end{la}
\begin{proof}
Since $\Phi_m(\gamma,\zeta)$ is linear in $\gamma$,
it suffices to show that $\Phi_m$ is $C^{0,\infty}$
(see \ref{basicCrs} (b) and (a)). We show by induction that $\Phi_m$ is $C^{0,k}$
for each $k\in \N_0$. Let $k=0$ first;
we have to show that $\Phi_m$ is continuous.
From the preceding proof, we know that
\[
\|\Phi_m(\gamma,\zeta)\|_{L^1,\|.\|_\infty}
\leq \|\gamma\|_{L^1,q}
\]
with $q:=\|.\|_{\cL^\infty,\|.\|_\infty}$.
Consider the map
$D\colon C^\infty_K(\R^n,\R^m)\to C^\infty_K(\R^n,\R^{m\times n})$
such that $D(f)(x):=f'(x)\in \cL(\R^n,\R^m)\cong \R^{m\times n}$
is the Jacobi matrix of $f$ at $x$.
Then $D$ is continuous linear,
entailing that $p:=\|.\|_{\cL^\infty,\|.\|_{op}}\circ D$
is a continuous seminorm on $C^\infty_K(\R^n, \R^m)$.
Thus $p(f)=\sup_{x\in \R^n}\|f'(x)\|_{op}$ for
$f\in C^\infty_K(\R^n, \R^m)$.
Let $\tilde{\gamma}\in
\cL^1([0,1],C^\infty_K(\R^n,\R^m))$
and $\gamma:=[\tilde{\gamma}]\in L^1([0,1],C^\infty_K(\R^n,\R^m))$.
We have for all
$\eta,\eta_1\in C([0,1],\R^n)$
and $t\in [0,1]$
\begin{equation}\label{reuseindusp}
\tilde{\gamma}(t)(\eta(t))-\tilde{\gamma}(t)(\eta_1(t))
=
\int_0^1 (D(\tilde{\gamma}(t))(\eta_1(t)+s(\eta(t)-\eta_1(t))).(\eta(t)-\eta_1(t))\,ds
\end{equation}
and thus
\begin{eqnarray}
\lefteqn{\|\tilde{\gamma}(t)(\eta(t))-\tilde{\gamma}(t)(\eta_1(t))\|_\infty}
\qquad\notag \\
&\leq &
\int_0^1\|(D(\tilde{\gamma}(t))(\eta_1(t)
+s(\eta(t)-\eta_1(t))).(\eta(t)-\eta_1(t))\|_\infty\,ds\notag \\
&\leq&
\|D(\tilde{\gamma}(t))\|_{\cL^\infty,\|.\|_{op}}\|\eta-\eta_1\|_{\cL^\infty,\|.\|_\infty}\notag \\
&=& p(\tilde{\gamma}(t))\|\eta-\eta_1\|_{\cL^\infty,\|.\|_\infty}.\label{yesyep}
\end{eqnarray}
Integrating over~$t$, we deduce that
\[
\|\Phi_m(\gamma)(\eta)-\Phi_m(\gamma)(\eta_1)\|_{L^1,\|.\|_\infty}
\leq
\|\gamma\|_{L^1,p}\|\eta-\eta_1\|_{\cL^\infty,\|.\|_\infty}.
\]
As a consequence,
\begin{eqnarray*}
\lefteqn{\|\Phi_m(\gamma,\zeta)-\Phi_m(\gamma_1,\zeta_1)\|_{L^1,\|.\|_\infty}}\qquad\\
&\leq& \|\Phi_m(\gamma,\zeta)-\Phi_m(\gamma,\zeta_1)\|_{L^1,\|.\|_\infty}
+\|\Phi_m(\gamma,\zeta_1)-\Phi_m(\gamma_1,\zeta_1)\|_{L^1,\|.\|_\infty}\\
&\leq&
\|\gamma\|_{L^1,p}\|\eta-\eta_1\|_{\cL^\infty,\|.\|_\infty}
+ \|\gamma-\gamma_1\|_{L^1,q}
\end{eqnarray*}
for all $\gamma_1\in
L^1([0,1],C^\infty_K(\R^n,\R^m))$ and $\gamma,\eta,\eta_1$
as before, which can be made arbitrarily small
for $\gamma_1$ close to $\gamma$ and $\eta_1$ close to~$\eta$.
Thus $\Phi_m$ is continuous
at each $(\gamma,\eta)$ and
thus $\Phi_m$ is continuous.\\[2.1mm]
Let $k\in \N$ now and assume that $\Phi_m$
is $C^{0,k-1}$ for each $m\in \N$.
Let
$\gamma=[\tilde{\gamma}]\in
L^1([0,1],C^\infty_K(\R^n,\R^m))$
with
$\tilde{\gamma}\in
\cL^1([0,1],C^\infty_K(\R^n,\R^m))$
and
$\eta,\eta_1\in C([0,1],\R^n)$.
To calculate
$d_2\Phi_m(\gamma,\eta;\eta_1)$,
we consider the corresponding directional
difference quotients first.
For $t\in [0,1]$ and $\tau\in \R\setminus \{0\}$, we
have
\begin{equation}\label{hiernkla}
\frac{\tilde{\gamma}(t)(\eta(t)+\tau \eta_1(t))-\tilde{\gamma}(t)(\eta(t))}{\tau}
=
\int_0^1 (D(\tilde{\gamma}(t))(\eta(t)+s\tau \eta_1(t)).\eta_1(t)\,ds
\end{equation}
by (\ref{reuseindusp}).
The map
\[
\alpha
\colon L^1([0,1],\R^{m\times n})\to L^1([0,1],\R^m),\quad
\alpha([f]):=[s\mto f(s)\eta_1(s)]
\]
for $f\in \cL^1([0,1],\R^{m\times n})$,
$s\in [0,1]$
(given pointwise by multiplication of matrices and vectors)
is linear and continuous,
with $\|\alpha\|_{op}\leq \|\eta_1\|_{\cL^\infty,\|.\|_\infty}$.
We abbreviate $\bar{\gamma}:=D \circ \tilde{\gamma}$
and identify $\R^{m\times n}$ with~$\R^{mn}$.
Then the mapping\linebreak
$h\colon \R\times [0,1]\to L^1([0,1],\R^m)$,
\[
h(\tau,s):=
\alpha([(\wh{\bar{\gamma}})_*(\eta+s\tau\eta_1)])
=\alpha(\Phi_{mn}([\bar{\gamma}],\eta+s\tau\eta_1))
\]
is continuous, by induction,
and we record that
\[
h(0,s)=\alpha(\Phi_{mn}([\bar{\gamma}],\eta))=\Phi_{mn}([\bar{\gamma}],\eta)\eta_1
\]
is independent of~$s\in [0,1]$.
Now the theorem on parameter-dependent integrals
(see \ref{pardep}) shows that
\[
g\colon \R\to L^1([0,1],\R^m),\quad
g(\tau):=\int_0^1 h(\tau,s)\,ds
\]
is continuous.
We claim that
\[
g(\tau)=\frac{\Phi_m(\gamma,\eta+\tau\eta_1)-\Phi_m(\gamma,\eta)}{\tau}
\quad\mbox{for all $\tau\in \R\setminus \{0\}$.}
\]
If this is true, then the continuity of~$g$ implies
that the limit as $\tau\to 0$ exists;
we have
\begin{eqnarray*}
d_2\Phi_m(\gamma,\eta;\eta_1)
&= &\lim_{\tau\to0}\frac{\Phi_m(\gamma,\eta+\tau\eta_1)-\Phi_m(\gamma,\eta)}{\tau}\\
&=&\lim_{\tau\to0}g(\tau)=g(0)=\int_0^1h(0,s)\,ds
=\Phi_{mn}([\bar{\gamma}],\eta)\eta_1.
\end{eqnarray*}
The map
\[
\beta\colon L^1([0,1],\R^{m\times n})\times C([0,1],\R^n)
\to L^1([0,1],\R^m),\;\,
\beta([f],g):=[t\mto f(t)g(t)]
\]
given by pointwise multiplication
of matrices and vectors
is continuous bilinear with
$\|\beta\|_{\op}\leq 1$,
and hence smooth.
By the preceding,
we have
\begin{equation}\label{goodrhs}
d_2\Phi_m(\gamma,\eta;\eta_1)=\beta(\Phi_{mn}([\bar{\gamma}],\eta),\eta_1).
\end{equation}
The map
\[
L^1([0,1],D)\colon L^1([0,1],C^\infty_K(\R^n,\R^m))\to
L^1([0,1],C^\infty_K(\R^n,\R^{m\times n}))
\]
sending
$\gamma=[\tilde{\gamma}]$ to $[\bar{\gamma}]=[D\circ \tilde{\gamma}]$
is continuous linear.
The map $\Phi_{mn}$ is $C^{0,k-1}$ by induction.
Hence also $(\gamma,\eta)\mto \Phi_{mn}([\bar{\gamma}],\eta)$
is $C^{0,k-1}$ (see \ref{basicCrs}\,(c)).
Looking at the right hand side
of (\ref{goodrhs}),
we deduce with \ref{basicCrs}\,(d) that $d_2\Phi_m$ is $C^{0,k-1,\infty}$
as a function of $(\gamma,\eta,\eta_1)$
and hence (by \ref{basicCrs}\,(e))
$C^{0,k-1}$ as a function of $(\gamma,(\eta,\eta_1))$.
Hence $\Phi_m$ is $C^{0,k}$, by \ref{basicCrs}\,(f).\\[2.3mm]
To prove the claim made above, we consider the continuous linear functionals
\[
I_{\lambda,\theta}\colon L^1([0,1],\R^m)\to \R,\quad
I_{\lambda,\theta}([f]):=\int_0^1 \lambda(f(t))\theta(t)\,dt
\]
for $\theta\in \cL^\infty([0,1],\R)$
and $\lambda$ in the dual space~$(\R^m)'$.
Then
\begin{eqnarray*}
I_{\lambda,\theta}(g(\tau)) &= &\int_0^1 I_{\lambda,\theta}(h(\tau,s))\,ds\\
&=&
\int_0^1\int_0^1
\lambda\big(D(\tilde{\gamma}(t))(\eta(t)+s\tau\eta_1(t))\eta_1(t)\big)\theta(t)
\,dt\,ds\\
&=&
\int_0^1\lambda\left(\int_0^1
D(\tilde{\gamma}(t))(\eta(t)+s\tau\eta_1(t))\eta_1(t)
\,ds\right) \,\theta(t) \, dt\\
&=&\int_0^1\lambda\left(\frac{\tilde{\gamma}(t)(\eta(t)+\tau \eta_1(t))
-\tilde{\gamma}(t)(\eta(t))}{\tau}\right)
\theta(t)\,dt\\
&=&
I_{\lambda,\theta}\left(\frac{\Phi_n(\gamma,\eta+\tau \eta_1)-
\Phi_n(\gamma,\eta)}{\tau}\right),
\end{eqnarray*}
using (\ref{hiernkla}) for the penultimate equality
and Fubini's Theorem for the third equality
(justified by Lemma~\ref{thusFub}).
As the $I_{\lambda,\theta}$ separate points on $L^1([0,1],\R^m)$,
the claim is established.
\end{proof}
We hasten to check that the hypotheses of Fubini's Theorem were satisfied
in the preceding situation.
\begin{la}\label{thusFub}
The function
$f\colon [0,1]^2 \to \R$,
\[
(t,s)\mapsto
\lambda(D(\tilde{\gamma}(t))(\eta(t)+s\tau\eta_1(t))\eta_1(t))\theta(t)
\]
is in $\cL^1([0,1]^2,\R)$ with respect to Lebesgue-Borel measure on $[0,1]^2$.
\end{la}
\begin{proof}
To see that $f$ is measurable,
write $\tilde{\gamma}(t)=(\tilde{\gamma}_1(t),\ldots,\tilde{\gamma}_m(t))$
(identifying $C^\infty_K(\R^n,\R^m)$ with $C^\infty_K(\R^n,\R)^m$).
Then $t\mto \frac{\partial \tilde{\gamma}_i(t)}{\partial x_j}$
is an element of $\cL^1([0,1],C^\infty_K(\R^n,\R))$
for each $j\in \{1,\ldots, n\}$.
Write $\eta_1=(\eta_{1,1},\ldots, \eta_{1,n})$
with continuous functions $\eta_{1,j}\colon [0,1]\to\R$
for $j\in\{1,\ldots, n\}$.
There are $\lambda_1,\ldots,\lambda_m\in \R$ such that
$\lambda(x_1,\ldots, x_m)=\lambda_1x_1+\cdots+\lambda_mx_m$.
The evaluation map
\[
\ve\colon C([0,1],\R^n)\times [0,1]\to \R^n,\quad (\kappa,t)\mto\kappa(t)
\]
is continuous and hence measurable.
The map $[0,1]^2\to C([0,1],\R^n)\times [0,1]$, $(s,t)\mto (\eta+s\tau\eta_1,t)$
is continuous and hence measurable.
Now the formula
\[
f(t,s)=\sum_{i=1}^m\sum_{j=1}^n\lambda_i
\frac{\partial\tilde{\gamma}_i(t)}{\partial x_j}(\ve(\eta+s\tau\eta_1,t))
\eta_{1,j}(t)\theta(t)
\]
shows that $f$ is measurable, being a sum of products of measurable
real-valued functions.
Using Fubini's theorem
for non-negative measurable functions on $[0,1]^2$, we find that
\begin{eqnarray*}
\lefteqn{\int_{[0,1]^2}|f(t,s)|\,d\lambda_2(t,s)}\qquad\qquad\\
&=& \int_0^1\int_0^1 |f(t,s)|\,ds\, dt\\
&\leq &
\|\theta\|_{\cL^\infty}
\sum_{i=1}^m\sum_{j=1}^n
\|\eta_{1,j}\|_{\cL^\infty}
|\lambda_i|
\int_0^1\int_0^1
\underbrace{\left|\frac{\partial\tilde{\gamma}_i(t)}{\partial x_j}
(\ve(\eta+s\tau\eta_1,t))\right|}_{\leq
\|\frac{\partial\tilde{\gamma}_i(t)}{\partial x_j}\|_{\cL^\infty}} ds\,dt\\
&\leq & \int_0^1 \left\|\frac{\partial\tilde{\gamma}_i}{\partial x_j}\right\|_{\cL^\infty} \,dt
= \|\partial/{\partial x_j}\circ \tilde{\gamma}_i\|_{\cL^1,p}<\infty
\end{eqnarray*}
with $p:=\|.\|_{\cL^\infty}:=\|.\|_{\cL^\infty,|.|}$.
\end{proof}
%
%
%
%
%
%
%
%
\begin{numba}\label{moreset}
Consider
$D\colon C^\infty_K(\R^n,\R^n)\to C^\infty_K(\R^n,\R^{n\times n})$, $f\mto f'$
and the continuous seminorm
$p:=\|.\|_{\cL^\infty,\|.\|_{op}}\circ D$
on $C^\infty_K(\R^n,\R^n)$;
thus $p(f)=\sup_{x\in \R^n}\|f'(x)\|_{op}$ for
$f\in C^\infty_K(\R^n, \R^n)$.
Fix $L\in \,]0,1[$.
Then
%
%
%
\[
Q_K:=
\{[\gamma]\in L^1([0,1],C^\infty_K(\R^n,\R^n))\colon
\|\gamma\|_{L^1,p}<L\}
\]
is an open $0$-neighbourhood in $L^1([0,1],C^\infty_K(\R^n,\R^n))$.\\[2.3mm]
%
%
%
%
We define a map
$\Psi_K \colon Q_K\times \R^n\times
C([0,1],\R^n)\to C([0,1],\R^n)$ via
\[
\Psi_K([\gamma],x,\kappa)(t):=x+\int_0^t \gamma(s)(\kappa(s))\,ds
\]
for $[\gamma]\in Q_K$ with $\gamma\in \cL^1([0,1],C^\infty_K(\R,\R^n))$,
$x\in \R^n$, $\kappa\in C([0,1],\R^n)$ and $t\in [0,1]$.
\end{numba}
\begin{la}\label{hfway}
The map
$\Psi_K \colon Q_K\times \R^n\times
C([0,1],\R^n)\to C([0,1],\R^n)$
is smooth and defines a uniform family of contractions
in the final variable, in the sense that
\[
\Lip(\Psi_K([\gamma],x,\sbull))\leq L
\]
for all $[\gamma]\in Q_K$ and $x\in \R^n$.
\end{la}
\begin{proof}
For $x\in \R^n$ let $c_x\colon [0,1]\to\R^n$
be the constant function $t\mto x$. 
The map $\R^n\to C([0,1],\R^n)$, $x\mto c_x$
is continuous linear and hence smooth.
Moreover, the operator
\[
J\colon L^1([0,1],\R^n)\to C([0,1],\R^n)
\]
determined by
$J([f])(t):=\int_0^t f(s)\, ds$
is continuous and linear and hence smooth.
Now the formula
\[
\Psi_K([\gamma],x,\kappa)=c_x+J(\Phi_n([\gamma],\kappa))
\]
(with the smooth map $\Phi_n$ from Lemma~\ref{phimglat})
shows that $\Psi_K$ is smooth.
Given $\eta,\eta_1\in C([0,1],\R^n)$, we deduce from (\ref{yesyep})
that
\begin{eqnarray*}
\|\Psi_K([\gamma],x,\eta)-\Psi_K([\gamma],x,\eta_1)\|_\infty &=&
\sup_{t\in [0,1]}
\left\|\int_0^t \gamma(s)(\eta(s))-\gamma(s)(\eta_1(s))\,ds\right\|_\infty\\
&\leq &
\sup_{t\in[0,1]}
\int_0^t \underbrace{\|\gamma(s)(\eta(s))-
\gamma(s)(\eta_1(s))\|_\infty}_{\leq p(\gamma(t))\|\eta-\eta_1\|_{L^\infty}} \,ds\\
&\leq &
\|\gamma\|_{\cL^1,p} \|\eta-\eta_1\|_{L^\infty}
\leq L\|\eta-\eta_1\|_{L^\infty}
\end{eqnarray*}
with $p$ as in \ref{moreset}.
This ends the proof.
\end{proof}
For each $([\gamma],x)\in Q_K\times\R^n$,
the contraction
\[
\Psi_K([\gamma],x,\sbull)\colon C([0,1],\R^n)\to C([0,1],\R^n)
\]
of the Banach space $C([0,1],\R^n)$ has a unique
fixed point $\zeta_{[\gamma],x}\in C([0,1],\R^n)$,
by Banach's Contraction Principle;
thus
\begin{equation}\label{willreco}
\Psi_K([\gamma],x,\zeta_{[\gamma],x})=\zeta_{[\gamma],x}.
\end{equation}
Since $\Psi_K$ is smooth, Lemma~\ref{PARFIX}
shows that also the map
\[
Q_K\times \R^n\to C([0,1],\R^n),\quad ([\gamma],x)\mto\zeta_{[\gamma],x}
\]
is smooth.
Define $F_K([\gamma])(t)(x):=\zeta_{[\gamma],x}(t)$
for $[\gamma]\in Q_K$, $x\in \R^n$ and $t\in [0,1]$.
Using the exponential laws from \cite{AaS},
we deduce:
\begin{itemize}
\item[(a)]
$F_K([\gamma])(t)\in C^\infty(\R^n,\R^n)$
for all $[\gamma]\in Q_K$ and $t\in[0,1]$;
\item[(b)] $F_K([\gamma])\in C([0,1],C^\infty(\R^n,\R^n))$ for all $[\gamma]\in Q_K$;
\item[(c)]
$F_K\colon Q_K \to C([0,1],C^\infty(\R^n,\R^n))$
is smooth.
\end{itemize}
As a consequence, also the map
\[
E_K\colon Q_K\to C([0,1],C^\infty(\R^n,\R^n)),\quad [\gamma]\mto F_K([\gamma])- I
\]
is smooth, where $I\colon [0,1]\to C^\infty(\R^n,\R^n)$ is the constant function
$t\mto\id_{\R^n}$.
If $x\in \R^n\setminus K$, then $\zeta_{[\gamma],x}$ is the fixed point
of the map $\Psi_K([\gamma],x,\sbull)$
determined by
\[
\Psi_K([\gamma],x,\kappa)(t)=x+\int_0^t\gamma(s)(\kappa(s))\,ds.
\]
Since $\gamma(s)(x)=0$ for each $s\in [0,1]$, also the constant map $c_x$
is a fixed point and thus $\zeta_{[\gamma],x}=c_x$ by uniqueness
of the latter.
As a consequence,
\[
E_K([\gamma])(t)\in C^\infty_K(\R^n,\R^n)
\]
for each $t\in [0,1]$ and thus $E_K$ can be considered as a smooth map
\[
E_K\colon Q_K\to C([0,1],C^\infty_K(\R^n,\R^n)).
\]
Since $E_K(0)=E_\emptyset(0)=0$ and $E_K$ is continuous,
there is an open $0$-neighbour\-hood $P_K\sub Q_K$
such that
\[
E_K(P_K)\sub \Omega_K.
\]
For each $[\gamma]\in P_K$, we have $\eta:=E_K([\gamma])\in C([0,1],\Omega_K)$
and $\zeta_x(t):=x+\gamma(t)(x)=F_K([\gamma])(t)(x)=\zeta_{[\gamma],x}$
satisfies~(\ref{niceinteqx}) by~(\ref{willreco}).
Hence $\eta\colon [0,1]\to \Omega_K$
satisfies~(\ref{checkthis}) by the discussion
in~(\ref{firsnum}) and thus $\eta=\Evol_{\Omega_K}([\gamma])$.
We now deduce from Proposition~\ref{lctoglb} and Lemma~\ref{setsubdiv}
that $(\Omega_K,*)$
(and hence also $\Diff_K(\R^n)$) is $L^1$-regular.\vfill\pagebreak

\noindent
{\bf{\boldmath$L^1$}-regularity of {\boldmath$\Diff(\bS_1)$}}\\[2.3mm]
In the Fr\'{e}chet space
\[
C^\infty_{2\pi}(\R,\R):=\{\gamma\in C^\infty(\R,\R)\colon (\forall x\in \R)\;
\gamma(x+2\pi)=\gamma(x)\},
\]
the set
\[
\Omega_{2\pi}:=\{\gamma\in C^\infty_{2\pi}(\R,\R)\colon
(\forall x\in [0,2\pi])\, \gamma'(x)>-1\}
\]
is open and convex (hence simply connected).
It is a well-known fact the $\Omega_{2\pi}$
is the universal covering group
of the identity component $\Diff(\bS_1)_0$ of $\Diff(\bS_1)$,
with the group multiplication
\[
\Omega_{2\pi}\times\Omega_{2\pi}\to \Omega_{2\pi},\quad
(\gamma,\eta)\mto \eta+\gamma\circ (\id_{\R}+\eta)
\]
(see, e.g., \cite{GaN}).
The universal covering map takes
$\gamma\in \Omega_{2\pi}$ to $\phi_\gamma\in \Diff(\bS_1)_0$,
\[
\phi_\gamma(e^{it}):=e^{i\gamma(t)}e^{it}=e^{i(t+\gamma(t))}\quad\mbox{for all $t\in\R$.}
\]
Setting $n=1$ and
replacing $\Omega_K$ with $\Omega_{2\pi}$
and $C^\infty_K(\R^n,\R^m)$ with $C^\infty_{2\pi}(\R,\R^m)$
in the preceding discussion of $\Diff_K(\R^n)\cong\Omega_K$,
we see that $\Omega_{2\pi}$ (and hence $\Diff(\bS_1)$)
is $L^1$-regular.\\[4mm]
%
%
%
%
%
%
{\bf{\boldmath$L^1$}-regularity of {\boldmath$\Diff_c(\R^n)$}}
\begin{la}\label{gvsLone}
Let $U\sub \R^n$ be an open set and $V\sub U$ be an open, convex subset
with compact closure $\wb{V}\sub U$.
Let $m\in \N_0$ and
$\gamma\in \cL^1([0,1],C^\infty(U,\R^m))$.
Then the map
\[
\wh{\gamma}\colon [0,1]\times U\to \R^m,\quad\wh{\gamma}(t,x):=\gamma(t)(x)
\]
is measurable. For any measurable function
$\zeta\colon [0,1]\to V$, define a function
\[
(\wh{\gamma})_*(\zeta)\colon [0,1]\to \R^m
\]
via $(\wh{\gamma})_*(\zeta)(t):=\wh{\gamma}(t,\zeta(t))=\gamma(t)(\zeta(t))$.
Then $(\wh{\gamma})_* (\zeta)\in \cL^1([0,1],\R^m)$.
\end{la}
\begin{proof}
Since $U$ is second countable, the Borel $\sigma$-algebra
$\cB(C^\infty(U,\R^m)\times U)$ coincides with the product $\sigma$-algebra
$\cB(C^\infty(U,\R^m))\tensor \cB(U)$ (see \ref{basicsmeas}\,(f)).
Therefore the map
\[
\gamma\times\id_U\colon [0,1]\times U\to C^\infty(U,\R^m)\times U,\;\;
(t,x)\mto (\gamma(t),x)
\]
is Borel measurable.
The evaluation map
\[
\ve\colon C^\infty(U,\R^m)\times U\to\R^m,\quad\ve(f,x):=f(x)
\]
is $C^\infty$ (see, e.g., \cite{ZOO} or \cite{GaN}),
hence continuous and hence measurable.
Thus $\wh{\gamma}=\ve\circ (\gamma\times \id_{\R^n})$ is measurable.
As $(\id_{[0,1]},\zeta)\colon [0,1]\to [0,1]\times U$,
$t\mto (t,\zeta(t))$
is measurable,
also the composition $(\wh{\gamma})_*(\zeta)=\wh{\gamma}\circ (\id_{[0,1]},\zeta)$
is measurable. On $\R^m$, we use the maximum-norm~$\|.\|_\infty$,
giving rise to a continuous norm $q:=\|.\|_{\wb{V},\|.\|_\infty}$
on $C^\infty(U,\R^m)$,
which in turn gives rise to a continuous seminorm $\|.\|_{\cL^1,q}$
on $\cL^1([0,1],C^\infty_(U,\R^m))$.
Now
\begin{equation}\label{immed3}
\|(\wh{\gamma})_*(\zeta)(t)\|_\infty=\|\gamma(t)(\zeta(t))\|_\infty
\leq \|\gamma(t)\|_{\wb{V},\|.\|_\infty}=q(\gamma(t)),
\end{equation}
whence $\int_0^1\|(\wh{\gamma})_*(\zeta)(t)\|_\infty\,dt\leq
\int_0^1q(\gamma(t))\,dt=\|\gamma\|_{\cL^1,q}<\infty$
and hence $(\wh{\gamma})_*(\zeta)\in \cL^1([0,1],\R^m)$,
with $\|(\wh{\gamma})_*(\zeta)\|_{\cL^1,\|.\|_\infty}\leq \|\gamma\|_{\cL^1,q}$.
\end{proof}
With notation as in the preceding lemma, we have:
\begin{la}\label{phimglatdp}
For each $m\in \N$,
the map
\[
\Phi_m\colon
L^1([0,1],C^\infty(U,\R^m))\times C([0,1],V)\to L^1([0,1],\R^m),
\]
$\Phi_m([\gamma],\zeta):=[(\wh{\gamma})_*(\zeta)]$, is smooth.
\end{la}
\begin{proof}
Since $\Phi_m(\gamma,\zeta)$ is linear in $\gamma$,
it suffices to show that $\Phi_m$ is $C^{0,\infty}$
(see \ref{basicCrs} (b) and (a)). We show by induction that $\Phi_m$ is $C^{0,k}$
for each $k\in \N_0$. Let $k=0$ first;
we have to show that $\Phi_m$ is continuous.
From the preceding proof, we know that
\[
\|\Phi_m(\gamma,\zeta)\|_{L^1,\|.\|_\infty}
\leq \|\gamma\|_{L^1,q}
\]
with $q:=\|.\|_{\wb{V},\|.\|_\infty}$.
Consider the map
$D\colon C^\infty(U,\R^m)\to C^\infty(U,\R^{m\times n})$
such that $D(f)(x):=f'(x)\in \cL(\R^n,\R^m)\cong \R^{m\times n}$
is the Jacobi matrix of $f$ at $x$.
Then $D$ is continuous linear,
entailing that $p:=\|.\|_{\wb{V},\|.\|_{op}}\circ D$
is a continuous seminorm on $C^\infty(U, \R^m)$.
Thus $p(f)=\sup_{x\in \wb{V}}\|f'(x)\|_{op}$ for
$f\in C^\infty(U, \R^m)$.
Let $\tilde{\gamma}\in
\cL^1([0,1],C^\infty(U,\R^m))$
and $\gamma:=[\tilde{\gamma}]\in L^1([0,1],C^\infty(U,\R^m))$.
We have for all
$\eta,\eta_1\in C([0,1],V)$
and $t\in [0,1]$
\begin{equation}\label{reuseinduspdp}
\tilde{\gamma}(t)(\eta(t))-\tilde{\gamma}(t)(\eta_1(t))
=
\int_0^1 (D(\tilde{\gamma}(t))(\eta_1(t)+s(\eta(t)-\eta_1(t))).(\eta(t)-\eta_1(t))\,ds
\end{equation}
and thus
\begin{eqnarray}
\lefteqn{\|\tilde{\gamma}(t)(\eta(t))-\tilde{\gamma}(t)(\eta_1(t))\|_\infty}
\qquad\notag \\
&\leq &
\int_0^1\|(D(\tilde{\gamma}(t))(\eta_1(t)
+s(\eta(t)-\eta_1(t))).(\eta(t)-\eta_1(t))\|_\infty\,ds\notag \\
&\leq&
\|D(\tilde{\gamma}(t))\|_{\wb{V},\|.\|_{op}}\|\eta-\eta_1\|_{\cL^\infty,\|.\|_\infty}\notag \\
&=& p(\tilde{\gamma}(t))\|\eta-\eta_1\|_{\cL^\infty,\|.\|_\infty}.\label{yesyepdp}
\end{eqnarray}
Integrating over~$t$, we deduce that
\[
\|\Phi_m(\gamma)(\eta)-\Phi_m(\gamma)(\eta_1)\|_{L^1,\|.\|_\infty}
\leq
\|\gamma\|_{L^1,p}\|\eta-\eta_1\|_{\cL^\infty,\|.\|_\infty}.
\]
As a consequence,
\begin{eqnarray*}
\lefteqn{\|\Phi_m(\gamma,\zeta)-\Phi_m(\gamma_1,\zeta_1)\|_{L^1,\|.\|_\infty}}\qquad\\
&\leq& \|\Phi_m(\gamma,\zeta)-\Phi_m(\gamma,\zeta_1)\|_{L^1,\|.\|_\infty}
+\|\Phi_m(\gamma,\zeta_1)-\Phi_m(\gamma_1,\zeta_1)\|_{L^1,\|.\|_\infty}\\
&\leq&
\|\gamma\|_{L^1,p}\|\eta-\eta_1\|_{\cL^\infty,\|.\|_\infty}
+ \|\gamma-\gamma_1\|_{L^1,q}
\end{eqnarray*}
for all $\gamma_1\in
L^1([0,1],C^\infty(U,\R^m))$ and $\gamma,\eta,\eta_1$
as before, which can be made arbitrarily small
for $\gamma_1$ close to $\gamma$ and $\eta_1$ close to~$\eta$.
Thus $\Phi_m$ is continuous
at each $(\gamma,\eta)$ and
thus $\Phi_m$ is continuous.\\[2.1mm]
Let $k\in \N$ now and assume that $\Phi_m$
is $C^{0,k-1}$ for each $m\in \N$.
Let
$\gamma=[\tilde{\gamma}]\in
L^1([0,1],C^\infty(U,\R^m))$
with
$\tilde{\gamma}\in
\cL^1([0,1],C^\infty(U,\R^m))$;
let
$\eta\in C([0,1],V)$
and $\eta_1\in C([0,1],\R^n)$.
Since $\eta([0,1])$ is a compact subset
of the open set $V$ and $\eta_1([0,1])$ is compact,
there is $\delta>0$ such that
\[
\eta(t)+r\eta_1(t)\in V\quad\mbox{for all $r\in [-\delta,\delta]$ and $t\in[0,1]|$.}
\]
To calculate
$d_2\Phi_m(\gamma,\eta;\eta_1)$,
we consider the corresponding directional
difference quotients first.
For $t\in [0,1]$ and $\tau\in \;]{-\delta},\delta[\; \setminus \{0\}$, we
have
\begin{equation}\label{hiernkladp}
\frac{\tilde{\gamma}(t)(\eta(t)+\tau \eta_1(t))-\tilde{\gamma}(t)(\eta(t))}{\tau}
=
\int_0^1 (D(\tilde{\gamma}(t))(\eta(t)+s\tau \eta_1(t)).\eta_1(t)\,ds
\end{equation}
by (\ref{reuseinduspdp}).
The map
\[
\alpha
\colon L^1([0,1],\R^{m\times n})\to L^1([0,1],\R^m),\quad
\alpha([f]):=[s\mto f(s)\eta_1(s)]
\]
for $f\in \cL^1([0,1],\R^{m\times n})$,
$s\in [0,1]$
(given pointwise by multiplication of matrices and vectors)
is linear and continuous,
with $\|\alpha\|_{op}\leq \|\eta_1\|_{\cL^\infty,\|.\|_\infty}$.
We abbreviate $\bar{\gamma}:=D \circ \tilde{\gamma}$
and identify $\R^{m\times n}$ with~$\R^{mn}$.
Then the mapping\linebreak
$h\colon \R\times [0,1]\to L^1([0,1],\R^m)$,
\[
h(\tau,s):=
\alpha([(\wh{\bar{\gamma}})_*(\eta+s\tau\eta_1)])
=\alpha(\Phi_{mn}([\bar{\gamma}],\eta+s\tau\eta_1))
\]
is continuous, by induction,
and we record that
\[
h(0,s)=\alpha(\Phi_{mn}([\bar{\gamma}],\eta))=\Phi_{mn}([\bar{\gamma}],\eta)\eta_1
\]
is independent of~$s\in [0,1]$.
Now the theorem on parameter-dependent integrals
(see \ref{pardep}) shows that
\[
g\colon \;]{-\delta},\delta[\; \to L^1([0,1],\R^m),\quad
g(\tau):=\int_0^1 h(\tau,s)\,ds
\]
is continuous.
Then
\[
g(\tau)=\frac{\Phi_m(\gamma,\eta+\tau\eta_1)-\Phi_m(\gamma,\eta)}{\tau}
\quad\mbox{for all $\tau\in \;]{-\delta},\delta[\; \setminus \{0\}$};
\]
this can be shown as in the proof of Lemma~\ref{phimglat}
(using the next lemma).
Now the continuity of~$g$ implies
that the limit as $\tau\to 0$ exists;
we have
\begin{eqnarray*}
d_2\Phi_m(\gamma,\eta;\eta_1)
&= &\lim_{\tau\to0}\frac{\Phi_m(\gamma,\eta+\tau\eta_1)-\Phi_m(\gamma,\eta)}{\tau}\\
&=&\lim_{\tau\to0}g(\tau)=g(0)=\int_0^1h(0,s)\,ds
=\Phi_{mn}([\bar{\gamma}],\eta)\eta_1.
\end{eqnarray*}
The map
\[
\beta\colon L^1([0,1],\R^{m\times n})\times C([0,1],\R^n)
\to L^1([0,1],\R^m),\;\,
\beta([f],g):=[t\mto f(t)g(t)]
\]
given by pointwise multiplication
of matrices and vectors
is continuous bilinear with
$\|\beta\|_{\op}\leq 1$,
and hence smooth.
By the preceding,
we have
\begin{equation}\label{goodrhsdp}
d_2\Phi_m(\gamma,\eta;\eta_1)=\beta(\Phi_{mn}([\bar{\gamma}],\eta),\eta_1).
\end{equation}
The map
\[
L^1([0,1],D)\colon L^1([0,1],C^\infty(U,\R^m))\to
L^1([0,1],C^\infty(U,\R^{m\times n}))
\]
sending
$\gamma=[\tilde{\gamma}]$ to $[\bar{\gamma}]=[D\circ \tilde{\gamma}]$
is continuous linear.
The map $\Phi_{mn}$ is $C^{0,k-1}$ by induction.
Hence also $(\gamma,\eta)\mto \Phi_{mn}([\bar{\gamma}],\eta)$
is $C^{0,k-1}$ (see \ref{basicCrs}\,(c)).
Looking at the right hand side
of (\ref{goodrhs}),
we deduce with \ref{basicCrs}\,(d) that $d_2\Phi_m$ is $C^{0,k-1,\infty}$
as a function of $(\gamma,\eta,\eta_1)$
and hence (by \ref{basicCrs}\,(e))
$C^{0,k-1}$ as a function of $(\gamma,(\eta,\eta_1))$.
Hence $\Phi_m$ is $C^{0,k}$, by \ref{basicCrs}\,(f).
\end{proof}
The following analogue of Lemma~\ref{thusFub} was used.
\begin{la}\label{thusFubdp}
The function
$f\colon [0,1]^2 \to \R$,
\[
(t,s)\mapsto
\lambda(D(\tilde{\gamma}(t))(\eta(t)+s\tau\eta_1(t))\eta_1(t))\theta(t)
\]
is in $\cL^1([0,1]^2,\R)$ with respect to Lebesgue-Borel measure on $[0,1]^2$.
\end{la}
\begin{proof}
To see that $f$ is measurable,
write $\tilde{\gamma}(t)=(\tilde{\gamma}_1(t),\ldots,\tilde{\gamma}_m(t))$
(identifying $C^\infty(U,\R^m)$ with $C^\infty(U,\R)^m$).
Then $t\mto \frac{\partial \tilde{\gamma}_i(t)}{\partial x_j}$
is an element of $\cL^1([0,1],C^\infty(U,\R))$
for each $j\in \{1,\ldots, n\}$.
Write $\eta_1=(\eta_{1,1},\ldots, \eta_{1,n})$
with continuous functions $\eta_{1,j}\colon [0,1]\to\R$
for $j\in\{1,\ldots, n\}$.
There are $\lambda_1,\ldots,\lambda_m\in \R$ such that
$\lambda(x_1,\ldots, x_m)=\lambda_1x_1+\cdots+\lambda_mx_m$.
The evaluation map
\[
\ve\colon C([0,1],\R^n)\times [0,1]\to \R^n,\quad (\kappa,t)\mto\kappa(t)
\]
is continuous and hence measurable.
The map $[0,1]^2\to C([0,1],\R^n)\times [0,1]$, $(s,t)\mto (\eta+s\tau\eta_1,t)$
is continuous and hence measurable.
Now the formula
\[
f(t,s)=\sum_{i=1}^m\sum_{j=1}^n\lambda_i
\frac{\partial\tilde{\gamma}_i(t)}{\partial x_j}(\ve(\eta+s\tau\eta_1,t))
\eta_{1,j}(t)\theta(t)
\]
shows that $f$ is measurable, being a sum of products of measurable
real-valued functions.
Using Fubini's theorem
for non-negative measurable functions on $[0,1]^2$, we find that
\begin{eqnarray*}
\lefteqn{\int_{[0,1]^2}|f(t,s)|\,d\lambda_2(t,s)}\qquad\qquad\\
&=& \int_0^1\int_0^1 |f(t,s)|\,ds\, dt\\
&\leq &
\|\theta\|_{\cL^\infty}
\sum_{i=1}^m\sum_{j=1}^n
\|\eta_{1,j}\|_{\cL^\infty}
|\lambda_i|
\int_0^1\int_0^1
\underbrace{\left|\frac{\partial\tilde{\gamma}_i(t)}{\partial x_j}
(\ve(\eta+s\tau\eta_1,t))\right|}_{\leq
\|\frac{\partial\tilde{\gamma}_i(t)}{\partial x_j}\|_{\wb{V},|.|}} ds\,dt\\
&\leq & \int_0^1 \left\|\frac{\partial\tilde{\gamma}_i}{\partial x_j}\right\|_{\wb{V},|.|} \,dt
= \|\partial/{\partial x_j}\circ \tilde{\gamma}_i\|_{\cL^1,p}<\infty
\end{eqnarray*}
with $p:=\|.\|_{\wb{V},|.|}$.
\end{proof}
\begin{numba}\label{moresetdp}
For $z\in \Z^n$
and $r>0$, let $B_r(z)\sub\R^n$ be the ball
with respect to $\|.\|_\infty$.
For $z\in \Z^n$, the balls $B_1(z)$
form a locally finite open cover of $\R^n$ by
relatively compact open sets $B_1(z)$;
likewise for $B_3(z)$.
Hence
\[
\rho_1\colon C^\infty_c(\R^n,\R^n)\to\bigoplus_{z\in \Z^n}C^\infty(B_1(0),\R^n),
\quad\gamma\mto (f|_{B_1(z)})_{z\in \Z^n}
\]
and the corresponding maps
\[
\rho_3\colon C^\infty_c(\R^n,\R^n)\to\bigoplus_{z\in \Z^n}C^\infty(B_3(0),\R^n)
\]
and $\rho_4\colon C^\infty_c(\R^n,\R^n)\to\bigoplus_{z\in \Z^n}C^\infty(B_4(0),\R^n)$
are linear topological embeddings with closed (and complemented) image (see, e.g., \cite{ZOO}).
Explicitly,
$\im(\rho_1)$ is the set of all $\{(\gamma_z)_{z\in\Z^n}
\in \bigoplus_{z\in \Z^n}C^\infty(B_1(z),\R^n)$ such that
\begin{equation}\label{imzw}
(\forall z,w\in \Z^n)(\forall x\in B_1(z)\cap B_1(w))\quad \gamma_z(x)=\gamma_w(x).
\end{equation}
As a consequence, also the maps
\begin{eqnarray}
R_3 & := & L^1([0,1],\rho_3)\colon
L^1([0,1],C^\infty_c(\R^n,\R^n))\to L^1\left([0,1],
\bigoplus_{z\in \Z^n}C^\infty(B_3(z),\R^n)\right)\notag \\
& & \hspace*{15mm}\cong
\bigoplus_{z\in \Z^n}L^1([0,1],C^\infty(B_3(z),\R^n)),\label{intosumm}
\end{eqnarray}
$R_4:=L^1([0,1],\rho_4)\colon L^1([0,1],C^\infty_c(\R^n,\R^n))\to
\bigoplus_{z\in \Z^n}L^1([0,1],C^\infty(B_4(z),\R^n))$
and
\begin{eqnarray}
R_1& :=& C([0,1],\rho_1)\colon
C([0,1],C^\infty_c(\R^n,\R^n))\to C\left([0,1],
\bigoplus_{z\in \Z^n}C^\infty(B_1(z),\R^n)\right)\notag\\
& & \hspace*{15mm}\cong
\bigoplus_{z\in \Z^n}C([0,1],C^\infty(B_1(z),\R^n))\label{intosumm2}
\end{eqnarray}
are linear topological embeddings
with closed image (where we use Mujica's Theorem
and its analogue for Lebesgue spaces
discussed above to rewrite the spaces as direct sums).\\[2.3mm]
Since $\rho_1$ is a topological embedding and $\rho_1(0)=0$,
there exist open $0$-neighbourhoods $W_z\sub C^\infty(B_1(z),\R^n)$
such that
\begin{equation}\label{inOmc}
(\rho_1)^{-1}\left(\bigoplus_{z\in \Z^n}W_n\right)\sub \Omega.
\end{equation}
\end{numba}
\begin{numba}
Consider
the continuous seminorm
$p_z:=\|.\|_{\wb{B}_2(z),\|.\|_{op}}\circ D+\|.\|_{\wb{B}_2(z),\|.\|_\infty}$
on $C^\infty(B_3(z),\R^n)$, where
$D\colon C^\infty(B_3(z),\R^n)\to C^\infty(B_3(z),\R^{n\times n})$, $f\mto f'$;
thus
\[
p_z(f)=\sup_{x\in \wb{B}_2(z)}(\|f'(x)\|_{op}+\|f(x)\|_\infty)\quad\mbox{for
$f\in C^\infty(B_3(z), \R^n)$.}
\]
Fix $L\in \,]0,1[$.
Then
\[
Q_z
:=\{\gamma \in L^1([0,1],C^\infty(B_3(0),\R^n))
\colon \|\gamma\|_{L^1,p_z}<L\}
\]
is an open $0$-neighbourhood in $L^1([0,1],C^\infty(B_3(z),\R^n))$.
We define a map
$\Psi_z \colon Q_z\times B_1(z)\times
C([0,1], B_2(z))\to C([0,1],B_2(z))$ via
\[
\Psi_z([\gamma],x,\kappa)(t):=x+\int_0^t \gamma(s)(\kappa(s))\,ds
\]
for $[\gamma]\in Q_z$ with $\gamma\in \cL^1([0,1],C^\infty(B_3(z),\R^n))$,
$x\in B_1(z)$, $\kappa\in C([0,1],B_2(z))$ and $t\in [0,1]$.
\end{numba}
\begin{la}\label{hfwaydp}
The map
$\Psi_z \colon Q_z\times B_1(z)\times
C([0,1],B_2(z))\to C([0,1],B_2(z))$
is smooth and defines a uniform family of contractions
in the final variable, in the sense that
\[
\Lip(\Psi_z([\gamma],x,\sbull))\leq L
\]
for all $[\gamma]\in Q_z$ and all $x\in B_1(z)$.
\end{la}
\begin{proof}
For $x\in \R^n$, let $c_x\colon [0,1]\to \R^n$
be the constant function $t\mto x$. 
The map $\R^n\to C([0,1],\R^n)$, $x\mto c_x$
is continuous linear and hence smooth.
Moreover, the operator
\[
J\colon L^1([0,1],\R^n)\to C([0,1],\R^n)
\]
determined by
$J([f])(t):=\int_0^t f(s)\, ds$
is continuous and linear and hence smooth.
Consider the mapping
\[
\Phi_z\colon
L^1([0,1],C^\infty(B_3(z),\R^m))\times C([0,1],B_2(0))\to L^1([0,1],\R^n),
\]
$\Phi_z([\gamma],\zeta):=[(\wh{\gamma})_*(\zeta)]$,
which is smooth
by Lemma~\ref{phimglatdp}).
Now the formula
\[
\Psi_z([\gamma],x,\kappa)=c_x+J(\Phi_z([\gamma],\kappa))
\]
shows that $\Psi_z$ is smooth.
Given $\eta,\eta_1\in C([0,1],B_2(0))$, we deduce from (\ref{yesyepdp})
that
\begin{eqnarray*}
\|\Psi_z([\gamma],x,\eta)-\Psi_z([\gamma],x,\eta_1\|_\infty &=&
\sup_{t\in [0,1]}
\left\|\int_0^t \gamma(s)(\eta(s))-\gamma(s)(\eta_1(s))\,ds\right\|_\infty\\
&\leq &
\sup_{t\in[0,1]}
\int_0^t \underbrace{\|\gamma(s)(\eta(s))-
\gamma(s)(\eta_1(s))\|_\infty}_{\leq p(\gamma(t))\|\eta-\eta_1\|_{L^\infty}} \,ds\\
&\leq &
\|\gamma\|_{\cL^1,p_z} \|\eta-\eta_1\|_{L^\infty}
\leq L\|\eta-\eta_1\|_{L^\infty}
\end{eqnarray*}
with $p_z$ as in \ref{moresetdp}.
It only remains to observe that $\Psi_z([\gamma],x,\eta)(t)\in B_2(z)$
always since
$\|\int_0^t \gamma(s)(\eta(s))\,ds\|_\infty
\leq \int_0^t\|\gamma(s)(\eta(s))\|_\infty\,ds\leq
\int_0^t \|\gamma(s)\|_{\wb{B}_2(z),\|.\|_\infty}\,ds$
$\leq
\int_0^t \|\gamma(s)\|_{\wb{B}_2(z),\|.\|_\infty}\,ds\leq
\int_0^t p_z(\gamma(s))\,ds\leq
\|\gamma\|_{\cL^1,p_z}\leq 1$.
\end{proof}
\begin{numba}\label{nowlocsol}
If $([\gamma],x)\in Q_z\times B_1(z)$,
we have $x\in \wb{B}_r(z)$ for some $r\in\;]0,1[$
and $\Psi_z([\gamma],x,\sbull)$ restricts to a contraction
\[
C([0,1], \wb{B}_{1+r}(z))\to C([0,1],\wb{B}_{1+r}(z))
\]
of the complete metric space $C([0,1],\wb{B}_{1+r}(z))=\{\zeta\in C([0,1],\R^n)\colon
\|\zeta\|_\infty\leq 1+r\}$.
The latter has a unique fixed point $\zeta_{[\gamma],x}$
by Banach's Contraction Principle,
which then also is the unique fixed point of the contraction
$\Psi_z([\gamma],x,\sbull)$ of $C([0,1],B_2(0))$.
Thus
\begin{equation}\label{willrecodp}
\Psi_z([\gamma],x,\zeta_{[\gamma],x})=\zeta_{[\gamma],x}.
\end{equation}
Since $\Psi_z$ is smooth, Lemma~\ref{PARFIX}
shows that also the map
\[
Q_z\times B_1(0)\to C([0,1],B_2(0)),\quad ([\gamma],x)\mto\zeta_{[\gamma],x}
\]
is smooth.
Define $F_z([\gamma])(t)(x):=\zeta_{[\gamma],x}(t)$
for $[\gamma]\in Q_z$, $x\in B_1(0)$ and $t\in [0,1]$.
Using the exponential laws from \cite{AaS},
we deduce:
\begin{itemize}
\item[(a)]
$F_z([\gamma])(t)\in C^\infty(B_1(z),\R^n)$
for all $[\gamma]\in Q_z$ and $t\in [0,1]$;
\item[(b)] $F_z([\gamma])\in C([0,1],C^\infty(B_1(z),\R^n))$ for all $[\gamma]\in Q_z$;
\item[(c)]
$F_z\colon Q_z \to C([0,1],C^\infty(B_1(z),\R^n))$
is smooth.
\end{itemize}
\end{numba}
\begin{numba}
As a consequence, also the map
\[
E_z\colon Q_z\to C([0,1],C^\infty(B_1(z),\R^n),\quad [\gamma]\mto F_z([\gamma])- I
\]
is smooth, where $I\colon [0,1]\to C^\infty(B_1(z),\R^n)$ is the constant function
$t\mto\id_{B_1(z)}$.
If $\gamma=0$, then $\zeta_{[\gamma],x}$ is the fixed point
of the map $\Psi_z([\gamma],x,\sbull)$
determined by
\[
\Psi_z([0],x,\kappa)(t)=x.
\]
Hence $\zeta_{0,x}=c_x$.
As a consequence,
\[
E_z(0)(t)=0
\]
and thus $E_z(0)=0$.
Since $E_z$ is continuous,
there is an open $0$-neighbourhood $P_z\sub Q_z$
such that
\[
E_z(P_z)\sub W_z
\]
(where $W_z$ is as in (\ref{inOmc})).
\end{numba}
\begin{numba}
Consider
$D\colon C^\infty(B_4(z),\R^n)\to C^\infty(B_4(z),\R^{n\times n})$, $f\mto f'$
and the continuous seminorm
$q_z:=\|.\|_{\wb{B}_3(z),\|.\|_{op}}\circ D+\|.\|_{\wb{B}_3(0),\|.\|_\infty}$
on $C^\infty(B_4(z),\R^n)$ for $z\in \Z^n$;
thus
\[
q_z(f)=\sup_{x\in \wb{B}_3(z)}(\|f'(x)\|_{op}+\|f(x)\|_\infty)\quad\mbox{for
$f\in C^\infty(B_4(z), \R^n)$.}
\]
Then
\[
S_z
:=\{\gamma \in L^1([0,1],C^\infty(B_4(0),\R^n))
\colon \|\gamma\|_{L^1,q_z}<L\}
\]
is an open $0$-neighbourhood in $L^1([0,1],C^\infty(B_4(z),\R^n))$.
We define a map
$\Theta_z \colon S_z\times B_2(z)\times
C([0,1], B_3(z))\to C([0,1],B_3(z))$ via
\[
\Theta_z([\gamma],x,\kappa)(t):=x+\int_0^t \gamma(s)(\kappa(s))\,ds
\]
for $[\gamma]\in S_z$ with $\gamma\in \cL^1([0,1],C^\infty(B_4(z),\R^n))$,
$x\in B_2(z)$, $\kappa\in C([0,1],B_3(z))$ and $t\in [0,1]$.
\end{numba}
The following lemma can be shown like Lemma~\ref{hfwaydp}.
\begin{la}\label{hfwaydpdp}
For all $z\in \Z^n$, $[\gamma]\in S_z$ and $x\in B_2(z)$, the map
\[
\Theta_z([\gamma],x,\sbull)\colon
C([0,1],B_3(z))\to C([0,1],B_3(z)),\quad\kappa\mto \Theta_z([\gamma],x,\kappa)
\]
is a contraction, with $\Lip(\Theta_z([\gamma],x,\sbull))\leq L$.\,\Punkt
\end{la}
\begin{numba}
Now consider the open $0$-neighbourhood
\[
\cU:=
R_3^{-1}\left(\bigoplus_{z\in \Z^n} P_z\right)\cap R_4^{-1}\left(\bigoplus_{z\in \Z^n}S_z\right)
\]
in $L^1([0,1],C^\infty_c(\R^n,\R^n))$.
Since $E_z$ is smooth for each $z\in \Z^n$ and $E_z(0)=0$, also the map
\[
\oplus_{z\in \Z^n} E_z\colon \bigoplus_{z\in \Z^n} Q_z\to \bigoplus_{z\in\Z^n}
C([0,1],C^\infty(B_1(z),\R^n)),\quad (\gamma_z)_{z\in \Z^n}\mto
(E_z(\gamma_z))_{z\in \Z^n}
\]
is smooth (see \cite{MEA}).
We claim that
\begin{equation}\label{dffclaim}
(\bigoplus_{z\in \Z^n}E_z)(R_3([\gamma]))\in R_1(C([0,1],\Omega))
\end{equation}
for each $[\gamma]\in \cU$.
If this is true, then $(\bigoplus_{z\in \Z^n}E_z)\circ R_3|_{\cU}$
is smooth also as a map $\cU\to \im(R_1)$ (since $\im(R_1)$
is a closed vector subspace and thus \cite[Lemma~10.1]{BGN} applies).
As a consequence, also the map
\[
E:=R_1^{-1}\circ (\bigoplus_{z\in \Z^n}E_z)\circ R_3|_{\cU}\colon \cU\to
C([0,1],C^\infty_c(\R^n,\R^n))
\]
is smooth. Let $[\gamma]\in \cU$ with $\gamma\in \cL^1([0,1], C^\infty_c(\R,\R^n))$.
If $x\in \R^n$, we can find $w\in \Z^n$ such that $x\in B_1(w)$.
Let us write $h\colon [0,1]\to C^\infty(B_1(w),\R^N)$ for
the $w$-component of
\[
R_1(E([\gamma]))=(\bigoplus_{z\in \Z^n}E_z)(R_3([\gamma]).
\]
Then $h=E_w([s\mto \gamma(s)|_{B_3(w)}])$
and
\[
h(t)=E([\gamma])(t)|_{B_1(w)},
\]
entailing that $E([\gamma])(t)(x)=h(t)(x)=E_w([s\mto \gamma(s)|_{B^3(w)}])(t)$
for all $t\in [0,1]$.
Thus $[0,1]\to E([\gamma])(t)(x)$ is a Carath\'{e}odory solution to
\[
y'(t)=\gamma(t)(y(t)),\quad y(0)=x
\]
and thus~(\ref{checkthis}) is satisfied by the continuous function
$\eta:=E([\gamma])\colon [0,1]\to \Omega\sub C^\infty_c(\R^n,\R^n)$.
As a consequence,
$\eta=\Evol_\Omega^r([\gamma])$ is the right evolution
of $[\gamma]$ (see \ref{thereduyea}).
Since $L^1$ has the subdivision property by Lemma~\ref{setsubdiv},
we deduce with (d)$\impl$(c) in Proposition~\ref{lctoglb}
that $\Evol^r$ exists on all of $L^1([0,1],C^\infty_c(\R^n,\R^n))$.
Since $\Evol^r|_{\cU}=E$ is smooth, (b)$\impl$(a) in Proposition~\ref{lctoglb}
now shows
that the Lie group $(\Omega,*)$ (and hence also $\Diff_c(\R^n)$)
is $L^1$-regular.
\end{numba}
\begin{numba}
To complete the proof of $L^1$-regularity for $\Diff_c(\R^n)$,
it only remains to prove~(\ref{dffclaim}).
To this end, let $[\gamma]\in \cU$ with $\gamma\in \cL^1([0,1],C^\infty_c(\R^n,\R^n))$.
For $z\in \Z^n$, define $\gamma_z\in \cL^1([0,1],C^\infty(B_3(z),\R^n))$
via $\gamma_z(t):=\gamma(t)|_{B_3(z)}$.
Define $\tilde{\gamma}_z\in \cL^1([0,1],C^\infty(B_4(0),\R^n))$
via $\tilde{\gamma}_z(t):=\gamma(t)|_{B_4(z)}$.
Let $z,w\in \Z^n$ and $x\in B_1(z)\cap B_1(w)$.
Let us write
$\zeta_{[\gamma_z],x}^z\colon [0,1]\to B_2(z)$ for the unique
fixed point of $\Psi_z([\gamma_z],x,\sbull)$
and $\zeta_{[\gamma_w],x}^w\colon [0,1]\to B_2(w)$ for the unique
fixed point of $\Psi_w([\gamma_w],x,\sbull)$.
Thus $\zeta^z_{[\gamma_z],x}$ is
a Carath\'{e}odory
solution to
\begin{equation}\label{Car1}
y'(t)=x+\int_0^t \gamma(t)|_{B_3(z)}(y(t))\,dt,\quad y(0)=x.
\end{equation}
and $\zeta^w_{[\gamma_w],x}$ is a
Carath\'{e}odory
solution to
\begin{equation}\label{Car2}
y'(t)=x+\int_0^t \gamma(t)|_{B_3(w)}(y(t))\,dt,\quad y(0)=x.
\end{equation}
Since $B_2(z)$ and $B_2(w)$ are subsets of $B_3(z)$,
both $\zeta^z_{[\gamma_z],x}$ and $\zeta^w_{[\gamma_w],x}$ can be considered
as elements of $C([0,1],B_3(z))$.
Since both $B_3(z)$ and $B_3(w)$ are subsets of $B_4(z)$,
we deduce from (\ref{Car1}) and (\ref{Car2}) that both
$\zeta^z_{[\gamma_z],x}$ and $\zeta^w_{[\gamma_w],x}$ are Carath\'{e}odory
solutions to
\[
y'(t)=x+\int_0^t \gamma(t)|_{B_4(w)}(y(t))\,dt,\quad y(0)=x
\]
and hence fixed points of $\Theta_z([\tilde{\gamma}_z],x,\sbull)$.
As the contraction
\[
\Theta_z([\tilde{\gamma}_z],x,\sbull)\colon C([0,1],B_3(z))\to
C([0,1],B_3(z))
\]
has at most one fixed point, we deduce that
$\zeta^z_{[\gamma_z],x}=\zeta^w_{[\gamma_w],x}$.
Thus
\[
E_z([\gamma_z])(t)(x)=\zeta^z_{[\gamma_z],x}(t)-x=\zeta^w_{[\gamma_w],x}(t)-x=E_w([\gamma_w])(t)(x)
\]
for each $t\in [0,1]$. Hence
\[
(E_z([\gamma_z])(t))_{z\in \Z^n}
\in \im(\rho_1)
\]
for each $t\in [0,1]$, using~(\ref{imzw}).
As a consequence,
\[
(E_z([\gamma_z]))_{z\in \Z^n}=(\oplus_{z\in \Z^n}E_z)(R_3([\gamma]))\in
\im(R_1).
\]
Since $(E_z([\gamma_z])(t))_{z\in \Z^n}\in W_z$,
we deduce that $\rho_1^{-1}((E_z([\gamma_z])(t))_{z\in \Z^n})\in \Omega$
for each $t$ (see (\ref{inOmc})), whence
$(E_z([\gamma_z]))_{z\in \Z^n}=(\oplus_{z\in \Z^n}E_z)(R_3([\gamma]))\in
R_1(\Omega)$. Thus~(\ref{dffclaim}) holds and the proof for $L^1$-regularity of
$\Diff_c(\R^n)$ is complete.
\end{numba}
\section{Local {\boldmath$L^1$}-Lipschitz condition and uniqueness for
initial value problems}\label{L1lip}
It is well known from the classical Picard-Lindel\"{o}f Theorem
that solutions to initial value problems in normed
spaces are unique if the right hand side of the differential equation
is continuous and satisfies a local Lipschitz condition.
As a preparation for the discussion of $L^1$-regularity for $\Diff_c(M)$,
we now describe weaker conditions ensuring uniqueness.\footnote{The alert reader
may notice that the condition
was already satisfied in the preceding section.
However, making it explicit would not much shorten the proof, as we
used Banach's Contraction Theorem anyway, and could exploit its uniqueness assertion.}
For differential equations on subsets
of finite-dimensional spaces, similar
conditions have been used e.g.\ in
\cite[Appendix C.3, Theorem 54]{Son}.
\begin{defn}\label{defl1lip}
Let $(E,\|.\|)$ be a normed space, $U\sub E$ be a subset
and $a<b$ be real numbers. We say that
a function $f\colon [a,b]\times U\to E$
satisfies a (global) \emph{$L^1$-Lipschitz condition}
if there exists a measurable function $g\colon [a,b]\to
[0,\infty]$ with $L:=\int_a^b g(t)\,d\lambda_1(t)<\infty$
such that
\[
\Lip(f(t,\sbull))\leq g(t)\quad \mbox{for all $t\in [a,b]$.}
\]
\end{defn}
\begin{rem}
\begin{itemize}
\item[(a)]
Here $\Lip(f(t,\sbull))\in [0,\infty]$ means the infimum of all Lipschitz constants
for the mapping $f_t:=f(t,\sbull)\colon U\to E$, $y\mto f(t,y)$.
\item[(b)]
If the function
$h\colon [a,b]\to [0,\infty]$, $t\mto \Lip(f(t,\sbull))$ is measurable,
then $g$ as required in Definition~\ref{defl1lip} exists if and only if $g:=h$
can be chosen there (i.e., if and only if $h$
is integrable). In all of our applications,
$h$ is measurable, but we do not need this requirement
to formulate Definition~\ref{defl1lip}.
\end{itemize}
\end{rem}
\begin{defn}\label{deflcl1lip}
Let $M$ be a $C^1$-manifold modelled on a normed space $(E,\|.\|)$,
$J\sub \R$ be a non-degenerate interval
and $f\colon J \times M\to TM$ be a function with $f(t,p)\in T_p(M)$
for all $(t,p)\in J \times M$.
We say that $f$ \emph{satisfies a local $L^1$-Lipschitz condition}
if for all $t_0\in J$ and $p\in M$, there exists a chart
\[
\kappa\colon U_\kappa\to V_\kappa\sub E
\]
of $M$ with $p\in U_\kappa$ and a relatively open
subinterval $[a,b]\sub J$ which is a neighbourhood of $t_0$ in~$J$
such that the map
\begin{equation}\label{fkapp}
f_\kappa\colon [a,b]\times V_\kappa\to E,\quad (t,y)\mto
d\kappa (f(t,\kappa^{-1}(y)))
\end{equation}
satisfies an $L^1$-Lipschitz condition.
\end{defn}
\begin{rem}\label{onethnoll}
If $(E,\|.\|)$ is a normed space and
$g\colon W\to E$ a $C^1$-map on an open subset $W\sub E$,
then each $x\in W$ has an open neighbourhood $W_0\sub W$
such that $\sup_{y\in W_0}\|dg(y,\sbull)\|_{op}<\infty$
(e.g., any $W_0$ such that $dg(W_0\times B^E_r(0))\sub B^E_1(0)$
for some $r>0$, which exists by continuity of $dg\colon W\times E\to E$).
Choosing $g$ as the transition map (change of charts) from one chart
to another, we deduce that
$f_\kappa$ will actually satisfy an $L^1$-Lipschitz condition
for \emph{each} chart around $p$ in the situation of Definition~\ref{deflcl1lip},
for suitable $[a,b]$.
\end{rem}
\begin{prop}[Uniqueness of solutions]\label{propunq}
Let $a<b$ be real numbers,
$M$ be a $C^1$-manifold modelled on a normed space $(E,\|.\|)$
and
\[
f\colon [a,b] \times M\to TM
\]
be a function with $f(t,p)\in T_p(M)$
for all $(t,p)\in [a,b] \times M$, which satisfies a local
$L^1$-Lipschitz condition.
If $\gamma\colon [a,b]\to M$ and $\eta\colon [a,b]\to M$
are absolutely continuous curves which are Carath\'{e}odory solutions to
\[
y'=f(t,y)
\]
and satisfy $\gamma(t_0)=\eta(t_0)$ for some $t_0\in [a,b]$,
then $\gamma=\eta$.
\end{prop}
\begin{proof}
We show that $\gamma|_{[t_0,b]}=\eta|_{[t_0,b]}$;
the proof that $\gamma|_{[a,t_0]}=\eta_{[a,t_0]}$ is similar.
We may therefore assume now that $t_0=a$.
The set
\[
A:=\{t\in [a,b]\colon \gamma(t)=\eta(t)\}
\]
is closed in $[a,b]$ as $\gamma$ and $\eta$ are continuous
and $M$ is Hausdorff. Moreover, $A$ is non-empty as $t_0\in A$.
If we can show that $A$ is also open in $[a,b]$,
then $A=[a,b]$ will follow from the connectedness of $[a,b]$
(completing the proof).
Let $\tau\in A$. By Definition~\ref{deflcl1lip},
there exist $\alpha<\beta$ in $[a,b]$ such that $[\alpha,\beta]$ is a neighbourhood
of $\tau$ in $[a,b]$ and a chart $\kappa\colon U_\tau\to V_\tau\sub E$
such that
\[
\gamma(\tau)=\eta(\tau)\in U_\kappa
\]
and the map $f_\kappa\colon [\alpha,\beta]\times V_\kappa\to E$
(analogous to (\ref{fkapp})) satisfies an $L^1$-Lipschitz condition.
Let $g\colon [\alpha,\beta]\to[0,\infty]$ be an integrable
function with $\Lip(f_\kappa(t,\sbull))\leq g(t)$ for all $t\in [\alpha,\beta]$.
Since
\[
\lim_{r\to 0} \int_{[\alpha,\beta]\cap [\tau-r,\tau+r]}g(s)\,ds=0,
\]
after shrinking the neighbourhood
$[\alpha,\beta]$ of $\tau$, we may assume that
\[
L :=\int_\alpha^\beta g(s)\,ds<1.
\]
Abbreviate
$\|\gamma-\eta\|_{[\alpha,\beta]}:=\sup\{\|\gamma(t)-\eta(t)\|\colon t\in [\alpha,\beta]\}$.
For each $t\in [\tau,\beta]$, we obtain
\begin{eqnarray*}
\|\gamma(t)-\eta(t)\| &=& \left\|\int_\tau^t (f(s,\gamma(s))-f(s,\eta(s)))\,ds\right\|\\
&\leq& \int_\tau^s\underbrace{\|f(s,\gamma(s))-f(s,\eta(s))\|}_{\leq\Lip(f(s,\sbull))\|\gamma(s)-\eta(s)\|\leq g(s)\|\gamma(s)-\eta(s)\|}\, ds\\
&\leq & \int_\tau^t g(s)\|\gamma(s)-\eta(s)\|\,ds
\leq L\|\gamma-\eta\|_{[\alpha,\beta]}.
\end{eqnarray*}
A similar argument shows that
$\|\gamma(t)-\eta(t)\|\leq L\|\gamma-\eta\|_{[\alpha,\beta]}$
also for $t\in [\alpha,\tau]$. Hence
$\|\gamma(t)-\eta(t)\|\leq L\|\gamma-\eta\|_{[\alpha,\beta]}$ for all $t\in [\alpha,\beta]$
and thus
\[
\|\gamma-\eta\|_{[\alpha,\beta]}\leq L \|\gamma-\eta\|_{[\alpha,\beta]},
\]
which is impossible unless $\|\gamma-\eta\|_{[\alpha,\beta]}=0$ and thus
$\gamma|_{[\alpha,\beta]}=\eta|_{[\alpha,\beta]}$. Thus $[\alpha,\beta]\sub A$,
entailing that $A$ is a neighbourhood of $\tau$ and hence open in $[a,b]$
(as $\tau\in A$ was arbitrary), which completes the proof.
\end{proof}
\begin{defn}\label{defgloflo}
Let $M$ be a $C^1$-manifold modelled on a normed space $(E,\|.\|)$. Let
$t_0,T$ be real numbers
and $f\colon [t_0,t_0+T] \times M\to TM$ be a function with $f(t,p)\in T_p(M)$
for all $(t,p)\in [t_0,t_0+T] \times M$, which satisfies a local
$L^1$-Lipschitz condition.
If the initial value problem
\[
y'(t)=f(t,y'(t)),\quad y(t_0)=p
\]
has a $($necessarily unique$)$ solution
$\gamma_p\colon [t_0,t_0+T]\to M$ for each $p\in M$,
then we say that \emph{$f$ admits a global flow
for initial time~$t_0$}
and write
\[
\Phi^f_{t,t_0}(p):=\gamma_p(t)
\]
for $t\in [t_0,t_0+T]$ and $p\in M$.
In this way, for each $t\in [t_0,t_0+T]$ we obtain a mapping
$\Phi_{t,t_0}^f \colon M\to M$.
\end{defn}
\section{{\boldmath$L^1$}-regularity of {\boldmath$\Diff_c(M)$}
and {\boldmath$\Diff_K(M)$}}\label{secDiM}
%
Throughout this section, let $M$ be a
paracompact finite-dimensional smooth manifold
and $K\sub M$ be a compact subset
(starting with \ref{nowsigm}, we shall assume that~$M$ is $\sigma$-compact).
Our goal is to see that
$\Diff_c(M)$ and $\Diff_K(M)$
are $L^1$-regular.
For general information on the Lie group structure of $\Diff_c(M)$,
the reader is referred to \cite{Mic}, \cite{DIF}, and \cite{GaN}.
The Lie group $\Diff_c(M)$ can also be regarded
as a special case of the diffeomorphism groups of orbifolds
discussed in \cite{Shm}.
\begin{numba}\label{bscsdf}
Let $g$ be a smooth Riemannian metric on~$M$
and $\exp\colon \cD\to M$ be the Riemannian
exponential function, defined on its maximal domain
$\cD$ which is an open neighbourhood
of the zero-section $M\to TM$, $p\mto 0_p\in T_pM$
in $TM$. Let $\pi_{TM}\colon TM\to M$ be the bundle
projection.
For some open neighbourhood $\cW\sub \cD$ of the zero-section, the map
\begin{equation}\label{intloca}
A=(A_1,A_2) \colon \cW\to M\times M,\quad v\mto (\pi_{TM}(v),\exp(v))
\end{equation}
has open image and is a diffeomorphism onto its image;
it is called a \emph{local addition}.
In particular, for each $p\in M$ the set
$\cW_p:=\cV\cap T_pM$ is open in $T_pM$ and the
function $\exp_p:=\exp|_{T_pM}$ restricts to a $C^\infty$-diffeomorphism
from $\cW_p$ onto the open neighbourhood
$\exp_p(\cW_p)$ of~$p$ in~$M$.
Let us write $\cX_c(M)$ for the space of compactly supported smooth
vector fields on~$M$.
There is an open $0$-neighbourhood $\cV\sub \cX_c(M)$
with $X(M)\sub \cW$ for all $X\in \cV$
such that
\[
\cU:=\{\exp \circ X\colon X\in \cV\}\sub \Diff_c(M)
\]
and the map
\[
\Phi\colon \cU\to\cV,\quad
\Phi(\phi)(p):=(\exp_p|_{\cW_p})^{-1}(\phi(p))
\]
is a chart for $\Diff_c(M)$, with inverse
\[
\Phi^{-1}\colon \cV\to \cU,\quad X\mto A_2\circ X=\exp\circ X.
\]
Since $\exp(0_p)=p$, we have $A_2(0_p)=p$ and $A^{-1}(p,p)=0_p$,
for each $p\in M$, entailing that
\[
\Phi(\cU\cap \Diff_K(M))=\cV\cap \cX_K(M).
\]
Hence $\Diff_K(M)$ is a submanifold of $\Diff_c(M)$
modelled on the Fr\'{e}chet space $\cX_K(M)$ of smooth
vector fields supported in~$K$.
\end{numba}
\begin{numba}
Let $\cC$ be the set of connected components $C$ of~$M$.
The above set $\cV$ can be chosen as $\cV=\bigoplus_{C\in \cC}\cV_C$
with open $0$-neighbourhoods $\cV_C\sub \cX_c(C)$,
making it clear that the weak direct product
\[
\bigoplus_{C\in \cC}\Diff_c(C)
\]
can be considered as an open subgroup
of $\Diff_c(M)$.
By
Proposition~\ref{mrgsum}, the weak direct product (and hence also $\Diff_c(M)$)
will be $L^1$-regular if we can show that $\Diff_c(C)$ is
$L^1$-regular for each connected component
$C\sub M$. Moreover, $\Diff_K(M)$ has $\bigoplus_{C\in \cC}\Diff_{K\cap C}(C)$
as an open subgroup and hence is $L^1$-regular if each $\Diff_{K\cap C}(C)$
is so.\\[2.3mm]
\emph{We may} (\emph{and will}) \emph{therefore assume throughout
the rest of this section that $M$ is a connected
and $\sigma$-compact finite-dimensional smooth manifold}.\\[2.3mm]
Thus $\cX_c(M)$ is a strict (LF)-space;
if $(K_j)_{j\in \N}$ is a sequence of compact subsets $K_j \sub M$ such that
$M=\bigcup_{j\in \N}K_j$ and $K_j\sub (K_{j+1})^0$ (the interior)
for each $j\in \N$, then
\[
\cX_c(M)=\dl_K \,\cX_K(M)=\dl_{j\in \N} \, \cX_{K_j}(M).
\]
\end{numba}
\begin{numba}\label{nowsigm}
We shall identify the Lie algebra of $G:=\Diff_c(M)$ with
the locally convex space $\cX_c(M)$ (with the negative of the traditional
Lie bracket of vector fields) by means of the isomorphism
\[
d\Phi|_{T_{\id_M}G}\colon L(G)\to \cX_c(M).
\]
Likewise, we identify the Lie algebra of $G_K:=\Diff_K(M)$
with $\cX_K(M)$.
We shall see later that $\Diff_c(M)$ is
$L^1$-regular, with smooth right evolution
\[
\Evol^r\colon L^1([0,1],\cX_c(M))\to AC([0,1],\Diff_c(M)),
\]
and
observe in Remark~\ref{fornowsig}
that
$\Evol^r([\lambda\circ \gamma])(t)\in \Diff_K(M)$
for $[\gamma]$ in an open $0$-neighbourhood $\Omega_K$
in $L^1([0,1],\cX_K(M))$, where $\lambda\colon \cX_K(M)\to \cX_c(M)$
is the inclusion map.
Then
\begin{equation}\label{trivstu}
\Evol^r([\gamma])
\in AC([0,1],\Diff_K(M)),
\end{equation}
as the conclusion of Lemma~\ref{intsubmf}
is available by Remark~\ref{L1inpart}.
Moreover,
the map
\[
h\colon \Omega_K\to AC([0,1],\Diff_K(M)),\quad [\gamma]\mto \Evol^r([\gamma])
\]
is smooth, as $AC([0,1],\Diff_K(M))$ is a submanifold
of $AC([0,1], \Diff_c(M))$ since $\Diff_K(M)$ is a submanifold
of $\Diff_c(M)$ (using that the conclusion of Lemma~\ref{ACHsubACG}
applies by Remark~\ref{L1inpart}).
For $[\gamma]\in L^1([0,1],\cX_K(M))$,
consider $\eta:= \Evol^r([\lambda\circ \gamma])$ as an element of
$AC([0,1],\Diff_K(M))$, as in \ref{trivstu}.
If $j \colon G_K\to G$
is the inclusion map, then we can identify $L(j)$
with~$\lambda$,
and thus
\[
L^1([0,1],\lambda)(\delta^r_{G_K}(\eta))=\delta^r_G(j\circ \eta)=[\lambda\circ
\gamma]=L^1([0,1],\lambda)([\gamma])
\]
(where we wrote the Lie group $G_K$
as an index for clarity).
Hence $\delta^r_{G_K}(\eta)=[\gamma]$.
We deduce that $h$ is the right evolution map
$\Evol^r_{G_K}$ for $G_K=\Diff_K(M)$ on $\Omega_K$.
Hence $\Diff_K(M)$ is $L^1$-regular, by
Proposition~\ref{lctoglb}.
%
%
Thus, it only remains to show that $\Diff_c(M)$
is $L^1$-regular, and to show the validity of Remark~\ref{fornowsig}.
\end{numba}
The proof of the following proposition,
which is similar to the familiar case
of a $C^k$-curve $\gamma$
(see, e.g., \cite{DIF};
cf.\ \cite{Shm} and \cite{KaM})
has been relegated to the appendix (Appendix~\ref{appDiff}).
\begin{prop}\label{handsonEv}
Let $\gamma\in \cL^1([0,1],\cX_c(M))$
such that
\[
f:=\wh{\gamma}\colon [0,1]\times M\to TM, \quad
(t,p)\mto
\gamma(t)(p)
\]
satisfies a local $L^1$-Lipschitz condition.
Let $\eta\in AC_{L^1}([0,1], \Diff_c(M))$ with $\eta(0)=\id_M$.
Then $\eta=\Evol^r([\gamma])$ if and only if
$f$ admits global flow for initial time $t_0=0$
and
\[
\eta(t)(p)=\Phi^f_{t,0}(p)
\]
for all $t\in [0,1]$ and $p\in M$,
with notation as in Definition~\emph{\ref{defgloflo}}.\,\Punkt
\end{prop}
Let $n$ be the dimension of~$M$.
\begin{numba}\label{onballQ}
Consider
the continuous seminorm
$p:=\|.\|_{\wb{B}_4(0),\|.\|_{op}}\circ D+\|.\|_{\wb{B}_4(0),\|.\|_\infty}$
on $C^\infty(B_5(0),\R^n)$, where
$D\colon C^\infty(B_5(0),\R^n)\to C^\infty(B_5(0),\R^{n\times n})$, $f\mto f'$;
thus
\[
p(f)=\sup_{x\in \wb{B}_4(0)}(\|f'(x)\|_{op}+\|f(x)\|_\infty)\quad\mbox{for
$f\in C^\infty(B_5(0), \R^n)$.}
\]
Fix $L\in \,]0,1[$.
Then
\[
Q
:=\{\gamma \in L^1([0,1],C^\infty(B_5(0),\R^n))
\colon \|\gamma\|_{L^1,p}<L\}
\]
is an open $0$-neighbourhood in $L^1([0,1],C^\infty(B_5(0),\R^n))$.
We define a map
$\Psi \colon Q\times B_3(0)\times
C([0,1], B_4(0))\to C([0,1],B_4(0))$ via
\[
\Psi([\gamma],x,\kappa)(t):=x+\int_0^t \gamma(s)(\kappa(s))\,ds
\]
for $[\gamma]\in Q$ with $\gamma\in \cL^1([0,1],C^\infty(B_5(0),\R^n))$,
$x\in B_3(0)$, $\kappa\in C([0,1],B_4(0))$ and $t\in [0,1]$
(recalling that the integrand always is an $\cL^1$-function of~$s$
by Lemma~\ref{gvsLone}).
\end{numba}
The following lemma can be proved exactly as Lemma~\ref{hfwaydp}
(increasing all radii by $2$):
\begin{la}\label{hfwaydpM}
The map
$\Psi \colon Q \times B_3(0)\times
C([0,1],B_4(0))\to C([0,1],B_4(0))$
is smooth and defines a uniform family of contractions
in the final variable, in the sense that
\[
\Lip(\Psi([\gamma],x,\sbull))\leq L
\]
for all $[\gamma]\in Q$ and all $x\in B_3(0)$.\,\Punkt
\end{la}
\begin{numba}
There exists a locally finite cover
$(U_j)_{j\in J}$ of $M$ by relatively compact,
open subsets $U_j\sub M$ such that
charts
\[
\kappa_j\colon U_j\to B_5(0)\sub\R^n
\]
with image $B_5(0)$ can be defined on~$U_j$
and the smaller sets $\kappa_j^{-1}(B_1(0))$
form an open cover of~$M$
(this follows from \cite[Chapter II, Theorem 3.3]{Lan},
as there exist diffeomorphisms $B_3(0)\to B_5(0)$
which leave $B_1(0)$ invariant).
Since we assume that $M$ is $\sigma$-compact,
the set~$J$ is countable.
\end{numba}
\begin{numba}
For each $j\in J$ and vector field $X\in \cX_c(M)$, we write
\[
X^{(j)}:=d\kappa_j\circ X\circ \kappa_j^{-1}\colon B_5(0)\to\R^n
\]
for its representative in the local chart~$\kappa_j$.
Then
\[
\cX_c(M)\to C^\infty(B_5(0),\R^n),\quad X\mto X^{(j)}
\]
is a continuous linear map.
Moreover, for each $r\in [1,5]$, the map
\[
\rho_r
\colon \cX_c(M)\to\bigoplus_{j\in J}C^\infty(U_j,\R^n),\quad \rho(X):=(X^{(j)}|_{B_r(0)})_{j\in J}
\]
is a linear topological embedding which admits a continuous linear
right inverse, whence its image is closed and complemented
as a topological vector space (see Lemma~\ref{toviasum}).
Hence also
\begin{eqnarray*}
R_5&:=& L^1([0,1],\rho_5)\colon L^1([0,1],\cX_c(M))\to L^1\left([0,1],\bigoplus_{j\in J}C^\infty(B_5(0),\R^n)\right)\\
& & \hspace*{55mm}\cong \bigoplus_{j\in J}L^1([0,1],C^\infty(B_5(0),\R^n))
\end{eqnarray*}
and
\begin{eqnarray*}
R_1&:=& C([0,1],\rho_1)\colon C([0,1],\cX_c(M))\to
C\left([0,1],\bigoplus_{j\in J}C^\infty(B_1(0),\R^n)\right)\\
& & \hspace*{52mm}\cong \bigoplus_{j\in J}C([0,1],C^\infty(B_1(0),\R^n))
\end{eqnarray*}
are linear topological embeddings with
closed (and complemented) image.
\end{numba}
\begin{numba}
Let $Q$ be as in \ref{onballQ};
then
\[
\cQ:=R_5^{-1}\left(\bigoplus_{j\in J}Q\right)=\{[\gamma]\in L^1([0,1],\cX_c(M)) \colon
(\forall j\in J)\; [t\mto (\gamma(t))^{(j)}]\in Q\}
\]
is an open $0$-neighbourhood in $L^1([0,1],\cX_c(M))$.
\end{numba}
\begin{numba}\label{defgmmn}
Let us write $R_5([\gamma])=([\gamma_j])_{j\in J}$
with $\gamma_j\in\cL^1([0,1],C^\infty(B_5(0),\R^n))$.
\end{numba}
\begin{la}
For each $[\gamma]\in \cQ$, the map
\[
\wh{\gamma}\colon [0,1]\times M\to TM,\quad \wh{\gamma}(t,p):=\gamma(t)(p)
\]
satisfies a local $L^1$-Lipschitz condition.
\end{la}
\begin{proof}
Abbreviate $f:=\wh{\gamma}$.
If $p\in M$, then $p\in \kappa_j^{-1}(B_3(0))$ for some $j\in J$.
Let $U_\kappa:=\kappa_j^{-1}(B_3(0))$ and
$\kappa:=\kappa_j|_{U_\kappa}\colon U_\kappa\to B_2(0)$.
Define
\[
f_\kappa\colon [0,1]\times B_3(0)\to\R^n,\quad f_\kappa(t,y):=
d\kappa(f(t,\kappa^{-1}(y)))
\]
as in Definition~\ref{deflcl1lip}.
For $t\in [0,1]$, we have
$f_\kappa(t,\sbull)=(\rho_j\circ \gamma)(t)|_{B_2(0)}$.
Hence
\begin{eqnarray*}
\Lip (f_\kappa(t,\sbull))&=& \Lip(\rho_j\circ \gamma)(t)|_{B_3(0)})
=\sup_{y\in B_3(0)} \|(\rho_j(\gamma(t)))'(y)\|_{op}\\
& \leq & p(\rho_j(\gamma(t)))=:h(t)
\end{eqnarray*}
with
\[
\int_0^1h(t)\,dt=\|[\rho_j\circ \gamma]\|_{L^1,p}
=\|L^1([0,1],\rho_j)([\gamma])\|_{L^1,p}\leq L<1
\]
as $L^1([0,1],\rho_j)([\gamma])\in Q$ by definition of~$\cQ$.
\end{proof}
\begin{numba}\label{wneF}
As in \ref{nowlocsol},
we see that if
$([\gamma],x)\in Q\times B_3(0)$,
then the contraction
$\Psi([\gamma],x,\sbull)\colon C([0,1],B_4(0))\to C([0,1],B_4(0))$
has a unique fixed point
$\zeta_{[\gamma],x}\in C([0,1],B_4(0))$.
Thus
\begin{equation}\label{willrecodpD}
\Psi([\gamma],x,\zeta_{[\gamma],x})=\zeta_{[\gamma],x}.
\end{equation}
Since $\Psi$ is smooth, Lemma~\ref{PARFIX}
shows that also the map
\[
Q\times B_3(0)\to C([0,1],B_4(0)),\quad ([\gamma],x)\mto\zeta_{[\gamma],x}
\]
is smooth.
Define $F([\gamma])(t)(x):=\zeta_{[\gamma],x}(t)$
for $[\gamma]\in Q$, $x\in B_3(0)$ and $t\in [0,1]$.
Using the exponential laws from \cite{AaS},
we deduce:
\begin{itemize}
\item[(a)]
$F([\gamma])(t)\in C^\infty(B_3(0),\R^n)$
for all $[\gamma]\in Q$ and $t\in [0,1]$;
\item[(b)] $F([\gamma])\in C([0,1],C^\infty(B_3(0),\R^n))$ for all $[\gamma]\in Q$;
\item[(c)]
$F\colon Q \to C([0,1],C^\infty(B_3(0),\R^n))$
is smooth.
\end{itemize}
Define $H([\gamma])(t):=F([\gamma])(t)|_{B_2(0)}$ for $[\gamma]\in Q$ and $t\in [0,1]$.
Thus
\[
H=C([0,1],\varrho)\circ F \colon Q\to C([0,1],C^\infty(B_2(0),\R^n)),
\]
where $\varrho\colon C^\infty(B_3(0),\R^n)\to C^\infty(B_2(0))$,
$\kappa\mto\kappa|_{B_2(0)}$ and
\[
C([0,1],\varrho)\colon C([0,1],C^\infty(B_3(0),\R^n))\to
C([0,1],C^\infty(B_2(0),\R^n)), \;\;\kappa\mto \varrho\circ\kappa
\]
are continuous linear maps (cf.\ \cite{ZOO}). Hence $H$ is smooth, being a composition
of smooth maps.
\end{numba}
\begin{la}\label{wgvAC}
For each $[\gamma]\in Q$, the map $H([\gamma])\colon [0,1]\to C^\infty(B_2(0),\R^n)$
is absolutely continuous.
\end{la}
\begin{proof}
The set
\[
Z:=\{\kappa\in C^\infty(B_3(0),\R^n)\colon \kappa(\wb{B}_2(0)\sub B_4(0)\}
\]
is an open $0$-neighbourhood
in $C^\infty(B_3(0),\R^n)$, whence $C([0,1],Z)$
is an open $0$-neighbourhood in $C([0,1],C^\infty(B_3(0),\R^n))$.
By \cite[Lemma~11.4]{ZOO}, the map
\[
f\colon Z\times C^\infty(B_5(0),\R^n)\to C^\infty(B_2(0),\R^n),\quad
f(\tau,\sigma):=\sigma\circ\tau
\]
is smooth. Moreover, $f(\tau,\sbull)$ is linear for
each $\tau\in Z$, and we have $F([\gamma])\in Z$.
Hence
\[
f\circ (F([\gamma]),\gamma) \in \cL^1([0,1],C^\infty(B_2(0),\R^n)),
\]
by Lemma~\ref{operders}\,(b), entailing that
$\kappa\colon [0,1]\to C^\infty(B_2(0),\R^n)$,
\[
\kappa(t):=
\id_{B_2(0)}+\int_0^t
f(F([\gamma])(s),\gamma(s))\,ds
=
\id_{B_2(0)}+\int_0^t
\gamma(s)\circ F([\gamma])(s)\,ds
\]
is an absolutely continuous function.
Hence, the proof will be complete if we can show that
\[
H([\gamma])=\kappa.
\]
It suffices to show that $H([\gamma])(t)(x)=\kappa(t)(x)$ for each $t\in [0,1]$
and $x\in B_2(0)$. Since $\ve_x\colon C^\infty(B_2(0),\R^n)\to \R^n$,
$\tau\mto \tau(x)$ is a continuous linear map,
we have
\begin{eqnarray*}
\kappa(t)(x)&=&\ve_x(\kappa(t))=x+\int_0^t \ve_x(\gamma(s)\circ F([\gamma])(s))\,ds\\
&=& x+\int_0^t \gamma(s)(F([\gamma])(s)(x))\,ds
=x+\int_0^t \gamma(s)(\zeta_{[\gamma],x})\,ds\\
&=&\Psi([\gamma],x,\zeta_{[\gamma],x}(t))=\zeta_{[\gamma],x}(t)
=H([\gamma])(t)(x),
\end{eqnarray*}
exploiting for the second equality that
weak integrals and continuous linear maps 
can be interchanged. This finishes the proof.
\end{proof}
\begin{numba}
Given $[\gamma]\in \cQ$, we have $[\gamma_j]\in Q_j$
for each $j\in J$ (using the notation from~\ref{defgmmn}).
For $p\in M$, we have $p\in \kappa_j^{-1}(B_3(0))$ for some $j\in J$.
Let $x:=\kappa(p)$.
By~(\ref{willrecodpD})
and definition of $\Psi$, we have $\zeta_{[\gamma_j],x}\in AC_{L^1}([0,1],B_4(0))$
and
\[
\zeta_{[\gamma_j],x}(t)= x+\int_0^t \gamma_j(s)(\zeta_{[\gamma_j],x}(s))\,ds\quad\mbox{for all
$t\in [0,1]$,}
\]
i.e., $\zeta_{[\gamma_j],x}$ is a Carath\'{e}odory solution to
\[
y'(t)=\gamma_j(t)(y(t))=\wh{\gamma_j}(t,y(t)),\quad y(0)=x.
\]
Hence
\begin{equation}\label{etawillwell}
\eta_{[\gamma],p}\colon [0,1]\to M, \quad
t\mto \kappa_j^{-1}(\zeta_{[\gamma_j],x}(t))
\end{equation}
is a Carath\'{e}odory solution to
\[
y'(t)=\gamma(t)(y(t))=\wh{\gamma}(t,y(t)),\quad y(0)=p;
\]
since $\wh{\gamma}$ satisfies a local $L^1$-Lipschitz condition
by the first half of Proposition~\ref{handsonEv},
such solutions are unique by Proposition~\ref{propunq}
and thus $\eta_{[\gamma],p}$ is well defined,
independent of the choice of~$j$.
By the preceding,
$\wh{\gamma}$ admits a global flow for the initial time $t_0=0$,
and the latter is given by
\[
\Phi_{t,0}^{\wh{\gamma}}(p)=\eta_{[\gamma],p}(t)\quad\mbox{for $\, t\in [0,1]$, $p\in M$.}
\]
We shall also write
\begin{equation}\label{defnett}
\eta(t)(p)\quad\mbox{ or } \quad \eta_{[\gamma]}(t)(p)
\end{equation}
for $\Phi^{\wh{\gamma}}_{t,0}(p)$.
\end{numba}
\begin{numba}
Since $\rho_1$ is a topological embedding, after shrinking $\cV$
we may assume that
\[
\cV=\rho_1^{-1}\left(\bigoplus_{j\in J}\cV_j\right)
\]
with suitable open zero-neighbourhoods $\cV_j\sub C^\infty(B_1(0),\R^n)$.
\end{numba}
\begin{numba}
For $j\in J$, we obtain a Riemannian metric $g_j$ on $B_5(0)\sub\R^n$
via
\[
g_j((x,y),(x,z)):=g(T\kappa_j^{-1}(x,y),T\kappa_j^{-1}(x,z))
\]
for $x\in B_5(0)$, $y,z\in \R^n$. This metric makes $\kappa_j$ an isometry.
If $\exp_j\colon \cD_j\to B_5(0)$ is the Riemannian exponential map
for $(B_5(0),g_j)$, then
\[
T\kappa_j^{-1}(\cD_j)\sub \cD
\]
and
\begin{equation}\label{standumr}
\kappa_j^{-1}\circ \exp_j=\exp\circ T\kappa_j^{-1}|_{\cD_j}.
\end{equation}
Let $\pr_1\colon B_5(0)\times\R^n\to B_5(0)$ be the projection onto the first component.
We consider the smooth map
\[
(\pr_1,\exp_j)\colon \cD_j\to B_5(0)\times B_5(0),\quad (x,y)\mto (x,\exp_j(x,y)).
\]
For each $x\in \wb{B}_4(0)\times \{0\}$, the smooth map $(\pr_1,\exp_j(x,y))$
has invertible derivative
\[
(v,w)\mto (v,v+w)
\]
at $(x,0)$ and hence is a local diffeomorphism
around $(x,0)$. 
Since $(\pr_1,\exp_j)$
is injective on the compact set $\wb{B}_4(0)\times\{0\}$,
there is an open subset $O_j\sub \cD_j$ containing $\wb{B}_4(0)\times\{0\}$
such that $(\pr_1,\exp_j)(O_j)$
is open in $B_5(0)\times B_5(0)$ and
$\psi_j:=(\pr_1,\exp_j)|_{O_j}$ is a $C^\infty$-diffeomorphism
onto its image
(see, e.g., \cite[Lemma~4.6]{DaG}).
After shrinking $O_j$ if necessary, we may assume that
\begin{equation}\label{thscompb}
T\kappa_j^{-1}(O_j)\sub\cW
\end{equation}
(with $\cW$ as in (\ref{intloca})).
As the compact set $\Delta:=\{(x,x)\colon x\in \wb{B}_4(0)\}$
is contained in the open set
$(\pr_1,\exp_j)(O_j)$,
there is $s_j\in \;]0,1]$ such that
\[
\Delta+(\{0\}\times B_{s_j}(0))
=\bigcup_{x\in \wb{B}_3(0)} \{x\}\times B_{s_j}(x) \sub (\pr_1,\exp_j)(O_j).
\]
Thus $\psi_j^{-1}$ restricts to a $C^\infty$-diffeomorphism of the form
\[
(\id_{B_4(0)},\theta_j)\colon
\bigcup_{x\in B_4(0)} \{x\} \times B_{s_j}(x)\to
\psi_j^{-1}\left( \bigcup_{x\in B_4(0)} \{x\}\times B_{s_j}(x) \right)
\]
whose range is an open subset of $O_j$.
As a consequence, $\exp(x,\sbull)$
takes the open $0$-neighbourhood $O_{j,x}:=\{y\in \R^n\colon (x,y)\in O_j\}$
diffeomorphically onto an open subset of $B_5(0)$
which contains $B_{s_j}(x)$, and
\begin{equation}\label{litpiec}
(\exp(x,\sbull)|_{O_{j,x}})^{-1}|_{B_{s_j}(x)}=\theta_j(x,\sbull).
\end{equation}
\end{numba}
\begin{numba}
Then also the map
\[
h_j\colon B_2(0)\times B_{s_j}(0)\to B_4(0),\quad (x,z)\mto \theta_j(x,x+z)
\]
is smooth, and
\[
Z_j:=\{\gamma\in C^\infty(B_2(0),\R^n)\colon \gamma(\wb{B}_1(0))\sub B_{s_j}(0)\}
\]
is an open subset of $C^\infty(B_2(0),\R^n)$.
As a consequence, also
\[
\id_{B_2(0)}+Z_j=\{\gamma\in C^\infty(B_2(0),\R^n)\colon (\forall x\in \wb{B}_1(0))\;
\gamma(x)\in B_{s_j}(x)\}
\]
is an open subset of $C^\infty(B_2(0),\R^n)$.
Consider the map
\[
(h_j)_*\colon Z_j\to C^\infty(B_1(0),\R^n)
\]
determined by $(h_j)_*(\gamma)(x):= h_j(x,\gamma(x))$ for $\gamma\in Z_j$
and $x\in B_1(0)$. Then $(h_j)_*$ is smooth,
by \cite[Proposition~4.23]{ZOO}.
Now consider the map
\[
(\theta_j)_*\colon \id_{B_2(0)}+Z_j\to C^\infty(B_1(0),\R^n)
\]
determined by $(\theta_j)_*(\gamma)(x):= \theta_j(x,\gamma(x))$
for $\gamma\in \id_{B_2(0)}+Z_j$ and $x\in B_1(0)$.
Since $(\theta_j)_*(\gamma)=(h_j)_*(\gamma-\id_{B_2(0)})$,
also $(\theta_j)_*$ is smooth.
In particular, $(\theta_j)_*$ is continuous.
Since $(\theta_j)_*(\id_{B_2(0)})=0$,
there exists an open neighbourhood $Y_j\sub \id_{B_2(0)}+Z_j$
of $\id_{B_2(0)}$ such that
\begin{equation}\label{YjVj}
(\theta_j)_*(Y_j)\sub \cV_j.
\end{equation}
Since $H$ is continuous by
(\ref{wneF}) and $H(0)(t)=\id_{B_2(0)}$ for all $t\in [0,1]$,
there is an open $0$-neighbourhood $P_j\sub Q$ such that
\begin{equation}\label{QjYj}
H(P_j)\sub C([0,1],Y_j).
\end{equation}
Now
\[
\cP:=R_5^{-1}\left(\bigoplus_{j\in J}P_j\right)
\]
is an open $0$-neighbourhood in $L^1([0,1],\cX_c(M))$,
and $\cP\sub\cQ$.
\end{numba}
\begin{numba}
Given $[\gamma]\in \cP$, write
$R_5([\gamma])=([\gamma_j])_{j\in J}$ with $\gamma_j\in \cL^1([0,1],C^\infty(B_5(0),\R))$.
Let $\eta(t):=\Phi^{\wh{\gamma}}_{t,0}$
If $p\in M$ such that $p\in \kappa_j^{-1}(B_1(0))$,
let $x:=\kappa_j(p)$. Then
\begin{eqnarray*}
(\theta_j)_*(H([\gamma_j])(t))(x)
&=& \theta_j(x,H([\gamma_j])(t)(x))=
\theta_j(x,\zeta_{[\gamma_j],x}(t))\\
&=& \theta_j (\kappa_j(p), \kappa_j(\eta(t)(p)))
= d\kappa_j (\exp_p|_{\cW_p})^{-1}(\eta(t)(p))
\end{eqnarray*}
by (\ref{standumr}), (\ref{thscompb}),
and (\ref{litpiec}).
Therefore
\[
(\theta_i)_*(H([\gamma_j])(t))(x)=
d(\kappa_j\circ\kappa_i^{-1})((\theta_j)_*(H([\gamma_j])(t))(x))
\]
if $p\in \kappa_j^{-1}(B_1(0))\cap \kappa_i^{-1}(B_1(0))$.
Thus
\[
((\theta_j)_*(H([\gamma_j])(t)))_{j\in J}\in \im(\rho_1)
\]
and hence
\[
((\theta_j)_*(H([\gamma_j])))_{j\in J}\in \im (R_1),
\]
whence we find a unique $\theta=\theta_{[\gamma]}\in C([0,1],\cX_c(M))$
such that
\[
R_1(\theta)
=
((\theta_j)_*(H([\gamma_j])))_{j\in J}.
\]
Then $\rho_1(\theta(t))=
((\theta_j)_*(H([\gamma_j]))(t))_{j\in J}
\in\bigoplus_{j\in J}\cV_j$,
whence
\[
\theta(t)\in \cV
\]
and hence
\[
\exp\circ \theta(t)\in \Diff_c(M).
\]
Define $\eta(t):=\eta_{[\gamma]}(t):=\Phi^{\wh{\gamma}}_{t,0}$.
If $p\in M$ and $j\in J$ with $p\in \kappa_j^{-1}(B_1(0))$,
then
\begin{eqnarray*}
\exp(\theta(t)(p))&=& \kappa_j^{-1}\exp_jT\kappa_j \theta(t)(\kappa_j^{-1}(x)))
= \kappa_j^{-1}\exp_j (\theta_j)_*(H([\gamma_j]))(t)(x)\\
&=& \kappa_j^{-1}\exp_j (\exp_j|_{\cW_j})^{-1}
H([\gamma_j])(t)(x)=\kappa_j^{-1}H([\gamma_j])(t)(x)\\
&=&
\kappa_j^{-1}(\zeta_{[\gamma_j],x}(t))=\Phi^{\wh{\gamma}}_{t,0}(p)
=\eta(t)(p)
\end{eqnarray*}
with $x:=\kappa_j(p)$.
Hence
\[
\eta(t)=\exp\circ\theta(t)=\Phi(\theta(t))\in \Diff_c(M),
\]
enabling us to consider $\eta$ as a map $\eta\colon [0,1]\to \Diff_c(M)$.
\end{numba}
\begin{numba}
Since $\eta=\Phi\circ \theta$,
the map $\eta\colon [0,1]\to \Diff_c(M)$
will be absolutely continuous
if we can show that $\theta\colon [0,1]\to \cX_c(M)$
is absolutely continuous, which is equivalent to
absolute continuity of the map
\[
[0,1]\to\bigoplus_{j\in J} C^\infty(B_1(0),\R^n),\quad
t\mto \rho_1(\theta(t)),
\]
by Lemma~\ref{intsubmf} (since $\im(\rho_1)$ is complemented
in the direct sum).
Note that $J_0:=\{j\in J\colon [\gamma_j]\not=0\}$
is a finite set. If $j\in J\setminus J_0$, then $\zeta_{[\gamma_j],x}(t)=x$
for all $x\in B_2(0)$ and $t\in [0,1]$, entailing that
\[
\eta(t)(p)=p
\]
for all $p\in $
(see (\ref{etawillwell}) and (\ref{defnett}))
Therefore
\[
\theta(t)(p)=(\exp_p|_{\cW_p})^{-1}(\eta(t)(p))=0
\]
for all $j\in J\setminus J_0$, $t\in [0,1]$, and $p\in U_j$.
Hence
\[
\im(\rho_1\circ \theta)\sub\bigoplus_{j\in J_0}C^\infty(B_1(0),\R^n)
\]
and thus $\rho_1\circ \theta$ (and hence $\theta$)
will be absolutely continuous if we can show
that its components
\[
[0,1]\to C^\infty(B_1(0),\R^n),\quad t\mto (\theta_j)_*(H([\gamma_j])(t))
\]
are absolutely continuous for all $j\in J_0$.
Since $(\theta_j)_*$ is smooth, it suffices to know that $H([\gamma_j])$
is absolutely continuous for each $j\in J_0$.
But this is the case by Lemma~\ref{wgvAC}.
Hence $\eta$ is absolutely continuous,
and thus
\[
\eta=\Evol^r([\gamma]),
\]
by Proposition~\ref{handsonEv}.
\end{numba}
\begin{numba}
In view of Proposition~\ref{lctoglb}
and the final assertion of Proposition~\ref{givsordr},
%
%
%
to complete the proof of $L^1$-regularity of $\Diff_c(M)$,
it only remains to check that the map
\[
\cP\to C([0,1],\Diff_c(M)),\quad [\gamma]\mto\Evol^r([\gamma])
\]
is smooth. Since $\Evol^r([\gamma])=\eta_{[\gamma]}=\Phi\circ \theta_{[\gamma]}
=C([0,1],\Phi)(\eta_{[\gamma]})$,
where $C([0,1],\Psi\colon C([0,1],\cV)\to C([0,1],\cU)$
is smooth (being a the inverse of a chart for the Lie group
$C([0,1],\Diff_c(M))$), we need only show that
\[
\cP\to C([0,1],\cX_c(M)),\quad [\gamma]\mto\theta_{[\gamma]}
\]
is smooth. This will hold if the map
\begin{equation}\label{furmpp}
\cP\to \bigoplus_{j\in J}
C([0,1],C([0,1],C^\infty(B_1(0),\R^n)),\quad [\gamma]\mto R_1(\theta_{[\gamma]})
\end{equation}
is smooth. But
\[
R_1(\theta_{[\gamma]})=((\theta_j)_*(H([\gamma_j])))_{j\in J}
=\left(\oplus_{j\in J}((\theta_j)_*\circ H)\right)(R_5([\gamma]),
\]
showing that the map in (\ref{furmpp}) coincides with
\[
\left(\oplus_{j\in J}((\theta_j)_*\circ H)\right)\circ R_5|_{\cP}.
\]
This map is smooth since $R_5$ is continuous linear
and $\oplus_{j\in J}((\theta_j)_*\circ H)$ is smooth (by \cite[Proposition~7.1]{MEA})
as each of the maps $(\theta_j)_*\circ H$
is so.
\end{numba}
\begin{rem}\label{fornowsig}
Given a compact subset $K\sub M$,
let
\[
[\gamma]\in \Omega_K:=\cL^1([0,1],\cX_K(M))\cap\cP.
\]
If $p\in K$, choose $j\in J$ with $p\in \kappa_j^{-1}(B_1(0))$
and set $x:=\kappa_j(p)$. Then
\begin{eqnarray*}
\Evol^r([\gamma])(t)(p)&=& \eta_{[\gamma]}(t)(p)
=\eta_{[\gamma],t}(p)\\
&=& \kappa_j^{-1}(\zeta_{[\gamma_j],x}(t))-\kappa_j^{-1}(x)=p
\end{eqnarray*}
for each $t\in [0,1]$,
using that $\zeta_{[\gamma],x}(t)=x$ for all $t$
since also the constant fucntion $t\mto x$
is a fixed point of
\[
\Psi([\gamma_j],x,\sbull): \kappa\mto x+\int_0^t \gamma_j(s)(\kappa(s)(x))\,ds
\]
(as $\gamma(s)(x)=0$).
Hence $\Evol^r([\gamma])\in C([0,1],\Diff_K(M))$,
as used in~\ref{nowsigm}.
\end{rem}
\section{Consequences of regulated regularity}\label{secregreg}
In this section, we prove
Theorems~H and~I from the introduction,
devoted to the strong Trotter property
and the strong commutator property
(as defined in the introduction).
We begin with a lemma from \cite{GaN}
due to K.-H. Neeb.
\begin{la}\label{rootC1}
Let $E$ be a locally convex space,
$U\sub E$ be an open set, $r>0$,
$\gamma\colon [0,r]\to U$
be a $C^1$-curve and
$f\colon U\to F$ be a $C^2$-map with $df(\gamma(0),.)=0$.
Then
\[
\eta\colon [0,r^2]\to U,\quad t\mto f(\gamma(\sqrt{t}))
\]
is $C^1$ with $\eta'(0)=\frac{1}{2}d^2f(\gamma(0),\gamma'(0),\gamma'(0))$.
\end{la}
\begin{proof}
We may assume that $\gamma(0)=0$ and $f(0)=0$.
Noting that
\[
\gamma(\sqrt{t})=
\sqrt{t}\, \frac{\gamma(\sqrt{t})-\overbrace{\gamma(\sqrt{0})}^{=0}}{\sqrt{t}}
=\sqrt{t}\, \gamma^{[1]}(0,1,\sqrt{t}),
\]
we get for $t\in \, ]0,r^2]$
\begin{eqnarray*}
\eta'(t) &=& \frac{1}{2\sqrt{t}}df(\gamma(\sqrt{t}); \gamma'(\sqrt{t}))
\underbrace{-\frac{1}{2\sqrt{t}}df(0, \gamma'(\sqrt{t}))}_{=0}\\
&=& \frac{1}{2}(df)^{[1]}(0,\gamma'(\sqrt{t});\gamma^{[1]}(0,1,\sqrt{t}),
0;\sqrt{t})
\end{eqnarray*}
The right-hand-side makes sense also for $t=0$
and is continuous on $[0,r^2]$.
So $\eta$ is $C^1$ and
$\eta'(0) =
\frac{1}{2}(df)^{[1]}(0,\gamma'(0);\gamma'(0),0;0)
= \frac{1}{2}d^2f(0,\gamma'(0), \gamma'(0))$.
\end{proof}
The following consequence (see also
\cite[proof of Proposition II.6.3]{Nee})
is relevant for our ends.
\begin{la}\label{givesC1here}
If $G$ is a Lie group, $r>0$
and $\gamma_1,\gamma_2\in C^1([0,r],G)$ with $\gamma_1(0)=\gamma_2(0)=e$,
then $\eta\colon [0,r^2]\to G$,
\[
\eta(t):=\gamma_1(\sqrt{t})\gamma_2(\sqrt{t})
\gamma_1(\sqrt{t})^{-1}\gamma_2(\sqrt{t})^{-1}
\]
is $C^1$, and $\eta'(0)=[\gamma_1'(0),\gamma_2'(0)]$
in the Lie algebra $L(G)=T_e(G)$.
\end{la}
\begin{proof}
It is clear that $\eta|_{]0,r^2]}$ is $C^1$.
Let $U\sub G$, $V\sub U$ be open identity neighbourhoods with
$VVV^{-1}V^{-1}\sub U$. Identify $U$ with an open set in~$E$ using a chart,
such that $e=0$.
The map
\[
f\colon V\times V\to U,\quad f(x,y):=xyx^{-1}y^{-1}
\]
is smooth with $df(0,0,v,w)=0$ and
\begin{equation}\label{hencen}
d^2f(0,0;x,y;x,y)=2[x,y].
\end{equation}
After shrinking~$r$, we may assume that $\eta([0,r^2])\sub U$.
The assertions now follow from Lemma~\ref{rootC1} and (\ref{hencen}).
\end{proof}
{\bf Proof of Theorem~H.}
Assume that $G$ has the strong Trotter property
and let $\gamma,\eta\colon [0,1]\to G$ be
$C^1$-curves such that $\gamma(0)=\eta(0)=e$.
Then
\[
\zeta\colon [0,1]\to G,\quad \zeta(t):=\gamma(\sqrt{t})\eta(\sqrt{t})
(\gamma(\sqrt{t}))^{-1}(\eta(\sqrt{t}))^{-1}
\]
is a $C^1$-curve with $\zeta'(0)=[\gamma'(0),\eta'(0)]$ (see
Lemma~\ref{givesC1here}).
By the strong Trotter property,
\begin{eqnarray*}
\lefteqn{\left(
\gamma\Big(\frac{\sqrt{t}}{n}\Big)
\eta\Big(\frac{\sqrt{t}}{n}\Big)
\gamma\Big(\frac{\sqrt{t}}{n}\Big)^{-1}
\eta\Big(\frac{\sqrt{t}}{n}\Big)^{-1}\right)^{n^2}}\\
&=& \zeta_{n^2}(t/{n^2})\to
\exp_G(t\zeta'(0))=\exp_G(t[\gamma'(0),\eta'(0)])
\end{eqnarray*}
as $n\to\infty$, uniformly in compact subsets
of $[0,\infty[$.\,\vspace{1.3mm}\Punkt

\noindent
{\bf Proof of Theorem I.}
If $\zeta\colon [0,1]\to G$ is a $C^1$-curve with
$\zeta(0)=e$ and $m\in \N$,
we define
$\zeta_n\colon [0,m]\to G$ via
\[
\zeta_n(t):=(\zeta(t/n))^n
\]
for $n\in \N$ such that $n\geq m$.
We claim that
\begin{equation}\label{AllTro}
\zeta_n(t)\to \exp_G(t\zeta'(0)) \quad \mbox{as $n\to \infty$,}
\end{equation}
uniformly in $t\in [0,m]$ (entailing
that $G$ has the strong
Trotter property).\\[2.3mm]
To establish the claim, let $U$ be an open identity neighbourhood
in~$G$. We show that there exists $n_0\geq m$ such that
\begin{equation}\label{concreteunif}
(\forall n\geq n_0)(\forall t\in [0,m])
\quad \zeta_n(t)\in U\exp_G(t\zeta'(0))\cap\exp_G(t\zeta'(0))U.
\end{equation}
For $v\in \cg$, let
$c_v\colon [0,1]\to \cg$ be the constant curve given by
\[
c_v(s):=v\quad \mbox{for all $\,s\in [0,1]$.}
\]
Because the map $\cg\to C([0,1],\cg)\sub \cR([0,1],\cg)$,
$v\mto c_v$ is continuous, the set
\[
K:=\{c_{t\zeta'(0)}\colon t\in [0,m]\}
\]
is compact. Also the map
$\evol\colon \cR([0,1],\cg)\to G$, $\evol(\sigma):=\Evol(\sigma)(1)$
is continuous. Hence
\[
\evol(K)\sub G
\]
is compact, whence there exists an open
identity neighbourhood~$V\sub U$ such that
$gVg^{-1}\sub U$ and thus
$gV\sub gU\cap Ug$,
for all $g\in \evol(K)$. Since $\evol(c_{t\zeta'(0)})=\exp_G(t\zeta'(0))$,
we deduce that
\begin{equation}\label{easierunif}
\exp_G(t\zeta'(0))V\sub U\exp_G(t\zeta'(0))\cap \exp_G(t\zeta'(0))U,
\quad\mbox{for all $t\in [0,m]$.}
\end{equation}
Next, we show that there is an open $0$-neighbourhood
$Q$ in $\cR([0,1],\cg)$ such that
\begin{equation}\label{wannagt}
\evol(\theta+\sigma)\in \evol(\theta)V\quad\mbox{for all
$\theta\in K$ and $\sigma\in Q$.}
\end{equation}
To this end,
let $W\sub G$ be an open identity neighbourhood such that
$W^{-1}W\sub V$.
Again using that $\evol\colon \cR([0,1],\cg)\to G$ is continuous,
for each $\theta\in K$ we find an open $0$-neighbourhood
$P_\theta\sub \cR([0,1],\cg)$ such that
\[
\evol(\theta+P_\theta)\sub \evol(\theta)W.
\]
Let $Q_\theta\sub \cR([0,1],\cg)$ be an open $0$-neighbourhood
such that $Q_\theta+Q_\theta \sub P_\theta$.
Then $K\sub \bigcup_{j=1}^\ell (\theta_j+Q_{\theta_j})$
for some finite subset $\{\theta_1,\ldots,\theta_\ell\}\sub K$.
Moreover,
\[
Q:=\bigcap_{j=1}^\ell Q_\theta
\]
is an
open $0$-neighbourhood in $\cR([0,1],\cg)$.
Then (\ref{wannagt}) holds.
In fact, for $\theta\in K$ we have $\theta\in \theta_j+Q_j$
for some $j\in\{1,\ldots,\ell\}$.
Since $\theta-\theta_j\in Q_j\sub P_j$, we have
\[
\evol(\theta)=\evol(\theta_j+(\theta-\theta_j))\in\evol(\theta_j)W
\]
and thus
\begin{equation}\label{revrs}
\evol(\theta_j)\in\evol(\theta)W^{-1}.
\end{equation}
For $\sigma\in Q\sub Q_j$, we have $(\theta-\theta_j)+\sigma
\in Q_j+Q_j\sub P_j$ and thus
\[
\evol(\theta+\sigma)=\evol(\theta_j+(\theta-\theta_j)+\sigma)
\in \evol(\theta_j)W\sub \evol(\theta)W^{-1}W\sub
\evol(\theta)V,
\]
using (\ref{revrs}) for the penultimate inclusion.
Thus (\ref{wannagt}) is established.\\[2.3mm]
For $t\in [0,m]$,
consider the continuous curve
$\alpha_{n,t} \colon [0,1] \to G$
defined piecewise for $s\in [k/n, (k+1)/n]$
with $k\in \{0,\ldots, n-1\}$ as
\begin{equation}\label{explpiec}
\alpha_{n,t}(s):=
\zeta(t/n)^k \zeta((s-k/n)t).
\end{equation}
Then $\alpha_{n,t}|_{[k/n,(k+1)/n]}$ is $C^1$
for each $k\in \{1,\ldots,n-1\}$,
entailing that $\alpha_{n,t}\in AC_{\cR}([0,1],G)$
with $\beta_{n,t}:=\delta^\ell(\alpha_{n,t})\in \cR([0,1],\cg)$.
By construction,
\[
\zeta_n(t) = \alpha_{n,t}(1) = \evol(\beta_{n,t}).
\]
The explicit formula (\ref{explpiec}) for $\alpha_{n,t}|_{[k/n,(k+1)/n]}$
shows that
\begin{equation}\label{finste1}
\beta_{n,t}(s) = t \delta^\ell(\zeta)((s-k/n)t)
\end{equation}
for all $k\in \{0,1,\ldots, n-1\}$ and $t\in [k/n,(k+1)/n]$.
We have
\[
\{\tau\in \cR([0,1],\cg)\colon \|\tau\|_{L^\infty,q}\leq 1\}\sub Q
\]
for a continuous seminorm $q$ on $\cg$.
Since $\delta^\ell\zeta\colon [0,1]\to \cg$
is continuous, there is $\ve\in \,]0,1]$ such that
\begin{equation}\label{finste2}
q(\delta^\ell\zeta(x)-\delta^\ell(\zeta(0))\leq \frac{1}{m}
\quad\mbox{for
all $x\in [0,\ve]$.}
\end{equation}
Choose $n_0\geq m$ so large that $\frac{m}{n_0}\leq\ve$.
Let $n\geq n_0$.
Then
\begin{equation}\label{finste3}
{\textstyle (s-\frac{k}{n})t\; \leq \; \frac{m}{n_0}\; \leq\; \ve}
\end{equation}
for all $t\in [0,m]$ and $s\in [0,1]$,
with $k\in \{0,1,\ldots, n-1\}$ such that
$s\in [k/n,(k+1)/n]$.
Combining (\ref{finste3}) with (\ref{finste2})
and (\ref{finste1}), we see that
\begin{eqnarray*}
\beta_{n,t}(s) &=&
t \delta^\ell(\zeta)(0)+
t (\delta^\ell(\zeta)((s-k/n)t) -\delta^\ell(\zeta)(0))\\
&\in&
t \delta^\ell(\zeta)(0)+t \wb{B}^q_{1/m}(0)\, \sub\,
t \delta^\ell(\zeta)(0)+ \wb{B}^q_1(0).
\end{eqnarray*}
Since $\delta^\ell(\zeta)(0)=\zeta'(0)$,
we deduce that
\[
\|\beta_{n,t}-c_{t\zeta'(0)}\|_{L^\infty,q}\leq 1
\]
for all $n\geq n_0$ and $t\in [0,m]$.
Thus $\beta_{n,t}-c_{t\zeta'(0)}\in Q$ and hence
\begin{eqnarray*}
\zeta_n(t) &=& \evol(\beta_{n,t})
\,=\,\evol(c_{t\zeta'(0)}+(\beta_{n,t}-c_{t\zeta'(0)}))\\
&\in&
\evol(c_{t\zeta'(0)})V=\exp_G(t\zeta'(0))V
\sub U\exp_G(t\zeta'(0))\cap \exp_G(t\zeta'(0))U
\end{eqnarray*}
for all $n\geq n_0$ and $t\in [0,m]$,
using~(\ref{wannagt}) and (\ref{easierunif}).
Thus (\ref{concreteunif}) holds
and the proof is complete.\,\Punkt
%
%
%
%
%
%
%
%
%
%
%
%
%
%
%
%
%
%
%
%
%
%
\appendix
\section{Proofs for Section~\ref{prels}}\label{appsecprel}
We provide proofs for results compiled in the indicated section,
and some auxiliary results.\\[4mm]
{\bf Proof of \ref{topPL}.}
(a) The topology on~$X$ is initial with respect to the mappings $\phi_j$ for $j\in J$.
Therefore, intersections
\[
\bigcap_{j\in F}\phi_j^{-1}(U_j)
\]
form
a basis of open neighbourhoods of~$x$,
for finite subsets $F\sub J$ and
open neighbourhoods $U_j\sub X_j$ of $\phi_j(x)$
for $j\in F$. Since $J$ is directed, we find $j_0\in J$
such that $j\leq j_0$ for all $j\in F$.
Then $U:=\bigcap_{j\in F}(\phi_{j,j_0})^{-1}(U_j)$
is an open neighbourhood of $\phi_{j_0}(x)$ in~$X_{j_0}$
and $(\phi_{j_0})^{-1}\sub
\bigcap_{j\in F}\phi_j^{-1}(U_j)$,
as $\phi_j(\phi_{j_0}^{-1}(U))=\phi_{j,j_0}(\phi_{j_0}(\phi_{j_0}^{-1}(U)))
\sub\phi_{j,j_0}(U)\sub \phi_{j,j_0}(\phi_{j,j_0}^{-1}(U_j))\sub U_j$
for each $j\in J$.

(b) If $D$ is dense in~$X$,
then $\phi_j(D)$ is dense in $\phi_j(X)$,
by continuity. Conversely, assume that $\phi_j(D)$ is dense
in $\phi_j(X)$ for all $j\in J$.
Let $x\in X$ and $V$ be an open neighbourhood of $x$ in~$X$.
By (a), we may assume that $V=\phi_j^{-1}(U)$
for some $j\in J$ and open neighbourhood~$U$ of $\phi_j(x)$ in~$X_j$.
Since $\phi_j(D)$ is dense in $\phi_j(X)$,
we find $y\in D$ such that $\phi_j(y)\in U$.
Then $y\in \phi_j^{-1}(U)=V$
and thus $D$ is dense in~$X$.\,\Punkt
\begin{la}\label{Sepp}
Let $(X,\Sigma)$ be a measurable space, $n\in \N$,
$Y_1,\ldots, Y_n$ be metric spaces
and $\gamma_j\colon (X,\Sigma)\to (Y_j,\cB(Y_j))$ be a measurable
map with separable image, for $j\in \{1,\ldots, n\}$.
Then the following holds:
\begin{itemize}
\item[\rm(a)]
$\gamma:=(\gamma_1,\ldots,\gamma_n)\colon X\to Y_1\times\cdots\times Y_n$
is measurable with respect to the $\sigma$-algebra $\cB(Y)$
on the direct product topological space $Y:=Y_1\times\cdots\times Y_n$.
\item[\rm(b)]
If $Z$ is a topological space
and $f\colon Y_1\times\cdots\times Y_n\to Z$ a continuous map,
then the map $f\circ (\gamma_1,\ldots,\gamma_n)\colon (X,\Sigma)\to (Z,\cB(Z))$
is measurable.
\end{itemize}
\end{la}
\begin{proof}
(a) By \ref{basicsmeas}\,(b),
we need only show that $\gamma$ is measurable as a map
to $\im(\gamma_1)\times\cdots\times \im(\gamma_n)$,
equipped with the trace of $\cB(Y)$.
The latter coincides with the Borel-$\sigma$-algebra with respect to
the induced topology~$\cO$ (see \ref{basicsmeas}\,(d)).
But $\cO$ coincides with the product topology
on $\gamma_1(X)\times\cdots\times \gamma_n(X)$,
if we use the topology induced by~$Y_j$ on $\gamma_j(X)$,
for each~$j$.
After replacing~$Y_j$ with $\gamma_j(X)$ if necessary,
we may therefore assume that each $Y_j$ is separable and hence
second countable. As a consequence,
$\cB(Y_1\times\cdots\times Y_n)=\cB(Y_1)\tensor\cdots\tensor \cB(Y_n)$
is the product $\sigma$-algebra (see \ref{basicsmeas}\,(f)).
But $\gamma$ is measurable with respect to this product $\sigma$-algebra,
as all of its components $\gamma_j$ are measurable (see \ref{basicsmeas}\,(e)).

(b) The maps $\gamma\!:\!(X,\Sigma)\to (Y,\cB(Y))$
and $f\colon (Y,\cB(Y))\to (Z,\cB(Z))$
are measurable,
hence also $f\circ\gamma$.
\end{proof}
\begin{rem}\label{spaceandnorm}
The preceding lemma
entails that $\cL^p(X,\mu,E)$
and $\cL^\infty_{rc}(X,\mu,E)$
are vector subspaces of
$E^X$, for each Fr\'{e}chet space~$E$ and $p\in [1,\infty]$.
\end{rem}
The next two lemmas and Lemma~\ref{labasemea}\,(a)
will help us to prove Lemma~\ref{labase1}.
Afterwards, we develop machinery for the proof of
Lemma~\ref{funda}.
\begin{la}\label{limmea}
Let $(X,\Sigma)$ be a measurable space, $(Y,d)$ be a metric space
and $\gamma_n\colon X\to Y$ be a measurable function
with separable image, for each $n\in \N$.
\begin{itemize}
\item[\rm(a)]
If $\gamma(x):={\displaystyle \lim_{n\to\infty}}\gamma_n(x)$ exists for
each $x\in X$, then $\gamma\colon X\to Y$, $x\mto\gamma(x)$
is a measurable map with separable image.
\item[\rm(b)]
If the metric space $(Y,d)$ is complete, then
\[
C:=\{x\in X\colon \mbox{$\gamma_n(x)$ converges as $n\to\infty$}\}
\]
is a measurable set.
\end{itemize}
\end{la}
\begin{proof}
After replacing~$Y$ with the closure of $\bigcup_{n\in\N}\gamma_n(X)$,
we may assume that the metric space~$Y$ is separable (see \ref{basicsmeas}
(b) and (d)).

(a) Is a special case of \cite[Lemma~2.5]{MEA}.

(b) Since~$Y$ is separable and hence second countable,
we have $\cB(Y\times Y)=\cB(Y)\tensor \cB(Y)$ (see, e.g., \cite[Lemma~2.7]{MEA}).
For $\ell\in \N$, the set
\[
O_\ell:=\{(y,z)\in Y\times Y\colon d(y,z)<1/\ell\}
\]
is open in $Y\times Y$. Hence $O_\ell\in \cB(Y\times Y)=\cB(Y)\tensor \cB(Y)$.
Fix $x\in X$. Because $Y$ is complete, the sequence $(\gamma_n(x))_{n\in\N}$
converges if and only if it is a Cauchy sequence,
requiring that for each $\ell\in \N$, there is $N\in \N$ such that
$(\gamma_n(x),\gamma_m(x))\in O_\ell$ for all $n,m\geq N$.
Hence
\[
C=\bigcap_{\ell\in \N}\bigcup_{N\in \N}\bigcap_{n,m\geq N}(\gamma_n, \gamma_m)^{-1}(O_\ell)\,.
\]
Since $(\gamma_n,\gamma_n)\colon (X,\Sigma)\to  (Y\times Y,\cB(Y)\tensor \cB(Y))$
is a measurable map (as it has measurable components) and each $O_\ell$
is measurable, $C$ is measurable.
\end{proof}
{\bf Proof of Lemma~\ref{imagesepmet}.}
For $m\in \N$,
let~$K_m$ be the $m$-fold sum
\[
K_m:=[{-m},m]K+\cdots+[{-m},m]K:=\{x_1+\cdots+x_m\colon x_1,\ldots, x_m\in[{-m},m]K\}.
\]
Lemma~\ref{quotcpmet} implies that $K_m$ is compact and metrizable,
as it is the image of a
continuous map $[{-m},m]^m\times K^m\to E$
on the metrizable compact space $[{-m},m]^m\times K^m$.
Let $d_m$ be a metric on $K_m$ defining its topology and set
\[
K_{m,j}:=\{(x,y)\in K_m\times K_m\colon d_m(x,y)\geq 1/j\}
\]
for $j\in \N$.
Then $\bigcup_{j\in \N} K_{m,j}=(K_m\times K_m)\setminus \Delta_m$,
where $\Delta_m:=\{(x,x)\colon x\in K_m\}\sub K_m\times K_m$
is the diagonal. Moreover, each of the sets $K_{m,j}$
is compact.
For each $(x,y)\in K_{m,j}$, we have $x\not= y$
and hence find a continuous seminorm $q_{m,j,x,y}\in P(E)$
such that $q_{m,j,x,y}(x-y)\not=0$ and thus $q_{m,j,x,y}(x-y)=1$
without loss of generality.
Then the sets
\[
B^{q_{m,j,x,y}}_{1/2}(x)\times B^{q_{m,j,x,y}}_{1/2}(y)
\]
form an open cover of $K_{m,j}$ for $(x,y)\in K_{m,j}$,
and hence there is a finite subset $F_{m,j}\sub K_{m,j}$ such that
\[
K_{m,j}\sub\bigcup_{(x,y)\in F_{m,j}}(B^{q_{m,j,x,y}}_{1/2}(x)\times B^{q_{m,j,x,y}}_{1/2}(y)).
\]
Thus, if $(v,w)\in K_{m,j}$, we find $(x,y)\in F_{m,j}$
such that
\begin{equation}\label{nbhdinlarge}
(v,w)\in B^{q_{m,j,x,y}}_{1/2}(x)\times B^{q_{m,j,x,y}}_{1/2}(y)=:V\times W.
\end{equation}
Then $V$ and $W$ are open neighourhoods
of $v$ and $w$, respectively, in~$E$. Since $q(y-x)=1$,
these neighbourhoods are disjoint (otherwise
a contradiction would result from the triangle
inequality).
Now $\Gamma:=\{q_{m,j,x,y}\colon m,j\in \N,\; (x,y)\in F_{m,j}\}\sub P(E)$
is a countable set,
\[
E_K:=\Spann(K)=\bigcup_{m\in \N}K_m\quad\mbox{and}\quad \bigcup_{m,j\in \N}K_{m,j}=
(E_K\times E_K)\setminus \Delta,
\]
where $\Delta:=\{(x,x) \colon x \in E_K\}$ is the diagonal in~$E_K\times E_K$.
We give $E_K$ the locally convex vector topology~$\cO'$
defined by the countable set~$\Gamma$
of seminorms. Then $\cO'$ is Hausdorff,
since any $(v,w)\in (E_K\times E_K)\setminus \Delta$
is contained in $K_{m,j}$ for some $m,j\in \N$
and hence $E_K\cap B^{q_{m,j,x,y}}_{1/2}(x)$
and $E_K\cap B^{q_{m,j,x,y}}_{1/2}(y)$
(with notation as in (\ref{nbhdinlarge}))
are disjoint open neighbourhoods of~$v$
and~$w$ in $(E_K,\cO')$.\,\vspace{1.3mm}\Punkt

\noindent
Let $\cF(X,\Sigma,E)$ be the space of all measurable
maps $\gamma\colon (X,\Sigma)\to (E,\cB(E))$ with finite image.
Then $\cF(X,\Sigma,E)\cap \cL^1(X,\mu,E)$
is the vector subspace of all $\gamma\in \cF(X,\Sigma,E)$
such that $\mu(\gamma^{-1}(E\setminus\{0\}))<\infty$.\\[2.4mm]
If $X$ is a locally compact topological space, we write
$C_c(X,E)$ for the space of all compactly supported continuous $E$-valued
functions on~$X$. Recall that a \emph{Radon measure on~$X$} is a measure
$\mu \colon \cB(X)\to[0,\infty[$
which is finite on compact sets and inner regular (see, e.g. \cite{BCR}).
\begin{la}\label{labasemea}
Let $(X,\Sigma,\mu)$ be a measure space $E$ be a Fr\'{e}chet space.
Then:
\begin{itemize}
\item[\rm(a)]
$\cF(X,\Sigma, E)$ is dense in $\cL^\infty_{rc}(X,\Sigma,E)$
and
$\cF(X,\Sigma,E)\cap\cL^1(X,\mu,E)$ is dense in $\cL^1(X,\mu,E)$.
\item[\rm(b)]
The weak integral
\[
\int_X \gamma(x)\,d\mu(x)
\]
exists in $E$, for each $\gamma\in \cL^1(X,\mu,E)$,
and
\begin{equation}\label{standintest}
p\left(\int_X\gamma(x)\,d\mu(x)\right)\leq \int_Xp(\gamma(x)\,d\mu(x)
\end{equation}
for each continuous seminorm
$p$ on~$E$.
\item[\rm(c)]
If $X$ is a locally compact space
and $\mu$ a Radon measure on~$X$, then $C_c(X,E)$ is dense
in $\cL^1(X,\mu,E)$.
\end{itemize}
\end{la}
\begin{proof}
(a) Let $p$ be a continuous seminorm on~$E$, and $\ve>0$.
If $\gamma\in \cL^1(X,\mu,E)$,
then $X_n:=\{x\in X\colon p(\gamma(x))\geq\frac{1}{n}\}\in \Sigma$
(cf.\ Remark~\ref{spaceandnorm})
and $(p\circ \gamma)^{-1}(]0,\infty[)=\bigcup_{n\in \N}X_n$,
where $X_1\sub X_2\sub\cdots$.
Thus $\mu(X_n)<\infty$, as
\[
\frac{1}{n}\mu(X_n)\leq\int_{X_n}p(\gamma(x))\,d\mu(x)\leq
\int_Xp(\gamma(x))\,d\mu(x)<\infty.
\]
Moreover,
\begin{eqnarray*}
\|\gamma\|_{\cL^1,p}&=&\int_Xp(\gamma(x))\,d\mu(x)=
\int_{(p\circ \gamma)^{-1}(]0,\infty[)}p(\gamma(x))\,d\mu(x)\\
&=& \lim_{n\to\infty}\int_{X^n}p(\gamma(x))\,d\mu(x),
\end{eqnarray*}
whence there exists $N\in \N$ such that
\[
\int_{X\setminus X_n}p(\gamma(x))\,d\mu(x)\,<\,\frac{\ve}{3}
\]
for all $n\geq N$.
If $\mu(X_N)=0$, then $\eta:=0$ is an element of
$\cF(X,\Sigma,E)\cap \cL^1(X,\mu,E)$ such that
$\|\gamma-\eta\|_{\cL^1,p}<\ve$.
If $m:=\mu(X_N)>0$,
let $D:=\{y_n\colon n\in \N\}$ be a countable
dense subset of~$\gamma(X)$.
Then
\[
\gamma(X_N)\sub \wb{\gamma(X)}=\wb{D}\sub D+B^p_{\ve/(3m)}(0)
=\bigcup_{n\in \N}B^p_{\ve/(3m)}(y_n).
\]
We define $A_1:=\{x\in X_N\colon \gamma(x)\in B^p_{\ve/(3m)}(y_1)\}\in\Sigma$
and $A_n:=\{x\in X_N\colon \gamma(x)
\in B^p_{\ve/(3m)}(y_n)\} \setminus\bigcup_{k=1}^{n-1}A_k\in \Sigma$
for integers $n\geq 2$.
Then
\[
X_N=\bigcup_{n\in \N}A_n
\]
is a countable union of disjoint sets.
Arguing as above, we find $n_0\in \N$ such that
\[
\int_{X_N\setminus \bigcup_{n=1}^{n_0}A_n}p(\gamma(x))\,d\mu(x)\,<\,\frac{\ve}{3}.
\]
Define $\eta:=\sum_{n=1}^{n_0}y_n \one_{A_n}$, using the characteristic function
$\one_{A_n}\colon X\to \{0,1\}$ of $A_n$.
Then $\eta\in \cF(X,\Sigma,E)\cap \cL^1(X,\mu,E)$ and
\begin{eqnarray*}
\|\gamma-\eta\|_{\cL^1,p} &=& \int_Xp(\gamma(x)-\eta(x))\,d\mu(x)\\
&=& \int_{X\setminus X_N}p(\gamma(x))\,d\mu(x)
+ \int_{X_N\setminus \bigcup_{n=1}^{n_0}A_n} p(\gamma(x))\,d\mu(x)\\
& & \quad +\int_{\bigcup_{n=1}^{n_0}A_n}\underbrace{p(\gamma(x)-\eta(x))}_{\leq \ve/(3m)}
\,d\mu(x)\\
&<&\frac{\ve}{3}+\frac{\ve}{3}+\frac{\ve}{3m}
\mu\underbrace{\left(\bigcup_{n=1}^{n_0}A_n\right)}_{\leq\mu(X_N)=m}
<\ve.
\end{eqnarray*}
Hence $\cF(X,\Sigma,E)\cap \cL^1(X,\mu,E)$ is dense
in $\cL^1(X,\mu,E)$.\\[2.3mm]
If $\gamma\in \cL^\infty_{rc}(X,\mu,E)$,
then $\gamma(X)$ is precompact in~$E$,
whence there exists a finite set $\{y_1,\ldots, y_n\}\sub \gamma(X)$
such that
\[
\gamma(x)\sub \bigcup_{k=1}^nB^p_\ve(y_k).
\]
Then $A_1:=\gamma^{-1}(B^p_\ve(y_1))\in \Sigma$
and $A_k:= \gamma^{-1}(B^p_\ve(y_k))\setminus (A_1\cup\cdots\cup A_{k-1})
\in \Sigma$ for $k\in \{2,\ldots, n\}$.
Moreover,
\[
\eta:=\bigcup_{k=1}^n y_k\one_{A_k}\in \cF(X,\Sigma,E)
\]
and $\|\gamma-\eta\|_{\cL^\infty,p}<\ve$ by construction.
Hence $\cF(X,\Sigma,E)$ is dense in $\cL^\infty_{rc}(X,\mu,E)$.\vspace{1mm}

(b) If $\gamma\in \cF(X,\Sigma,E)\cap \cL^1(X,\mu,E)$,
then $\gamma=\sum_{k=1}^ny_k\one_{A_k}$
for some $n\in \N_0$, $y_k\in E$ and disjoint sets
$A_1,\ldots,A_n\in \Sigma$ such that $\mu(A_k)<\infty$
for each $k\in \{1,\ldots, n\}$
(for example, if $\gamma(X)=\{y_1,\ldots, y_n$
with pairwise distinct elements $y_1,\ldots, y_n$,
we can take $A_k:=\gamma^{-1}(\{y_k\})$).
We define
\[
I(\gamma):=\sum_{k=1}^n\mu(A_k)y_k
\]
(declaring the empty sum as $0$, if $n=0$).
Without changing $(\gamma)$, we may omit those
indices $k$ such that $y_k=0$ or $A_k=\emptyset$.
We may therefore assume that $y_k\not=0$
and $A_k\not=\emptyset$ for each
$k\in \{1,\ldots, n\}$.
Then $I(\gamma)$ is well-defined. In fact, assume that also
$\gamma=\sum_{j=1}^mz_j\one_{B_j}$ with
disjoint measurable subsets $B_1,\ldots, B_m$ of finite measure,
such that each $z_j$ is non-zero and
each $B_j\not=\emptyset$.
Then
\begin{equation}\label{henceun}
\bigcup_{k=1}^nA_k=
\{x\in X\colon \gamma(x)\not=0\}=\bigcup_{j=1}^mB_j.
\end{equation}
If $k\in \{1,\ldots, n\}$ and $j\in \{1,\ldots, m\}$,
then either $A_k\cap B_j=\emptyset$ (whence
$\mu(A_k\cap B_j)=0$ and thus
$y_k\mu(A_k\cap B_j)=0=z_j\mu(A_k\cap B_j)$)
or there exists $c\in A_k\cap B_j$, whence
$y_k=\gamma(c)=z_j$ and again $y_k\mu(A_k\cap B_j)=z_j\mu(A_k\cap B_j)$.
Hence, using that $A_k=\bigcup_{j=1}^mA_k\cap B_j$
(by (\ref{henceun}) which is a disjoint union,
\[
\sum_{k=1}^ny_k\mu(A_k)=
\sum_{k=1}^n\sum_{j=1}^m y_k\mu(A_k\cap B_j)
=\sum_{k=1}^n\sum_{j=1}^m z_j\mu(A_k\cap B_j)
=\sum_{j=1}^m z_j \mu(B_j).
\]
If also $\eta\in \cF(X,\Sigma,E)\cap\cL^1(X,\mu,E)$
and $r,s\in \R$, write $\eta=\sum_{i=1}^\ell w_\ell\one_{C_\ell}$
with disjoint measurable subsets $C_\ell$ of finite
measure and $w_\ell\in E$. Then $\gamma=\sum_{k=1}^n\sum_{i=1}^\ell
y_k\one_{A_k\cap C_i}$. Writing $\eta$ and $r\gamma+s\eta$
likewise as a linear combination of the characteristic
functions $\one_{A_k\cap C_i}$, we easily find that
\[
I(r\gamma+s\eta)=rI(\gamma)+sI(\eta).
\]
Thus $I$ is linear. Moreover,
$I$ is continuous with respect to the
topology induced by $\cL^1(X,\mu,E)$ on
$\cF(X,\Sigma,E)\cap\cL^1(X,\mu,E)$,
as
\[
p(I(\gamma))=p\left(\sum_{k=1}^n\mu(A_k)y_k\right)
\leq\sum_{k=1}^np(y_k)\mu(A_k)=\int_Xp(\gamma(x))\,d\mu(x)=
\|\gamma\|_{\cL^1,p}
\]
for each $\gamma$ as above and each continuous seminorm
$p$ on~$E$. Since $E$ is complete
and $\cF(X,\Sigma,E)\cap \cL^1(X,\mu,E)$ is dense
in $\cL^1(X,\mu,E)$,
the continuous linear map~$I$ has a unique continuous
linear extension
\[
J\colon \cL^1(X,\mu,E)\to E.
\]
For each continuous linear functional $\lambda\colon E\to \R$,
both $\lambda\circ J$ and the map
\[
h\colon \cL^1(X,\mu,E)\to \R,\qquad \gamma\mto \int_X\lambda(\gamma(x))\,d\mu(x)
\]
are continuous linear extensions of $\lambda\circ I$,
whence $J=h$ by density of $\cF(X,\Sigma,E)\cap \cL^1(X,\mu,E)$.
Thus
\[
\lambda(J(\gamma))=\int_X\lambda(\gamma(x))\,d\mu(x)
\]
for each $\lambda\in E'$ and thus $J(\gamma)$ is the weak
integral $\int_X\gamma(x)\,d\mu(x)$ in~$E$.
If $p$ is a continuous seminorm
on~$E$, then $\cL^1(X,\mu,E)\to \R$, $\gamma\mto p(J(\gamma))$
and $\gamma\mto \|\gamma\|_{\cL^1,p}$ are continuous
functions on $\cL^1(X,\mu,E)$ such that $p(J(\gamma))\leq \|\gamma\|_{\cL^1,p}$
for each $\gamma$ in the dense subset $\cF(X,\Sigma,E)\cap \cL^1(X,\mu,E)$.
Hence $p(J(\gamma))\leq \|\gamma\|_{\cL^1,p}$ for all
$\gamma\in \cL^1(X,\mu,E)$,
establishing~(\ref{standintest}).\vspace{1mm}

(c) Since $\cF(X,\cB(X),E)\cap\cL^1(X,\mu,E)$ is dense in $\cL^1(X,\mu,E)$
and every $\gamma$ in the former space is a linear combination
of maps of the form $v\one_A$ with $A\in\cB(X)$ and $\mu(A)<\infty$,
it suffices to show that $v\one_A$ is in the closure of
$C_c(X,E)$. Since $\mu$ is inner regular,
there exists an ascending sequence $K_1\sub K_2\sub\cdots$
of compact subsets of $A$ such that $\mu(K_n)\geq \mu(A)-\frac{1}{n}$
and thus $\mu(A\setminus K_n)\to 0$ as $n\to\infty$.
Let $V_n\sub X$ be a relatively compact, open
subset of~$X$ such that $K_n\sub V_n$.
Since $\mu|_{\wb{V_n}}$ is outer regular,
there exists a relatively open subset $W_n\sub \wb{V_n}$
such that $K_n\sub W_n$ and $\mu(W_n)\leq \mu(K_n)+\frac{1}{n}$.
After replacing $W_n$ with its intersection with~$V_n$,
we may assume that~$W_n$ is open in~$C$.
Now Urysohn's Lemma \cite[2.12]{Ru1}
provides $\eta_n\in C_c(X,\R)$ such that
$\one_{K_n}\leq \eta_n\leq \one_{W_n}$.
Then $v\eta_n\in C_c(X,E)$.
If~$q$ is a continuous seminorm on~$E$, then
\begin{eqnarray*}
\|v\eta_n-v\one_{K_n}\|_{\cL^1,p}& =& q(v)\int_{V_n\setminus K_n}|\eta_n(x)|\,d\mu(x)\\
& \leq & q(v)(\mu(V_n\setminus A_n)+\mu(A_n\setminus K_n))
\leq \frac{2q(v)}{n},
\end{eqnarray*}
which tends to~$0$ as $n\to\infty$.
\end{proof}
{\bf Proof of Lemma~\ref{labase1}.}
By Lemma~\ref{labasemea},
the weak integral $\int_X\gamma\,d\mu$
exists in~$E$ for each $\gamma\in \cL^1(X,\mu,E)$,
and can be estimated as desired.
The remaining assertions
on $L^\infty_{rc}(X,E)$
are covered by \cite[Proposition~3.21]{MEA}.\footnote{Where
the symbol $L^\infty$ is used in place of $L^\infty_{rc}$.}
If $E$ is a Banach space space
(resp., a Fr\'{e}chet space),
then $L^p(X,\mu,E)$
are normable (resp., have the property that
the vector topology can be defined
using a countable set of seminorms).
It only remains to show
that
$L^p(X,\mu,E)$ is complete.
To this end, let $(q_n)_{n\in \N}$
be a sequence of seminorms on~$E$
defining its vector topology.
Let $\gamma_n\in \cL^p(X,\mu,E)$
such that $([\gamma_n])_{n\in \N}$
is a Cauchy sequence in $L^p(X,\mu,E)$.
The Cauchy sequence will converge
if we can show that it has a convergent
subsequence. After passing to a subsequence,
we may therefore assume that
\[
(\forall n_0\in \N)\;(\forall n,m\geq n_0)\quad
\|\gamma_n-\gamma_m\|_{\cL^p,q_{n_0}}\leq 2^{-n_0}.
\]
We claim that $\sum_{n=1}^\infty (\gamma_{n+1}-\gamma_n)$
converges in $L^p(X,\mu,E)$. If this is true,
then $\gamma_m=\gamma_1+\sum_{n=1}^{m-1}(\gamma_{n+1}-\gamma_n)$
converges to $\gamma_1+\sum_{n=1}^\infty (\gamma_{n+1}-\gamma_n)$
as $m\to\infty$, whence $L^p(X,\mu,E)$ is complete.
We first prove the claim if
$p=\infty$. After replacing the representatives
by~$0$ on a set of measure~$0$,
we may assume that
$\|\gamma_n-\gamma_m\|_{\cL^\infty,q_{n_0}}=\sup_{x\in X}q_{n_0}(\gamma_n(x)-\gamma_m(x))$
for all $n_0\in \N$ and $n,m\geq n_0$.
Let $x\in X$.
For each $n_0\in \N$,
we have
\[
\sum_{n=n_0}^\infty q_{n_0}(\gamma_{n+1}(x)-\gamma_n(x))
\leq\sum_{n=n_0}^\infty 2^{-n}<\infty
\]
and thus $\sum_{n=1}^\infty q_{n_0}(\gamma_{n+1}(x)-\gamma_n(x))<\infty$.
Therefore $\sum_{n=1}^\infty(\gamma_{n+1}(x)-\gamma_n(x))$
is an absolutely convergent series
in the Fr\'{e}chet space~$E$
and hence convergent.
By Lemma~\ref{limmea}, the function
\[
\gamma\colon X\to E,\quad \gamma(x):=\sum_{n=1}^\infty(\gamma_{n+1}(x)-\gamma_n(x))
=\lim_{N\to\infty}\sum_{n=1}^N(\gamma_{n+1}(x)-\gamma_n(x))
\]
is measurable
and has separable image.
For each $n_0\in \N$ and $x\in X$, we have
\begin{eqnarray*}
q_{n_0}(\gamma(x))
&\leq&
\sum_{n=1}^\infty q_{n_0}(\gamma_{n+1}(x)-\gamma_n(x))\\
&=&
\sum_{n=1}^{n_0-1}q_{n_0}(\gamma_{n+1}(x)-\gamma_n(x))
+ \sum_{n=n_0}^\infty q_{n_0}(\gamma_{n+1}(x)-\gamma_n(x))\\
&\leq&
\sum_{n=1}^{n_0-1}\|\gamma_n\|_{\cL^\infty,q_{n_0}}
+ \sum_{n=n_0}^\infty 2^{-n}.
\end{eqnarray*}
As a consequence,
$\im(\gamma)$ is bounded in~$E$
and thus $\gamma\in \cL^\infty(X,\mu,E)$,
with $\|\gamma\|_{\cL^\infty,q_{n_0}}\leq
\sum_{n=1}^{n_0-1}\|\gamma_n\|_{\cL^\infty,q_{n_0}}
+ \sum_{n=n_0}^\infty 2^{-n}$.
If $n_0\in \N$,
we have for all $m\in \N$ with $m\geq n_0-2$
and $x\in X$
\begin{eqnarray*}
q_{n_0}(\gamma(x)-\sum_{n=0}^m (\gamma_{n+1}(x)-\gamma_n(x)))
&\leq & \sum_{n=m+2}^\infty q_{n_0}(\gamma_{n+1}(x)-\gamma_n(x)))\\
&\leq& \sum_{n=m+2}^\infty 2^{-n}=
\frac{2^{-m-2}}{1-1/2}=2^{-m-1}.
\end{eqnarray*}
Thus $\|\gamma-\sum_{n=1}^m (\gamma_{n+1}-\gamma_n)\|_{\cL^\infty,q_{n_0}}\leq
2^{-m-1}$ tends to $0$ as $m\to\infty$
and hence $\gamma=\lim_{m\to\infty}\sum_{n=1}^m(\gamma_{n+1}-\gamma_n)$
in $\cL^\infty(X,\mu,E)$ (establishing the claim).

If $p\in [1,\infty[$,
for each $n_0\in \N$ and $N\in \N$ we have
\begin{eqnarray*}
\lefteqn{\sqrt[p]{\int_X \Big(\sum_{n=1}^Nq_{n_0}(\gamma_{n+1}-\gamma_n)\Big)^p\,d\mu}}\\
&=&
\|\sum_{n=1}^N q_{n_0}\circ(\gamma_{n+1}-\gamma_n)\|_{\cL^p}
\leq \sum_{n=1}^N\|q_{n_0}\circ(\gamma_{n+1}-\gamma_n)\|_{\cL^p}\\
&=& \sum_{n=1}^N\|\gamma_{n+1}-\gamma_n\|_{\cL^p,q_{n_0}}
\leq \sum_{n=1}^\infty\|\gamma_{n+1}-\gamma_n\|_{\cL^p,q_{n_0}}.
\end{eqnarray*}
Letting $N\to\infty$, the Monotone Convergence Theorem
entails that
\[
\sqrt[p]{\int_X \Big(\sum_{n=1}^\infty q_{n_0}(\gamma_{n+1}-\gamma_n)\Big)^p
\,d\mu}
\leq \sum_{n=1}^\infty\|\gamma_{n+1}-\gamma_n\|_{\cL^p,q_{n_0}}<\infty.
\]
Hence $\Big(\sum_{n=1}^\infty q_{n_0}(\gamma_{n+1}-\gamma_n)\Big)^p\in \cL^1(X,\R)$.
Hence, after replacing each of the maps $\gamma_n$ by~$0$
on a set of measure zero, we may assume that
\[
\sum_{n=1}^\infty q_{n_0}(\gamma_{n+1}-(x)\gamma_n(x))<\infty
\]
for all $n_0\in \N$ and all $x\in X$.
Hence, for each $x\in X$, the series $\sum_{n=1}^\infty
(\gamma_{n+1}(x)-\gamma_n(x))$ in~$E$ is absolutely
convergent and hence convergent to some $\gamma(x)\in E$.
By Lemma~\ref{limmea}, the function
\[
\gamma\colon X\to E,\quad \gamma(x):=\sum_{n=1}^\infty(\gamma_{n+1}(x)-\gamma_n(x))
=\lim_{N\to\infty}\sum_{n=1}^N(\gamma_{n+1}(x)-\gamma_n(x))
\]
is measurable
and has separable image.
Since
\begin{eqnarray*}
\int_X (q_{n_0}(\gamma(x)))^p\,d\mu(x)
&\leq & \int_X \Big(\sum_{n=1}^\infty q_{n_0}(\gamma_{n+1}(x)-\gamma_n(x))\Big)^p\,d\mu(x)\\
&\leq & \Big(\sum_{n=1}^\infty\|\gamma_{n+1}-\gamma_n\|_{\cL^p,q_{n_0}}\Big)^p<\infty
\end{eqnarray*}
for each $n_0\in \N$, we have $\gamma\in \cL^p(X,\mu,E)$.
Finally, $\sum_{n=1}^m(\gamma_{n+1}-\gamma_n)\to \gamma$ in $\cL^p(X,\mu,E)$
as $m\to \infty$ since
\[
\|\gamma-\sum_{n=1}^m(\gamma_{n+1}-\gamma_n)\|_{\cL^p,q_{n_0}}
= \sqrt[p]{\int_X q_{n_0}\Big(\sum_{n=n_0}^\infty (\gamma_{n+1}(x)-\gamma_n(x))\Big)\,d\mu}
\to 0
\]
by dominated convergence, using that the integrands are majorized
by the integrable function
$\Big(\sum_{n=1}^\infty q_{n_0}\circ (\gamma_{n+1}-\gamma_n)\Big)^p$.\vspace{1.3mm}\Punkt

\noindent
{\bf Proof of Lemma~\ref{indeedmetr}.}
Let $M^1_+(K)$
be the set of all Radon probability measures
on~$K$. Endow $M^1_+(K)$
with the vague topology,
which makes it a compact topological
space and turns the map $\Phi\colon M^1_+(K)\to C(K)'$, $\mu\mapsto
I_\mu$ (with $I_\mu(f):=\int_K f\,d\mu$)
into a topological embedding
with respect to the weak $*$-topology
on the dual $C(K)$ of the Banach space $C(K)$
of continuous real-valued functions on~$K$
with the supremum norm
(see \cite[Chapter 2, Corollary~4.7]{BCR}).
Because $K$ is metrizable,
$M^1_+(K)$ is metrizable
(see \cite[Satz 31.5\,(a)]{Bau}).
The weak $*$-topology on $C(K)'$ is initial
with respect to the linear maps
\[
\ve_\gamma\colon C(K)'\to\R,\quad \lambda\mto\lambda(\gamma).
\]
Hence $\ve_\gamma\circ\Phi\colon M^1_+(K)\to\R$,
$\mu\mto\int_K\gamma\,d\mu$
is continuous for each $\gamma\in C(K)$.
For each $\mu\in M^1_+(K)$,
the barycentre
\[
b_\mu:=\int_Kx\,d\mu(x)
\]
exists in the compact set $\wb{\conv(K)}$
(see \cite[Chapter 2, Proposition~5.3]{BCR}
or \cite[Theorem 3.27]{Ru2}).
The map
\[
\beta\colon M^1_+(K)\to \wb{\conv(K)}, \quad \mu\mto b_\mu
\]
is continuous: Because $\wb{\conv(K)}=:L$
is compact, the topology induced by~$E$ on
$L$ coincides with the weak topology,
which is initial with respect to the mappings
$\lambda|_L$ for $\lambda\in E'$.
Hence $\beta$ will be continuous
if we can show that $\lambda\circ \beta$ is continuous
for each $\lambda\in E'$.
But
\[
(\lambda\circ\beta)(\mu)=\int_K\lambda(x)\,d\mu(x)=I_\mu(\lambda|_L),
\]
showing that
$\lambda\circ \beta=\ve_{\lambda|_L}\circ\Phi$,
which indeed is continuous.
Now Lemma~\ref{quotcpmet} shows that
$\im(\beta)$ is compact and
metrizable.
It only remains to recall that $\im(\beta)=\wb{\conv(K)}$,
see \cite[Chapter 2, Proposition~5.3]{BCR}.\,\vspace{1.3mm}\Punkt

\noindent
{\bf Proof of Lemma~\ref{labaserc}.}
If $\mu=0$, then all assertions
are trivial. We may therefore assume that
$\mu(X)>0$. As weak integrals are linear in~$\mu$,
we may assume that $\mu(X)=1$.
If $\gamma\in \cL^\infty_{rc}(X,\mu,E)$,
then $K:=\wb{\im(\gamma)}\sub E$ compact and metrizable.
Since~$E$ is assumed integral complete,
it satisfies the metric CCP.
Thus $\wb{\conv(K)}$ is compact. Now consider
the image measure $\gamma_*(\mu)\colon \cB(\wb{\conv(K)})\to [0,1]$
of $\mu$ under the measurable map
$\gamma\colon X\to \wb{\conv(K)}$;
thus $\gamma_*(\mu)(A):=\mu(\gamma^{-1}(A))$
for Borel sets $A\sub \wb{\conv(K)}$.
By \cite[Theorem 3.27]{Ru2},
the weak integral
\[
b:=\int_{\wb{\conv(K)}}x\, d\gamma_*(\mu)(x)
\]
exists in~$E$. Testing with continuous
linear functions~$\lambda\in E'$
and using the Transformation Theorem
for integrals with respect to image measures
(\cite[19.2, Korollar 1]{Bau}),
we see that $b=\int_X\gamma\,d\mu$.
If $q\in P(E)$, the Hahn-Banach theorem
provides $\lambda\in E'$ such that
$q(b)=\lambda(b)$ and $\lambda(B^q_1(0))\sub [{-1},1]$.
Thus
$q(b)=\lambda(b)=\int_X\lambda (\gamma(x))\,d\mu
\leq \int_X|\lambda (\gamma(x))|\,d\mu
\leq \int_X q(\gamma(x))\,d\mu=\|\gamma\|_{\cL^1,q}
\leq \|\gamma\|_{\cL^\infty,q}\mu(X)$.\,\vspace{1.3mm}\Punkt

\noindent
Let $\|.\|:=\|.\|_2$ be the euclidean norm on $\R^k$ and let $B_r(x):=\{y\in \R^k
\colon \|y-x\|<r\}$ for $x\in \R^k$ and $r>0$.
Abbreviate $B_r:=B_r(0)$.
Then $\lambda_k(B_r(x))=\lambda_k(B_r)$
for all $x\in \R^k$.
The following discussion of Lebesgue points and
absolutely continuous functions
was inspired by the treatment of the scalar-valued case in
\cite[\S7]{Ru1}.
\begin{defn}
Let $E$ be a Fr\'{e}chet space
and $\gamma\in \cL^1(\R^k,\lambda_k,E)$.
A point $x\in \R^k$ is called a \emph{Lebesgue point}
of~$\gamma$ if
\[
\lim_{r\to 0}
\frac{1}{\lambda_k(B_r)}
\int_{B_r(x)}p(\gamma(y)-\gamma(x))\,d\lambda_k(y)\,=\,0
\]
for each continuous seminorm~$p$ on~$E$.
\end{defn}
\begin{rem}
If $x\in \R^k$ is a Lebesgue point for
$\gamma\in \cL^1(\R^k,\lambda_k,E)$, then
\[
\lim_{r\to 0}\frac{1}{\lambda_k(B_r)}
\int_{B_r(x)}\gamma(y)\,d\lambda_k(y)\,=\, \gamma(x)
\]
in particular, as
\begin{eqnarray*}
\lefteqn{p\left(
\frac{1}{\lambda_k(B_r)}
\int_{B_r(x)}\gamma(y)\,d\lambda_k(y) - \gamma(x)\right)}\qquad\qquad\\
&=&
p\left(
\frac{1}{\lambda_k(B_r)}
\int_{B_r(x)}(\gamma(y)- \gamma(x))\, d\lambda_k(y)\right)\\
&\leq&
\frac{1}{\lambda_k(B_r)}
\int_{B_r(x)}p(\gamma(y)-\gamma(x))\,d\lambda_k(y)\to 0
\end{eqnarray*}
as $r\to 0$, for each continuous seminorm $p$ on~$E$.
\end{rem}
Lemma~\ref{Lebesguenshri}
implies that the same Lebesgue points
are obtained if the euclidean norm is replaced with
any norm $\|.\|$ on $\R^k$.
\begin{defn}
Let $\rho\colon \R^k\to[0,\infty]$
be a measurable function.
We write $M_\rho\colon \R^k\to[0,\infty]$
for the \emph{maximal function} of the measure $\rho d\lambda_k$,
defined via
\[
M_\rho(x):=\sup_{r\in \,]0,\infty[}\frac{1}{\lambda_k(B_r)}\int_{B_r(x)}\rho(y)\,d\lambda_k(y).
\]
Then $M_\rho$ is lower semicontinuous and hence
Borel measurable (see \cite[7.2]{Ru1}).
\end{defn}
\begin{la}\label{Leblarge}
If $E$ is a Fr\'{e}chet space and
$\gamma\in \cL^1(\R^k,\lambda_k,E)$,
then the set
$L_\gamma$ of all Lebesgue points of $\gamma$
is a Borel set in $\R^k$,
and $\lambda_k(\R^k\setminus L_\gamma)=0$.
\end{la}
\begin{proof}
Let $q_1\leq q_2\leq\cdots$ be an ascending sequence
of continuous seminorms on~$E$ defining the locally convex vector
topology of~$E$. Let $j\in \N$.
For $r>0$, the map $h_r\colon E \times \R^k \to \cL^1(\R^k,E)$,
$h_r(v,x):=v\one_{B_r(x)}$ is continuous,
and also the map
\[
\R^k\to \cL^1(\R^k,E),\quad x\mto \gamma\cdot \one_{B_r(x)}
\]
is continuous (exploiting that $\lambda_k(B_r(y)\Delta B_r(x))\to 0$
as $r\to0$, where
$A\Delta  B:=(A\cup B)\setminus (A\cap B)$
denotes the symmetric difference).
Hence
\[
\R^k\to \cL^1(\R^k,E),\;\;
x \mto \gamma(x)\one_{B_r(x)}-\gamma\cdot \one_{B_r(x)}
=h_r(\gamma(x),x)-\one_{B_r(x)}
\]
is a measurable map. Since $\|.\|_{\cL^1,q_j}$ is a continuous
seminorm, we deduce that the map
\[
g_{j,r}\colon \R^k\to \R,\;\,x\mto \|\gamma(x)\one_{B_r(x)}
-\gamma\cdot\one_{B_r(x)}\|_{\cL^1,q_j}
=\int_{B_r(x)}q_j(\gamma(x)-\gamma(y))\,d\lambda_k(y)
\]
is measurable.
Hence also $f_{j,r}\colon \R^k\to \R$,
\[
f_{j,r}(x):=\frac{1}{\lambda_k(B_r)}
\int_{B_r(x)}q_j(\gamma(x)-\gamma(y))\,d\lambda_k(y)
\]
is measurable. We shall later write $f_{j,r,\gamma}:=f_{j,r}$
to emphasize the dependence on~$\gamma$.
For fixed $x\in \R^k$, the map
\[
]0,\infty[\,\to\, [0,\infty[,\quad
r\mto g_{j,r}(x)=\int_{\R^k}q_j(\gamma(x)\one_{B_r(x)}-
\gamma(y)\one_{B_r(y)})\,d\lambda_k(y)
\]
is continuous, exploiting that $\lambda_k(B_r(x))$
is continuous in~$r$.
Hence also the map
\[
]0,\infty[\,\to\, [0,\infty[,\quad
r\mto f_{j,r}(x)=\frac{1}{\lambda_k(B_r)}
\int_{\R^k}q_j(\gamma(x)\one_{B_r(x)}-
\gamma(y)\one_{B_r(y)})\,d\lambda_k(y)
\]
is continuous and thus
\[
T_{j,r}(x):=\sup\{f_{j,s}(x)\colon s\in \,]0,r]\}
=\sup\{f_{j,s}(x)\colon s\in \,]0,r]\cap\Q\}.
\]
Being the pointwise supremum of a countable
family of measurable functions, the function
\[
T_{j,r}\colon \,]0,\infty[\,\to [0,\infty]
\]
is measurable. By definition, $T_t\leq T_r$ if $t\leq r$.
Hence
\begin{eqnarray*}
L_\gamma &=& \{x\in \R^k\colon (\forall j\in \N)\,\mbox{$T_{j,r}(x)\to 0$ as $r\to 0$}\}\\
&=& \bigcap_{j\in \N}\{x\in \R^k\colon \mbox{$T_{j,1/n}(x)\to 0$ as $n\to \infty$}\}.
\end{eqnarray*}
As the set $A_j$ of all $x\in \R^k$
such that $(T_{j,1/n}(x))_{n\in \N}$ converges is a Borel set
(see Lemma~\ref{limmea}\,(b))
and $T_{j,\gamma} \colon A_j\to [0,\infty]$, $x\mto\lim_{n\to\infty}T_{j,1/n}(x)$
is a measurable map, we see that
$L_\gamma=\bigcap_{j\in \N} T^{-1}_{j,\gamma}(\{0\})$
is a Borel set in~$\R^k$.

Let $\ve>0$. Let $j,n\in \N$.
By Lemma~\ref{labasemea}\,(c),
there exists $\eta\in C_c(\R^k,E)$ such that
$\|\gamma-\eta\|_{\cL^1,q_j}<\frac{1}{n}$. Put $\xi:=\gamma-\eta$.
Since $\eta$ is continuous, $T_{j,\eta}=0$.
Because
\[
f_{j,r,\zeta}(x)\leq \frac{1}{\lambda_k(B_r)}\int_{B_r(x)}q_j(\zeta(y))\,d\lambda_k(y)
+p_j(\zeta(x)),
\]
we have $T_{j,\zeta}\leq M_{q_j\circ \zeta}+ p_j\circ \zeta$.
Since $f_{j,r,\gamma}\leq f_{j,r,\eta}+f_{j,r,\xi}$, it follows that
\[
T_{j,\gamma}\leq M_{p_j\circ\xi}+p_j\circ \xi.
\]
Therefore
\begin{eqnarray*}
\lefteqn{\{x\in \R^k\colon T_{j,\gamma}(x)>2\ve\}}\qquad\qquad\\
&\sub & \{x\in \R^k\colon
M_{p_j\circ \xi}(x)>\ve\}\cup\{x\in \R^k\colon p_j(\xi(x))>\ve\}=:S_{\ve,n}.
\end{eqnarray*}
Since $\|p_j\circ \xi\|_{\cL^1}=\|\xi\|_{\cL^1,p_j}<\frac{1}{n}$,
\cite[7.5\,(1) and Theorem 7.4]{Ru1}
show that
\[
\lambda_k(S_{\ve,n})\leq \frac{3^k}{\ve}\|p_j\circ\xi\|_{\cL^1}
+\frac{1}{\ve}\|p_j\circ\xi\|_{\cL^1}\leq
\frac{3^k+1}{\ve n}.
\]
Hence $\lambda_k(\{x\in \R^k\colon T_{j,\gamma}(x)>2\ve\})
\leq\frac{3^k+1}{\ve n}$.
As $n$ was arbitrary, $\lambda_k(\{x\in \R^k\colon T_{j,\gamma}(x)>2\ve\})=0$
follows, and we deduce that
\[
\lambda_k(\{x\in \R^k\colon T_{j,\gamma}(x)>0\})=0.
\]
Thus also $\R^k\setminus L_\gamma=\bigcup_{j\in \N}\{x\in \R^k\colon T_{j,\gamma}(x)>0\}$
has measure zero.
\end{proof}
The following concept is well known (see, e.g., \cite[7.9]{Ru1}:
\begin{defn}
Let $x\in \R^k$. A sequence $(A_n)_{n\in \N}$ of Borel sets in $\R^k$
is said to shrink to $x$ nicely
if there exist $\alpha>0$ and a sequence of balls $B_{r_n}(x)$ with $r_n\to 0$
such that $A_n\sub B_{r_n}(x)$ for all $n\in \N$ and
\[
\lambda_k(E_n)\geq \alpha\lambda_k(B_{r_n}(x)).
\]
\end{defn}
Analogous to \cite[7.10]{Ru1}, also in the vector-valued case we have:
\begin{la}\label{Lebesguenshri}
Let $E$ be a Fr\'{e}chet space
and $\gamma\in \cL^1(\R^k,E)$.
For each $x\in \R^k$, let $(A_n(x))_{n\in \N}$
be a sequence of Borel sets which shrinks to~$x$ nicely.
Then
\[
\gamma(x)=\lim_{n\to\infty}\frac{1}{\lambda_k(A_n(x))}\int_{A_n(x)}\gamma\,d\lambda_k
\]
at every Lebesgue point $x$ of $\gamma$,
and thus $\lambda_k$-almost everywhere.
\end{la}
\begin{proof}
Let $x$ be a Lebesgue point of~$\gamma$ and $\alpha>0$ and $(B_{r_n}(x))_{n\in \N}$
be the positive number and the balls associated to the sequence
$(E_n(x))_{n\in \N}$. Let $q$ be a continuous seminorm on~$E$.
Then
\begin{eqnarray}
\lefteqn{\frac{1}{\lambda_k(A_n(x))}\int_{A_n(x)}q(\gamma(y)-\gamma(x))\,d\lambda_k(y)}
\qquad \notag \\
&\leq&
\frac{1}{\alpha \lambda_k(B_{r_n}(x))}\int_{B_{r_n}(x)}q(\gamma(y)
-\gamma(x))\,d\lambda_k(y).\label{righlef}
\end{eqnarray}
Since $x$ is a Lebesgue point of~$\gamma$, the right hand side
of (\ref{righlef}) tends to~$0$ as $n\to\infty$,
entailing that also the left hand side tends to~$0$.
\end{proof}
{\bf Proof of Lemma~\ref{funda}}
(compare, e.g., \cite[Theorem 7.11]{Ru1}
for the well-known scalar-valued case).
We may assume that $J=\R$ (as we can extend
$\gamma$ by $0$ outside~$J$).
The weak integrals needed to define $\eta$
exist by (\ref{labasemea}).
To complete the proof,
we need only show
that $\eta$ is differentiable with derivative $\gamma(x)$
at each Lebesgue point~$x$
of~$\gamma$ (recalling Lemma~\ref{Leblarge}).
Since
\[
p(\eta(y)-\eta(x))\leq \left| \int_x^y p(\gamma(t))\,d\lambda_1(t)\right|
\to 0
\]
as $y\to x$ for each continuous seminorm $p$ on $E$,
the map $\eta$ is continuous.
If $x\in \R$ and
$(t_{x,n})_{n\in \N}$ is a sequence of real number
$t_{x,n}>x$ converging to~$x$,
then the sets $A_n(x):=[x,t_{x,n}]$
shrink nicely to~$x$,
entailing that
\[
\frac{\eta(t_{x,n})-\eta(x)}{t_{x,n}-x}=
\frac{1}{\lambda_1([x,t_{x,n}])}\int_{[x,t_{x,n}]}\gamma\,d\lambda_1
\to \gamma(x)
\]
at each Lebesgue point $x$ of~$\gamma$.
We have shown that right-sided derivative of~$\eta$
exists at~$x$ and equals~$\gamma(x)$.
Likewise, if $(s_{x,n})_{n\in \N}$ is a sequence of real number $s_n<x$
converging to $x$, then the sets $[s_{x,n},x]$
shrink to~$x$ nicely and thus
\[
\frac{\eta(x)-\eta(s_{x,n})}{x-s_{x,n}}=
\frac{1}{\lambda_1([s_{x,n},x])}\int_{[s_{x,n},x]}\gamma\,d\lambda_1
\to \gamma(x)
\]
at each Lebesgue point~$x$.
Hence $\gamma(x)$ is also the left-sided derivative of~$\eta$
at~$x$. As a consequence, $\eta$ is differentiable at $x$
and $\eta'(x)=\gamma(x)$.\,\vspace{1.3mm}\Punkt

\noindent
{\bf Proof of Lemma~\ref{funda2}.}
Let $\gamma_j\in \cL^\infty_{rc}(J,E)$
for $j\in \{1,2\}$ such that
$\eta(t):=\int_{t_0}^t\gamma_1(s)\,ds
=\int_{t_0}^t\gamma_2(s)\,ds$ for all $t\in J$.
Since $E$ is integral complete
and hence has the metric CCP,
the set $K:=\wb{\conv(\wb{\im(\gamma_1)}\cup\wb{\im(\gamma_2)})}\sub E$
is metrizable and compact (see Lemma~\ref{indeedmetr}).
By Lemma~\ref{imagesepmet},
there is a metrizable vector topology~$\cO'$
on $F:=\Spann(K)$
which is coarser than the topology induced by~$E$.
As $(F,\cO')$ and $E$ induce the same topology on~$K$,
we have
$\gamma_j\in \cL^\infty_{rc}(J,F)$
for $j\in \{1,2\}$.
By the proof of Lemma~\ref{labaserc},
the weak integrals
$\int_{t_0}^t\gamma_1(s)\,ds$ and
$\int_{t_0}^t\gamma_2(s)\,ds$ also exist in~$F$,
for all $t\in J$, and coincide with
$\eta(t)$.
Let $\wt{F}$ be a completion of~$F$
such that $F\sub \wt{F}$
and consider
$\gamma_j$ as an element of $\cL^\infty_{rc}(J,\wt{F})$
for $j\in \{1,2\}$.
Since $\wt{F}$ is a Fr\'{e}chet space,
we deduce with Lemma~\ref{funda} that $[\gamma_1]=[\gamma_2]$
in $L^\infty_{rc}(J,\wt{F})$.
As a consequence, $[\gamma_1]=[\gamma_2]$ also in $L^\infty_{rc}(J,F)$
and hence also in $L^\infty_{rc}(J,E)$.\,\Punkt
\begin{la}
Let $E$ be a Fr\'{e}chet space, $J\sub \R$ a non-degenerate
interval, $\eta\colon J\to E$ be a continuous
map and $A$ the set of all $t\in J$
such that $\gamma$ is differentiable at~$t$.
Then $A$ is a Borel set in $J$.
\end{la}
\begin{proof}
Let $p_1\leq p_2\leq\cdots$ be a sequence of continuous
seminorms on~$E$ defining its locally convex vector topology.
For $j\in \N$ and $\ve>0$, let
$h_{j,\ve}(t)$ be the supremum
of the real numbers
\begin{equation}\label{willfixs}
p_j\left(\frac{\eta(s_1)-\eta(t)}{s_1-t}
-\frac{\eta(s_2)-\eta(t)}{s_2-t}
\right)=:g_{j,s_1,s_2}(t),
\end{equation}
for $s_1, s_2\in (J\cap\,]x-\ve,x+\ve[)\setminus\{t\}$.
Note that (\ref{willfixs})
is a continuous function of $s_1$ and $s_2$.
Therefore the same supremum is obtained if
we take $s_1, s_2\in \Q\cap (J\cap\,]x-\ve,x+\ve[)\setminus\{t\}$
instead. If we set $g_{j,s_1,s_2}(t):=0$ if $t\in \{s_1,s_2\}$,
then
\[
g_{j,s_1,s_2}\colon J\to [0,\infty[
\]
is a measurable function.
Since $h_{j,\ve}$ is the supremum of
these for countably many $(s_1,s_2)$
as just described, also $h_{j,\ve}$ is measurable.
Now $x\in A$ if and only if $(\frac{\gamma(s)-\gamma(t)}{s-t})_{s\not=t}$
is a Cauchy net (using the preorder given by
$s_1\precsim s_2$ if and only if $|s_2-x|<|s_1-x|$
to make $J\setminus\{t\}$ a directed set).
The latter holds if and only if $h_{j,\ve}(t)\to 0$
as $\ve\to0$, for all $j\in \N$.
Since $h_{j,\ve}(t)$ is a decreasing functions of~$\ve$,
equivalently $h_{j,1/n}(t)\to 0$ as $n\to\infty$, for all $j\in \N$.
Thus
\[
A=\bigcap_{j\in \N}\Big\{t\in J\colon \lim_{n\to\infty}\, h_{j,1/n}(t)=0\Big\},
\]
entailing that $A$ is a Borel set.
\end{proof}
%
%
%
%
%
%
%
%
%
%
{\bf Proof of Lemma~\ref{stepswithin}.}
Let $\gamma\in \cR([a,b],E)$ and $(\eta_n)_{n\in \N}$
be a sequence in $\cT([a,b],E)$
such that $\eta_n\to\gamma$ uniformly.
By Lemma~\ref{imagesepmet},
there is a sequence $q_1\leq q_2\leq\cdots$ of continuous seminorms
on~$E$ such that the vector topology $\cO'$ on $F:=\Spann(\wb{\gamma([a,b])})$
defined by $(q_n|_F)_{n\in \N}$
is Hausdorff. Then $E$ and $(F,\cO')$ induce the same topology on
$K:=\wb{\gamma([a,b])}$.
After increasing the seminorms
if necessary, we may assume that $2q_n\leq q_{n+1}$
for each $n\in \N$.
Hence, for each $x\in K$
\begin{equation}\label{yeah1}
K\cap B^{q_n}_2(x)\quad\mbox{with $n\in \N$}
\quad \mbox{is a basis of neighbourhoods of~$x$ in~$K$.}
\end{equation}
For each $n\in \N$,
there is $m_n\in \N$ such that
\begin{equation}\label{enabapprox}
\sup_{t\in [a,b]}q_n(\gamma(t)-\eta_k(t))\;\leq 1\quad\mbox{for all $k\geq m_n$.}
\end{equation}
We have $\eta_{m_n}([a,b])=\{y_{n,1},\ldots, y_{n,\ell_n}\}$ for some $\ell_n\in \N$
and pairwise distinct elements $y_{n,1},\ldots, y_{n,\ell_n}\in E$.
By (\ref{enabapprox}), we find $z_{n,j}\in \gamma([a,b])$
for $j\in \{1,\ldots,\ell_n\}$ such that $q_n(y_{n,j}-z_{n,j})\leq 1$.
Define $\gamma_n\colon [a,b]\to E$ via
\[
\gamma_n(t):=z_{n,j}\quad\mbox{if $\eta_{m_n}(t)=y_{n,j}$.}
\]
Then $\gamma_n\in \cT([a,b],E)$,
$\gamma_n([a,b])\sub \gamma([a,b])$
and $\sup_{t\in [a,b]}q_n(\gamma(t)-\gamma_n(t))\leq 2$.
By (\ref{yeah1}), this implies
that $\gamma_n\to \gamma$ uniformly.\,\vspace{1.3mm}\Punkt

\noindent
{\bf Proof of Lemma~\ref{psemiondsum}.}
Write $Q$ for the map described in (\ref{formulasemn}).
Since
\[
Q|_{E_1\times\cdots E_N}\colon (x_1,\ldots, x_N)\mto \left(\sum_{n=1}^N q_n(x_n)^p\right)^{1/p}
\]
(resp.\ $(x_1,\ldots, x_N)\mto \max\{q_1(x_1),\ldots, q_N(x_N)\}$) is a continuous seminorm on $E_1\times\cdots\times E_N$
for each $N\in\N$, and $\bigoplus_{n\in\N}E_n=\dl(E_1\times\cdots\times E_N)$\vspace{-.5mm},
we deduce that $Q$ is a continuous seminorm on the direct sum. Thetrefore, the topology defined by the seminorms $Q$
is coarser than the locally convex direct sum topology.
That the topologies coincide if $p=\infty$ is well-known and corresponds to the fact
that the topology on a countable direct sum is the box topology.
If $p<\infty$, then $Q_\infty\leq Q_p$ holds pointwise for the seminorms
defined by (\ref{formulasemn}), applied with $\infty$ (in place of $p$) and $p$,
respectively. The topology defined by the $Q_p$ is therefore finer than the
topology defined by the $Q_\infty$, which is the locally convex direct sum topology.
As it is also coarser, the topologies coincide.\,\Punkt\vspace{2.3mm}

\noindent
{\bf Proof of Lemma~\ref{pbasisindl}.}
It is well-known that the summation map $\bigoplus_{n\in \N}E_n\to E$
is a topological quotient map, if we use the locally convex direct sum topology on the left.
Since linear quotient maps between locally convex spaces are open maps,
the assertion follows from Lemma~\ref{psemiondsum}.\,\Punkt\vspace{2.3mm}

\noindent
{\bf Proof of Lemma~\ref{metrinsep}.}
Let $D$ be a countable dense subset of $E$
and $\{U_n\colon n\in\N\}$ be a basis of $0$-neighbourhoods
for~$F$.
For each $n\in\N$, there is a continuous seminorm $q_n$ on~$E$
such that $F\cap B^{q_n}_1(0)\sub U_n$.
After replacing $q_n$ with the pointwise maximum $\max\{q_1,\ldots, q_n\}$,
we may assume that $q_1\leq q_2\leq\cdots$.
Now $\{q_n|_F\colon n\in\N\}$ is a directed set of seminorms on $F$ defining its topology.
For $n,m\in\N$, let
\[
D_{n,m}:=\{y\in D\colon B^{q_n}_{1/m}(y)\cap F\not=\emptyset\}.
\]
For $y\in D_{n,m}$, pick $z_{n,m,y}\in B^{q_n}_{1/m}(y)\cap F$.
Then
\[
D_F:=\{z_{n,m,y}\colon n,m\in\N,y\in D_{n,m}\}
\]
is a countable subset of $F$. To see that $D_F$ is dense in $F$, it suffices to show that
\[
B^{q_n}_{2/m}(x)\cap D_F
\]
is non-empty for all $x\in F$ and $n,m\in \N$.
By density of $D$ in $E$, we find $y\in B^{q_n}_{1/m}(x)\cap D$.
Then $x\in B^{q_n}_{1/m}(y)\cap F$, entailing that $y\in D_{n,m}$.
Now $z_{n,m,y}\in D_F$ is an element such that $z_{n,m,y}\in B^{q_n}_{1/m}(y)$,
whence
\[
q_n(x-z_{n,m,y})\leq q_n(x-y)+q_n(y-z_{n,m,y})< \frac{1}{m}+\frac{1}{m}=\frac{2}{m}
\]
and thus $z_{n,m,y}\in B^{q_n}_{2/m}(x)\cap D_F$.\,\Punkt\vspace{2.3mm}

\noindent
{\bf Proof of \ref{Lpfepvsp}.}
Let $\gamma_1,\gamma_2\in \cL^p(X,\mu,E)$.
Then $\Spann(\gamma_1(X)\cup\gamma_2(X))$
is a separable vector subspace of~$E$,
entailing that
\[
\wb{\Spann(\gamma_1(X)\cup\gamma_2(X))}=\bigcup_{n\in \N}F_n
\]
with topological vector subspaces $F_1\sub F_2\sub\cdots$ of~$E$
which are eseparable Fr\'{e}chet spaces.
Since $F_n$ is closed in~$E$ and hence a Borel set, we deduce that
\[
A_n:=(\gamma_1)^{-1}(F_n)\cap (\gamma_2)^{-1}(F_n)\in\Sigma
\]
for all $n\in \N$. Moreover, $A_1\sub A_2\sub\cdots$
with $\bigcup_{n\in\N}A_n=X$. Since $F_n$ is a separable
metric space and the addition map $\alpha_n\colon F_n\times F_n\to F_n$
is continuous, we see with Lemma~\ref{Sepp}\,(b)
that
\[
\gamma_1|_{A_n} +\gamma_2|_{A_n}=\alpha_n\circ (\gamma_1|_{A_n}^{F_n},
\gamma_2|_{A_n}^{F_n})
\]
is a measurable map to $F_n$ and hence to~$E$,
for each~$n$. Since $(\gamma_1+\gamma_2)|_{A_n}$
is measurable and $(A_n)_{n\in \N}$ is a countable cover of~$X$
by measurable sets, we deduce that $\gamma_1+\gamma_2$ is measurable.
Now $\im(\gamma_1+\gamma_2)=\bigcup_{n\in \N} (F_n\cap \im(\gamma_1+\gamma_2))$
is a countable union of separable sets and hence
separable. Moreover, $\|\gamma_1+\gamma_2\|_{\cL^p,q}\leq
\|q\circ \gamma_1+q\circ \gamma_2\|_{\cL^p}\leq
\|q\circ \gamma_1\|_{\cL^p}
\|q\circ \gamma_2\|_{\cL^p}=\|\gamma_1\|_{\cL^p,q}+\|\gamma_2\|_{\cL^p,q}<\infty$
for each continuous seminorm $q\in P(E)$. Hence
$\gamma_1+\gamma_2\in \cL^p(X,\mu,E)$.
If $\gamma\in \cL^p(X,\Sigma,\mu)$ and $\|\gamma\|_{\cL^p,q}=0$
for each $q\in P(E)$,
pick separable Fr\'{e}chet spaces $F_1\sub F_2\cdots$ in~$E$
with
\[
\wb{\Spann (\gamma(X))}=\bigcup_{n\in \N}F_n.
\]
Then $\|\gamma|_{\gamma^{-1}(F_n)}\|_{\cL^p,q}=0$
for all $q\in P(F_n)$,
entailing that $\gamma_{\gamma^{-1}(F_n)}\colon \gamma^{-1}(F_n)\to F_n$
is $0$ outside a set $N_n\sub X$ of measure $\mu(N_n)=0$.
Then $\mu(\bigcup_{n\in \N}N_n)=0$ and $\gamma(x)=0$
for all $x\in X\setminus\bigcup_{n\in \N}N_n$.
Hence $[\gamma]=0$ and thus $L^p(X,\mu,E)$
is Hausdorff.\,\Punkt\vspace{2.3mm}

\noindent
\begin{numba}
Let $M$ be a smooth manifold modelled
on a locally convex space~$X$ and $F$ be a locally convex space.
Recall that a smooth vector bundle over~$M$, with typical
fibre~$F$ is a smooth manifold~$E$, together
with a surjective smooth map $\pi\colon E\to M$
and a vector space structure on
\[
E_x:=\pi^{-1}(\{x\})
\]
for each $x\in M$, such that for each $x\in M$
there exists an open neighbourhood
$U$ of~$x$ in~$M$ and
a so-called local trivialization $\theta\colon E_U\to U\times F$,
with $E_U:=\pi^{-1}(U)$. Thus $\theta$ is a $C^\infty$-diffeomorphism
such that
\[
\pr_1\circ \theta=\pi
\]
(where $\pr_1\colon U\times F\to U$, $(x,y)\mto x$)
and $\pr_2\circ \theta|_{E_x}\colon E_cx\to F$ is linear
(and hence an isomorphism of topological vector spaces)
for each $x\in U$,
where $\pr_2\colon U\times F\to F$, $(x,y)\mto y$.
\end{numba}
\begin{numba}\label{deftosec}
If $\pi\colon E\to M$ is a smooth vector bundle
with typical fibre~$F$
A smooth section of~$E$ is a smooth map
$\sigma\colon M\to E$ such that $\pi\circ \sigma=\id_M$.
We write $\Gamma(E)$ for the vector space
of all smooth sections of~$E$ (with pointwise addition
and multiplication with scalars).
We endow $\Gamma(E)$ with the vector topology
making the map
\[
\Gamma(E)\to \prod_{\theta}C^\infty(U_\theta,F),\quad
\sigma\mto(\pr_2\circ \theta\circ \sigma|_{U_\theta})
\]
a topological embedding onto a closed vector subspace
(for $\theta\colon E_{U_\theta}\to U_\theta\times F$
ranging through the set of all local trivializations
of~$E$).\\[2.3mm] Now assume that $M$ is paracompact and finite-dimensional.
If $K\sub M$ is a compact set,
we write $\Gamma_K(E)$ for the closed vector subspace
of all $\sigma\in \Gamma(E)$ such that $\sigma(x)=0\in E_x$ for all
$x\in M\setminus K$.
We endow
\[
\Gamma_c(E)=\bigcup_{K}\Gamma_K(E)
\]
with the locally convex direct limit topology
(for $K$ ranging through the set of all
compact subsets of~$M$).
As the inclusion map $\Gamma_K(E)\to \Gamma(E)$
for each compact set $K\sub M$
is continuous linear, also the inclusion map
$\Gamma_c(E)\to \Gamma(E)$
is continuous. Since $\Gamma(E)$ is Hausdorff,
we deduce that $\Gamma_c(E)$ is Hausdorff.
\end{numba}
\begin{numba}
If $\pi\colon E\to M$ is a smooth vector bundle
and $U\sub M$ is an open subset, then $E|_U:=\pi^{-1}(U)\sub E$
with the given vector space structure on each fibre
and the restriction $\pi|_{E|_U}\colon E|_{U}\to U$.
If $K\sub M$ is compact and $K\sub U$,
then the linear map
\[
\Gamma_K(E)\to \Gamma_K(E|_U),\quad \sigma\mto\sigma|_U
\]
is a bijection and in fact an isomorphism
of topological vector spaces (as is clear from the definition
of the topologies).
Hence also the inverse map
\begin{equation}\label{exbyzer}
\Gamma_K(E|_U)\to \Gamma_K(E),\quad \sigma\mto \wt{\sigma}
\end{equation}
(with $\wt{\sigma}(x)=\sigma(x)$ if $x\in U$,
$\wt{\sigma}(x)=0$ if $x\in M\setminus K$)
is an isomorphism of topological vector spaces.
\end{numba}
\begin{la}\label{toviasum}
Let $F$ be a locally convex space and $\pi\colon E\to M$ be a smooth vector bundle
over a finite-dimensional paracompact smooth manifold~$M$,
with typical fibre~$F$.
Let $(U_j)_{j\in J}$ be a locally finite cover of~$M$
by open, relatively compact subsets $U_j\sub M$.
Then the map
\[
\Phi\colon \Gamma_c(E)\to \bigoplus_{j\in J}\Gamma(E|_{U_j}),\quad
\sigma\mto (\sigma|_{U_j})_{j\in J}
\]
is a topological embedding onto a closed
vector subspace which is complemented
in the direct sum as a topological vector space.
\end{la}
\begin{proof}
It is clear from the definition of the topology that the linear map
\[
\Gamma(E)\to \Gamma(E|_{U_j}),\quad \sigma\mto \sigma|_{U_j}
\]
is continuous for each $j\in J$. Hence also the restriction to
$\Gamma_K(E)$ is continuous linear for each compact
set $K\sub M$, entailing that the linear map $\Phi$ is continuous.
Pick a partition of unity $(h_j)_{j\in J}$
subordinate to $(U_j)_{j\in J}$,
with compact supports $K_j:=\wb{\{x\in M\colon h_j(x)\not=0\}}$.
Then the multiplication operator
\[
\Gamma(E|_{U_j})\to \Gamma_{K_j}(E|_{U_j}),\quad \sigma\mto h_j\,\sigma
\]
is linear and continuous
(see \cite{ZOO} or \cite{SEC})
and hence also the map
\[
\alpha_j\colon \Gamma(E|_{U_j})\to \Gamma(E),\quad \sigma\mto (h_j\sigma)\wt{\;}
\]
is continuous linear (with notation as in (\ref{exbyzer})).
This map takes its values in $\Gamma_K(E)$,
which injects continuously in $\Gamma_c(E)$.
We can therefore consider $\alpha_j$ as a continuous linear map
to $\Gamma_c(E)$.
By the universal property of the locally convex direct sum,
also the map
\[
\alpha\colon \bigoplus_{j\in J}\Gamma(E|_{U_j})\to \Gamma_c(E),\quad
(\sigma_j)_{j\in J}\mto \sum_{j\in J}\alpha_j(\sigma_j)
\]
is continuous linear.
Now $\alpha\circ \Phi=\id$, entailing that $\Phi$ is a topological
embedding and
\[
\bigoplus_{j\in J}\Gamma(E|_{U_j})=\im(\Phi)\oplus \ker(\alpha)
\]
as a topological vector space. Notably,
$\im(\Phi)$ is closed in $\bigoplus_{j\in J}\Gamma(E|_{U_j})$.
\end{proof}
\begin{rem}\label{alsnee}
The same argument applies if smooth sections are replaced with $C^k$-sections
with $k\in \N_0$. In particular, the map
\[
C^k_c(M,E)\to\bigoplus_{j\in J}C^k(U_j,E),\quad \gamma\mto (\gamma|_{U_j})_{j\in J}
\]
is a linear topological embedding onto a closed and complemented
topological vector subspace for each $k\in \N_0\cup\{\infty\}$, locally convex space $E$,
paracompact finite-dimensional smooth manifold~$M$
and locally finite cover $(U_j)_{j\in J}$ of~$M$
by relatively compact, open subsets $U_j\sub M$.
\end{rem}
{\bf Proof of Lemma~\ref{lafep}.}
(a)
If $S\sub\bigoplus_{j\in J}E_j=:E$
is a separable closed vector subspace,
let $D\sub S$ be a dense countable subset.
Since $D$ is countable, there exists
a countable subset $J_0\sub J$ such that $D\sub\bigoplus_{j\in J_0}E_j$.
As the projections $\pr_i\colon E\to E_i$, $(x_j)_{j\in J}\mto x_i$
are continuous linear, $F:=\bigoplus_{j\in J_0}E_j=\bigcap_{i\in J\setminus J_0}\ker\pr_i$
is a closed vector subspace of $E$ with $D\sub F$
and hence $S\sub F$. Now $S_i:=\wb{\pr_i(S)}\sub E_i$
is a separable closed vector subspace
for each $i\in J_0$
and hence $S_i=\bigcup_{k\in \N}F_{i,k}$
with separable Fr\'{e}chet subspaces $F_{i,1}\sub F_{i,2}\sub\cdots$
of~$E_i$, as $E_i$ has the (FEP).
If $J_0$ is infinite, enumerate $J_0=\{j_1,j_2,\ldots\}$;
if $J_0$ has a finite number, $m$, of elements,
write $J_0=\{j_1,\ldots, j_m\}$.
Then
\[
Y:=\bigoplus_{j\in J_0}S_j=\bigcap_{i\in J\setminus J_0}
\ker(\pr_i)\cap\bigcap_{j\in J_0}
(\pr_j)^{-1}(S_j)
\]
is a closed vector subspace which contains $D$
and thus $S \sub Y
=\bigcup_{n\in \N}Y_n$
with the separable Fr\'{e}chet spaces
$Y_n:=F_{j_1,n}\times \cdots\times F_{j_n,n}\sub \bigoplus_{j\in J}E_j$
(if $J_0$ is infinite), resp.,
\[
Y_n:=F_{j_1,n}\times\cdots\times F_{j_m,n}
\]
(if $J_0$ is finite).
Then $(Y_n\cap S)_{n\in \N}$
is an ascending sequence of separable Fr\'{e}chet subspaces
of~$\bigoplus_{j\in J}E_j$ with $S=\bigcup_{n\in \N} (Y_n\cap S)$.\\[2.3mm]
If $\gamma=(\gamma_j)_{j\in J}\in \cL^p(X,\mu, E)$,
then $S:=\wb{\Spann(\gamma(X))}$ is a closed
vector subspace of~$E$
and $S\sub \bigoplus_{j\in J_0}E_j$
with some countable subset $J_0\sub J$.
Then
\[
J_1:=\{j\in J_0\colon \mu((\gamma_j)^{-1}(E_j\setminus\{0\}))>0\}
\]
is a finite set, from which $L^p(X,\mu,E)=\bigoplus_{j\in J}L^p(X,\mu,E_j)$
follows. In fact, if $J_1$ was infinite,
we could choose a bijection $\N\to J$, $n\to j_n$.
Since $[\gamma_{j_n}]\not=0$, we would find a continuous
seminorm $q_{j_n}$ on $E_{j_n}$ such that $\|\gamma_{j_n}\|_{\cL^p,q_{j_n}}>0$;
after replacing $q_{j_n}$ with a multiple, we may assume that
\[
\|\gamma_{j_n}\|_{\cL^p,q_{j_n}}\geq n
\]
for all $n\in \N$. For $j\in J\setminus J_1$,
choose any $q_j\in P(E_j)$.
Then
\[
q\colon E\to [0,\infty[,\quad q((y_j)_{j\in J}):=\sum_{j\in J}q_j(y_j)
\]
is a continuous seminorm on~$E$.
For each $n\in \N$, we have
\[
\|\gamma\|_{\cL^p,q}
=\sqrt[p]{\int_X q(\gamma(x))^p\,d\mu(x)}
\geq \sqrt[p]{\int_X q_{j_n}(\gamma(x_{j_n}))^p\,d\mu(x)}
=\|\gamma_{j_n}\|_{\cL^p,q_{j_n}}\geq n.
\]
Hence $\|\gamma\|_{\cL^p,q}=\infty$, contradicting $\gamma\in \cL^p(X,\mu,E)$.\\[2.3mm]
As each inclusion map
\[
L^p(X,\mu,E_j)\to L^p(X,\mu, E)
\]
is continuous and linear, the universal property of the locally convex direct sums
provides the continuity of the linear summation map
\[
\Sigma \colon \bigoplus_{j\in J}L^p(X,\mu, E_j)\to L^p(X,\mu,E),\; (\gamma_j)_{j\in J}\mto\sum_{j\in J}\gamma_j.
\]
If $J$ is countable or $p=1$, let us show that also $\Sigma^{-1}$ is continuous.
To this end, let $Q$ be a continuous seminorm on $\bigoplus_{j\in J}L^p(X,\mu, E_j)$.
After increasing $Q$, we may assume that
\[
Q((\gamma_j)_{j\in J})=\sum_{j\in J}Q_j(\gamma_j)
\]
with continuous seminorm $Q_j$ on $L^1(X,\mu, E_j)$.
After increasing each $Q_j$ if necessary, we may assume that
\[
Q_j=\|.\|_{L^1,q_j}
\]
for a continuous seminorm $q_j$ on $E_j$, for each $j\in J$.
Then $q\colon E\to [0,\infty[$,
\[
q((v_j)_{j\in J}):=\sum_{j\in J}q_j(v_j)
\]
is a continuous seminorm on~$E$.
If $[\gamma]\in L^1(X,\mu, E)$,
we may assume that the representative $\gamma$
has been chosen in $\bigoplus_{j\in J}\cL^1(X,\mu,E_j)$
(as just shown). Thus $\gamma=\sum_{j\in F}\gamma_j$
for some finite subset $F\sub J$ and suitable $\gamma_j\in\cL^1(X,\mu,E_j)$.
Thus
\begin{eqnarray*}
Q(\Sigma^{-1}([\gamma]))&=&\sum_{j\in F}\|\gamma_j\|_{\cL^1,q_j}
=\sum_{j\in F}\int_X q_j(\gamma_j(x))\, d\mu(x)\\
&=&\int_X\sum_{j\in F}q_j(\gamma_j(x))\,d\mu(x)
=\int_Xq((\gamma_j(x))_{j\in J})\, d\mu(x)\\
&=& \int_Xq(\gamma(x))\,d\mu(x)=\|[\gamma]\|_{L^1,q}
\leq \|[\gamma]\|_{L^1,q}
\end{eqnarray*}
and thus $\Sigma^{-1}$ is continuous.

(b) If $S$ is a closed vector subspace
of~$F$, then $S$ is also closed in~$E$ and thus $S=\bigcup_{n\in\N}F_n$
with an ascending sequence $(F_n)_{n\in \N}$
of separable Fr\'{e}chet spaces. Hence~$F$ has the (FEP).

(c) By Proposition~9 in \cite[Chapter II, \S4, no.\,6]{Bou},
$E$ induces the given topology on each~$F_n$.
If $S\sub E$ is a closed vector subspace,
then $S=\bigcup_{n\in \N}(S\cal F_n)$ where $S\cap F_n$
is a Fr\'{e}chet space for each $n\in \N$.

(d) Let $(U_j)_{j\in J}$ be a locally finite cover of~$M$
by open, relatively compact subsets $U_j\sub M$
such that there exists a trivialization $\theta_j\colon E|_{U_j}\to U_j\times F$.
Then $\Gamma(E|_{U_j})\cong C^\infty(U_j,F)$
is a Fr\'{e}chet space for each $j\in J$ (cf.\ \cite{ZOO}
or \cite{SEC}).
Since, by Lemma~\ref{toviasum}, there exists a linear topological
embedding
\[
\Phi\colon \Gamma_c(E)\to\bigoplus_{j\in J}\Gamma(E|_{U_j})
\]
with closed image, we deduce with (a) and (b)
that $\Gamma_c(E)$ has the (FEP)
and with (a) that, for each $\gamma\in \cL^p(X,\mu,E)$,
there exists a finite subset $J_0\sub J$
and $A\in \Sigma$ with $\mu(X\setminus A)=0$
and
\[
(\Phi\circ \gamma)(A)\sub \bigoplus_{j\in J_0}\Gamma(E|_{U_j}).
\]
Let $\one_A\colon X\to \{0,1\}$ be the characteristic function (indicator function)
of~$A$.
Then
$K:=\wb{\bigcup_{j\in J_0}U_j}$ is a compact subset of~$M$ and
\[
[\gamma]=[\one_A\gamma]\in L^p(X,\mu, \Gamma_K(E)).\vspace{-1mm}
\]

(e) This is a special case of~(d).\,\Punkt\vspace{2.3mm}

\noindent
{\bf Proof of Lemma~\ref{fepweakint}.}
If $\gamma\in \cL^1(X,\mu,E)$,
then $S:=\wb{\Spann(\gamma(X))}$ is a separable closed
vector subspaces of~$E$ and thus $S=\bigcup_{n\in \N}F_n$
with separable Fr\'{e}chet spaces $F_1\sub F_2\sub\cdots$
in~$E$ (as $E$ has the (FEP)).
Since $F_n$ is complete, hence closed in~$E$
and thus Borel, we have $A_n:=\gamma^{-1}(F_n)\in \Sigma$
for each $n\in \N$,
$A_1\sub A_2\sub\cdots$ and $X=\bigcup_{n\in \N}A_n$.
Since $\gamma|_{A_n}\in \cL^1(A_n,\mu|_{A_n},F_n)$,
the weak integral
\[
z_n:=\int_{A_n}\gamma|_{A_n}\, d(\mu|_{A_n})=\int_X \one_{A_n}\gamma\,d\mu
\]
exists in~$F_n$ (and hence in~$E$), for each $n\in \N$.
We claim that $(z_n)_{n\in \N}$ is a Cauchy sequence in~$E$.
If this is true, then
\[
z:=\lim_{n\to\infty}z_n
\]
exists in~$E$ since $E$ is assumed sequentially complete.
For each $\lambda\in E'$, we have
\[
\lambda(z)=\lim_{n\to\infty}\lambda(z_n)=\lim_{n\to\infty}
\int_X\one_{A_n}(\lambda\circ \gamma)\,d\mu=
\int_X(\lambda\circ\gamma)\,d\mu
\]
by dominated convergence (with $|\lambda\circ\gamma|$
as a majorant) and thus $z=\int_X\gamma\,d\mu$.\\[2.3mm]
To prove the claim, let $q\in P(E)$.
Since $q\circ\gamma\in \cL^1(X,\mu,\R)$,
we have
\[
\int_X\one_{A_n}(q\circ\gamma)\,d\mu
\to \int_X(q\circ\gamma)\,d\mu
\]
as $n\to \infty$ by dominated convergence,
entailing that $(\int_X\one_{A_n}(q\circ\gamma)\,d\mu)_{n\in \N}$
is a Cauchy sequence. Hence, given $\ve>0$,
there exists $N\in \N$ such that for all $m\geq n\geq N$:
\begin{eqnarray*}
q(z_m-z_n)&=&q\left(
\int_X (\one_{A_m}-\one_{A_n})\gamma\,d\mu\right)\\
&\leq &\int_X (\one_{A_m}-\one_{A_n})(q\circ \gamma)\,d\mu\\
&=& \int_X\one_{A_m}(q\circ\gamma)\,d\mu-\int_X\one_{A_n}(q\circ\gamma)\,d\mu
<\ve
\end{eqnarray*}
Thus $(z_n)_{n\in \N}$ indeed is a Cauchy sequence.\,\Punkt\vspace{2.3mm}

\noindent
{\bf Proof of Lemma~\ref{simpappfep}.}
Let $\gamma\in \cL^p(X,\mu,E)$.
We have $\overline{\Spann\gamma(X)}=\bigcup_{n\in\N}F_n$
for suitable vector subspaces
$F_1\sub F_2\sub\cdots$ of $E$
which are separable Fr\'{e}chet spaces and hence closed.
Then $A_n:=\gamma^{-1}(F_n)$ are Borel sets with
$A_1\sub A_2\sub\cdots$ and $X=\bigcup_{n\in\N}A_n$.
For each continuous seminorm $q$ on $E$, we have
\[
\|\gamma-\gamma\cdot \one_{A_n}\|_{\cL^p,q}=\left(\int_X(q(\gamma(x)))^p(1-\one_{A_n}(x))\,d\mu(x)\right)^{1/p}\to 0
\]
by dominated convergence. Hence $\gamma\cdot \one_{A_n}\to \gamma$.
Hence, if each $\gamma\one_{A_n}$ is in the closure of $\cF(X,\mu,E)\cap \cL^p(X,\mu,E)$ (or the subset specified in (\ref{strabi}),
then also $\gamma$ is in the closure.
After replacing $\gamma$ with $\gamma\one_{A_n}$ and $E$ with $F_n$,
we may therefore assume that $E$ is a separable Fr\'{e}chet space.
This case was already settled.\,\Punkt\vspace{2.3mm}
%

%
%
%
%
%
%
\noindent
{\bf Proof of Lemma~\ref{funda3}.}
By Lemma~\ref{fepweakint},
the integrals needed to define~$\eta$
exist in~$E$.
To see that $\eta$ is continuous on the left
at each $t\in J$ (a similar
argument shows that $\eta$ is continuous on the right at~$t$).
Let $q\in P(E)$.
For each $t_1\in J$ such that $t_1\leq t$,
we have that
\[
q(\eta(t)-\eta(t_1))=q\left(\int_{t_1}^t\gamma(s)\,ds\right)
\leq \int_{t_1}^t(q\circ \gamma)(s)\,ds
= \int_J \one_{[t_1,t]}(s)(q\circ \gamma)(s)\,ds\to0
\]
as $t_1\to t$, by dominated convergence.\\[2.3mm]
To see that the linear map $L^1(J,E)\to C(J,E)$, $[\gamma]\mto \eta_\gamma$
(with $\eta_\gamma(t):=\int_{t_0}^t\gamma(s)\,d\mu(s)$)
is injective,
let $\gamma\in \cL^1(J,E)$ such that
$[\gamma]\not=0$.
Then $\|\gamma\|_{\cL^1,q}\not=0$ for some $q\in P(E)$.
Let $\pi_q\colon E\to \wt{E}_q$
and the norm $\|.\|_q$ on $\wt{E}_q$
be as in~\ref{banquot}.
Then
$\|\pi_q \circ \gamma\|_{\cL^1, \|.\|_q}=\|\gamma\|_{\cL^1,q}\not=0$,
whence $[\pi_q\circ \gamma]\not=0$ in $L^1(J,\wt{E}_q)$.
Since
\[
(\pi_q\circ \eta_\gamma)(t)=\int_{t_0}^t(\pi_q\circ \gamma)(s)\,d\mu(s)
\]
for each $t\in J$, Lemma~\ref{funda} shows that $\pi_q\circ \eta_\gamma\not=0$.
Hence $\eta_\gamma\not=0$.\,\Punkt\vspace{2.3mm}

\noindent
{\bf Proof of \ref{henceemba}.}
If $[\gamma]=[\gamma_1]$ with $\gamma_1$
in $\cL^1([a,b],F)$ or $L^\infty_{rc}([a,b],F)$,
then $\gamma$ can be replaced with $\gamma_1$
in the definition of~$\eta$.
But then
the weak integral defining $\eta(t)$
can be formed in~$F$, and coincide with those
in~$E$. Thus $\eta(t)\in F$ for all $t\in [a,b]$.
If, conversely,
\[
\lambda_1(\{t\in [a,b]\colon \gamma(t)\in E\setminus F\})>0,
\]
then
\[
[q\circ \gamma]\not=0
\]
in $L^1([a,b], E/F)$ (resp., $L^\infty_{rc}([a,b],E/F)$),
where
\[
q\colon E\to E/F,\quad x\mto x+F
\]
is the canonical quotient map.
Now
\[
(q\circ \eta)(t)=\int_a^tq(\gamma(s))\,ds
\]
for all $t\in [a,b]$.
If $E$ is a Fr\'{e}chet space,
then also $E/F$ is a Fr\'{e}chet space
and the uniqueness
assertion
in Lemma~\ref{funda} shows that $q\circ \eta\not=0$
and thus $\eta([a,b])\not\sub F$.
If $\gamma\in \cL^\infty_{rc}([a,b],E)$,
then $q\circ \gamma\in \cL^\infty_{rc}([a,b], E/F)$.
If $q\circ \eta$ would vanish, it would also be a primitive
of the constant curve with value $0$
and the contradiction $[q\circ \gamma]=[0]$
would follow
analogously to the uniqueness part of Lemma~\ref{funda2}
(whose proof only requires the existence
of the weak integrals at hand, not
integral completeness).
Thus $q\circ \eta\not=0$ and thus $\eta([a,b])\not\sub F$.\\[2.3mm]
Now assume that $E$ is a strict (LF)-space,
say $E=\dl\, E_n$ as a locally convex space with
an ascending sequence $E_1\sub E_2\sub\cdots$ of Fr\'{e}chet spaces
such that $E_{n+1}$ induces the given topology on~$E_n$ for each~$n$.
If a vector subspace $F\sub E$ is a Fr\'{e}chet space
in the induced topology, then $F\sub E_N$
for some $N\in \N$ as a consequence of the Grothendieck Factorization Theorem
\cite[24.33]{MaV} (or simply using that the locally convex direct limit
$E=\bigcup_{n\in \N}E_n$ is regular, whence $F$ cannot contain a zero-sequence
leaving each $E_n$).
By Lemma~\ref{lafep}\,(c), we may assume that $\gamma\in \cL^1([a,b],E_n)$
for some $n\in \N$; we may assume that $n\geq N$.
Then $F$ is a closed vector subspace of the Fr\'{e}chet space
$E_n$ and hence we can replace $\gamma$ by an element of
$\cL^1([a,b],F)$ by the special case of the lemma for Fr\'{e}chet spaces
already discussed.\,\Punkt\vspace{2.3mm}

\noindent
{\bf Proof of \ref{regulatedfrech}.}
To prove the assertion, let $(q_m)_{m\in \N}$
be a sequence in $P(E)$ defining the locally convex
vector topology on~$E$.
If $\gamma$ is in the closure,
then we find a sequence $\gamma_n\in \cT([a,b],E)$
such that $\|\gamma-\gamma_n\|_{\cL^\infty,q_m}\to 0$
for all $m\in \N$. For each $m$ and $n$,
there is a measurable set $A_{m,n}\sub X$ with $\mu(A_{m,n})=0$
such that $\|\gamma-\gamma_n\|_{\cL^\infty,q_m}
=\sup_{x\in X\setminus A_{m,n}}
q_m(\gamma(x)-\gamma_n(x))$.
Then also $A:=\bigcup_{n,m\in \N}A_{m,n}$ is measurable
and $\mu(A)=0$.
Define $\eta_n(x):=\gamma_n(x)$ if $x\in X\setminus A$,
$\eta_n(x):=\gamma(x)$ if $x\in A$.
Then $[\eta_n]=[\gamma_n]$ in $L^\infty_{rc}([a,b],E)$
and $\eta_n\to \gamma$ uniformly.\,\Punkt\vspace{1.3mm}

\noindent
{\bf Proof of \ref{Lebspprod}.}
In fact, the projections
$\pi_j\colon E_1\times E_2\to E_j$
onto the components are continuous linear for $j\in \{1,2\}$
and also the mappings
\[
\lambda_1\colon E_1\to E_1\times E_2,\quad x\mto (x,0)
\]
and $\lambda_2\colon E_2\to E_1\times E_2$, $y\mto (0,y)$
are continuous linear.
Thus
\[
\Phi:=(L^p(X,\mu,\pi_1),L^p(X,\mu,\pi_2)\colon
L^p(X,\mu,E_1\times E_2)
\to L^p(X,\mu,E_1)\times L^p(X,\mu, E_2)
\]
is continuous linear and also
\[
\Psi\colon L^p(X,\mu,E_1)\times L^p(X,\mu, E_2)\to L^p(X,\mu,E_1\times E_2),
\]
$\Psi([\gamma_1],[\gamma_2]):=L^p(X,\mu,\lambda_1)([\gamma_1])+
L^p(X,\mu,\lambda_2)([\gamma_2])$ is continuous linear.
Since $\Psi\circ\Phi$ is the identity map
on $L^p(X,\mu,E_1\times E_2)$
and $\Phi\circ \Psi$ is the identity map on
$L^p(X,\mu,E_1)\times L^p(X,\mu, E_2)$,
we see that $\Phi$ is an isomorphism of topological
vector spaces with inverse~$\Psi$.\,\Punkt\vspace{1.3mm}

\noindent
{\bf Proof of Lemma~\ref{chainpw}.}
The set $U^{[1]}:=\{(x,y,t)\in U\times E\times \R\colon x+ty\in U\}$
is open in $E\times E\times \R$.
Since $f\colon U\to F$ is $C^1$, the map
\[
f^{[1]}\colon U^{[1]}\to F,\quad
(x,y,t)\mto
\left\{
\begin{array}{cl}
\frac{f(x+ty)-f(x)}{t} &\mbox{if $t\not=0$;}\\
df(x,y) & \mbox{if $t=0$}
\end{array}
\right.
\]
is continuous (see \cite{BGN} or \cite{GaN}).
For $t\in \R\setminus\{0\}$ such that $t_0+t\in J$,
we have
\begin{eqnarray*}
\frac{f(\gamma(t_0+t))-f(\gamma(t_0))}{t}
&=&\frac{f\left(\gamma(t_0)+t \frac{\gamma(t_0+t)-\gamma(t_0)}{t}\right)-
f(\gamma(t_0))}{t}\\
&=& f^{[1]}\left(\gamma(t_0),\frac{\gamma(t_0+t)-\gamma(t_0)}{t},t\right),
\end{eqnarray*}
which converges to $f^{[1]}(\gamma(t_0),\gamma'(t_0),0)=
df(\gamma(t_0),\gamma'(t_0))$
as $t\to0$. Thus $(f\circ \gamma)'(t_0)
=df(\gamma(t_0),\gamma'(t_0))$.\,\Punkt\vspace{2.3mm}

\noindent
{\bf Proof of Lemma~\ref{locLip}.}
Since $df\colon V\times E\to F$ is continuous
and $df(x,0)=0\in B^p_1(0)$,
there is a convex open neighbourhood $V_1 \sub V$ of~$x$
and an open $0$-neighbourhood $W\sub E$ such that
$df(V_1\times W)\sub B^p_1(0)$.
We may assume that $W=B^q_1(0)$ for some $q\in P(E)$.
Then
\[
p(df(y,z))\leq q(z)\quad\mbox{for all $y\in V_1$ and $z\in E$.}
\]
For all $y,z\in V_1$, we obtain
\begin{eqnarray*}
p(f(z)-f(y))& =& q\left(
\int_0^1df(y+t(z-y),z-y)\,dt\right)\\
&\leq& \int_0^1 p(df(y+t(z-y),z-y))
\leq q(z-y),
\end{eqnarray*}
as desired.\,\vspace{1.3mm}\Punkt
\noindent
{\bf Proof of Lemma~\ref{locLip2}.}
The map $d_1f\colon V\times E_2\times E_1\to F$,
$d_1f(x,v,h):=(D_{(h,0)}f)(x,v)$ is continuous
and $d_1(K\times \{0\}\times\{0\})+\{0\}
\sub B^p_1(0)$.
Using the Wallace Lemma (see \ref{Wallace}),
we find an open subset $U\sub V$ such that $K\sub U$
and continuous seminorms
$q_1\in P(E_1)$ and $q_2\in P(E_2)$ such that
\[
d_1f(u,v,h)\in B^p_1(0)\quad
\mbox{for all $u\in U$, $h\in B^{q_1}_1(0)$
and $v\in B^{q_2}_1(0)$.}
\]
As a consequence,
\[
p(d_1f(u,v,h))\leq q_1(h)q_2(v)
\quad\mbox{for all $u\in U$, $h\in E_1$ and $v\in E_2$.}
\]
Since $f(K\times \{0\})=\{0\}\sub B^p_1(0)$,
the Wallace Lemma shows that we may assume that,
moreover,
\[
f(U\times B^{q_2}_1(0))\sub B^p_1(0)
\]
after shrinkling $U$ and increasing $q_2$.
Thus
\begin{equation}\label{easybutuse}
p(f(u,v))\leq q_2(v)
\quad \mbox{for all $u\in U$ and $v\in E_2$.}
\end{equation}
After increasing $q_1$ if necessary, we may assume
that $K+B^{q_1}_2(0)\sub U$.
Let $y,z\in K+B^{q_1}_1(0)$.
If $q_1(y-z)<1$, choose
$x\in K$ such that $y\in B^{q_1}_1(x)$.
Then $z\in B^{q_1}_2(x)$, by the triangle
inequality.
For all $v,w\in E_2$,
we deduce that
\begin{eqnarray*}
f(z,v)-f(y,w)
&=& f(z,v-w)+ f(z,w)-f(y,w)\\
&=&f(z,v-w)+\int_0^1 d_1f(y+t(z-y),w,z-y)\,dt.
\end{eqnarray*}
Hence
\begin{eqnarray*}
p(f(z,v)-f(y,w))
&\leq&
p(f(z,v-w))\\
& & \;\;+ \sup_{t\in [0,1]}
\underbrace{p(d_1f(y+t(z-y),w,z-y))}_{\leq q_1(z-y)q_2(w)}\\
&\leq& q_2(v-w)+ q_1(z-y)q_2(w).
\end{eqnarray*}
If $q_1(y-z)\geq 1$,
we estimate with (\ref{easybutuse})
\begin{eqnarray*}
p(f(z,v)-f(y,w))
&\leq & p(f(z,v-w))+ p(f(z,w))+p(f(y,w))\\
&\leq&q_2(v-w)+2q_2(w)\leq q_2(v-w)+2q_1(y-z)q_2(w).
\end{eqnarray*}
We therefore always have (\ref{easireqat})
if we replace $q_1$ with~$2q_1$.\,\vspace{2.3mm}\Punkt

\noindent
{\bf Proof of Lemma~\ref{banana}.}
Using Lemma~\ref{locLip},
we find a continuous seminorm
$q\in P(E)$ such that $B^q_2(x)\sub V$
and
\[
p(f(z)-f(y))\leq q(z-y)\quad\mbox{for all $z,y\in B^q_2(x)$.}
\]
Thus $p(f(z)-f(y))=0$ for all $z,y\in B^q_2(x)$ such that $q(z-y)=0$.
Equivalently, $\pi_p(f(z))=\pi_p(f(y))$ for
all $z,y\in B^q_2(x)$ such that $\pi_q(z)=\pi_q(y)$.
We therefore get a well-defined map
\[
\wt{f} \colon \pi_q(B^q_2(x))\to \wt{F}_p,\quad \pi_q(y)\mto \pi_p(f(y))
\]
on the open ball $\pi_q(B^q_2(x))=\{v\in E_q\colon \|v-\pi_q(x)\|_q<2\}$
in the normed space~$E_q$.
The map $\wt{f}$ is Lipschitz continuous with Lipschitz constant~$1$, as
\begin{eqnarray*}
\|\wt{f}(\pi_q(z))-\wt{f}(\pi_q(y))\|_p &=& \|\pi_p(f(z)-f(y))\|_p\\
&=& p(f(z)-f(y))\leq q(z-y)=\|\pi_q(z)-\pi_q(y)\|_q
\end{eqnarray*}
for all $z,y\in B^q_2(0)$. In particular, $wt{f}$
is uniformly continuous and hence
extends uniquely to a continuous (and indeed
Lipschitz continuous) map
\[
g\colon B^{\|.\|_q}_2(x)
\to \wt{F}_p
\]
on the open ball $B^{\|.\|_q}_2(x)\sub \wt{E}_q$.
Since $df\colon V\times E\to F$ is~$C^1$,
we can repeat the reasoning.
After increasing~$q$ if necessary,
we may assume that there is a continuous
map
\[
h\colon
B^{\|.\|_q}_2(x)
\times
B^{\|.\|_q}_2(0)
\to \wt{F}_p
\]
such that $h(\pi_q(y),\pi_q(z))=\pi_p(df(y,z))$
for all $(y,z)\in B^q_2(x)\times B^q_2(0)$.
Then
\[
h(v,rw)=rh(v,w)\quad\mbox{for all $(v,w)
B^{\|.\|_q}_2(x) \times B^{\|.\|_q}_2(0)$ and $r\in \,]0,1]$}
\]
as $h$ is continuous and equality holds
for all $(v,w)$ in the dense subset
$\pi_q(B^q_2(x))\times \pi_q(B^q_2(0))$.
Therefore
\[
H \colon B^{\|.\|_q}_2(x)
\times \wt{E}_q\to\wt{F}_p,\quad
H(v,w):=n h\left(v,\frac{1}{n}w\right)
\]
for all
$n\in \N$, $v\in B^{\|.\|_q}_2(x)$
and $w\in B^{\|.\|_q}_{2n}(0)$
is well-defined.
Since $H$ is continuous on the open sets
$B^{\|.\|_q}_2(x)\times B^{\|.\|_q}_{2n}(x)$ for $n\in \N$
which cover $B^{\|.\|_q}_2(x)\times \wt{E}_q$,
the map~$H$ is continuous.
Now, if $y\in B^q_2(x)$ and $y\in B^q_{2n}(0)$, then
\begin{eqnarray*}
H(\pi_q(y),\pi_q(z))&=&nh\left(\pi_q(y),\frac{1}{n}\pi_q(z)\right)
=nh\left(\pi_q(y),\pi_q\left(\frac{1}{n}z\right)\right)\\
&=& ndf\left(y,\frac{1}{n}z \right)
=df(y,z).
\end{eqnarray*}
Since $\pi_q(B^q_2(x))\times E_q$ is dense
in $B^{\|.\|_q}_2(x)\times \wt{E}_q$,
we deduce that
\[
H(v,.)\colon \wt{E}_q\to \wt{F}_p
\]
is linear for each $v\in B^{\|.\|_q}_1(x)$.
The parameter-dependent integral
\[
g_1 \colon B^{\|.\|_q}_1(x)\times B^{\|.\|_q}_1(0) \times \;]{-1},1[\;\to\;
\wt{F}_p,\quad a(v,w,t):=\int_0^1H(v+stw,w)\,ds
\]
is continuous, by~\ref{pardep}.
For all $y\in B^q_1(x)$,
$z\in B^q_1(0)$ and $t\in \;]{-1},1[\;\setminus \{0\}$,
we have
\begin{eqnarray*}
g_1(\pi_q(y),\pi_q(z)) & = & \int_0^1H(\pi_q(y+stz),\pi_q(z))\,ds
= \int_0^1df(y+stz,z)\,ds\\
&=& \frac{1}{t}(f(y+tz)-f(y))
=\frac{1}{t}(g(\pi_q(y)+t\pi_q(z))-g(\pi_q(y))).
\end{eqnarray*}
Hence
\[
g_1(v,w,t)= \frac{1}{t}(v+tw)-g(v))
\]
for all
$(x,y,t)\in
B^{\|.\|_q}_1(x)\times B^{\|.\|_q}_1(0) \times (]{-1},1[\;\setminus\{0\})$,
as both sides of the equation are continuous functions
of $(x,y,t)\in
B^{\|.\|_q}_1(x)\times B^{\|.\|_q}_1(0) \times (]{-1},1[\;\setminus\{0\})$
which agree on the dense subset $\pi_q(B^q(x))\times \pi_q(B^q_1(0))$.
Abbreviate $U:=B^{\|.\|_q}_1(\pi_q(x))$.
By the preceding, the map
\[
(g|_U)^{[1]}\colon U^{[1]}=U^{]1[}\cup
(U\times B^{\|.\|_q}_1(0) \times \;]{-1},1[)\to \wt{F}_p,
\]
\[
(v,w,t)\mto\left\{
\begin{array}{cl}
g_1(v,w,t)&\mbox{if $\,(v,w,t)\in U\times B^{\|.\|_q}_1(0) \times \;]{-1},1[$;}\\
\frac{1}{t}(g(v+tw)-g(v))&\mbox{if $\,(v,w,t)\in U^{]1[}$}
\end{array}
\right.
\]
is well-defined, and it is continuous
as it is piecewise defined
and continuous on the two open pieces.
From $\ref{BGN}$, we deduce that $g|_U$
is $C^1$, with $dg=g^{[1]}(.,0)=H|_{U\times \wt{E}_q}$.
By construction, $\pi_p\circ f|_{B^q_1(0)}=g|_U\circ \pi_q|_{B^q_1(0)}^U$.\,\Punkt
\section{Details for Section~\ref{secPL}}\label{appeB}
We study the compatibility of
Lebesgue spaces with countable projective limits.
\begin{la}\label{denseTG}
Let $((G_n)_{n\in \N}, (\phi_{n,m})_{n\leq m})$
be a projective system of metrizable, complete
topological groups $G_n$ and continuous
homomorphisms $\phi_{n,m}\colon G_m\to G_n$
such that $\phi_{n,m}(G_m)$ is dense in~$G_n$,
for all $n,m\in \N$ with $n\leq m$.
Let $G:=\pl\,G_n$\vspace{-1.3mm}
be a projective limit.
Then also each limit map
$\phi_n\colon G\to G_n$
has dense image.
\end{la}
\begin{proof}
It suffices to show that $\phi_1$ has dense
image. Let $x_1\in G_1$ and $V\sub G_1$ be an open identity neighbourhood.
We construct an element $y\in G$ such that $\phi_1(y)\in x_1\wb{V}$,
where $\wb{V}$ is the closure of~$V$.
There is a sequence $(U_{1,k})_{k\in \N_0}$
of open identity neighbourhoods $U_{1,k}\sub G_1$ such that
$U_{1,0}=V$ and $U_{1,k}U_{1,k}\sub U_{1,k-1}$ for all
$k\in \N$. Thus
\[
U_{1,k}U_{1,k+1}\cdots U_{1,\ell-1}U_{1,\ell}\sub U_{1,\ell}\sub U_{1,k-1}\quad
\mbox{for all $k\in \N$ and $\ell>k$.}
\]
Recursively, we choose sequences $(U_{n,k})_{k\in \N_0}$
of open identity neighbourhoods in $G_n$ for $n\in \{2,3,\ldots\}$
such that
\[
\phi_{n-1,n}(U_{n,k})\sub U_{n-1,k}\quad\mbox{for all $k\in \N_0$}
\]
and $U_{n,k}U_{n,k}\sub U_{n,k-1}$ for all $k\in \N$.
The open set $x_1U_{1,1}$ contains $\phi_{1,2}(x_2)$
for some $x_2\in G_2$. Recursively, we find
$x_n\in G_n$ for $n\in \{2,3,\ldots\}$
such that $\phi_{n-1,n}(x_n)\in x_{n-1}U_{n-1,n-1}$.
Let $k\in \N$ and $N\in \N$ with $N\geq k$.
For all $n,m\in \N$ with $m>n\geq N$, we then have
\[
\phi_{n,m}(x_m)\in x_nU_{n,n}U_{n,n+1}\cdots U_{n,m-2}U_{n,m-1}\sub x_nU_{n,n-1},
\]
entailing that $x_n^{-1}\phi_{n,m}(x_m)\in U_{n,n-1}$ and hence
\begin{equation}\label{furthlim}
\phi_{k,n}(x_n)^{-1}\phi_{k,m}(x_m)\in U_{k,n-1}\sub U_{k,N-1}\quad\mbox{for all $m>n\geq N$.}
\end{equation}
Hence $(\phi_{k,n}(x_n))_{n\geq k}$ is a Cauchy sequence
in $G_k$ and thus convergent, to $y_k\in G_k$,
say. Letting $m\to\infty$ in (\ref{furthlim}),
we find that $(\phi_{k,N}(\phi_{N,n}(x_n)))^{-1}y_k=
\phi_{k,n}(x_n)^{-1}y_k\in \wb{U_{k,N-1}}$.
Letting $n\to \infty$, we see that $(\phi_{k,N}(y_N))^{-1}y_k\in \wb{U_{k,N-1}}$
and thus
\begin{equation}\label{thucnv}
\lim_{N\to\infty}\phi_{k,N}(y_N)=y_k
\quad\mbox{for each $k\in \N$.}
\end{equation}
Consider $z_k:=(\phi_{1,k}(y_k),\phi_{2,k}(y_k),\ldots, \phi_{k-1,k}(y_k),y_k,e,e,\cdots)
\in \prod_{k\in \N}G_k$.
By (\ref{thucnv}), the sequence $(z_k)_{k\in \N}$
converges to some $z\in \pl\,G_k=G$.
If $\ell\in \N$,
taking $k=N=1$ in (\ref{furthlim})
and let $m\to\infty$. using that $\phi_{1,n}(x_m)=\phi_{1,\ell}(\phi_{\ell,m}(x_m))$
if $m\geq\ell$,
we see that $(x_1)^{-1}\phi_{1,\ell}(y_\ell)\in \wb{U_{1,0}}\sub \wb{V}$.
Since $\phi_{1,\ell}(y_\ell)$ converges to $\phi_1(z)$
as $\ell\to\infty$, we deduce that $\phi_1(z)\in x_1\wb{V}$.
\end{proof}
{\bf Proof of Lemma~\ref{compaPL}.}
After passing to an isomorphic locally convex space, we may assume that~$E$
is the vector subspace
\[
\pl\,E_n=\{(x_n)_{n\in \N}\in \prod_{n\in \N}E_n\colon (\forall n\leq m)\;
\phi_{n,m}(x_m)=x_n\}\vspace{-.3mm}
\]
of the direct product, with the projection onto the component~$E_n$
as the limit map~$\phi_n$.
We realize $\pl\, L^p(X,\mu,E_n)$
as the vector subspace
\[
\{([\gamma_n])_{n\in \N}\in \prod_{n\in \N}L^p(X,\mu,E_n)\colon
(\forall n\leq m)\; [\gamma_n]=[\phi_{n,m}\circ \gamma_m]\}
\]
of the direct product, with the projections $\pr_n$
as the limit maps. Then
\[
\Phi:=(L^p(X,\mu,\phi_n))_{n\in \N}\colon L^p(X,\mu,E)\to
\pl\, L^p(X,\mu,E_n)
\]
is a continuous linear map. If $\Phi(\gamma)=0$,
then there are subsets $A_n\sub X$ such that $\mu(A_n)=0$
and $\phi_n\circ \gamma|_{x\setminus A_n}=0$.
After replacing $\gamma(x)$ with $0$ for $x$ in the set
$\bigcup_{n\in \N}A_n$ of measure~$0$,
we obtain $\gamma=0$. Hence~$\Phi$
is injective. To see that $\Phi$
is surjective,
let $([\gamma_n])_{n\in \N}\in\pl\,L^p(X,\mu,E_n)$.\vspace{-.3mm}
For all $n\leq m$, there is a subset
$A_{n,m}\sub X$ such that $\mu(A_{n,m})=0$ and
\[
\phi_{n,m}\circ \gamma_m|_{A_{n,m}}=\gamma_n|_{A_{n,m}}.
\]
Then the countable union $A:=\bigcup_{n\in \N}\bigcup_{m\geq n}A_{n,m}$
has measure $0$ as well and
\[
\phi_{n,m}\circ \gamma_m|_A=\gamma_n|_A\quad\mbox{for all $n\leq m$ in $\N$.}
\]
After re-defining $\gamma_n(x):=0$
for $x\in A$, we may assume that $\gamma_n=\phi_{n,m}\circ \gamma_m$
for all $n,m\in \N$ with $n\leq m$.
For each $x\in X$, we have
\[
\gamma(x):=(\gamma_n(x))_{n\in \N}
\in E;
\]
thus $\phi_n(\gamma(x))=\gamma_n(x)$
for all $n\in \N$.
Note that
$\prod_{n\in \N}\Spann(\gamma_n(X))$ is a separable
metrizable vector space
and
$\im(\gamma)$ is contained
in
\[
F:=E\cap \prod_{n\in \N}\Spann(\gamma_n(X)),
\]
which is separable. Let $\{x_1,x_2,\ldots\}\sub F$ be a countable
dense subset and $(q_j)_{j\in \N}$ be a sequence
of seminorms $q_1\leq q_2\leq\cdots$ on~$E$
defining its vector topology. Then countable set
\[
\cE:=\{F\cap B_{1/i}^{q_j}(x_k)\colon i,j,k\in \N\}
\]
of balls is a basis for
the topology on~$F$,
whence the Borel $\sigma$-algebra is generated by~$\cE$,
i.e., $\cB(F)=\sigma(\cE)$.
After increasing each of the $q_j$ in turn,
we may assume that $q_j=Q_j\circ \phi_{n_j}$
for some $n_j\in \N$ and some continuous
seminorm~$Q_j$ on~$E_{n_j}$.
Since
\begin{eqnarray*}
\gamma^{-1}(B_{1/i}^{q_j}(x_k))
&=& \{x\in X\colon q_j(\gamma(x)-x_k)
<1/i\}\\
&=& \{x\in X\colon Q_j(\gamma_j(x)-\phi_j(x_k))<1/i\}
=\gamma_j^{-1}(B^{Q_j}_{1/i}(\phi_j(x_k)))
\end{eqnarray*}
is measurable for all $i,j,k\in \N$,
we deduce that $\gamma$ is measurable.
Since $\|\gamma\|_{\cL^p,q_j}=\|\gamma_j\|_{\cL^p,Q_j}<\infty$
for each $j\in \N$, we see that $\gamma\in \cL^p(X,\mu,E)$.
By construction, $\Phi([\gamma])=([\gamma_n])_{n\in\N}$.
Thus $\Phi$ is surjective.
To see that $\Phi$ is a topological
embedding, note that seminorms of the form
$\|.\|_{L^p,Q\circ \phi_j}$
define the vector topology on $L^p(X,\mu,E)$,
for $j\in \N$ and~$Q$ ranging through the continuous
seminorms on~$E_n$.
Since $\|.\|_{L^p, Q}\circ \pr_j$
is a continuous seminorm on $\pl\,L^p(X,\mu,E)$\vspace{-1.3mm}
and
\[
\|.\|_{L^p, Q}\circ \pr_j \circ \Phi= \|.\|_{L^p,Q\circ \phi_j},
\]
we see that $\Phi^{-1}$ is continuous
and thus $\Phi$ a topological embedding.\,\vspace{2.3mm}\Punkt

\noindent
{\bf Proof of Lemma~\ref{compaPLrc}.}
After passing to an isomorphic locally convex space, we may assume that~$E$
is the closed vector subspace
\[
\pl\,E_n=\{(x_n)_{n\in \N}\in \prod_{n\in \N}E_n\colon (\forall n\leq m)\;
\phi_{n,m}(x_m)=x_n\}\vspace{-.3mm}
\]
of the direct product, with the projection onto the component~$E_n$
as the limit map~$\phi_n$.
We realize $\pl\, L^p(X,\mu,E_n)$\vspace{-1.3mm}
as the vector subspace
\[
\{([\gamma_n])_{n\in \N}\in \prod_{n\in \N}L^p(X,\mu,E_n)\colon
(\forall n\leq m)\; [\gamma_n]=[\phi_{n,m}\circ \gamma_m]\}
\]
of the direct product, with the projections $\pr_n$
as the limit maps. Then
\[
\Phi:=(L^p(X,\mu,\phi_n))_{n\in \N}\colon L^p(X,\mu,E)\to
\pl\, L^p(X,\mu,E_n)\vspace{-1.3mm}
\]
is a continuous linear map. Let $\gamma\in L^p(X,\mu,E)$
such that $\Phi(\gamma)=0$.
Using Lemma~\ref{imagesepmet},
we find a sequence $(q_j)_{j\in \N}$
of continuous seminorms $q_j$ on~$E$
such that the $q_j|_F$ define a Hausdorff
vector topology on $F:=\Spann(\wb{\gamma(X)})$.
After increasing $q_j$ if necessary,
we may assume that $q_j=Q_j\circ \phi_{n_j}$
for some $n_j\in \N$ and some continuous seminorm
$Q_n$ on $E_{n_j}$. Since $\|\phi_{n_j}\circ \gamma\|_{\cL^\infty,Q_n}=0$,
there is a measurable set $A_n\sub X$ such that $\mu(A_n)=0$
and $\phi_{n_j}\circ \gamma|_{X\setminus A_n}=0$.
After redefining $\gamma(x):=0$
on the set $\bigcup_{n\in \N}A_n$ of measure~$0$,
we achieve that $q_j(\gamma(x))=0$
for all $x\in X$ and $j\in \N$,
whence $\gamma(x)=0$ by choice of the $q_j$.
Thus $[\gamma]=0$ and thus $\Phi$ is injective.

To see that $\Phi$
is surjective,
let $([\gamma_n])_{n\in \N}\in\pl\,L^\infty_{rc}(X,\mu,E_n)$.\vspace{-1.3mm}
As in the preceding proof, we may assume
that $\gamma_n=\phi_{n,m}\circ \gamma_m$
for all $n,m\in \N$ with $n\leq m$.
For each $x\in X$, we have
\[
\gamma(x):=(\gamma_n(x))_{n\in \N}
\in E;
\]
thus $\phi_n(\gamma(x))=\gamma_n(x)$
for all $n\in \N$.
By Tychonoff's Theorem,
the metrizable topological space
$\prod_{n\in \N}\wb{\gamma_n(X)}$ is
compact,
whence also the closed subset
\[
E\cap \prod_{n\in \N}\wb{\gamma_n(X)}
\]
is metrizable and compact.
Since $\im(\gamma)$ is contained in the latter
set, we deduce that $\wb{\im(\gamma)}$
is compact and metrizable.
If we can show that $\gamma$ is measurable,
the $\gamma\in \cL^\infty_{rc}(X,E)$
and $\Phi([\gamma])=([\gamma_n])_{n\in \N}$
by construction, completing the proof of surjectivity.
Using Lemma~\ref{imagesepmet},
we find a sequence $(q_j)_{j\in \N}$
of continuous seminorms $q_1\leq q_2\leq\cdots$ on~$E$
such that the $q_j|_F$ define
a separable Hausdorff vector topology~$\cO'$
on $F:=\Spann(\wb{\im(\gamma)})$.
As the latter induced the given topology on
the compact set $\wb{\gamma(X)}$,
we need only show that $\gamma$ is measurable
as a mapping to $(F,\cO')$.
Let $\{x_1,x_2,\ldots\}$ be a countable
dense subset of $(F,\cO')$.
As in the preceding proof,
we see that $\gamma\colon X\to (F,\cO')$
is measurable.

To see that $\Phi$ is a topological
embedding, note that seminorms of the form
$\|.\|_{L^\infty,Q\circ \phi_j}$
define the vector topology on $L^\infty_{rc}(X,\mu,E)$,
for $j\in \N$ and~$Q$ ranging through the continuous
seminorms on~$E_n$.
Since $\|.\|_{L^\infty, Q}\circ \pr_j$
is a continuous seminorm on $\pl\,L^\infty_{rc}(X,\mu,E)$\vspace{-1.3mm}
and
\[
\|.\|_{L^\infty, Q}\circ \pr_j \circ \Phi= \|.\|_{L^\infty,Q\circ \phi_j},
\]
we see that $\Phi^{-1}$ is continuous
and thus $\Phi$ a topological embedding.\,\vspace{2.3mm}\Punkt

\noindent
{\bf Proof of Lemma~\ref{compaPLrc2}.}
Lemma~\ref{compaPLrc}
and its proof apply
if we set $X:=[a,b]$ and let $\mu$
be Lebesgue-Borel measure on~$X$.
Let $\Phi\colon L^\infty_{rc}([a,b],E)\to\pl\, L^\infty_{rc}([0,1],E_n)$\vspace{-0.6mm}
be as in the proof
of Lemma~\ref{compaPLrc}.
We can realize the projective limit $\pl\,R([a,b],E_n)$\vspace{-0.6mm}
as a vector subspace of $\pl\,L^\infty_{rc}([a,b],E_n)$.\vspace{-1.3mm}
Then
\[
\Phi_R:=(R([a,b],\phi_n))_{n\in \N}=\Phi|_{R([a,b],E)}\colon R([a,b],E)\to
\pl\,R([a,b],E_n)\vspace{-1.3mm}
\]
is a topological embedding
(since $\Phi$ is a topological embedding).
It only remains to see that $\Phi_R$
is surjective.
If each $E_n$ is a Fr\'{e}chet space,
then also $E$, $R([a,b],E)$
and each $R([a,b],E_n)$ are Fr\'{e}chet spaces
(cf.\ \ref{regulatedfrech}).
Hence $\Phi_R$ has complete (and hence closed)
image. As a consequence, we need
only show that $\Phi_R$ has dense
image in $\pl\, R([a,b],E_n)$.\vspace{-0.6mm}
Now $\im(\Phi_R)$
is dense if and only if $R([a,b],\phi_j)(\im(\Phi_R))$ is
dense in $R([a,b],E_n)$
for each $n\in \N$.
It suffices to show that
$\pr_n(\im(\Phi_R))=\im(\pr_n \circ \Phi_R)=\im\, R([a,b],\phi_n)$
is dense in $R([a,b],E_n)$
for each $n\in \N$.
This will hold if we can show that
$\cT([a,b],E)\sub \wb{\im\,R([a,b],\phi_n)}$.
To do so, let $\gamma\in \cT([a,b],E_n)$.
Then $\im \gamma=\{y_1,\ldots,y_\ell\}$
with some $\ell\in \N$ and pairwise distinct elements
$y_1,\ldots, y_\ell\in E_n$.
If $V\sub E_n$ is an open $0$-neighbourhood,
we find $z_j\in E$ such that $\phi_n(z_j)\in y_j+V$
for all $j\in \{1,\ldots,\ell\}$,
since $\im(\phi_n)$ is dense
in~$E_n$ by Lemma~\ref{denseTG}.
If we set $\eta(t):=z_j$ for all
$t\in [a,b]$ such that $\gamma(t)=y_j$,
we obtain a function $\eta\in \cT([a,b],E)$
such that $(\phi_n\circ \eta)(t)-\gamma(t)=\phi_n(z_j)-y_j\in V$
for all $t\in [a,b]$ (where $j\in \{1,\ldots,\ell\}$
is chosen such that $\gamma(t)=y_j$).
Hence $\gamma\in \wb{\im R([a,b],\phi_n)}$,
as desired.\,\vspace{2.3mm}\Punkt

\noindent
{\bf Proof of Lemma~\ref{reformug}.}
If $G$ has a projective limit chart,
let $\alpha_{n,m}\colon E_m\to E_n$, $\alpha_n\colon E\to E_n$,
$\phi_n\colon U_n\to V_n\sub E_n$
and $\phi\colon U\to V\sub E$ be as in Definition~\ref{dplcha}.
Let $x_n:=q_n(e)\in V_n$; then
\[
\beta_n:=T_{x_n}(\phi_n)^{-1}\colon E_n\to L(G_n)
\]
is an isomorphism of topological vector spaces,
identifying $E_n$ with $\{x_n\}\times E_n=T_{x_n}E_n$
via $(x_n,y)\mto y$.
Therefore $W_n:=\beta_n(V_n)$ is open in~$L(G_n)$ and
\[
\psi_n:=\beta_n\circ \phi_n\colon U_n\to W_n
\]
is a $C^\infty$-diffeomorphism.
From
\[
\phi_n\circ q_{n,m}|_{U_m}=\alpha_{n,m}\circ \phi_m,
\]
we deduce that $(T_e\phi_n)\circ L(q_{n,m})=(T_{x_m}\alpha_{n,m})\circ (T_e\phi_m)$
and thus
\[
L(q_{n,m})\circ \beta_m=\beta_n\circ \alpha_{n,m},
\]
entailing that
\[
L(q_{n,m})\circ \psi_m = L(q_{n,m})\circ \beta_m\circ \phi_m
=\beta_n\circ \alpha_{n,m}\circ \phi_m
=\beta_n\circ \phi_n\circ q_{n,m}|_{U_m}
=\psi_n\circ q_{n,m}|_{U_m}.
\]
Thus (\ref{compaby}) holds.
Moreover,
\[
L(q_{n,m})(W_m)=L(q_{n,m})(\beta_m(V_m))
=\beta_n (\alpha_{n,m}(V_m))\sub
\beta_n(V_n)=W_n.
\]
Set $x:=\phi(e)\in E$ and define
\[
\beta:=T_x\phi^{-1}\colon E\to L(G),
\]
identifying $T_xE=\{x\}\times E$ with~$E$.
Then $W:=\beta(V)$ is open in $L(G)$
and
\[
\psi:=\beta\circ \phi\colon U\to W
\]
is a $C^\infty$-diffeomorphism.
From
$\alpha_n \circ \phi=\phi_n\circ q_n|_U$
we deduce that
\[
\alpha_n\circ d\phi|_{L(G)}=d\phi_n|_{L(G_n)}\circ L(q_n)
\]
and hence
\[
\beta_n\circ \alpha_n=L(q_n)\circ \beta.
\]
Thus
\[
L(q_n)(W)=L(q_n)(\beta(V))=\beta_n(\alpha_n(V))
\sub \beta_n(V_n)=W_n,
\]
i.e., (\ref{compby3}) holds.
Moreover,
\[
L(q_n)\circ \psi=L(q_n)\circ \beta\circ \phi
=\beta_n\circ \alpha_n\circ \phi
=\beta_n\circ \phi_n\circ q_n|_U
=\psi_n\circ q_n|_U,
\]
i.e., (\ref{compby2}) holds.
Since $W_n=\beta_n(V_n)$,
we have
\[
L(q_n)^{-1}(W_n)=(\beta_n^{-1}\circ L(q_n))^{-1}(V_n)
=((\alpha_n\circ \beta^{-1})^{-1}(V_n)=\beta(\alpha_n^{-1}(V_n))
\]
and hence
\[
\bigcap_{n\in \N}L(q_n)^{-1}(W_n)
=\beta\left(\bigcap_{n\in \N}\alpha_n^{-1}(V_n)\right)
=\beta(V)=W.
\]
The proof of the converse implication is similar;
given $\psi_n\colon U_n\to W_n\sub L(G_n)$ and $\psi\colon U\to W\sub L(G)$
as in Lemma~\ref{reformug}, let $E$ be the modelling space of $G$
and $E_n$ be the modelling space of~$E_n$.
Let $\beta\colon L(G)\to E$ and $\beta_n\colon L(G_n)\to E_n$
be an isomorphism of topological vector spaces.
Then the conditions described in Definition~\ref{dplcha}
are satisfied if we set
\[
V_n:=\beta_n(W_n) \qquad\mbox{and}\qquad
\alpha_n:= \beta_n\circ L(q_n)\circ \beta^{-1}
\]
for $n\in \N$,
\[
\alpha_{n,m}:= \beta_n\circ L(q_{n,m})\circ \beta^{-1}_m
\]
for positive integers $n\leq m$,
and $V:=\beta(W)$.\,\Punkt
\section{Details for Section~\ref{secDiM}}\label{appDiff}
In this appendix, we prove
Proposition~\ref{handsonEv} from Section~\ref{secDiM}
and discuss various concepts which are useful
for the proof.
\begin{numba}
Recall that, if $M$ is a $C^k$-manifold modelled on a locally convex space~$E$
with $k\in \N\cup\{\infty\}$ and $p\in M$,
then $T_pM$ is the space of all tangent vectors $v$ to $M$ at~$p$.
Interpreting these as geometric tangent vectors,
they are $\sim$-equivalence classes $[\gamma]$ of $C^k$-curves
$\gamma\colon ]{-\ve},\ve[\;\to M$ with $\gamma(0)=p$,
where $\gamma\sim \eta$ if and only if
\begin{equation}\label{eqgeotan}
(\phi\circ \gamma)'(0)=(\phi\circ \eta)'(0)
\end{equation}
for some (and hence every)
chart $\phi$ of~$M$ around $p$.
It is useful for us to deviate from this classical
definition and consider, instead, (larger) equivalence classes
$v=[\gamma]$ with continuous curves
\[
\gamma\colon J\to M
\]
defined on a non-degenerate interval $J\sub \R$
with $0\in J$ such that $\gamma(0)=p$ and $\gamma$
is \emph{differentiable at $0$} in the sense
that $\phi\circ \gamma$
is differentiable at $0$ for some (and hence every)
chart of $M$ around $p$ (cf.\ Lemma~\ref{chainpw}).
Also for such curves, we write $\gamma\sim\eta$
if and only if (\ref{eqgeotan})
holds for some (and hence every)\footnote{See Lemma~\ref{chainpw}.}
chart $\phi$ of $M$
around~$p$.
\end{numba}
\begin{defn}\label{pwdiffmf}
Let $M$ be a $C^k$-manifold modelled
on a locally convex space $E$ with $k\in \N\cup\{\infty\}$.
Let $J\sub \R$ be a non-degenerate
interval, $\eta\colon J\to M$ be a continuous curve
and $t\in J$.
We say that $\eta$ is \emph{differentiable at $t$}
if the $E$-valued curve $\phi\circ \eta$ is differentiable
at $t$ for some (and hence every) chart $\phi$ of $M$ around $p$.
In this case, we set $\dot{\eta}(t):=[s\mto\eta(t+s)]$.
\end{defn}
\begin{la}\label{havedersae}
Let $M$ be a $C^k$-manifold modelled
on a Fr\'{e}chet space $E$ with $k\in \N\cup\{\infty\}$
and $\eta\in AC_{L^1}([a,b],M)$ with real numbers
$a<b$.
Write $\dot{\eta}=[\gamma]$
as in Definition~\emph{\ref{defetadot}}.
Then
there is a Borel set $A\sub [a,b]$ with $\lambda_1(A)=0$
such that
$\eta$ is differentiable at each $t\in [a,b]\setminus A$
and
\[
\dot{\eta}(t)=\gamma(t)
\]
for all $t\in [a,b]\setminus A$,
with $\dot{\eta}(t)$ as in
Definition~\emph{\ref{pwdiffmf}}.
The same conclusion holds
if $M$ is a $C^k$-manifold modelled on a strict \emph{(LF)}-space
which is a union $M=\bigcup_{n\in\N}M_n$
of $C^k$-manifolds $M_1\sub M_2\sub\cdots$
modelled on Fr\'{e}chet spaces
such that the inclusion maps $M_n\to M_{n+1}$
and $j_n\colon M_n\to M$ are topological embeddings and $C^k$ for all $n\in \N$,
and $AC_{L^1}([a,b],M)=\bigcup_{n\in \N}AC_{L^1}([a,b],M_n)$
as a set.
\end{la}
\begin{proof}
In the first situation,
let $a=t_0<t_1<\cdots< t_m=b$
such that $\phi_i([t_{i-1},t_i])\sub U_i$
for a chart $\phi_i\colon U_i\to V_i\sub E$,
for each $i\in\{1,\ldots,m\}$.
Then $\eta_i:=\phi_i\circ \eta|_{[t_{i-1},t_i]}\in AC_{L^1}([t_{i-1},t_i],E)$
for $i\in \{1,\ldots,m\}$.
Write $\eta_i'=[\gamma_i]$ with $\gamma_i\in \cL^1([t_{i-1},t_i],E)$.
There is a Borel set $A_i\sub [t_{i-1},t_i]$
such that $\eta_i'(t)$ exists for all $t\in [t_{i-1},t_i]\setminus A_i$
and $\eta_i'(t)=\gamma_i(t)$
(see Lemma~\ref{funda}). There is a Borel set $B_i\sub [t_{i-1},t_i[$
of measure $0$ such that
$\gamma(t)=T(\phi_i^{-1})(\eta_i(t),\gamma_i(t))$
for all $t\in [t_{i-1},t_i[\;\setminus B_i$.
Then $A:=\{t_0,t_1,\ldots, t_m\}\cup \bigcup_{i=1}^m(A_i\cup B_i)$
has the required properties.\\[2.3mm]
In the second situation, let $\zeta\in AC_{L^1}([a,b],M)$.
By hypothesis, we have $\zeta\in AC_{L^1}([a,b], M_n)$ for some $n\in \N$;
write $\eta$ for $\zeta$, considered as a curve in $M_n$.
Using that~$M_n$ carries the topology induced by~$M$,
for suitable $a=t_0<t_1<\cdots<t_m=b$
we find charts $\phi_i\colon U_i\to V_i\sub E_n$ for $M_n$
and $\psi_i\colon X_i\to Y_i\sub E$ for~$M$ such that $U_i\sub X_i$
and $\eta([t_{i-1},t_i])\sub U_i$ for all $i\in \{1,\ldots,m\}$.
Let $\eta_i$, $\gamma_i$, $\gamma$ with $\dot{\eta}=[\gamma]$,
$A$, $A_i$ and $B_i$ be as in the first situation. 
Write $\dot{\zeta}=[\theta]$.
Set $\zeta_i:=\psi_i\circ\zeta|_{[t_{i-1},t_i]}$.
Then $\zeta_i=\tau_i\circ\eta_i$ with $\tau_i:=\psi_i\circ \phi_i^{-1}$
entails $\zeta_i'=[d\tau_i\circ (\eta_i,\gamma_i)]$,
whence there is a Borel set $C_i\sub [t_{i-1},t_i[$
of measure zero such that $\theta(t)=T\psi^{-1}_i(\zeta_i(t),d\tau_i(\eta_i(t),\gamma_i(t)))$ for all $t\in [t_{i-1},t_i[\;\setminus C_i$
and thus
\[
\theta(t)=T\psi_i^{-1}T\tau_i(\eta_i(t),\gamma_i(t))=Tj_nT\phi_i^{-1}(\eta_i(t),\gamma_i(t))
=Tj_n\gamma(t)
\]
for all $t\in [t_{i-1},t_i[\;\setminus (B_i\cup C_i)$.
Since $\zeta_i'(t)=d\tau_i(\eta_i(t),\eta_i'(t))$ for all
$t\in \;]t_{i-1},t_i[\;\setminus A_i$,
we also deduce that $\zeta'(t)$ exists for all
$t\in \;]t_{i-1},t_i[\;\setminus A_i$, and is given by
\[
\zeta'(t)=T\psi_i^{-1}(\zeta_i(t),\zeta_i'(t))=
T\psi_i^{-1}(\zeta_i(t),d\tau_i(\eta_i(t),\eta_i'(t)))=\theta(t)
\]
for $t\in \;]t_{i-1},t_i[\;\setminus (A_i\cup B_i\cup C_i)$.
Summing up, $\zeta'(t)$ exists for all $t\in [a,b]\setminus (A\cup C_1\cup\cdots\cup
C_m)$ and concides with $\theta(t)$ there.
\end{proof}
Let $M$ be a $\sigma$-compact finite-dimensional
smooth manifold. For $p\in M$, let $\ve_p\colon \Diff_c(M)\to M$
be the smooth map $\phi\mto \phi(p)$ (see, e.g., \cite{DIF}
for the smoothness).
Let $C^\infty_c(M,TM)$ be the set of all smoth maps
$X\colon M\to TM$ such that $\{p\in M\colon X(p)\not=0_{\pi_{TM}(X(p))}\}$
is relatively compact in~$M$.
For $\phi\in \Diff_c(M)$, let
\[
\Gamma_\phi:=\{X\in C^\infty_c(M,TM)\colon \pi_{TM}\circ X=\phi\},
\]
which is a vector space under pointwise operations.
For $\phi\in \Diff_c(M)$ and $[\gamma]\in T_\phi\Diff_c(M)$,
define
\[
\alpha_\phi([\gamma])\colon M\to TM,\quad p\mto T\ve_p([\gamma])=[\ve_p\circ\gamma],
\]
i.e., $\alpha_\phi([\gamma])=(T\ve_p([\gamma]))_{p\in M}$.
Given $\psi\in \Diff_c(M)$, let us write
\[
\rho_\psi\colon \Diff_c(M)\to \Diff_c(M),\quad \phi\mto \phi\circ \psi
\]
for right translation with~$\psi$.
\begin{la}\label{pinpointdr}
Let $\phi\in \Diff_c(M)$.
\begin{itemize}
\item[\rm(a)]
For each
$[\gamma]\in T_\phi\Diff_c(M)$,
we have that $\alpha_\phi([\gamma])\in \Gamma_\phi$.
\item[\rm(b)]
The map $\alpha_\phi\colon T_\phi\Diff_c(M)\to \Gamma_\phi$ is
an isomorphism of vector spaces.
\item[\rm(c)]
Let $\psi\in \Diff_c(M)$. For each $X\in \Gamma_\phi$, we have $X\circ\psi\in
\Gamma_{\phi\circ\psi}$.
The map
\[
R_\phi(\psi)\colon \Gamma_\phi\to \Gamma_{\phi\circ\psi},\quad X\mto X\circ \psi
\]
is an isomorphism of vector spaces such that
\[
R_\phi(\psi)\circ\alpha_\phi=\alpha_{\phi\circ \psi}\circ T_\phi \rho_\psi.
\]
\item[\rm(d)]
$\alpha_{\id_M}$ is the usual identification of $L(\Diff_c(M))=
T_{\id_M}\Diff_c(M)$ with $\cX_c(M)$, i.e., it coincides
with $d\Phi|_{L(\Diff_c(M))}$ where $\Phi$ is
one of the usual charts for $\Diff_c(M)$ with
$\Phi^{-1}\colon X\mto \exp_g \circ X$ \emph{(cf.\ \ref{bscsdf}).}
\end{itemize}
\end{la}
\begin{proof}
It is clear that $R_\phi(\psi)$ takes $\Gamma_\phi$ into $\Gamma_{\phi\circ\psi}$,
and is a linear map.
Moreover, $R_\phi(\id_M)=\id_{\Gamma_\phi}$ and
$R_{\phi\circ\psi}(\theta)\circ R_\phi(\psi)=R_\phi(\psi\circ\theta)$
if also $\theta\in \Diff_c(M)$.
We easily deduce that
\[
R_\phi(\psi)\colon \Gamma_\phi\to \Gamma_{\phi\circ\psi}
\]
is an isomorphism of vector spaces with inverse $R_{\phi\circ\psi}(\psi^{-1})$.
If $[\gamma]\in T_\phi\Diff_c(M)$,
then $T\ve_p([\gamma])\in T_{\ve_p(\phi)}M=T_{\phi(p)}M$
and thus $\pi_{TM}(T\ve_p([\gamma]))=\phi(p)$.
Hence $\pi_{TM}\circ\alpha_\phi([\gamma])=\phi$ and thus
\begin{equation}\label{almgamphi}
\pi_{TM}\circ \alpha_\phi([\gamma])=\phi.
\end{equation}
As we do not know yet that the map
$\alpha_\phi([\gamma])\colon M\to TM$ is smooth (nor compactly supported),
we cannot
conclude that $\alpha_\phi([\gamma])\in \Gamma_\phi$ yet.
Define
\[
V_\phi:=\im(\alpha_\phi)\sub TM^M.
\]
For $\psi\in \Diff_c(M)$ and $[\gamma]\in T_\phi\Diff_c(M)$, we have
\[
\alpha_{\phi\circ\psi}(T\rho_\psi([\gamma]))
=\alpha_{\phi\circ\psi}([t\mto \gamma(t)\circ \psi])
=(p\mto [t\mto \gamma(t)(\psi(p))])
=\alpha_\phi([\gamma])(\psi(p))
\]
and thus
\begin{equation}\label{praecc}
(\alpha_{\phi\circ\psi}\circ T\rho_\psi)([\gamma])= \alpha_\phi([\gamma])\circ \psi.
\end{equation}
As a consequence,
\[
f\circ \psi\in V_{\phi\circ\psi}\quad\mbox{for all $f\in V_\phi$ and $\psi\in\Diff_c(M)$,}
\]
enabling us to define a map
\[
r_\phi(\psi)\colon V_\phi\to V_{\phi\circ\psi},\quad f\mto f\circ \psi.
\]
Note that $r_\phi(\id_M)=\id_{V_\phi}$ and
$r_{\phi\circ\psi}(\theta)\circ r_\phi(\psi)=r_\phi(\psi\circ\theta)$
for each $\theta\in \Diff_c(M)$.
We easily deduce that
$r_\phi(\psi)$ is an isomorphism of vector spaces
for all $\phi,\psi\in \Diff_c(M)$
(with inverse $r_{\phi\circ\psi}(\psi^{-1})$).
It is clear that
\[
W_\phi:=\{f\in TM^M\colon \pi_{TM}\circ f=\phi\}
\]
is a vector space under the pointwise operations.
Since each of the maps $T_p\ve_p\colon T_\phi\Diff_c(M)\to T_{\phi(p)}M$
is linear, we deduce that $\alpha_\phi$ is linear
as a map to $W_\phi$. In particular, $V_\phi$
is a vector subspace of $W_\phi$ and we can consider $\alpha_\phi$
as a \emph{linear} surjective map
\[
\alpha_\phi\colon T_\phi\Diff_c(M)\to V_\phi.
\]
Fix a smooth Riemannian metric $g$ on~$M$, with Riemannian exponential function
$\exp\colon \cD\to M$ on an open neighbourhood
$\cD\sub TM$ of the zero-section. Consider the map
\[
h\colon \cX_c(M)\to T_{\id_M}\Diff_c(M),\quad X\mto [t\mto \exp\circ (tX)].
\]
Then $h=T\Phi^{-1}(0,\sbull)$ for a typical
chart $\Phi$ of $\Diff_c(M)$ around $\id_M$
such that $\Phi(\id_M)=0$,
with $\Phi^{-1}\colon X \mto \exp\circ X$.
Now
\[
(\alpha_{\id_M}\circ h)(X)=([t\mto \exp(tX(p))])_{p\in M}=X,
\]
using that $c\colon t\mto \exp(tX(p))$ is the geodesic starting
(for $t=0$) at $p$ with velocity $[t\mto c(t)]=\dot{c}(0)=X(p)$.
Since $h$ is an isomorphism of topological vector space,
it is surjective and hence the injectivity of
\begin{equation}\label{givesgive}
X\mto (\alpha_{\id_M}\circ h)(X)=X
\end{equation}
entails that $\alpha_{\id_M}$ is injective.
Moreover, since $h$ is surjective, we deduce from (\ref{givesgive})
that
\[
V_{\id_M}=\im (\alpha_{\id_M})=\im(\alpha_{\id_M}\circ h)=\cX_c(M)=\Gamma_{\id_M}.
\]
Thus
\[
\alpha_{\id_M}\colon T_{\id_M}\Diff_c(M)\to\Gamma_{\id_M}
\]
is an isomorphism of vector spaces
and $\alpha_{\id_M}=h^{-1}=d\Phi|_{L(\Diff_c(M))}$,
establishing~(d).
Now
\begin{eqnarray*}
V_\phi=\im(\alpha_\phi)&=&\im(\alpha_\phi\circ T_{\id_M}(\rho_\phi))
=\im(r_{\id_M}(\phi)\circ \alpha_{\id_M})\\
&=&\im(R_{\id_M}(\phi)\circ \alpha_{\id_M})=\im(R_{\id_M}(\phi))=\Gamma_\phi.
\end{eqnarray*}
Notably, $\im(\alpha_\phi)=V_\phi\sub \Gamma_\phi$ and so~(a)
holds. Further, $\alpha_\phi\colon T_\phi\Diff_c(M)\to V_\phi$ $=\Gamma_\phi$
is an isomorphism, establishing~(b).
By (\ref{praecc}),
we have~(c).
\end{proof}
\begin{rem}\label{mkeasia}
By Lemma~\ref{pinpointdr}\,(b),
we can identify $T_\phi\Diff_c(M)$
with $\Gamma_\phi$ by means of the isomorphism $\alpha_\phi$
(and we can give $\Gamma_\phi$ the unique locally convex vector topology
making $\alpha_\phi$ an isomorphism of topological vector spaces).
By Lemma~\ref{pinpointdr}\,(c),
the tangent map $T_\phi \rho_\psi$ of right translation with $\psi$ on $\Diff_c(M)$
corresponds to the right translation
$R_\phi(\psi)\colon \Gamma_\phi\to \Gamma_{\phi\circ\psi}$.
\end{rem}
{\bf Proof of Proposition~\ref{handsonEv}.}
Write $M=\bigcup_{j\in \N}K_j$ with compact subsets
$K_j\sub M$ such that $K_j$ is contained in the interior
of $K_{j+1}$, for each $j\in \N$.
We know that the Fr\'{e}chet-Lie group
$\Diff_{K_j}(M)$
is a submanifold of $\Diff_c(M)$
for each $j\in \N$, and
\begin{equation}\label{uiscoreg}
\Diff_c(M)=\bigcup_{j\in \N}\Diff_{K_j}(M).
\end{equation}
If $\eta\in AC_{L^1}([0,1],\Diff_c(M))$
and $\eta'=[\theta]$,
then $\eta([0,1])$ is a compact subset
of $\Diff_c(M)$ and hence
contained in $\Diff_{K_j}(M)$ for some $j\in \N$,
by the compact regularity of the union~(\ref{uiscoreg})
(which follows from \cite[Corollary 3.6]{HMO} and \cite[Remark~5.2]{JFA}).
Hence $\eta\in AC_{L^1}([0,1],\Diff_{K_j}(M))$
(see Lemma~\ref{intsubmf} and Remark~\ref{L1inpart}).
By Lemma~\ref{havedersae}, we find a Borel subset $B\sub [0,1]$
of measure $\lambda_1(B)=0$ such that
$\eta$ is differentiable at each $t\in [0,1]\setminus B$,
with $\dot{\eta}(t)=\theta(t)$.
Assume that $\eta(0)=\id_M$.\\[2.3mm]
If $\eta=\Evol^r([\gamma])$, i.e., $\eta$
is a Carath\'{e}odory solution to
\[
y'(t)=\gamma(t).y(t)
\]
(with multiplication in the tangent Lie group $T\Diff_c(M)$),
then
\[
\dot{\eta}(t)=\gamma(t).\eta(t)
\]
for all $t\in [0,1]\setminus A$, after increasing $A$
if necessary.
Note that we identify $L(\Diff_c(M))$ with $\cX_c(M)$
in Proposition~\ref{handsonEv} by means of $\alpha_{\id_M}$
(with notation as in Lemma~\ref{pinpointdr}.
Making this identification, $\alpha_{\id_M}$ turns
into an identity map.
Now
\begin{equation}\label{bizzzz}
([s\mto\eta(t+s)(p)])_{p\in M}=\alpha_{\eta(t)}(\dot{\eta}(t))
=\alpha_{\id_M}(\gamma(t))\circ\eta(t)=\gamma(t)\circ \eta(t)
\end{equation}
for each $t\in [0,1]\setminus A$,
exploiting Lemma~\ref{pinpointdr}\,(c).
For each $p\in M$, we have $\eta_p:=\ve_p\circ \eta\in AC_{L^1}([0,1],M)$.
Moreover,
$\eta_p'(t)=(\ve_p\circ \eta)'(t)$ exists for all $t\in [0,1]\setminus A$
and is given by
\[
\eta_p'(t)=(\ve_p\circ \eta)'(t)=[s\mto \eta(t+s)(p)]=\gamma(t)(\eta(t)(p))
=\gamma(t)(\eta_p(t))
\]
(see (\ref{bizzzz}).
Hence $\eta_p$ is a Carath\'{e}odory solution to
\[
y'(t)=\gamma(t)(y(t))=f(t,y(t)),\quad y(0)=p.
\]
Thus $f$ admits a global flow for initial time $t_0=0$
and $\eta(t)(p)=\eta_p(t)=\Phi_{t,0}^f(p)$
for all $p\in M$ and $t\in [0,1]$, i.e., $\eta(t)=\Phi^f_{t,t_0}$.\\[2.3mm]
Conversely, assume that $f$ admits a global flow
for initial time $t_0:=0$ and $\eta(t)=\Phi^f_{t,t_0}$
for each $t\in [0,1]$.
As we assume that the finite-dimensional smooth manifold~$M$
is $\sigma$-compact, there is a countable dense subset
$D\sub M$. For each $p\in D$, the curve
\[
\eta_p:=\ve\circ\eta\in AC_{L^1}([0,1],M)
\]
is a Carath\'{e}odory solution to
\[
y'(t)=\gamma(t)(y(t))=f(t,y(t)),\quad y(0)=p.
\]
Hence, there is a Borel set $A_p\sub [0,1]$ with $\lambda_1(D_p)=0$
such that $\eta_p$ is differentiable at each
$t\in [0,1]\setminus A_p$ and
\begin{equation}\label{nochma0}
\eta_p'(t)=f(t,\eta_p(t))=\gamma(t)(\eta_p(t)).
\end{equation}
Now $A:=B\cup\bigcup_{p\in D} A_p$ is a Borel set with $\lambda_1(A)=0$.
For each $t\in [0,1]\setminus A$, we know that $\eta$
is differentiable at~$t$.
Thus
\[
\alpha_{\eta(t)}(\eta'(t))=([t\mto \eta(t)(p)])_{p\in M}\in \Gamma_{\eta(t)}
\]
can be formed and is a smooth (and hence continuous)
function $M\to TM$. Also $\gamma(t)\circ\eta(t)\colon M\to TM$
is continuous. Therefore
\begin{equation}\label{havestb}
\alpha_{\eta(t)}(\dot{\eta}(t))=\gamma(t)\circ \eta(t)
\end{equation}
will hold if we can show that
\[
(\alpha_{\eta(t)}(\eta'(t)))(p)=\gamma(t)(\eta(t)(p))
\]
for all $p\in D$, which can be rewritten as
\begin{equation}\label{nochma}
\dot{\eta}_p(t)=\gamma(\eta_p(t)),
\end{equation}
noting that
$[t\mto \eta(t)(p)]=[t\mto \eta_p(t)]=\dot{\eta}_p(t)$.
Since (\ref{nochma}) holds by (\ref{nochma0}),
we have established~(\ref{havestb}).
Using Lemma~\ref{pinpointdr}\,(c),
we can rewrite (\ref{havestb}) as
\[
\dot{\eta}(t)=\gamma(t).\eta(t).
\]
Thus $\theta(t)=\dot{\eta}(t)=\gamma(t).\eta(t)$ for all $[0,1]\setminus A$,
entailing that $\eta$ is a Carath\'{e}odory solution
to $y'(t)=f(t,y(t))$, $y(0)=\id_M$
and thus $\eta=\Evol^r([\gamma])$.\,\Punkt\vspace{-1.5mm}
{\small
Helge Gl\"{o}ckner, Universit\"{a}t Paderborn, Institut f\"{u}r Mathematik,\\
Warburger Str.\ 100, 33098 Paderborn, Germany.
\,\,Email: {\tt  glockner\at{}math.upb.de}}
\end{document}